%% file: construction-clean.tex
\documentclass[11pt]{amsart}
\usepackage[foot]{amsaddr}

\IfFileExists{revision/revised/preamble.tex}
  {\def\RevisionAssets{revision/revised}}
  {\IfFileExists{revised/preamble.tex}
    {\def\RevisionAssets{revised}}
    {\def\RevisionAssets{.}}}
\input{\RevisionAssets/preamble}

\usepackage{amssymb}
\usepackage{tikz}
\usepackage{amsmath}
\usepackage{latexsym}
\usepackage{mathrsfs}
\usepackage{amsthm}
\usepackage{verbatim}
\usepackage{graphicx}
\usepackage{epstopdf}

\usepackage{enumitem}
\newcounter{listcounter}

\let\oldcaption\caption
\renewcommand{\caption}[1]{\oldcaption{#1}}

\usetikzlibrary{quotes}
\usepackage{soul}
\newcommand\FIG[2]{#1#2}
\newcommand\TBL[2]{#1#2}
\DeclareMathOperator{\supp}{supp} 

\title{Finite element form-valued forms: Construction}

\author{Kaibo Hu}
\address{Mathematical Institute, University of Oxford,
Radcliffe Observatory, Andrew Wiles Building,
Woodstock Rd, Oxford OX2 6GG, UK.}
\email{kaibo.hu@maths.ox.ac.uk}
\author{Ting Lin}
\address{School of Mathematics, Peking University, Beijing 100871, P.R.China}
\email{lintingsms@pku.edu.cn}
\thanks{The work of KH was supported by a Royal Society University Research Fellowship (URF$\backslash$R1$\backslash$221398) and an ERC Starting Grant (project 101164551, GeoFEM). The work of TL was supported by NSFC Project No. 123B2014.}

\DeclareMathOperator{\cS}{\mathcal S}

\DeclareMathOperator{\cP}{\mathcal P}

\DeclareMathOperator{\codim}{codim}

\begin{document}

\maketitle

\begin{abstract}
We provide a finite element discretization of $\ell$-form-valued $k$-forms on triangulation in $\mathbb{R}^{n}$ for general $k$, $\ell$ and $n$ and any polynomial degree. The construction generalizes finite element Whitney forms for the de~Rham complex and their high-order and distributional versions,  the Regge finite elements and the Christiansen--Regge elasticity complex, the TDNNS element for symmetric stress tensors, the MCS element for traceless matrix fields, the Hellan--Herrmann--Johnson (HHJ) elements for biharmonic equations, and discrete divdiv and Hessian complexes in [Hu, Lin, and Zhang,  2025]. The construction discretizes the Bernstein--Gelfand--Gelfand (BGG) diagrams. Applications of the construction include discretization of strain and stress tensors in continuum mechanics and metric and curvature tensors in differential geometry in any dimension. 
\end{abstract}



\tableofcontents

\section{Introduction}

Constructing finite element spaces (and more general discrete patterns) that encode the differential structures of continuous problems has drawn growing attention in recent decades. For solving PDEs and simulating physical systems, preserving the de~Rham complex (and its cohomology) provides stability, convergence, and structure-preserving properties. This viewpoint has become central in the area of Finite Element Exterior Calculus (FEEC) \cite{arnold2018finite,arnold2006finite,arnold2010finite}. Classical finite elements for the de~Rham complex, such as N\'ed\'elec and Raviart--Thomas spaces \cite{nedelec1980mixed,raviart2006mixed}, can be unified through the notion of Whitney forms and their high-order extensions \cite{hiptmair1999canonical,hiptmair2001higher,bossavit1988whitney,whitney2012geometric}. These elements have a canonical form: in the lowest order case, $k$-forms are discretized on $k$-cells. These elements and their associated numerical schemes form the standard toolkit in computational electromagnetism and other $\curl$--$\div$ problems (see, e.g., recent quantum computing hardware simulations and geophysics applications \cite{amazon-code,adams2019lfric,melvin2019mixed,rochlitz2019custem}). Moreover, discrete topology and discrete differential forms play a crucial role in computer graphics \cite{Wang:2023:ECIG} and topological data analysis \cite{lim2020hodge}.

\begin{center}
\includegraphics[width=4in]{\RevisionAssets/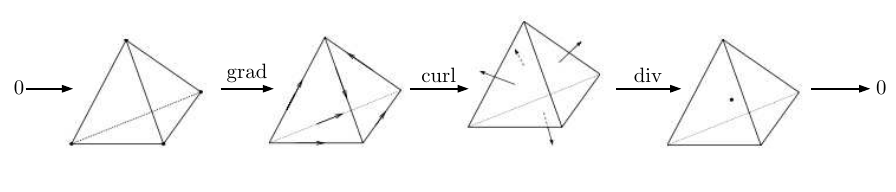} 
\end{center}

A wide range of problems involve tensors with more general symmetries (differential forms being tensors with full skew-symmetry) and more elaborate differential structures than the $\grad$--$\curl$--$\div$ operators in the de~Rham setting. For instance, elasticity typically introduces \emph{symmetric} $(0,2)$-tensors as strain and stress, while in differential geometry, the metric is a symmetric $(0,2)$-tensor and the Riemannian curvature, interpreted as a $(0,4)$-tensor, obeys multiple symmetries (skew-symmetry in the first two and the last two indices, symmetry between those two groups, plus the algebraic Bianchi identity). Related constructions (Ricci, Einstein, Weyl tensors, etc.) arise in general relativity, continuum defects, network theories, and beyond.  These tensors have various kinds of symmetries. A general approach for expressing tensor symmetries is by Young tableaux and representation theory \cite{vcap2022bgg,olverdifferential,young-tab-elements}. Inspired by the canonical form and wide applications of discrete or finite element differential forms on triangulation, a natural question is

\newquote{\label{question1}
{\it Are there discrete analogues of such tensors with symmetries and differential structures?}
}
For these tensorial objects, the Bernstein--Gelfand--Gelfand (BGG) sequences play a role analogous to that of the de~Rham complex for differential forms. Originally studied in algebraic geometry and representation theory \cite{bernstein1975differential,vcap2001bernstein,eastwood2000complex}, BGG sequences have recently been brought into functional analysis and numerical analysis \cite{arnold2006defferential,arnold2021complexes,vcap2022bgg,arnold2006finite}. Corresponding finite element discretizations have been explored in various works \cite{chen2022complexes,chen2022finiteelasticity,chen2022finitedivdiv,chen2022finite,chen2022finite2D,hu2022conforming,hu2021conforming,gong2023discrete,bonizzoni2023discrete,christiansen2023finite,christiansen2024discrete,hu2024finite,lin2025mixedfiniteelementmethod}, mostly focusing on {conforming} elements (piecewise polynomials with certain high intercell continuity). Except for one approach on cubical meshes using tensor product structures \cite{bonizzoni2023discrete}, these constructions are either dimension-specific or restricted to particular (the last several) slots in a complex. No systematic approach exists to cover all form indices in arbitrary dimensions. One possible reason is that extending conforming elements to arbitrary dimensions requires extremely high regularity on (sub)simplices \cite{hu2024construction,hu2025sharpness}. More importantly, while Whitney forms for the de~Rham complex exhibit a clear topological structure, such structures have yet to be fully discovered for tensors, either generally or more specifically in BGG-type constructions \cite{arnold2021complexes}.

The work in this paper aims to answer a more specific version of \eqref{question1}: 

\newquote{\label{question2}
\emph{\it Are there canonical finite elements for form-valued forms and BGG complexes on simplicial triangulation?} 
}

In other words, we aim to design finite elements that reflect the same differential and cohomological properties as their continuous counterparts, while also demonstrating discrete topological/geometry structures comparable to Whitney forms in the de~Rham context. Requiring these properties is not only mathematically appealing but also crucial for robust numerical solutions for tensor-valued problems, problems involving intrinsic geometry (e.g., shells, continuum defects, numerical relativity), and problems for discrete structures (e.g., networks, graphs, \cite{lim2020hodge}).

 In the same way that finite element differential forms were built based on works by {N\'ed\'elec} \cite{nedelec1980mixed}, Raviart--Thomas \cite{raviart2006mixed}, Brezzi--Douglas--Marini \cite{brezzi1985two}, etc., many building blocks are also available for  form-valued forms. {We hereby collect these building blocks as puzzles to the unified construction.} The Regge element \cite{li2018regge,christiansen2011linearization,regge1961general,williams1992regge}  represents a piecewise-flat metric using the length of edges. This leads to a distributional curvature on hinges, matching the angle-deficit interpretation of curvature in Regge geometry. In the lowest order case \cite{christiansen2011linearization}, the degrees of freedom are the evaluation of the tangential-tangential components on each edge, see the second slot in \Cref{fig:regge-complex}. 

\begin{figure}[h]
\FIG{
\includegraphics[width = 0.9\linewidth]{\RevisionAssets/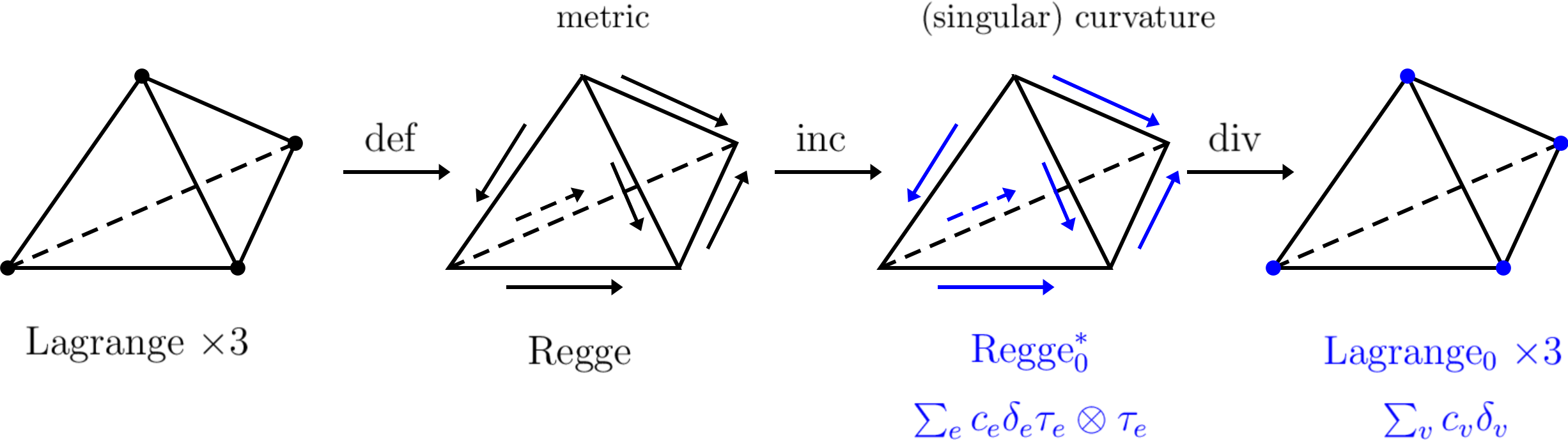}} 
{\caption{Christiansen-Regge complex: a discrete elasticity complex using Regge elements in three dimensions. The first space and the second space are the vector Lagrange and Regge finite elements, respectively. The third space and the fourth space are the dual of the Regge and vector Lagrange finite elements with vanishing boundary degrees of freedom, respectively.}\label{fig:regge-complex}}
\end{figure}

Independently, Sch\"oberl and collaborators developed {distributional} finite elements for equilibrated error estimators \cite{braess2008equilibrated} and for continuum mechanics,
 giving rise to the TDNNS method for elasticity \cite{pechstein2011tangential} and the MCS method \cite{gopalakrishnan2020mass} for fluids. Recently, a $\sym\curl$-nonconforming space MCS$^{\top}$ with a Koszul construction was presented in \cite{hu2025distributional}. 
 Moreover, the classical work of the Hellan--Herrmann--Johnson (HHJ) element \cite{hellan1967analysis,herrmann1967finite,johnson1973convergence} for biharmonic plate problems can be also interpreted in this spirit \cite{neunteufel2023hellan}. 
 These methods incorporate distributional derivatives and certain vector or matrix versions of Dirac measures.

These building blocks give an answer to \eqref{question2} in three dimensions for some specific form indices. The pattern also provides clues for more general constructions in this paper, leading to new intrinsic finite element  form-valued forms.
On the continuous level, 0-form-valued and $n$-form-valued de~Rham complexes can be interpreted as special (trivial) cases of the BGG complexes and fit in the same diagram, where the former is just the de~Rham complex and the latter can be identified with a de~Rham complex if a volume element is fixed (see \Cref{fig:form-form}). On the discrete level, the Whitney forms for the de~Rham complex \cite{bossavit1988whitney,hiptmair1999canonical,arnold2006finite,hiptmair2001higher} and the dual Whitney forms \cite{braess2008equilibrated} turn out to be the canonical discretization of the de~Rham complexes. This can be also generalized to higher polynomial degrees, see \cite{arnold2009geometric,hiptmair1999canonical,arnold2006finite,hiptmair2001higher} for Whitney forms and \cite{licht2017complexes} for distributional complexes. 

In two and three dimensions, some complexes involving tensor elements have been established. For example, for the elasticity complex, one may see the Christiansen--Regge complex \cite{christiansen2011linearization} as the {canonical} discretization. The canonicality can be seen from several aspects: (1) the degrees of freedom have canonical forms, (2) the spaces and operators have discrete geometric interpretations as discrete metric and curvature, (3) the discrete complex preserves the formal self-adjointness of the continuous version. The cohomology of the Christiansen--Regge complex and extensions to twisted complexes can be found in \cite{christiansen2023extended}. New finite elements and distributional spaces were needed to derive the Hessian and divdiv complexes in three dimensions \cite{hu2025distributional}. The Hessian complex starts with a Lagrange element, followed by Dirac measures. The divdiv complex is the formal adjoint of the Hessian complex, which contains the MCS$^{\top}$ element and the HHJ element before the divdiv operator. Moreover, \cite{hu2025distributional} used a diagram chase approach to establish the cohomology of the discrete complexes. The discrete complexes discussed above are summarized in \Cref{fig:table}.
\begin{figure}

\begin{center}
$
\begin{array}{c}
\Alt^{-1, -1} = 0\\
\\
\\
\end{array}
\begin{array}{c}
\tikz[baseline=(s.base)]{
  \node (s) at (0.,0.) {};
  \node (t) at (1,0) {};
  \draw[red, dashed, ->] (s) -- (t);
}\vspace{-0.3cm}\\\vspace{-0.2cm}
\tikz[baseline=(s.base)]{
  \node (s) at (0.,0.3) {};
  \node (t) at (-1,-0.3) {};
  \draw[red, dashed, ->] (s) -- (t);
}\vspace{-0.2cm}\\
\tikz[baseline=(s.base)]{
  \node (s) at (0.,0.) {};
  \node (t) at (1,0) {};
  \draw[red,dashed, ->] (s) -- (t);
}
\end{array}
\begin{array}{c}
\\
\\
\Alt^{0, 0} ~~\longrightarrow~~ \Alt^{1, 0}~~\longrightarrow~~ \Alt^{2, 0}~~\longrightarrow~~ \Alt^{3, 0}
\end{array}
$

$\begin{array}{c}
\Alt^{0, 0} \\
\\
\\
\end{array}$ $\color{red} \begin{array}{c}
\longrightarrow \\
  \swarrow\\
\longrightarrow 
\end{array}$ $\begin{array}{c}
\\
\\
\Alt^{1, 1} 
\end{array}$ $\begin{array}{c}
\\
\\
\longrightarrow 
\end{array}$  $\begin{array}{c}
\\
\\
\Alt^{2, 1} 
\end{array}$ {\vspace{+0.2cm}$
\begin{array}{c}
 \\
 \\
\longrightarrow 
\end{array}
$}    $\begin{array}{c}
 \\
\\
\Alt^{3, 1}
\end{array}$
\\
$\begin{array}{c}
\Alt^{0, 1} \\
\\
\\
\end{array}$ 
 $\begin{array}{c}
\longrightarrow \\
\\
\\
\end{array}$
 $\begin{array}{c}
\Alt^{1, 1} \\
\\
\\
\end{array}$
 $\color{red} \begin{array}{c}
\longrightarrow \\
  \swarrow\\
\longrightarrow 
\end{array}$  $\begin{array}{c}
\\
\\
\Alt^{2, 2} 
\end{array}$ $\begin{array}{c}
\\
\\
\longrightarrow 
\end{array}$  $\begin{array}{c}
\\
\\
\Alt^{3, 2} 
\end{array}$\\
$\begin{array}{c}
\Alt^{0, 2} \\
\\
\\
\end{array}$ 
$\begin{array}{c}
\longrightarrow \\
\\
\\
\end{array}$
$\begin{array}{c}
\Alt^{1, 2} \\
\\
\\
\end{array}$
$\begin{array}{c}
\longrightarrow \\
\\
\\
\end{array}$
 $\begin{array}{c}
\Alt^{2, 2} \\
\\
\\
\end{array}$ 
$\color{red} \begin{array}{c}
\longrightarrow \\
  \swarrow\\
\longrightarrow 
\end{array}$ 
  $\begin{array}{c}
 \\
\\
\Alt^{3, 3}
\end{array}$\\
$
\begin{array}{c}
\Alt^{0, 3} ~~\longrightarrow~~ \Alt^{1, 3}~~\longrightarrow~~ \Alt^{2, 3}~~\longrightarrow~~ \Alt^{3, 3}\\
~\\~\\
\end{array}
\begin{array}{c}
\tikz[baseline=(s.base)]{
  \node (s) at (0.,0.) {};
  \node (t) at (1,0) {};
  \draw[red, dashed, ->] (s) -- (t);
}\vspace{-0.3cm}\\\vspace{-0.2cm}
\tikz[baseline=(s.base)]{
  \node (s) at (0.,0.3) {};
  \node (t) at (-1,-0.3) {};
  \draw[red,dashed, ->] (s) -- (t);
}\vspace{-0.2cm}\\
\tikz[baseline=(s.base)]{
  \node (s) at (0.,0.) {};
  \node (t) at (1,0) {};
  \draw[red, dashed, ->] (s) -- (t);
}
\end{array}
 \begin{array}{c}
 \\
\\
\Alt^{4, 4} = 0
\end{array}
$
\end{center}

	 \caption{De Rham and BGG complexes with form-valued forms in $\mathbb{R}^{3}$ on the continuous level in a unified language. Here $\Alt^{k, \ell}$ stands for $\ell$-form-valued $k$-forms with certain symmetries. The complexes are the  de Rham complex, the Hessian complex, the elasticity complex, the divdiv complex, and the 3-form-valued de  Rham complex, respectively. The spaces $\Alt^{-1, -1}$ and $\Alt^{4, 4}$ are {zero} in three dimensions.} 	  \label{fig:form-form}
\end{figure}
\begin{figure}
\flushleft
 {\tikzset{every node/.append style={text=red}}
\begin{tikzpicture}

    \node (deRham3D1) at (-2,0) {\includegraphics[width=0.8\textwidth]{\RevisionAssets/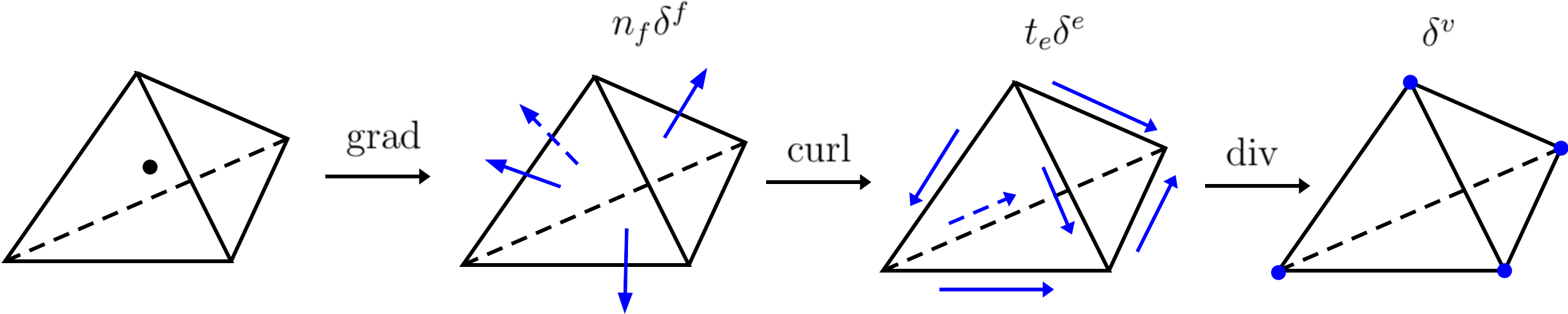}};
  \node at (-5,0.7) {$\Lambda^{0,0}$};   \node at (-2.4,0.7) {$\Lambda^{1,0}$};   \node at (0.5,0.7) {$\Lambda^{2,0}$};   \node at (3.2,0.7) {$\Lambda^{3,0}$};     \node at (5.4,0.) { \small distributional de Rham};         \node at (5.4,-0.4) { \small Braess-Sch\"oberl 2008};                 \node at (5.4,-0.8) { currents};
                                \node at (1.7,1.8) { \textbf{distributions, Dirac measures}}; 
  \node (hessian) at (-2,-2.5) {\includegraphics[width=0.8\textwidth]{\RevisionAssets/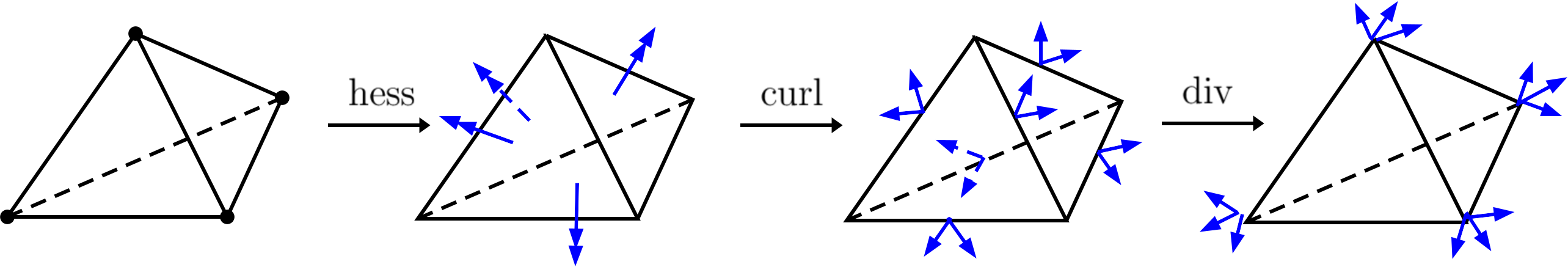}};
  \node at (-5,-1.7) {$\Lambda^{0,0}$};   \node at (-2.4,-1.7) {$\Lambda^{1,1}$};   \node at (0.5,-1.7) {$\Lambda^{2,1}$};   \node at (3.2,-1.7) {$\Lambda^{3,1}$};     \node at (5.4,-2.3) {Hu-Lin-Zhang 2025};     
    \node (regge0) at (-2,-5) {\includegraphics[width=0.8\textwidth]{\RevisionAssets/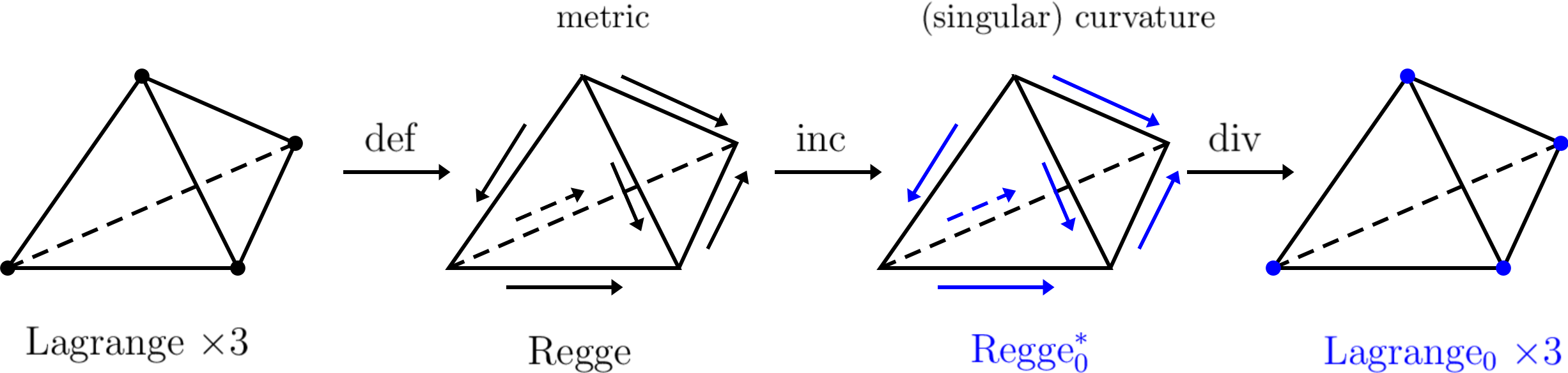}};
  \node at (-5,-4.2) {$\Lambda^{0,1}$};   \node at (-2.4,-4.2) {$\Lambda^{1,1}$};   \node at (0.5,-4.2) {$\Lambda^{2,2}$};   \node at (3.2,-4.2) {$\Lambda^{3,2}$};       \node at (5.4,-4.6) {Regge, Christiansen};             \node at (5.4,-5.) {1961, 2011};                         \node at (5.4,-5.4) {metric, curvature}; \node at (-5,-12) { \textbf{finite elements, p.w. polynomials}};
    \node (divdiv) at (-2,-7.5) {\includegraphics[width=0.8\textwidth]{\RevisionAssets/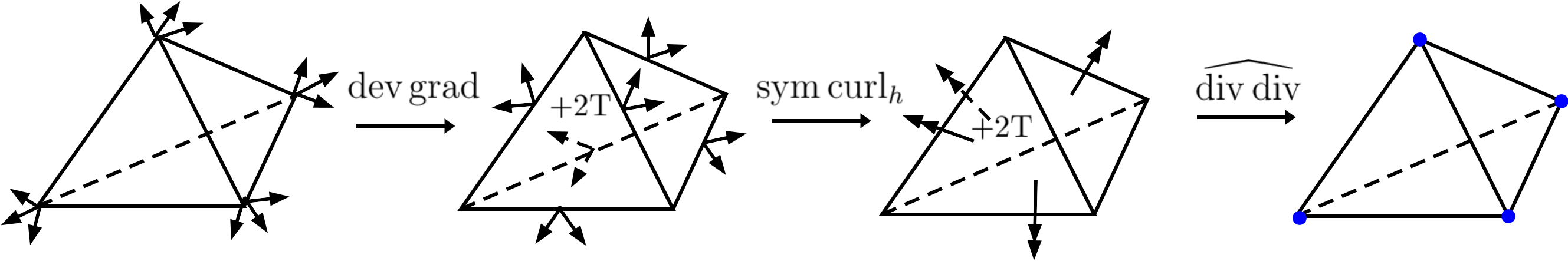}};
  \node at (-5,-6.7) {$\Lambda^{0,2}$};   \node at (-2.4,-6.7) {$\Lambda^{1,2}$};   \node at (0.5,-6.7) {$\Lambda^{2,2}$};   \node at (3.2,-6.7) {$\Lambda^{3,3}$};         \node at (5.6,-7.2) { \small Hellan, Herrmann, Johnson };                 \node at (5.4,-7.6) {\small Pechstein-Sch\"oberl 2011};             \node at (5.4,-8) { \small Hu-Lin-Zhang 2025};   \node (deRham3D) at (-2,-10) {\includegraphics[width=0.8\textwidth]{\RevisionAssets/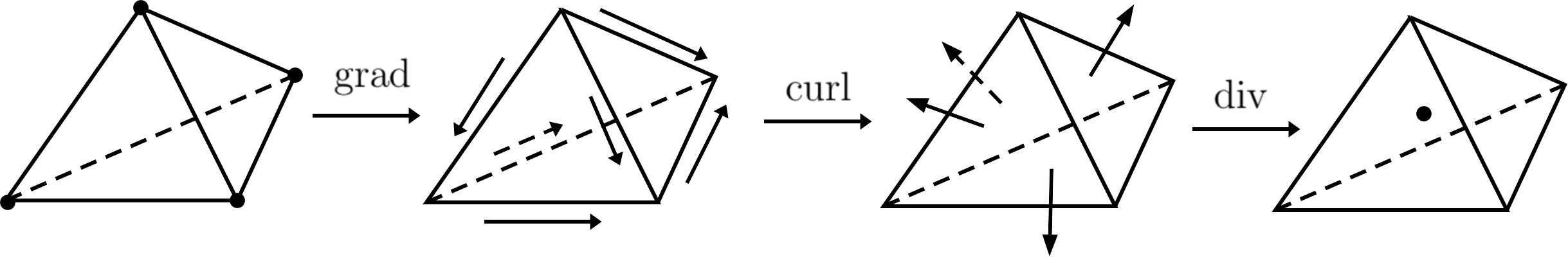}};
  \node at (-5,-9.2) {$\Alt^{0,3}$};   \node at (-2.4,-9.2) {$\Alt^{1,3}$};   \node at (0.5,-9.2) {$\Alt^{2,3}$};   \node at (3.2,-9.2) {$\Alt^{3,3}$};           \node at (5.8,-9.5) {\small N\'ed\'elec, Raviart--Thomas };                 \node at (5.4,-9.9) {\small 1980, 1977};             \node at (5.4,-10.3) {\small Whitney forms 1957};                         \node at (5.4,-10.7) { \small Bossavit, Hiptmair}; 
                        \node at (5.4, -11.1) { \small  Arnold-Falk-Winther};
  \draw[line width=2pt, red!30, rounded corners=5pt,dotted] (-7.8,-1.4) -- (-7.8,-11) -- (3.8,-11) -- (3.8, -9.7) -- (-6.6, -1.4) -- cycle;
    \draw[line width=2pt, blue!30, rounded corners=5pt, dotted] (-7.8,1.2) -- (3.8,1.2) -- (3.8,-8.4) -- (2.6, -8.4) -- (-7.8, -0.1) -- cycle;
\end{tikzpicture}}

\caption{Finite elements (including currents) of the lowest order for  \Cref{fig:form-form}. The first row is the distributional de Rham complex (dual Whitney forms) \cite{braess2008equilibrated,licht2017complexes}; the last row consists of Whitney forms; the elasticity complex is discretized by the Christiansen--Regge complex \cite{christiansen2011linearization}; the Hessian and divdiv complexes are due to \cite{hu2025distributional}.}
 \label{fig:table}
\end{figure}

In this paper, we identify the patterns in \Cref{fig:table} and extend them to \emph{any dimension, any form-valued form, and any polynomial degree}.  Moreover, we discretize the iterated BGG constructions (leading to the $\grad\curl$ complex, the $\curl\div$ complex, and the $\grad\div$ complex in three dimensions, see \Cref{fig:iterative-bgg}). The complexes shown in \Cref{fig:iterative-bgg-fe} glue together Whitney forms and the MCS element with high-order differentials.
  \begin{figure}[htbp]\FIG{\includegraphics[width=0.8\linewidth]{\RevisionAssets/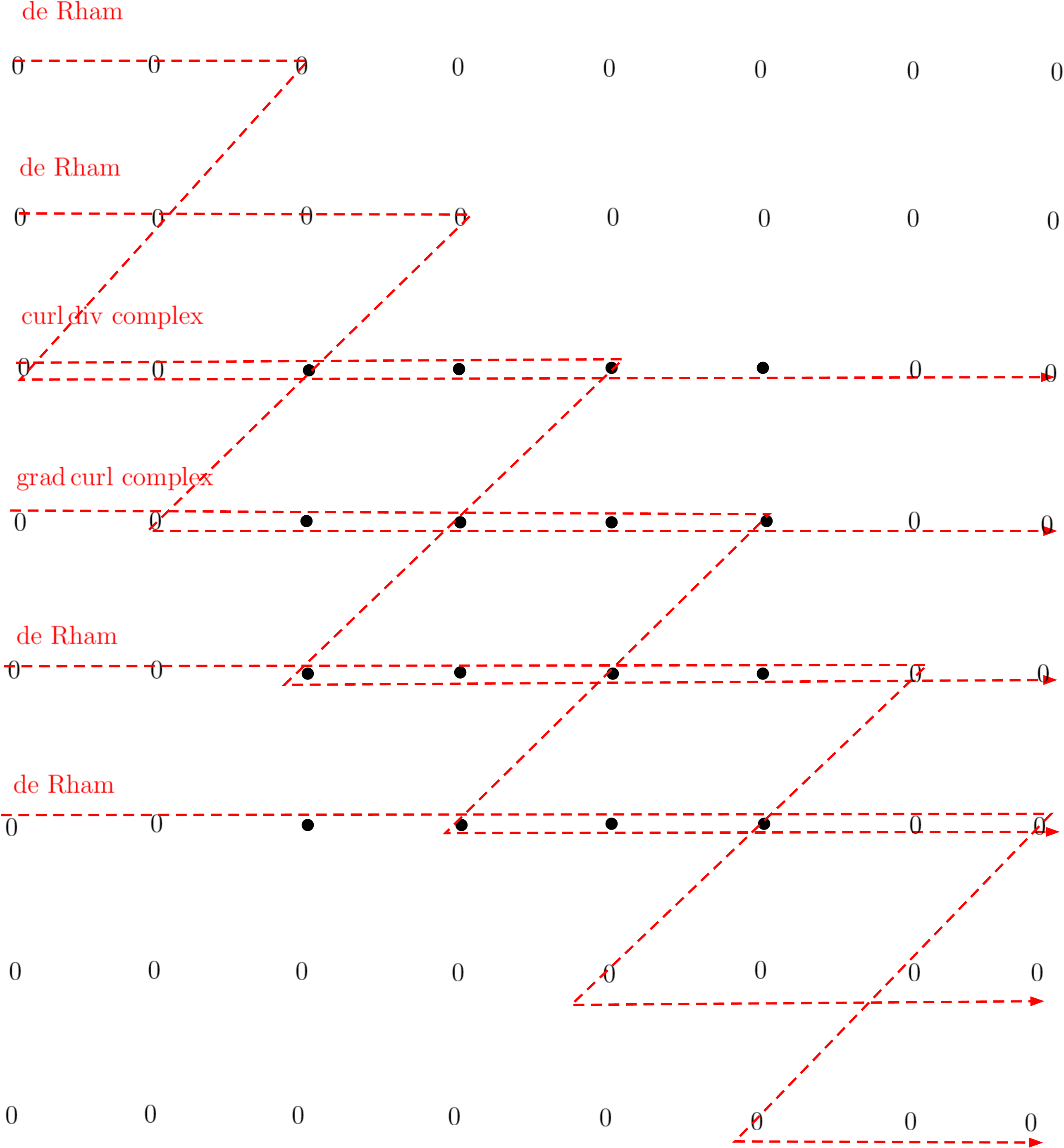}}
{\caption{Iterated BGG constructions are shown in red dotted line. The thick dots denote spaces $C^{\infty}\otimes \Alt^{k, \ell}$ with $0\leq k, \ell\leq 3$, {which are nonzero in three dimensions}. The diagram is extended by zero.}
\label{fig:iterative-bgg}}
\end{figure}
  \begin{figure}[htbp]\FIG{\includegraphics[width=0.8\linewidth]{\RevisionAssets/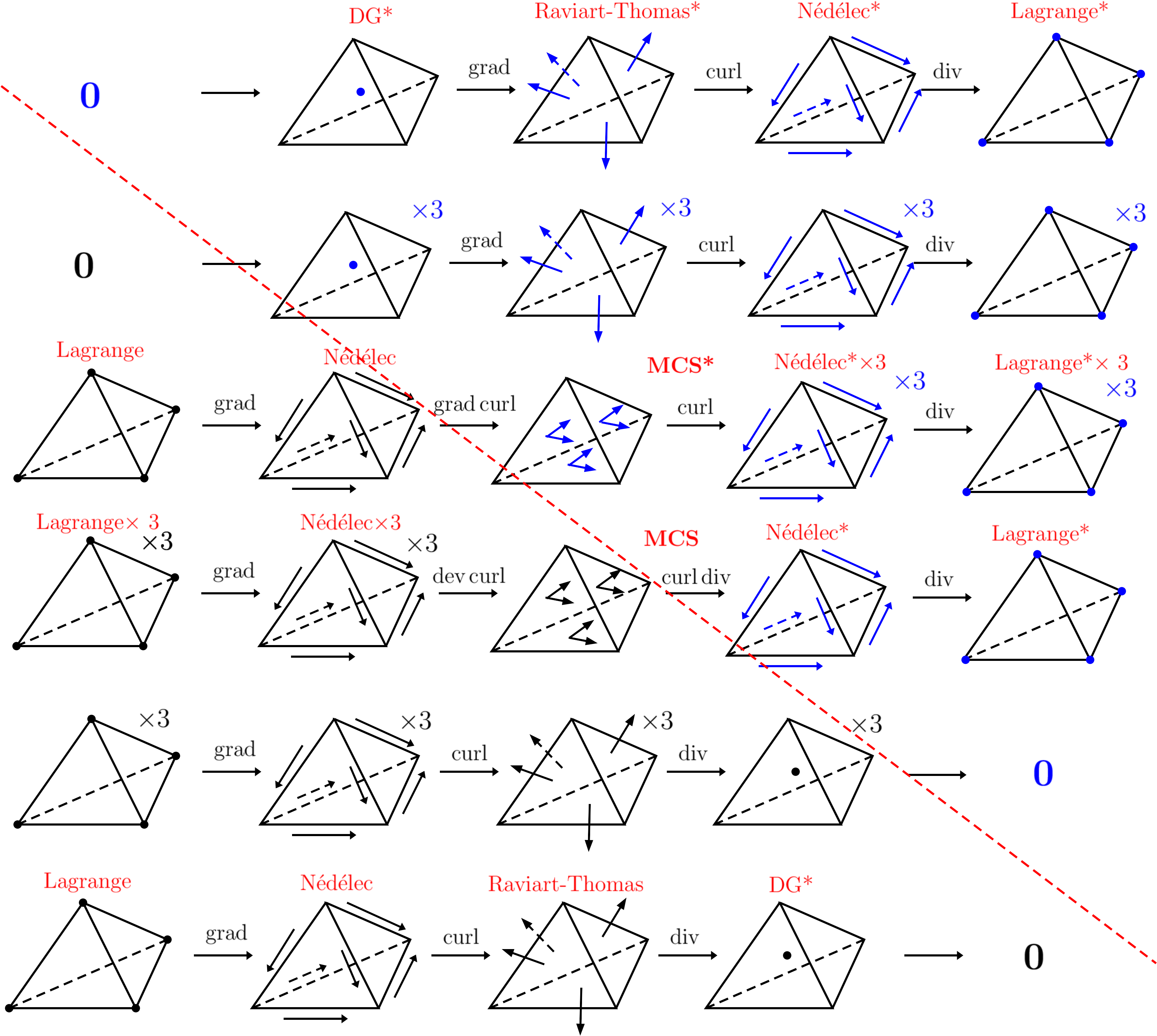}}
{\caption{Distributional finite element complexes for the iterated BGG constructions. The complex of the MCS element (the fourth row in the diagram) and its dual (the third row in the diagram) are new, to the best of our knowledge. The complex of the MCS element also fits in and can be derived from a BGG diagram; see \Cref{fig:MCS-BGG}.}
\label{fig:iterative-bgg-fe}}
\end{figure}

  \begin{figure}[htbp]\FIG{\includegraphics[width=0.8\linewidth]{\RevisionAssets/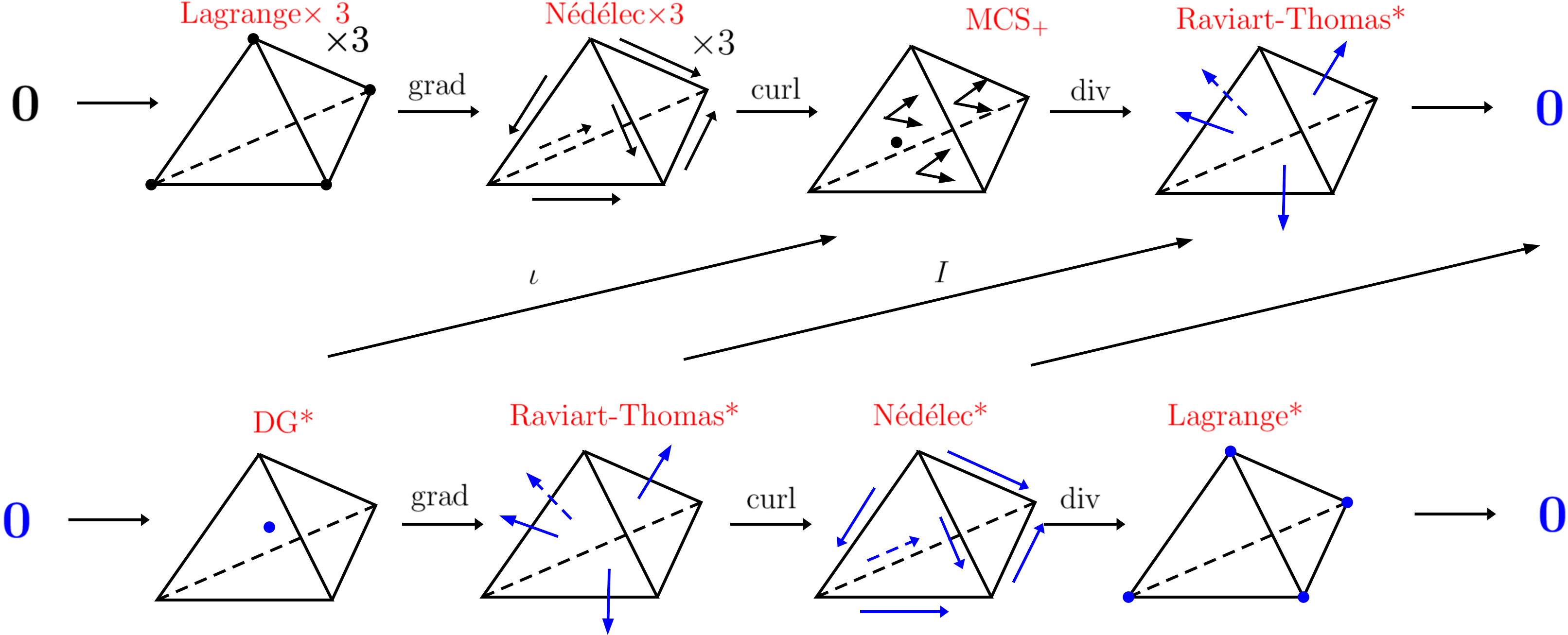}}
{\caption{BGG diagram for the MCS element and the MCS complex (the fourth row of \Cref{fig:iterative-bgg-fe}). Here $\mathrm{MCS}_{+}$ denotes the ``full matrix MCS element''. The full diagram can be useful for discretizing more complete models encoded in the twisted complex \cite{vcap2022bgg}.}
\label{fig:MCS-BGG}}
\end{figure}

We show the unisolvency of the resulting finite element spaces. This paper leaves the complex and cohomological issues open, i.e., in this paper, we do not prove that the resulting spaces fit in a complex and their cohomology is isomorphic to the continuous versions (although they do in three and lower dimensions). This is because some of the differential operators have to be interpreted discretely, and a full explanation is beyond the scope of this paper. However, we provide a dimension count as a strong indication that such results will hold in any dimension. 

Before diving into the details of the construction, we mention motivations for investigating a general construction in arbitrary dimensions.
\begin{itemize}
\item The canonical patterns in arbitrary dimensions for general form-valued forms can be identified, which also shed light on the constructions and applications in three dimensions, e.g. HHJ and MCS elements.
\item Important problems from differential geometry and general relativity require discretizing tensor fields (such as the metric and various notions of curvature tensors) in four and even higher dimensions.  For example, the Einstein equation describes a relation between the Einstein tensor of the four-dimensional spacetime and the matter field.  
\item The twisted complexes \cite{arnold2021complexes,vcap2022bgg}, which involve all the spaces in the BGG diagram, play a fundamental role in their own right. The twisted complexes incorporate richer physics and geometry. For example, the twisted complex characterizes micropolar and Cosserat models while the BGG complex describes the standard elasticity \cite{vcap2022bgg}; the twisted complex involves Riemann-Cartan geometry with curvature and torsion, while the torsion is eliminated in the BGG complex \cite{christiansen2023extended}.  Applications require discretizing the twisted complexes \cite{dziubek2024intrinsic}. The general construction in this paper discretizes the entire diagram, and therefore sheds light on discretizing generalized models in micropolar continuum, Riemann-Cartan geometry, and continuum defects \cite{yavari2012riemann,yavari2013riemann,yavari2012weyl} etc. 
\item Cliques (analogues of simplices) of any dimensions exist on graphs or hypergraphs \cite{bick2023higher}. A simplicial construction with full generality sheds light on investigating objects and applications from graph and network theory, such as the notion of graph curvature \cite{joharinad2023mathematical,lin2011ricci} {and torsion}. 
\end{itemize}
Concerning the last point, Hodge-Laplacian and discrete differential forms can be established on graphs \cite{lim2020hodge,he2020graph,ribando2024combinatorial,grigor2024eigenvalues}. The theory has a close relation with the lowest order Whitney forms as they share the same degrees of freedom. To carry tensor finite elements to other discrete structures such as graphs, one desires intrinsic finite elements with canonical degrees of freedom, leading to discrete geometric and topological interpretations. This is another reason for the preference of a construction mimicking the Whitney forms with relaxed conformity (for the de~Rham complex, the Whitney forms happen to have enough conformity for $L^{2}$ spaces with exterior derivatives in $L^{2}$; however, {for BGG complexes we need to sacrifice some conformity to obtain the geometric structure and topological interpretations}). 

\medskip

The systematic finite element discretization of form-valued forms presented in this paper has a number of corollaries. 

\paragraph{\it Finite element spaces with constant shape functions.} In our motivating examples, the Regge elements ((1,1)-forms), the HHJ elements ((2,2)-forms), and the MCS elements ((2,1)-forms) are all piecewise constant. These finite element spaces demonstrate an elegant match of the simplest possible shape functions and degrees of freedom. It is natural to ask: 

\newquote{\label{question:constant}
\emph{\it What piecewise constant finite element spaces can we construct for form-valued forms?} 
}

A similar question is posed in \cite{berchenkodraft}. The tools developed in this paper allow us to provide an answer to this question in our framework: \emph{There exists an $\iota^{\ast}\iota^{\ast}$-conforming finite element space with constant $\mathbb W^{k,\ell}_{[p]}$ as the shape function, for any $k$, $\ell$, and $p$, provided  $p\leq k\leq\ell\leq n-p$}. See \Cref{thm:constant-construction}.

\smallskip

\paragraph{\it Tensor decomposition in differential geometry.} The second remark concerns tensor decompositions in differential geometry. Form-valued forms have various connections to differential geometry, encoding important algebraic and differential structures therein. In particular, the $(2, 2)$-forms $\Alt^{2, 2}$ are closely related to the Riemannian tensor. Using our framework, the symmetries can be imposed from the BGG diagrams to obtain $\W^{2, 2}$, which has the algebraic symmetries of the Riemannian tensor in any dimension. An orthogonal decomposition of $\Alt^{2, 2}$ in differential geometry can be interpreted using BGG diagrams ({especially, the connecting mappings,} see \eqref{alt22} below). Since we discretize the BGG diagrams, we obtain a discrete version of this decomposition with finite elements.  Moreover, these connections open a door for exploring further connections between finite elements and differential geometry, such as {conformal metrics and} Weyl tensor in higher dimensions; see \Cref{rmk:weyl}.

\subsection{Overview of the construction}
\label{sec:overview}
Each BGG complex involves a ``zig-zag'' at some slot, connecting two rows of the diagram. 
From the examples in \Cref{fig:table} (see also \Cref{fig:form-form}), we see that each BGG complex consists of finite element spaces (piecewise polynomials) before the zig-zag, and then Dirac measures of certain types (generalized currents) after it. The sequence of Whitney forms and its dual are two special border cases: for the former, all the spaces are finite elements in the classical sense (piecewise polynomial functions) as appeared in FEEC \cite{arnold2006finite,arnold2010finite,arnold2018finite,hiptmair1999canonical} and the periodic table of finite elements \cite{arnold2014periodic}; for the latter, all the spaces are currents (distributions) \cite{braess2008equilibrated}. To generalize this pattern in the general construction, each sequence is also split into two parts: first the finite elements and then the generalized currents. The construction of currents is relatively straightforward, as we can extend the sequences via derivatives. However, constructing the finite element spaces calls for special care in choosing local shape functions and degrees of freedom that match each other (unisolvency) and yield the desired interelement continuity.

\smallskip
{\noindent \bf Generalized trace operators.}
For a finite element space, specifying  the interelement continuity, i.e., the conformity (and hence the degrees of freedom) is essential. For the Whitney forms, the conformity condition demands that the \emph{trace} (denoted as $\iota^*$ in this paper; see \eqref{def:trace} for the definition) of a differential form from both sides of a face is single-valued on that face. Correspondingly, the degrees of freedom for Whitney forms can be given by moments of this trace over subcells. The first challenge for form-valued forms is to generalize the notion of the trace. A straightforward approach is to project each vector onto the face’s tangent space (see $\iota^{\ast}\iota^{\ast}$ below). However, this is not necessarily what we need. For example, consider the first space in the elasticity complex, which is in $\alt^{0,1}$ (1-form-valued 0-forms). The $\iota^{\ast}\iota^{\ast}$-trace vanishes at vertices. Yet the canonical Christiansen--Regge complex starts with a Lagrange space, requiring vertex evaluations.

To resolve this, we introduce {generalized} trace operators. In particular, we allow evaluating a $k$-form on an $m$-dimensional cell where $m<k$ (via the operator $\jmath^{*}$ \eqref{def:jstar}). The idea is to \emph{use tangent vectors as much as possible}. For example, to evaluate a 3-form on a 1-dimensional cell, we feed its single tangent vector plus two vectors normal to the cell into the 3-form. This definition sits between the classical trace (which feeds only tangent vectors) and the restriction operator (which can feed any vectors of the ambient space). 

For iterated BGG complexes (see \Cref{sec:iterated} for the details), we must generalize further, leading to $\jmath^{*}_{[p]}$ (the above case corresponds to $p=1$). Increasing $p$ moves the definition closer to the restriction operator, allowing $p-1$ tangent vectors to remain unused. In the example of evaluating a 3-form on a one-dimensional cell, $\jmath^{*}_{[2]}$ permits either tangent or normal vectors. On a one-dimensional cell, this reduces effectively to the restriction operator. On two-dimensional cells, for $p=2$, one must feed at least one tangent vector to the 3-form, while the remaining two slots can be tangent or normal; for $p=3$, $\jmath^{*}_{[p]}$ boils down to the restriction. The notation of $\jmath^{\ast}_{[p]}$ has not appeared in existing literature on finite elements in three dimensions since for $p > 1$,  $\jmath_{[p]}^*$ boils down to either $\rho^*$ or $\iota^*$.
The first non-trivial examples {where one sees differences between $\jmath^{\ast}_{[p]}$-conforming finite elements with $p>1$ and $\rho^{\ast}$-conforming or $\iota^{\ast}$-conforming ones} appear in four dimensions. The definition of the generalized traces and their properties are discussed in \Cref{sec:trace}.

These generalized traces recover existing elements, such as TDNNS, MCS, Regge, and Hu--Lin--Zhang in three dimensions, and facilitate new constructions in higher dimensions. However, a direct generalization to higher dimensions is challenging, as vector proxies become more complicated — beyond the tangential and normal components of vectors used in two and three dimensions — and are far less explored in the literature.

In this paper, we investigate finite element spaces for form-valued forms $\Alt^{k, \ell}$ and its variants $\mathbb W^{k,\ell}$ with symmetries encoded in the BGG diagram. To define generalized traces for these spaces, we must address the continuity conditions for the two indices separately. If we represent $\Alt^{k, \ell}$ as a matrix as we usually do in two and three dimensions, this entails applying a row-wise operator to one index and a column-wise operator to the other. The conformity of finite element spaces is characterized by the single-valuedness of these generalized trace operators (rather than belonging to certain Sobolev spaces, which is left as a further question). In principle, we can consider all combinations of the operators $\iota^{\ast}$ and $\jmath_{[p]}^{\ast}$ (namely, $\iota^{\ast}\iota^{\ast}$, $\iota^{\ast}\jmath_{[p]}^{\ast}$, $\jmath_{[p]}^{\ast}\iota^{\ast}$, and $\jmath_{[p]}^{\ast}\jmath_{[p']}^{\ast}$) as potential generalized trace operators for form-valued forms. However, our goal is to construct spaces that fit within Bernstein-Gelfand-Gelfand (BGG) complexes. As we will demonstrate, $\iota^{\ast}\jmath_{[p]}^{\ast}$-conforming spaces --- that is, finite element spaces where the generalized trace $\iota^{\ast}\jmath_{[p]}^{\ast}$ (with $\iota^{\ast}$ applied to the $k$-form part and $\jmath_{[p]}^{\ast}$ to the $\ell$-form ``value'' part) is single-valued across cell boundaries --- are suitable candidates for these complexes. Consequently, we primarily focus on $\iota^{\ast}\jmath_{[p]}^{\ast}$-conforming spaces in this paper, though we occasionally discuss $\iota^{\ast}\iota^{\ast}$-conforming spaces.

\begin{remark}
For standard finite element differential forms, conformity is typically characterized by continuity (i.e., single-valuedness of the trace) across codimension-one faces. For instance, in three dimensions, N\'ed\'elec spaces exhibit continuous tangential components across faces. However, for the finite element spaces constructed in this paper, continuity on codimension-one faces alone is insufficient, as some spaces share the same continuity on codimension-one faces but differ on lower-dimensional subcells. Thus, when we refer to $\iota^{\ast}\jmath_{[p]}^{\ast}$-conforming spaces, we mean finite element spaces where the generalized trace $\iota^{\ast}\jmath_{[p]}^{\ast}$ is single-valued across \emph{all} subcells, even though $\iota^{\ast}\jmath_{[p]}^{\ast}$ may degenerate in certain low-dimensional cases. This subtlety is analogous to the distinction between Lagrange and Hermite elements in their respective complexes \cite{christiansen2018nodal}: both exhibit global $C^{0}$ continuity, but Hermite elements possess additional smoothness at vertices.
\end{remark}

\medskip

The overall idea of arriving at these $\iota^{\ast}\jmath_{[p]}^{\ast}$-conforming spaces in this paper is to adapt Whitney forms from FEEC theory to tensor elements, following the steps outlined below.

\smallskip
\noindent{\bf Step 1: $\iota^{\ast}\iota^{\ast}$-conforming elements.}
For $\ell$-form-valued $k$-forms, we begin by tensoring Whitney $k$-forms with  constant $\ell$-forms, giving $\mathcal{P}^{-}\alt^{k,\ell} := \mathcal{P}^{-}\alt^{k}\otimes \alt^{\ell}$. The resulting space is $\iota^{\ast}\rho^{\ast}$-conforming (where $\rho^{\ast}$ is the restriction operator). We then \emph{weaken} continuity to obtain an $\iota^{\ast}\iota^{\ast}$-conforming space. For instance, to build 1-form-valued 1-forms in three dimensions, we start with three copies of the N\'ed\'elec space (tangential continuity) and weaken the continuity, leading to tangential--tangential continuity.
Intuitively, this general procedure is feasible because certain degrees of freedom can be transferred from lower-dimensional subcells to higher-dimensional ones.  For a precise argument, we compute the dimension of the bubble space (bubbles have vanishing $\iota^*\iota^*$-traces at each level) and demonstrate unisolvency and conformity using a standard argument. The resulting finite element spaces $C_{\iota^*\iota^*} \mathcal P^- \alt^{k,\ell}$  are detailed in \Cref{prop:Pminus-ii}.

\smallskip
\noindent{\bf Step 2: Symmetry reduction.}
The spaces from Step 1 do not yet reflect the tensor symmetries. To impose the desired tensor symmetry, we follow the BGG construction on the discrete level. We therefore reduce these spaces to lie in $\ker(\mathcal{S}_{\dagger})$, which appears in the BGG diagrams. This requires reducing both the shape function spaces and their degrees of freedom. We discuss these two issues below.

To reduce the local shape function spaces, we verify that the BGG machinery is compatible with the polynomial spaces $\mathcal{P}^{-}\alt^{k,\ell}$; i.e., $\mathcal{S}_{\dagger}$ maps onto from $\mathcal{P}^{-}\alt^{k,\ell}$ to $\mathcal{P}^{-}\alt^{k-1,\ell+1}$ (\Cref{lem:S-poly}). For example, in three dimensions, $\mathcal S_{\dagger}^{1,1}:= -2\vskw$ maps $\mathbb M+ \bm x \times \mathbb M$ to $\mathbb V + \bm x \cdot \mathbb M$, and the kernel is the constant symmetric matrix space $\mathbb S$, which is exactly the local shape function space of the Regge element. \Cref{fig:shapefunc-reduction} shows the reduction on the shape function spaces in three dimensions. 

\begin{figure}
\begin{equation*}
\begin{tikzcd}[column sep = small]
~ & \text{Whitney 0-form} & \text{Whitney 1-form} & \text{Whitney 2-form} & \text{Whitney 3-form} \\ 
\text{$\otimes$ const 0-form} & \mathbb{R} + \bm x \cdot \mathbb V \arrow{r}{\grad} & \mathbb{V}+ \bm x \times \mathbb V \arrow[ld, "I",swap]  \arrow{r}{\curl} & \mathbb{V} + \bm x \otimes \mathbb R \arrow[ld, "\mskw",swap] \arrow{r}{\div} & \mathbb{R} \arrow[ld, "I",swap]\\
\text{$\otimes$ const 1-form}   & \mathbb{V} + \bm x \cdot \mathbb M \arrow{r}{\grad} & |[shape=rectangle, color=red, draw, minimum width=1cm, minimum height=0.6cm]| \mathbb{M} + \bm x \times \mathbb M \arrow[ld, "-2\vskw",swap]\arrow{r}{\curl}  &  \mathbb{M} + \bm x \otimes \mathbb V \arrow[ld, "S",swap]  \arrow{r}{\div}  & \mathbb{V} \arrow[ld, "\mskw",swap]  \\
\text{$\otimes$ const 2-form}  & |[shape=rectangle, color=red, draw, minimum width=1cm, minimum height=0.6cm]| \mathbb{V} + \bm x \cdot \mathbb M \arrow{r}{\grad}  & |[shape=rectangle, color=blue, draw, minimum width=1cm, minimum height=0.6cm]|  \mathbb{M} + \bm x \times \mathbb M \arrow[ld, "\tr", swap] \arrow{r}{\curl}  & |[shape=rectangle, color=green!50!black, draw, minimum width=1cm, minimum height=0.6cm]|   \mathbb{M} + \bm x \otimes \mathbb V \arrow[ld, "-2\vskw",swap] \arrow{r}{\div}  &\mathbb{V}  \arrow[ld, "I",swap] \\
\text{$\otimes$ const 3-form}  &  |[shape=rectangle, color=blue, draw, minimum width=1cm, minimum height=0.6cm]|  \mathbb{R} + \bm x \cdot \mathbb V \arrow{r}{\grad}  & |[shape=rectangle, color=green!50!black, draw, minimum width=1cm, minimum height=0.6cm]|  \mathbb{V} + \bm x \times \mathbb V\arrow{r}{\curl}  & \mathbb{V} + \bm x \otimes \mathbb R \arrow{r}{\div}  & \mathbb{R} 
\end{tikzcd}
\end{equation*}

\[\boxed{~~\text{Finite element } (k,\ell) \text{-form } = \text{Whitney } k\text{-form } \otimes \text{constant } \ell \text{-form}.~~}\]
\caption{The symmetry reduction of shape function spaces following the BGG construction. 
First, $(k, \ell)$-forms are constructed as the tensor product of Whitney $k$-forms (constant plus Koszul of constant) and constant $\ell$-forms. 
 The diagonal arrows are the vector proxies of the $\mathcal S_{\dagger}$ operators. The spaces involved in the reduction are illustrated in the figure with different colors. Red: reducing  (0,2)-forms from (1,1)-forms, giving $\mathbb S$ (the shape function space of the lowest order Regge element); Blue: reducing (0,3)-forms from (1,2)-forms, giving $\mathbb T + \bm x \times \mathbb S$ (the shape function space of the lowest order Hu-Lin-Zhang element); Green: reducing (1,3)-forms from (2,2)-forms, giving $\mathbb S$ (the shape function space of the lowest order HHJ/TDNNS element).} 
\label{fig:shapefunc-reduction}
\end{figure}

The degrees of freedom for $\mathcal{P}^{-}\alt^{k,\ell}$ involve moments against bubbles on each subcell. We remove certain bubbles likewise. Consequently, the degrees of freedom for the reduced finite elements can be defined by moments against the {reduced} bubble spaces. The key is to check that $\mathcal{S}_{\dagger}$ indeed maps {\it onto} from $\mathcal{B}^{-}\alt^{k,\ell}$ to $\mathcal{B}^{-}\alt^{k-1,\ell+1}$ (see  \Cref{lem:S-bubble}). This gives the dimension of the reduced spaces. 

This process yields spaces $C_{\iota^{\ast}\iota^{\ast}}\mathcal P^- \mathbb W^{k,\ell}$ together with their degrees of freedom, described in \Cref{prop:sym-reduction}. 
\smallskip
\noindent{\bf  Step 3: $\iota^{\ast}\jmath^{\ast}$-conforming elements.}
We then move certain degrees of freedom from higher-dimensional subcells to lower-dimensional ones, \emph{enhancing} the continuity of the finite element space. This works because:
\begin{enumerate}
\item The total dimension of the space remains unchanged.
\item Single-valuedness on lower-dimensional subcells guarantees single-valuedness on higher-dimensional subcells.
\end{enumerate}
This procedure applied to the reduced spaces $\W^{k,\ell}$ gives the final desired finite elements \Cref{prop:moving-dofs-W}. {Technically, the same operation of moving degrees of freedom can also be applied to the full $\alt^{k,\ell}$ spaces (\Cref{prop:moving-dofs-alt}).}

We apply the same recipe to obtain spaces $\mathcal{P}^{-}\W_{[p]}^{k,\ell}$ in the iterated BGG complexes (\Cref{ijpW}). The above recipe is demonstrated in \Cref{fig:regge-reduction,fig:hlz-reduction,fig:mcs-reduction,fig:hhj-reduction} for the Regge ($k = 1, \ell = 1, p = 1$), HLZ ( $k = 1, \ell = 2, p = 1$), MCS ( $k = 2, \ell = 1, p = 2$), and HHJ (TDNNS) ($k = 2, \ell = 2, p = 1$) elements, respectively. This exhausts all the nontrivial  constructions in three dimensions, see \Cref{example:3D-allcases}.

\begin{figure}
\FIG{\includegraphics[width=0.9\linewidth]{\RevisionAssets/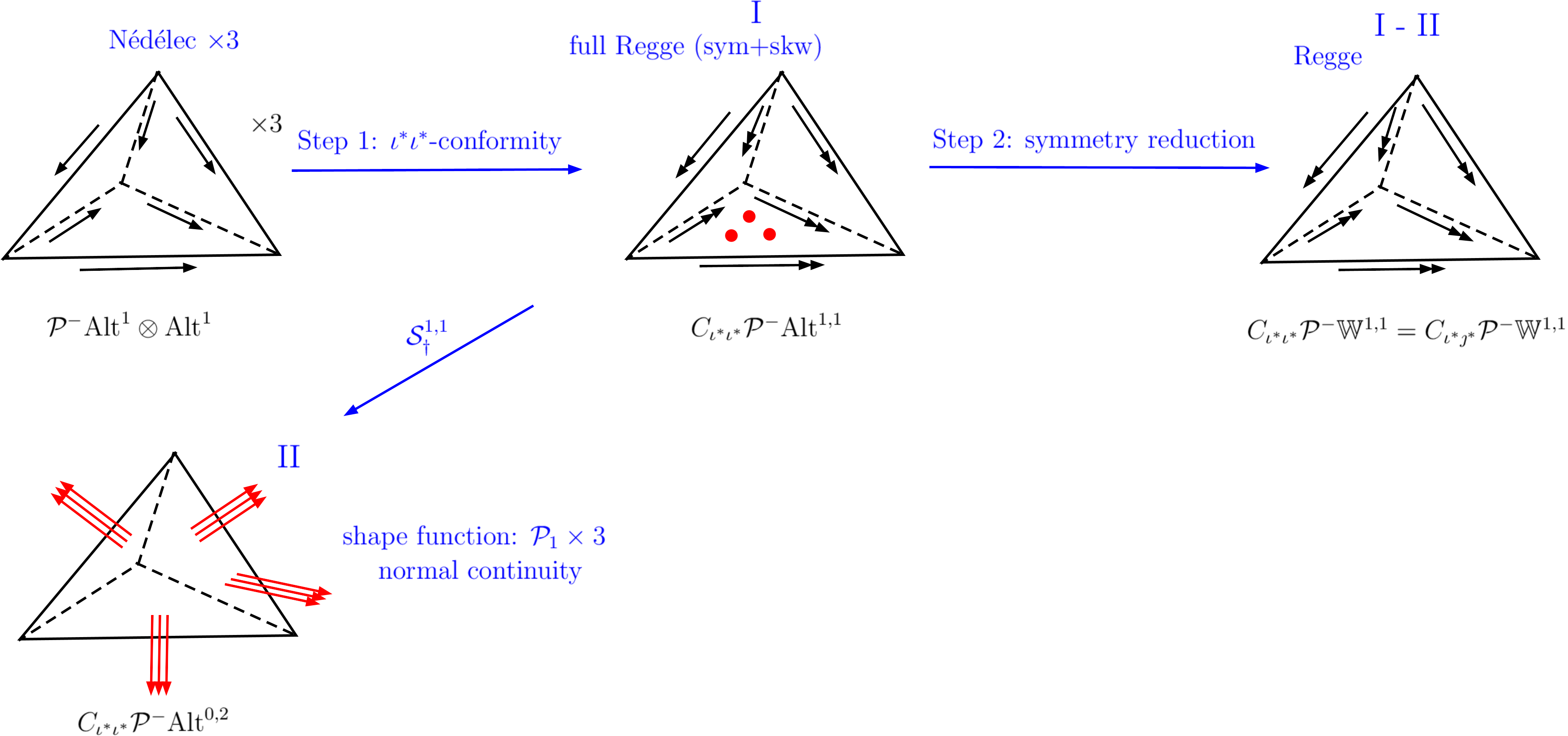}}
{\caption{Deriving the Regge element (tangential-tangential continuity) from a vector-valued N\'ed\'elec element (tangential continuity). Step 1: weakening the continuity of the N\'ed\'elec element to tangential-tangential. Step 2: eliminating the face degrees of freedom by those of a weakened vector Lagrange element connected by a $\mathcal{S}_{\dagger}$ operator.}
 \label{fig:regge-reduction}}
\end{figure}
\begin{figure}
\FIG{\includegraphics[width=1\linewidth]{\RevisionAssets/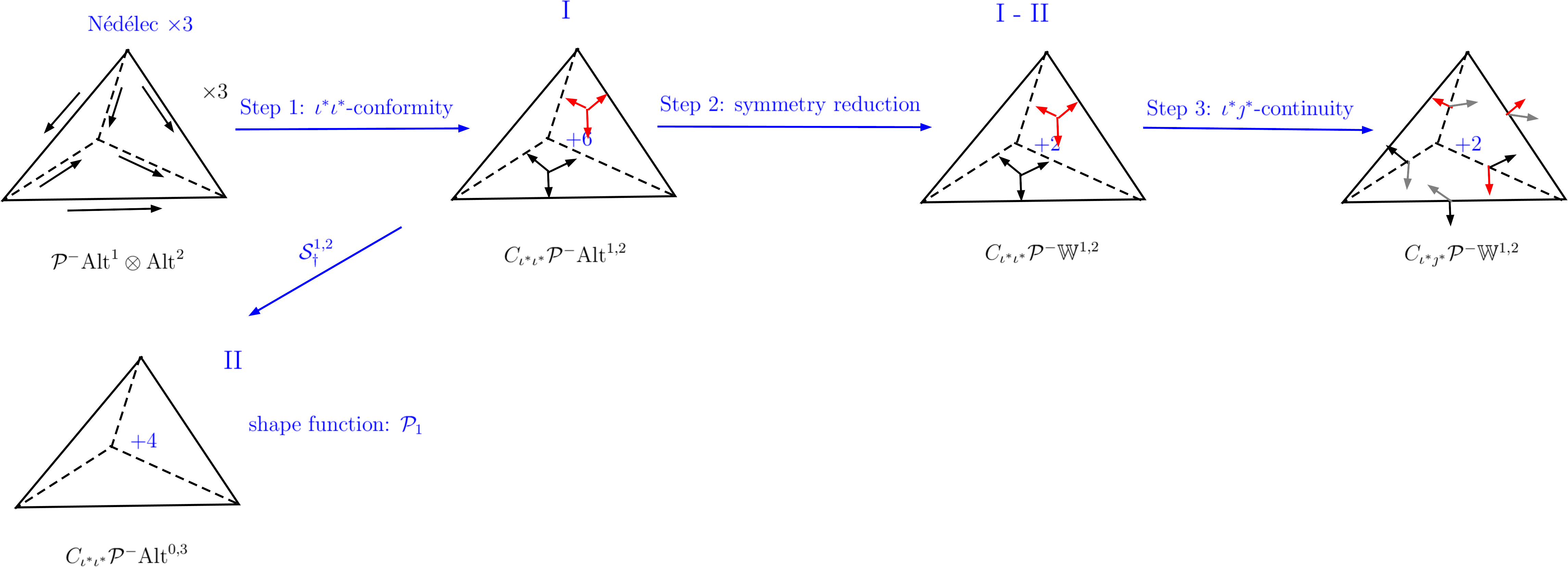}}
{\caption{Deriving the Hu-Lin-Zhang element (tangential-normal continuity) from a vector-valued N\'ed\'elec element (tangential continuity). Step 1: weakening the continuity of the N\'ed\'elec element to tangential-normal. Step 2: eliminating part of the interior degrees of freedom by those of a $\mathcal{P}_{1}$ connected by an $\mathcal{S}_{\dagger}$ operator. Step 3: moving the three degrees of freedom on each face to the three edges of the face; each edge gathers two tangential-normal degrees of freedom from its two neighbouring faces. \\
In general, we obtain $C_{\iota^*\jmath^*} \mathcal P^- \alt^{k,\ell}(K)$ from $C_{\iota^*\iota^*} \mathcal P^- \alt^{k,\ell}(K)$ by moving degrees from $\ell$-dimensional cells to $k$-dimensional ones. On each $\ell $-face $F$, the degrees of freedom are the inner product with respect to the space $\mathcal P^- \alt^{k}(F)$. To see that these degrees of freedom can be redistributed to $k$-cells, note that each $\sigma \in \mathcal T_{k}(F)$ receives $\binom{n-k}{\ell - k}$ degrees of freedom (one from each $\ell$-face containing $\sigma$), which is exactly the dimension of $\alt^{\ell - k}(\sigma^{\perp})$.
 }
 \label{fig:hlz-reduction}
 }
\end{figure}
\begin{figure}
\FIG{\includegraphics[width=0.9\linewidth]{\RevisionAssets/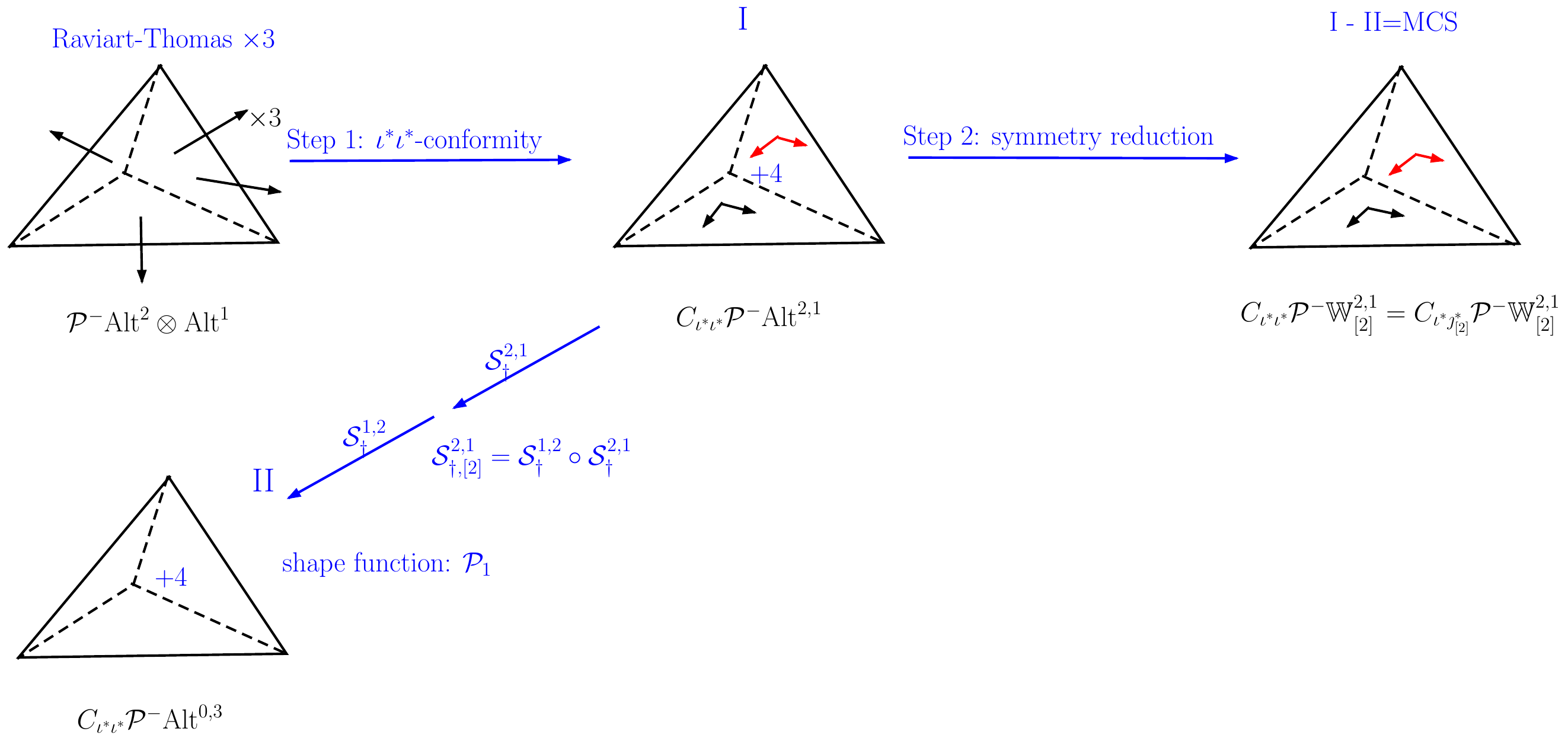}}
{\caption{Deriving the MCS element (normal-tangential continuity) from a vector-valued Raviart--Thomas element (normal continuity). Step 1: weakening the continuity of the Raviart--Thomas element to normal-tangential. Step 2: eliminating part of the interior degrees of freedom by those of a $\mathcal{P}_{1}$ connected by an $\mathcal{S}_{\dagger,[2]}$ operator.  Note that the MCS element appears in an iterated BGG construction as $k=2, \ell=1$ and $k>\ell$. Thus the reduction is by a space down two rows.}   \label{fig:mcs-reduction}}
\end{figure}

\begin{figure}
\FIG{\includegraphics[width=1\linewidth]{\RevisionAssets/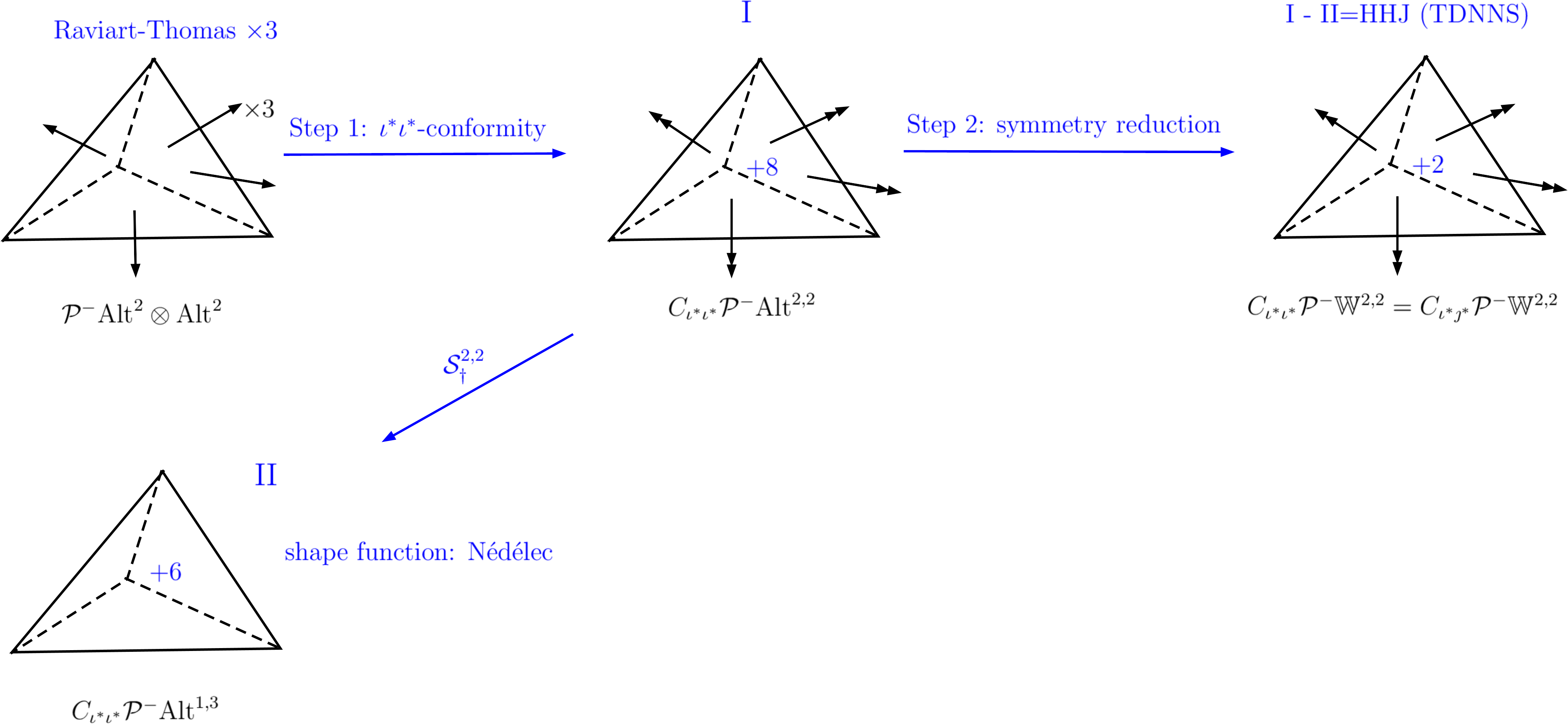}}
{\caption{Deriving the HHJ (TDNNS) element (normal-normal continuity) from a vector-valued Raviart--Thomas element (normal continuity). Step 1: weakening the continuity of the Raviart--Thomas element to normal-normal. Step 2: eliminating part of the interior degrees of freedom by N\'ed\'elec shape functions connected by a $\mathcal{S}_{\dagger}$ operator.}   \label{fig:hhj-reduction}}
\end{figure}

\bigskip

Finally, high-order constructions (including the families $\mathcal{P}\alt^{k,\ell}$ and $\mathcal{P}\W_{[p]}^{k,\ell}$) follow analogously, using the same sequence of steps. In this case, we can generate two different families: the incomplete polynomial finite element family  $\mathcal P_r^-\alt^{k,\ell}$ and the complete polynomial family   $\mathcal P_r\alt^{k,\ell}$. 

\bigskip

We summarize the main steps in \Cref{fig:main-steps}. The spaces produced in the construction are the following:
\begin{enumerate}
	\item $C_{\iota^*\iota^*} \mathcal P^- \alt^{k,\ell}$ for $k \le \ell$ {(\Cref{prop:Pminus-ii})}; \label{list}
	\item $C_{\iota^*\iota^*} \mathcal P^- \W^{k,\ell}$ for $k \le \ell$ {(\Cref{prop:sym-reduction})}; 
	\item $C_{\iota^* \jmath^*} \mathcal P^- \alt^{k,\ell}$ for $k \le \ell$ {(\Cref{prop:moving-dofs-alt})}; 
	\item $C_{\iota^*\jmath^*} \mathcal P^- \W^{k,\ell}$ for $k \le \ell$ ({\Cref{prop:moving-dofs-W}}).\\
\noindent  For the iterated constructions, we have 
\item $C_{\iota^*\iota^*}\mathcal P^- \W^{k,\ell}_{[p]}$ for $k \le \ell + p -1$ {(\Cref{iiWp})};
{\item  $C_{\iota^*\jmath^*_{[p]}} \mathcal P^-\Alt^{k,\ell}$ for $k \le \ell + p -1$ (\Cref{prop:Wp});} 
\item $C_{\iota^*\jmath^*_{[p]}} \mathcal P^-\W^{k,\ell}_{[p]}$ for $k \le \ell + p -1$ {(\Cref{ijpW})}. 
\end{enumerate}
The spaces of symmetries (2)(4)(7) are candidates for discrete BGG complexes.

 \begin{figure}[htbp]
 \centering
 {\tikzset{every node/.append style={text=red}}
\begin{tikzpicture}[
    level distance = 3cm,
    sibling distance = 3.5cm,
    every node/.style = {rectangle, minimum size = 0.6cm},
    invisible/.style = {draw=none, minimum size=0pt},
    edge from parent/.style={draw, ->}
]
\node (root) [label=left:\textcolor{red}{Lagrange, N\'ed\'elec, Raviart--Thomas}]  {$(C_{\iota^\ast}\Alt^k)\otimes\Alt^\ell=C_{\iota^\ast\rho^\ast}\Alt^{k, \ell}$}
    child {
        node (node1) [label=left:\textcolor{red}{full Regge}]  {(1) $C_{\iota^*\iota^*} \Alt^{k,\ell}$}
        child { 
            node (node2)[label=left:\textcolor{red}{\shortstack{Regge\\  $\mathrm{MCS}^{\top}$\\TDNNS/HHJ}}] {(2)$C_{\iota^{*}\iota^*} \W^{k,\ell}$}
            child { 
                node (node4)[label=left:\textcolor{red}{\shortstack{Regge\\ HLZ\\TDNNS/HHJ}}] {(4)$C_{\iota^{*}\jmath^*} \W^{k,\ell}$}
            }
        }
        child[edge from parent/.style={draw=none}] { 
            node [invisible] {} 
            child[edge from parent/.style={draw=none}] { 
                node (node3) {(3)$C_{\iota^{*}\jmath^*} \Alt^{k,\ell}$} 
            } 
        }
        child { 
            node (node5)[label=right:\textcolor{red}{MCS ($p=2$)}] {(5)$C_{\iota^{*}\iota^*} \W_{[p]}^{k,\ell}$}
            child { 
                node (node7) {(7)$C_{\iota^{*}\jmath_{[p]}^*} \W_{[p]}^{k,\ell}$}
            }
        }
        child[edge from parent/.style={draw=none}] { 
            node [invisible] {} 
            child[edge from parent/.style={draw=none}] { 
                node (node6)  {(6)$C_{\iota^{*}\jmath_{[p]}^*} \Alt^{k,\ell}$} 
            } 
        }
    };

\path (root) -- (node1) node[pos=0.5, right=10pt, text=red, font=\small] {moving DoFs};
\path (node1) -- (node2) node[pos=0.5, left=10pt, text=red, font=\small] {sym reduction};
\path (node3) -- (node4) node[pos=0.5, below=10pt, text=red, font=\small] {sym reduction};
\path (node2) -- (node4) node[pos=0.5, left=10pt, text=red, font=\small] {moving DoFs};
\path (node1) -- (node3) node[pos=0.5, left=10pt, text=red, font=\small] {moving DoFs};
\path (node6) -- (node3) node[pos=0.5, below=15pt, below=10pt, text=red, font=\small] {$p=1$};
\path (node7) -- (node4) node[pos=0.5, above=7pt, fill=white, inner sep=2pt, text=red, font=\small] {$p=1$};
\path (node1) -- (node5) node[pos=0.5, right=10pt, text=red, font=\small] {};
\path (node3) -- (node6) node[pos=0.5, right=10pt, text=red, font=\small] {};
\path (node5) -- (node7) node[pos=0.5, left=5pt, text=red, font=\small] {moving DoFs};
\path (node1) -- (node6) node[pos=0.5, right=20pt, below=15pt, text=red, font=\small] {moving DoFs};
\path (node6) -- (node7) node[pos=0.5, below=10pt, text=red, font=\small] {sym reduction};

\draw [->] (node1) -- (node3) node[midway, above, text=red, font=\small] {};
\draw [dashed, ->] (node3) -- (node4) node[midway, above, text=red, font=\small] {};
\draw [dashed, ->] (node6) -- (node7) node[midway, above, text=red, font=\small] {};
\draw [->] (node1) -- (node6);
\draw [dashed, ->] (node6) to[bend left=30] (node3) node[midway, below=5pt, text=red, font=\small] {}; 
\draw [dashed, ->] (node7) to[bend left=30] (node4) node[midway, below=5pt, text=red, font=\small] {}; 
\end{tikzpicture}}
\caption{Main steps of the construction by moving DoFs and symmetry reduction. The numbers indicate the items in the list on Page \pageref{list}.}
\label{fig:main-steps}
 \end{figure}

\bigskip

\noindent{\bf Remarks on complexes and cohomologies.} 
 As we mentioned before, the ultimate goal of this study is to provide a cohomology-preserving discretization of the BGG complexes. 
For a complex of finite-dimensional vector spaces,
\[
\begin{tikzcd}
0 \arrow{r} & V^{0 } \arrow{r}{d^{0}} & V^{1}\arrow{r}{d^{1}} &\cdots \arrow{r}{} & V^{n} \arrow{r}{} & 0, 
\end{tikzcd}
\]
a necessary condition for it to be exact is that the Euler characteristic is zero. That is,
\begin{equation}\label{dim-condition}
\sum_{i=0}^{n}(-1)^{i}\dim V^{i}=0. 
\end{equation}

For the lowest-order case (referred to as tensorial Whitney forms) we construct the tensor-valued distribution spaces, which serve as the candidates for the spaces following the zig-zag pattern. With these specifications, we can formally define the complete sequence, generalizing the diagrams for the BGG complexes (see \Cref{fig:table}) and the iterated BGG complexes (see \Cref{fig:iterative-bgg-fe}) to arbitrary dimensions.

Although this paper does not establish the result that the cohomologies of the finite element complexes are isomorphic to their continuous counterparts (except in dimensions up to three, as demonstrated in \cite{hu2025distributional} and \cite{christiansen2023extended}), we verify that  the alternating dimension sum equals the Euler characteristic of the corresponding smooth BGG complex when the domain has trivial topology (homeomorphic to a ball). This dimension condition also extends to the iterated case. This should be a strong indication that the discrete BGG and iterated BGG complexes possess the correct cohomologies. A comprehensive analysis of the operators and cohomologies is left for future work.

\subsection{Notations}

 In this subsection, we recall some notations in the study of form-valued forms and vector proxies. Let $V$ be a vector space,  typically $\mathbb R^n$ in this paper.
We use $\Alt^k(V)$ to denote the algebraic alternating $k$-forms on $V$,
and $\Alt^{k,\ell}(V) := \Alt^{k}(V) \otimes \Alt^{\ell}(V)$. \footnote{Throughout this paper, $\otimes$ is the tensor product with respect to $\mathbb R$.} 
When there is no danger of confusion, we also drop $V$ and write $ \Alt^{k}$ and $\Alt^{k,\ell}$.  
Then the space $C^{\infty}(\Omega) \otimes \Alt^{k}$ consists of smooth differential $k$-forms on  an open subset $\Omega\subset\mathbb R^n$. We use $d^{\bs}$ to denote the exterior derivative  $d^{k}: C^{\infty}\otimes\Alt^{k,\ell}\to C^{\infty}\otimes\Alt^{k+1,\ell}$. Note that $d^{\bs}$ acts on the first index ($k$, rather than $\ell$). 

Hereafter, $\mathcal T$ is a triangulation and $\mathcal T_{<n}$ denotes the set of all subsimplices of $\mathcal T$ with dimension less than $n$. Similarly, we define $\mathcal T_{>n}, \mathcal T_{\le n}, \mathcal T_{\ge n}$, and define $\mathcal T_{[a:b]} := \mathcal T_{\ge a} \cap \mathcal T_{\le b}$. 
In the sequel, we use $\VV:=\mathbb{R}^{n}$, $\MM:=\mathbb{R}^{n\times n}$,  $\SS:=\mathbb{R}^{n\times n}_{\sym}$, and $\TT:=\mathbb{R}^{n\times n}_{\dev}$ to denote the spaces of vectors, matrices, symmetric matrices, and traceless matrices, respectively.
We introduce notation for several linear algebraic operations on $\mathbb{R}^{n\times n}$:
\begin{itemize}
\item $\skw: \mathbb{M}\to \mathbb{K}$ and $\sym:\mathbb{M}\to \mathbb{S}$ denote taking the skew-symmetric and symmetric part of a matrix.
\item $\tr:\mathbb{M}\to\mathbb{R}$ is the trace, given by summing the diagonal entries of a matrix.
\item $I: \mathbb{R}\to \mathbb{M}$ is defined by $I(u) := uI$, identifying a scalar $u$ with the corresponding diagonal matrix.
\item $\dev:\mathbb{M}\to \mathbb{T}$ is the deviator (trace-free part), $\dev(u):= u - \tfrac{1}{n}\tr(u)I$.
\end{itemize}

In three dimensions only, there is an isomorphism between skew-symmetric matrices in $\mathbb{K}$ and vectors in $\mathbb{V}$ via
\[
 \mskw
 \left (
 \begin{array}{c}
 v_{1}\\v_{2}\\v_{3}
 \end{array}
 \right )
 =
 \begin{pmatrix}
0 & -v_{3} & v_{2}\\
v_{3} & 0 & -v_{1}\\
-v_{2} & v_{1} & 0
\end{pmatrix}.
\]
This map $\mskw: \mathbb{V}\to \mathbb{K}$ is an isomorphism satisfying $\mskw(v)\,w = v\times w$ for any $v,w\in \mathbb{V}$; the vector $v$ is called the \emph{axial vector} of $\mskw(v)$. We also define $\vskw := \mskw^{-1}\circ \skw: \mathbb{M}\to \mathbb{V}$, taking a matrix to the axial vector of its skew-symmetric part.

Finally, let $S: \mathbb{M}\to\mathbb{M}$ be the linear map given by $S(u) := u^T - \tr(u)\,I$. One can verify that $S$ is invertible (with $S^{-1}(u) = u^T - \frac{1}{n-1} \tr(u)I$) for any $n>1$.

We summarize the notations and terms that will be used below, see \Cref{tab:notation-1,tab:notation-2,tab:notation-3}, with references to the place where they first appear.

\begin{table}[htbp]
 \TBL{
\begin{center}
\begin{tabular}{|c|c|}
\hline
Notations & Meanings
\\ \hline 
$\Omega$ & $n$ dimensional domain\\
$\mathcal T$ & triangulation of $\Omega$ \\
$\mathcal T_{s}, \mathcal T_{\le a}, \mathcal T_{<a}, \mathcal T_{>a}, \mathcal T_{\ge a}, \mathcal T_{[a,b]}$ & the collection of the faces \\ & with dimension $s, \le a, <a, >a, \ge a $ and $\in [a,b]$ \\
$\mathcal H_{dR}$ & de Rham cohomology\\ 
$\lambda_{i}$ & barycentric coordinates \\ 
$\phi_{\sigma}$, \eqref{def:whitney}, p.\pageref{def:whitney} & Whitney forms \\ 
$[n]$ & $\{1,2,\cdots, n\}$ \\ 
$X(n,k)$ & increasing $k$ sets in $\{1,2,\cdots,n\}$ \\ 
DoFs & abbreviation for degrees of freedom \\
BGG & Bernstein-Gelfand-Gelfand \\
\hline
\end{tabular}
\end{center}}
{\caption{General notations.}
\label{tab:notation-3}}
\end{table}

\begin{table}[htbp]
 \TBL{
\begin{center}
\begin{tabular}{|c|c|}
\hline
Notations & Meanings
\\ \hline 
\textbf{Forms} & \\
$\Alt^{k}$ & alternating $k$-forms \\ 
$\mathcal P_r\alt^k$, \Cref{sec:high-order}, p.\pageref{sec:high-order} & polynomial $k$-forms \\
$\mathcal B_r\alt^k$, \Cref{sec:high-order}, p.\pageref{sec:high-order}& bubble space of polynomial $k$-forms \\
 $\mathcal P_r^-\alt^k$, \Cref{sec:high-order}, p.\pageref{sec:high-order} & incomplete polynomial $k$-forms \\
$\mathcal B_r^-\alt^k$, \Cref{sec:high-order}, p.\pageref{sec:high-order} & bubble space of incomplete polynomial $k$-forms \\
$N^{\ell}(\sigma, K)$, \eqref{def:Nell}, p.\pageref{def:Nell} & auxiliary space for bubbles \bigskip\\ 
\textbf{Form-valued Forms} & \\ 
$\Alt^{k, \ell}:=\Alt^{k}\otimes \Alt^{\ell}$ & alternating $\ell$-form-valued alternating $k$-forms (($k,\ell$) forms) \\
$\W^{k, \ell}$, \eqref{def:W}, p.\pageref{def:W} & subspace of $\Alt^{k, \ell}$ in $\ker(\mathcal{S}^{\bs}_{\dagger})$\\
$\widetilde{\W}^{k, \ell}$, \eqref{def:tildeW}, p.\pageref{def:tildeW} & subspace of $\Alt^{k, \ell}$ in $\ker(\mathcal{S}^{\bs})$\\
$\W^{k, \ell}_{[p]}$, \eqref{def:Wp}, p.\pageref{def:Wp} & subspace of $\Alt^{k, \ell}$ in $\ker(\mathcal{S}^{\bs}_{\dagger,[p]})$, \\
$\widetilde{\W}^{k, \ell}_{[p]}$, \eqref{def:tildeWp}, p.\pageref{def:tildeWp}&  subspace of $\Alt^{k, \ell}$ in $\ker(\mathcal{S}^{\bs}_{[p]})$ \\
$\mathcal P^- \alt^{k,\ell}$,  \eqref{def:p1-}, p.\pageref{def:p1-}& Whitney $(k,\ell)$ forms \\ 
$\mathcal B^- \alt^{k,\ell}$, \eqref{def:B-}, p.\pageref{def:B-} & bubbles of Whitney $(k,\ell)$ forms \\
$\mathcal P_r^- \Alt^{k,\ell}$,  \eqref{def:P-}, p.\pageref{def:P-}& incomplete polynomial $(k,\ell)$ forms \\
$\mathcal B_r^- \alt^{k,\ell}$,  \eqref{eq:Br-}, p. \pageref{eq:Br-} &  bubbles of incomplete polynomial $(k,\ell)$ forms \\ 
$\mathcal P_r^- \W^{k,\ell}$,  \eqref{def:Pr-}, p.\pageref{def:Pr-} & incomplete polynomial subspace in $\ker(\mathcal{S}^{\bs}_{\dagger})$\\
  $\mathcal P_r^{-} \W^{k,\ell}_{[p]}$,  \Cref{lem:Pr-Wp}, p.\pageref{lem:Pr-Wp}   & incomplete polynomial subspace in  $\ker(\mathcal S^{\bs}_{\dagger,[p]})$ \\
$\mathcal P_r \alt^{k,\ell}$ & polynomial $(k,\ell)$ forms $\mathcal P_r \otimes \alt^{k,\ell}$ \\ 
$\mathcal B_r \alt^{k,\ell}$,  \eqref{eq:Br}, p.\pageref{eq:Br} & bubbles of polynomial $(k,\ell)$ forms  \\ 
$\mathcal P_r \W^{k,\ell}$, $\mathcal P_r \W^{k,\ell}_{[p]}$  & polynomial subspaces $\mathcal P_r \otimes \W^{k,\ell}$ and $\mathcal P_r \otimes \W^{k,\ell}_{[p]}$ \\ 
& \\ 
\hline 
\end{tabular}
\end{center}
}
{\caption{Notations for continuous and discrete spaces.}
\label{tab:notation-1}}
\end{table}

\begin{table}[htbp]
\TBL{
\begin{center}
\begin{tabular}{|c|c|}
\hline
Notations & Meanings
\\ \hline 

$d$ & differential operators for forms and form-valued forms \\
$\mathcal S$, \eqref{def:S}, p.\pageref{def:S} & connecting maps in the BGG diagram \\ 
$\mathcal S_{\dagger}$, \eqref{eq:Sdagger}, p.\pageref{eq:Sdagger}& adjoint operator of $\mathcal S$  \\ 
$\mathcal S_{[p]}$, \eqref{def:Sp}, p.\pageref{def:Sp} & iterated connecting maps, composition of $\mathcal S$ \\ 
  $\mathcal S_{\dagger,[p]}$, \eqref{def:Sdaggerp}, p.\pageref{def:Sdaggerp}& adjoint operator of $\mathcal S_{[p]}$ \\ 
$\kappa$, \eqref{def:kappa}, p.\pageref{def:kappa} & Koszul operator for forms and form-valued forms \\
$\iota^*$, \eqref{def:trace}, p. \pageref{def:trace} & traces / pullback of the inclusion of forms \\ 
$\iota^*\iota^*$, \eqref{double-trace}, p.\pageref{double-trace} & two-sided traces for form-valued forms \\ 
$\jmath^*$, \eqref{def:jstar}, p. \pageref{def:jstar} & generalized trace operators\\
 $\vartheta_{E,q}^*$, \eqref{def:vartheta}, p.\pageref{def:vartheta} & generalized trace (edge normal, etc.) \\ 
$\rho^*$, \eqref{def:rho}, p.\pageref{def:rho} & restriction (value on edges, etc.) \\
$\jmath^*_{[p]}$, \eqref{def:jp}, p.\pageref{def:jp} & interpolated generalized trace \\ 
$C_{\iota^*}$ & (prefix) $\iota^*$ conforming finite element forms \\
$C_{\rho^*}$ & (prefix) $\rho^*$ conforming finite element forms  \\ 
$C_{\iota^*\iota^*}$ & (prefix) $\iota^*\iota^*$ conforming finite element form-valued forms \\ 
$C_{\iota^*\jmath^*}$, $C_{\iota^*\jmath_{[p]}^*}$ & (prefix) $\iota^*\jmath^*$ (and $\iota^*\jmath_{[p]}^*)$ conforming finite element form-valued forms \\  
\hline
\end{tabular}
\end{center}}
{\caption{Notations for the operators and the prefixes.}
\label{tab:notation-2}}
\end{table}

\section{BGG complexes and form-valued  forms revisited}\label{sec:bgg}

In this section, we first revisit the BGG machinery in the setting of \cite{arnold2021complexes,vcap2022bgg}. Then for the purpose of this paper, we provide several generalizations, including introducing the $\mathcal{S}_{\dagger}$ operators and deriving the Koszul version of the BGG complexes. Moreover, we introduce the iterated construction, by which we obtain more complexes in the BGG diagram.

The goal of this paper is to give a canonical discretization of the form-valued forms $C^{\infty} \otimes \Alt^{k,\ell}$. 
Moreover, we incorporate more symmetries than the skew-symmetry of the first $k$ indices and the last $\ell$ indices. 
Specifically, the symmetry considered in this paper is given by the operators $\mathcal S$ and $\mathcal S_{\dagger}$ in the framework of the BGG construction 
\cite{arnold2021complexes,vcap2022bgg}:
\[
\mathcal S^{k,\ell} : \Alt^{k,\ell} \to \Alt^{k+1,\ell-1}, \quad \mathcal S_{\dagger}^{k,\ell} : \Alt^{k,\ell} \to \Alt^{k-1,\ell+1}.
\]

The definition follows from \cite{arnold2021complexes}: for $\omega\in \alt^{k, \ell}(V)$ and $v_{1}, \cdots, v_{k+1}\in V$,  $u_{1}, \cdots, u_{\ell-1}\in V$, let the connecting mapping be defined as
\begin{equation}
\label{def:S}
\mathcal S^{k,\ell}\omega(v_{1}, \cdots, v_{k+1})(u_{1}, \cdots, u_{\ell-1}):=\sum_{j=1}^{k+1}(-1)^{j+1}\omega(v_{1}, \cdots, \widehat{v_{j}}, \cdots, v_{k+1})(v_{j}, u_{1}, \cdots, u_{\ell-1}).
\end{equation}
Here, the hat notation means that $v_j$ is deleted from the sequence $v_1,\cdots, v_{k+1}$.

The BGG construction in \cite{arnold2021complexes} follows the diagram below:
\begin{equation}\label{diagram-nD}
\begin{tikzcd}[column sep=tiny]
0 \arrow{r} & C^{\infty}\otimes\alt^{0,0}  \arrow{r}{d} &C^{\infty}\otimes\alt^{1,0}  \arrow{r}{d} & \cdots \arrow{r}{d} & C^{\infty}\otimes \alt^{n,0} \arrow{r}{} & 0\\
0 \arrow{r} & C^{\infty}\otimes\alt^{0,1}  \arrow{r}{d} \arrow[ur, "\mathcal S^{0,1}"] &C^{\infty}\otimes\alt^{1,1}  \arrow{r}{d}  \arrow[ur, "\mathcal S^{1,1}"] & \cdots \arrow{r}{d}  \arrow[ur, "\mathcal S^{n-1,1}"] & C^{\infty}\otimes \alt^{n,1} \arrow{r}{} & 0\\[-15pt]
 & \vdots & \vdots & {} & \vdots & {} \\[-15pt]
0 \arrow{r} & C^{\infty}\otimes\alt^{0,n-1}  \arrow{r}{d} &C^{\infty}\otimes\alt^{1,n-1}  \arrow{r}{d} & \cdots \arrow{r}{d} & C^{\infty}\otimes \alt^{n,n-1} \arrow{r}{} & 0\\
0 \arrow{r} & C^{\infty}\otimes\alt^{0,n}  \arrow{r}{d} \arrow[ur, "\mathcal S^{0,n}"] &C^{\infty}\otimes\alt^{1,n}  \arrow{r}{d}  \arrow[ur, "\mathcal S^{1,n}"] & \cdots \arrow{r}{d}  \arrow[ur, "\mathcal S^{n-1,n}"] & C^{\infty}\otimes \alt^{n,n} \arrow{r}{} & 0.
\end{tikzcd}
\end{equation}
Here, $d^{k}: C^{\infty}\otimes \alt^{k,\ell} \to C^{\infty}\otimes \alt^{k+1,\ell}$ acts on the first index.

We introduce the spaces of alternating forms with symmetries: for $0 \le k, \ell \le n$, 
 {
\begin{equation}\label{def:W}
\W^{k, \ell} := 
 \ran(\cS^{k-1,\ell+1})^{\perp} \subseteq \alt^{k,\ell} , \text{ when } k \le \ell, 
\end{equation}}
and
\begin{equation}\label{def:tildeW}
\widetilde{\W}^{k, \ell}:=
\ker(\cS^{k,\ell})  \subseteq \alt^{k,\ell}   \text{ when } k \geq {\ell}.
\end{equation}
Note that {the indices $k=\ell$ appear} in both \eqref{def:W} and \eqref{def:tildeW}. 
In this paper, the spaces involving $\widetilde\W$ will be discretized as distributions, while the spaces involving $\mathbb W$ will be discretized by piecewise functions.

The following theorem follows from \cite{arnold2021complexes}.
\begin{theorem}
The following sequence is a complex (referred to as the BGG complex linking rows $\ell$ and $\ell+1$ of \eqref{diagram-nD} hereafter)
\begin{equation} \label{BGG-seq}
\begin{tikzcd}[column sep = small]
0\ar[r] & C^{\infty} \otimes \W^{0,\ell} \ar[r,"\pi\circ d"] & C^{\infty} \otimes \W^{1,\ell} \ar[r,"\pi\circ d"] & \cdots   \ar[r,"\pi\circ d"] & C^{\infty} \otimes \W^{\ell,\ell} \ar[r,"d"] & ~ \ar[llld,"\mathcal S^{-1}"] \\
&&  ~ \ar[r,"d"] & C^{\infty} \otimes \widetilde{\W}^{\ell+1,\ell+1} \ar[r,"d"] &   C^{\infty} \otimes \widetilde{\W}^{\ell+2,\ell+1} \ar[r,"d"] & \cdots  \ar[r,"d"] & C^{\infty} \otimes \widetilde{\W}^{n,\ell+1}\ar[r] & 0,
\end{tikzcd}
\end{equation}
where the operators $\pi$ are the projections to the tensor spaces with symmetries $\W^{\bs, \bs}$ (with respect to the Frobenius norm). 
The cohomology of \eqref{BGG-seq} is isomorphic to $ \mathcal H_{dR}^{\bs}(\Omega)\otimes (\Alt^{\ell} \oplus \Alt^{\ell+1})$, where $ \mathcal H_{dR}^{\bs}(\Omega)$ is the de~Rham cohomology. 
\end{theorem}

However, complexes of the form of \eqref{BGG-seq} do not exhaust all the possibilities. We may compose the $\mathcal{S}^{\bs}$ operators in \eqref{diagram-nD}, leading to new connecting maps, and connect any two rows in \eqref{diagram-nD}. In three space dimensions, this {\it iterated BGG construction} leads to the $\grad\curl$, $\curl\div$ and $\grad\div$ complexes, which were derived in  \cite{arnold2021complexes}. For general dimensions, we show that $\mathcal S^{\bs}$ also enjoys the desired injectivity/surjectivity properties, leading to more BGG complexes.

\subsection{The $\mathcal S_{\dagger}$ operator and adjointness}

In the above framework, the spaces in the BGG complex take value in $\ran(\mathcal S^{k-1,\ell+1})^{\perp}$. The orthogonal  complement is not straightforward to work with for the purpose of this paper. Below we will instead use $\ker(\mathcal S_{\dagger}^{k,\ell})$, the kernel of the adjoint operator of $\mathcal{S}$. The introduction of $\mathcal S_{\dagger}$ is closer to the BGG construction in an algebraic and geometric context \cite{vcap2001bernstein}.

We define $\mathcal S_{\dagger}^{k,\ell}: \Alt^{k,\ell} \to \Alt^{k-1,\ell+1}$ as follows:  for $\omega\in \alt^{k, \ell}(V)$ and $v_{1}, \cdots, v_{k-1}\in V$,  $u_{1}, \cdots, u_{\ell+1}\in V$, 

\begin{equation}
\label{eq:Sdagger}	
\mathcal S_{\dagger}^{k,\ell}\omega(v_{1}, \cdots, v_{k-1})(u_{1}, \cdots, u_{\ell+1})=\sum_{j=1}^{\ell+1}(-1)^{j+1}\omega(u_{j}, v_{1}, \cdots, v_{k-1})(u_{1}, \cdots,  \widehat{u_{j}}, \cdots, u_{\ell+1}).
\end{equation}

\begin{lemma}
\label{lem:S-inj-suj}
We have the following properties.
	\begin{enumerate}
    \item $\mathcal S^{k,\ell}_{\dagger}$ and $\mathcal S^{k-1, \ell+1}$ are adjoint with respect to the  induced exterior inner products (increasing wedge monomials are orthonormal), and therefore $\ker(\mathcal S_{\dagger}^{k,\ell}) =\ran (\mathcal S^{k-1, \ell+1})^{\perp}$.  Here $\perp$ is the orthogonal complement with  this inner product.
			\item  When $k \le \ell+1$, $\mathcal S_{\dagger}^{k,\ell}$ is surjective, and $\mathcal S^{k-1, \ell+1}$ is injective.
				\item  When $k \geq \ell-1$, $\mathcal S^{k,\ell} $ is surjective while $\mathcal S_{\dagger}^{k+1, \ell-1}$ is injective.
	\end{enumerate}
\end{lemma}
The proof for the surjectivity and injectivity ((2) and (3) above) can be found in \cite{arnold2021complexes}. For clarity, we present the proof of \Cref{lem:S-inj-suj} in the appendix.
\begin{proof}
We prove (1) in \Cref{lem:adjointness}, and (2) and (3) in \Cref{subsec:inj/sur}.
\end{proof}

With the above properties of $\mathcal S_{\dagger}$, we can reformulate $\mathbb W^{k,\ell}$ as 
 {
\begin{equation}\label{def:W-kernel}
\W^{k, \ell} :=  
 \ker(\cS_{\dagger}^{k,\ell}) \subseteq \alt^{k,\ell} ,  \text{ when } k \le \ell.
\end{equation}}

In three dimensions, the form-valued forms $\Alt^{k,\ell}$ and the symmetric ones $\W^{k,\ell}$ can be illustrated via vector proxies. 
\begin{table}[htbp]
\TBL{
\begin{center}
\begin{tabular}{ccccc}
\diagbox[width=\dimexpr \textwidth/20+\tabcolsep\relax, height=0.7cm]{ $k$ }{$\ell$}&  0 & 1 & 2 & 3 \\
0 & $\RR$ & $\VV$ & $\VV$ &  $\RR$ \\
1 & $\VV$ & $\MM$ & $\MM$  & $\VV$ \\
2  & $\VV$ & $\MM$ & $\MM$  & $\VV$ \\
3 & $\RR$ & $\VV$ & $\VV$ &  $\RR$  \\
\end{tabular}	
\hspace{2em}
\begin{tabular}{ccccc}
\diagbox[width=\dimexpr \textwidth/20+\tabcolsep\relax, height=0.7cm]{ $k$ }{$\ell$}&  0 & 1 & 2 & 3 \\
0 & $\RR$ & $\VV$ & $\VV$ &  $\RR$ \\
1 & $\VV$ & $\SS $ & $\TT$  & $\VV$ \\
2  & $\VV$ & $\TT$ & $\SS$  & $\VV$ \\
3 & $\RR$ & $\VV$ & $\VV$ &  $\RR$  \\
\end{tabular}	\\
\end{center}
}
{\caption{Left: vector/matrix proxies of $\Alt^{k, \ell}$ in $\mathbb{R}^{3}$. Right: vector/matrix proxies of $\W^{k, \ell}$ and $\widetilde{\W}^{k,\ell}$ in $\mathbb{R}^{3}$ (see \eqref{def:W}). We highlight that two definitions $\W^{k,k}$ and $\widetilde{\W}^{k,k}$ lead to the same proxy in $\mathbb{R}^{3}$. Any space $\W^{k, \ell}$ with $k\neq \ell$ only appears in  \eqref{def:W} once. Thus listing all the cases as in the table on the right will not lead to ambiguity.}}
\end{table}

In general dimensions, $\W^{1,1}$ can be identified with the space of symmetric matrices $\mathbb S$ in $n$ dimensions;  $\widetilde\W^{n-1,1}$ corresponds to the space of traceless matrices $\mathbb T$ in $n$ dimensions; $\W^{2,2}$ corresponds to the space of algebraic curvature tensors $\mathbb{AC}$, encoding the symmetries of the Riemannian tensors ($(2,2)$-forms satisfying the algebraic Bianchi identity). More spaces and results can be found in \Cref{sec:4D-results}.

\subsection{Koszul operators and local polynomial spaces}

The Koszul operators (Poincar\'e operators on polynomial spaces) are an important tool for establishing exact sequences of polynomials \cite{eisenbud2013commutative}.   {They serve as an essential building block} for constructing Whitney forms and high-order versions for the de~Rham complex \cite{arnold2006finite,arnold2010finite,arnold2018finite,hiptmair1999canonical,hiptmair2001higher}. In this section, we further develop this technique and apply it to the construction of local shape function spaces of form-valued forms.

Recall that the Koszul operators $\kappa: C^{\infty}(\Omega)\otimes \Alt^{k}\to C^{\infty}(\Omega)\otimes \Alt^{k-1}$ are defined as 
\begin{equation}
\label{def:kappa}
\kappa\omega(v_{1}, \cdots, v_{k-1}):=\omega(x, v_{1}, \cdots, v_{k-1}), \quad \forall v_{1}, \cdots, v_{k-1}\in C^{\infty}(\Omega)\otimes V, 
\end{equation}
where $x$ is the Euler vector field (the vector field $\bm x:=(x_{1}, \cdots, x_{n})$). In vector proxies in $\mathbb{R}^{3}$, the $\kappa$ operator corresponds to $\otimes \bm x$, $\times \bm x$, and $\cdot \bm x$, respectively, for $k=3,2,1$.
For simplicity, we consider smooth forms in the presentation below. For smooth forms, we introduce the following {\it Koszul complex}:
\begin{equation}\label{deRham-koszul}
\begin{tikzcd}[column sep = small]
0\arrow{r} & C^{\infty} \otimes \alt^{n} \ar[r,"\kappa^{n}"] &  	C^{\infty} \otimes \alt^{n-1} \ar[r,"\kappa^{n-1}"] & \cdots  \ar[r,"\kappa^{2}"] & C^{\infty} \otimes \alt^{1} \ar[r,"\kappa^{1}"] & C^{\infty} \otimes \alt^{0}  \ar[r] & 0.
\end{tikzcd}
\end{equation}
 An important property is that $\kappa$ maps polynomials of degree $r$ to those of degree $r+1$.
The relationship between $d$ and $\kappa$ has been investigated in various contexts. See \cite{arnold2006finite,arnold2010finite,arnold2018finite} for applications in Finite Element Exterior Calculus.   

The operators in the Koszul complex \eqref{deRham-koszul} are, up to an overall sign, the symbols of the formal adjoints of the exterior derivatives, where one replaces derivatives $\frac{\partial}{\partial x^{i}}$ by the  multiplication by $x^{i}$. Next, we develop the BGG complex versions of \eqref{deRham-koszul}. For the BGG versions, we lose the interpretation of $\kappa$ as the contraction with the Eulerian vector field, but the interpretation of $\kappa$ as symbols of differential operators is still valid. We thus refer to the resulting complexes as the {\it symbol complexes}, or still the Koszul complexes. For example, in the symbol (Koszul) version of the elasticity complex, the operators $\sym\grad$, $\inc$ and $\div$ are replaced by $\sym(\bm x\otimes \bs)$, $\bm x\times \bs \times \bm x$, and $\bm x\cdot \bs$, respectively.

In two and three dimensions, symbol complexes with polynomial spaces for special instances of the BGG complexes have been studied in \cite{chen2022finiteelasticity,chen2022finitedivdiv}. These works provided direct proofs for the exactness of these polynomial symbol complexes. For the purpose of this paper, we will show the exactness, except at index $0$, of polynomial symbol complexes in any dimension for any forms. To this end, we view BGG diagrams from a different angle: the Koszul operators as ``differentials'' and the exterior derivatives as the null-homotopy operators. This serves as a new derivation of the BGG symbol complexes and establishes the desired properties.

The general forms of the symbol complexes call for more notations. For form-valued forms $C^{\infty}\otimes\alt^{k,\ell}$, there are two indices $k$ and $\ell$. Correspondingly, exterior derivatives and the Koszul operators can be defined for each of the slots.
To unify the notation, we use $\kappa^{k,\ell}: C^{\infty}\otimes\alt^{k,\ell}\to C^{\infty}\otimes\alt^{k-1,\ell}$ to denote the Koszul operator with respect to the first index. Recall that $d^{k,\ell}:C^{\infty}\otimes\alt^{k,\ell}\to C^{\infty}\otimes\alt^{k+1,\ell}$ is the exterior derivative in the first index. We also introduce Koszul type algebraic operators $K^{k, \ell}: C^{\infty}\otimes\alt^{k,\ell}\to C^{\infty}\otimes\alt^{k,\ell-1}$ and the exterior derivatives for the second index $D^{k, \ell}: C^{\infty}\otimes\alt^{k,\ell}\to C^{\infty}\otimes\alt^{k,\ell+1}$. The following identities are crucial for the construction in \cite{arnold2021complexes}:

\begin{equation}
\mathcal S^{k, \ell}=d^{k, \ell-1}K^{k, \ell}-K^{k+1, \ell}d^{k, \ell},
\end{equation}
and consequently, 
\begin{equation}
d^{k+1, \ell-1} \cS^{k, \ell}=-\cS^{k+1, \ell}d^{k, \ell}.
\end{equation}
As the two indices in form-valued forms play symmetric roles, we similarly have for the other two operators:
\begin{equation}\label{DPPD}
\cS_{\dagger}^{k, \ell}=D^{k-1, \ell}\kappa^{k, \ell}-\kappa^{k, \ell+1}D^{k, \ell},
\end{equation}
and consequently, 
\begin{equation}\label{SdaggerP}
\cS_{\dagger}^{k-1, \ell} \kappa^{k, \ell}=-\kappa^{k-1, \ell+1} \cS_{\dagger}^{k, \ell}.
\end{equation}

With the identities \eqref{DPPD} and \eqref{SdaggerP}, viewing \eqref{diagram-nD} from bottom to top and from right to left, we get
 \begin{equation}{\label{diagram-nD-2}
\begin{tikzcd}[column sep=tiny]
0 \arrow{r} & C^{\infty}\otimes\alt^{n, n}  \arrow{r}{\kappa} &C^{\infty}\otimes\alt^{n-1, n}  \arrow{r}{\kappa} & \cdots \arrow{r}{\kappa} & C^{\infty}\otimes \alt^{0,n} \arrow{r}{} & 0\\
0 \arrow{r} & C^{\infty}\otimes\alt^{n, n-1}  \arrow{r}{\kappa} \arrow[ur, "\cS_{\dagger}^{n,n-1}"] &C^{\infty}\otimes\alt^{n-1, n-1}  \arrow{r}{\kappa}  \arrow[ur, "\cS_{\dagger}^{n-1,n-1}"] & \cdots \arrow{r}{\kappa}  \arrow[ur, "\cS_{\dagger}^{1,n-1}"] & C^{\infty}\otimes \alt^{0, n-1} \arrow{r}{} & 0\\[-15pt]
 & \vdots & \vdots & {} & \vdots & {} \\[-15pt]
0 \arrow{r} & C^{\infty}\otimes\alt^{n,1}  \arrow{r}{\kappa} &C^{\infty}\otimes\alt^{n-1,1}  \arrow{r}{\kappa} & \cdots \arrow{r}{\kappa} & C^{\infty}\otimes \alt^{0,1} \arrow{r}{} & 0\\
0 \arrow{r} & C^{\infty}\otimes\alt^{n,0}  \arrow{r}{\kappa} \arrow[ur, "\cS_{\dagger}^{n, 0}"] &C^{\infty}\otimes\alt^{n-1,0}  \arrow{r}{\kappa}  \arrow[ur, "\cS_{\dagger}^{n-1, 0}"] & \cdots \arrow{r}{\kappa}  \arrow[ur, "\cS_{\dagger}^{1,0}"] & C^{\infty}\otimes \alt^{0,0} \arrow{r}{} & 0.
\end{tikzcd}
}\end{equation}

 Compared to the framework in \cite{arnold2021complexes}, here $\kappa$, $D$, and $\mathcal{S}_{\dagger}$ play the role of $d$, $K$, and $\mathcal{S}$ in \cite{arnold2021complexes}, respectively, thanks to the identities \eqref{DPPD} and \eqref{SdaggerP}. Therefore we can carry out a similar construction as in \cite{arnold2021complexes} to derive a Koszul version of the BGG complexes as follows.
\begin{theorem}[Koszul BGG Complexes]
The following sequence is a complex 
 \begin{equation} {\label{reduced-complex}
\begin{tikzcd}[column sep = small]
0 & C^{\infty} \otimes \W^{0,\ell-1}\ar[l] & \cdots  \ar[l,"\kappa",swap] & C^{\infty} \otimes \W^{\ell-2,\ell-1} \ar[l,"\kappa",swap] & C^{\infty} \otimes \W^{\ell-1,\ell-1} \ar[l,"\kappa",swap] & ~ \ar[l,"\kappa",swap] \\
 & ~&~ \ar[urrr,"\mathcal S_{\dagger}^{-1}"] &  C^{\infty} \otimes \widetilde{\W}^{\ell,\ell} \ar[l,"\kappa",swap] & \cdots   \ar[l,"\pi\circ \kappa",swap] & C^{\infty} \otimes \widetilde{\W}^{n,\ell} \ar[l,"\pi\circ \kappa",swap] & 0.\ar[l] 
\end{tikzcd}
}\end{equation}
Moreover, the polynomial Koszul complex is exact  except at index $0$ (the terminal space $\mathcal P_r\otimes\W^{0,\ell-1}$). 
\begin{equation} {\label{reduced-complex-poly}
\begin{tikzcd}[column sep = small]
0 & \mathcal P_r \otimes \W^{0,\ell-1}\ar[l] & \cdots  \ar[l,"\kappa",swap] & \mathcal P_{r-\ell + 2} \otimes \W^{\ell-2,\ell-1} \ar[l,"\kappa",swap] & \mathcal P_{r-\ell+1} \otimes \W^{\ell-1,\ell-1} \ar[l,"\kappa",swap] & ~ \ar[l,"\kappa",swap] \\
~&~ & \ar[urrr,"\mathcal S_{\dagger}^{-1}"] & \mathcal{P}_{r-\ell-1} \otimes \widetilde{\W}^{\ell,\ell} \ar[l,"\kappa",swap] & \cdots   \ar[l,"\pi\circ \kappa",swap]  & \mathcal P_{r-n-1} \otimes \widetilde{\W}^{n,\ell} \ar[l,"\pi\circ \kappa",swap] & 0.\ar[l] 
\end{tikzcd}
}\end{equation}
\end{theorem}
Note that for $i \le \ell -1$, $\kappa$  maps $C^{\infty} \otimes \mathbb W^{i, \ell - 1} $ to $C^{\infty} \otimes \mathbb W^{i-1, \ell - 1}$ due to the anticommutativity \eqref{SdaggerP}.

Polynomial Koszul complexes will be the local shape functions of finite element complexes. Let $\cP_r$ be the polynomial space with degree $\le r$, and $\mathcal H_r$ be the  homogeneous polynomial space with degree $ = r$. We first recall the Koszul version of the de~Rham complex with polynomials \cite{arnold2018finite}:
 \begin{equation*}
{\begin{tikzcd}[column sep = small]
0 \arrow{r}&\mathcal P_r^-  \alt^{n} \ar[r,"\kappa"] &  	\mathcal P_r^- \alt^{n-1} \ar[r,"\kappa"] & \cdots \ar[r,"\kappa"] & \mathcal P_r^- \alt^{1} \ar[r,"\kappa"] & \mathcal P_r^-  \alt^{0} \ar[r] & 0,
\end{tikzcd}
}\end{equation*}
where  $\mathcal P_r^-  \alt^{k}:=\mathcal P_{r-1}  \alt^{k}+\kappa^{k+1}\mathcal P_{r-1}  \alt^{k+1} = \mathcal P_{r-1} \alt^k \oplus \kappa^{k+1} \mathcal H_{r-1} \alt^{k+1}$.

Similarly, the Koszul spaces for form-valued forms are defined by
\begin{equation}\label{def:P-}
\mathcal P^{-}_{r} \Alt^{k,\ell} := \mathcal P_{r-1} \Alt^{k,\ell} + \kappa^{k+1} \mathcal P_{r-1}\Alt^{k+1,\ell} = (\mathcal P_{r-1}\Alt^k + \kappa^{k+1} \mathcal P_{r-1} \Alt^{k+1}) \otimes \Alt^\ell.
\end{equation}

By the  anticommutativity of $\kappa$ and $\mathcal S_{\dagger}$, the  restriction of $\mathcal S_{\dagger}$ to trimmed polynomial spaces is well defined.
 \begin{lemma}\label{lem:S-poly}
Let $A_r^{k,\ell}$ denote the restriction of $\mathcal S_{\dagger}^{k,\ell}$  from $\mathcal P_r^-\alt^{k,\ell}$ to $\mathcal P_r^-\alt^{k-1,\ell+1}$. Then $A_r^{k,\ell}$ is surjective for $k\le\ell$ and injective for $k\ge\ell+1$. 
\end{lemma}
\begin{proof}
 Well-definedness follows from \eqref{SdaggerP}. Surjectivity follows by applying the algebraic surjectivity to the constant-degree part and to a Koszul preimage. Injectivity in the stated range follows from algebraic injectivity pointwise.
\end{proof}
%

Moreover, we have the following characterization of 
\begin{equation}\label{def:Pr-}
\mathcal P_r^- \W^{k,\ell} := \ker(\mathcal S_{\dagger}^{k,\ell} : \mathcal P_r^- \alt^{k,\ell} \to \mathcal P_r^- \alt^{k-1,\ell+1})
\end{equation}
 whenever $ k \le \ell$.
\begin{lemma}\label{charact-ker-Sdagger}
For $0<k<\ell$, we have 
\begin{equation}
	{\mathcal P_r^- \W^{k,\ell}} =   \cP_{r-1} \W^{k,\ell} \oplus \kappa \mathcal H_{r-1} \W^{k+1,\ell}.
\end{equation}

For $k=\ell>0$, we have 
 {
\[{\mathcal P_r^- \W^{k,\ell}} = \cP_{r-1} \W^{k,\ell} \oplus  \kappa (\mathcal S^{k+1,\ell}_{\dagger})^{-1} \kappa \mathcal H_{r-2} \alt^{k+1,\ell+1}.\]} Note that $\mathcal S_{\dagger}^{k+1,\ell}$ is an isomorphism.
For $k=0$, $\mathcal P_r^-\W^{0,\ell}=\mathcal P_r^-\alt^{0,\ell}$.
\end{lemma}
\begin{proof}
We prove the lemma by diagram chasing argument; see the following diagram. 
\begin{equation}
    \begin{tikzcd}
        ~ & \alt^{k,\ell} \ar[phantom, " \color{blue}a", pos=0, shift={(0.5,0.5)}]
         \ar[dl,"{\mathcal S_{\dagger}}",swap] 
         & \alt^{k+1,\ell} \ar[l,"{\kappa}",swap] \ar[dl,"{\mathcal S_{\dagger}}"]  \ar[phantom, "\color{blue} b", pos=0, shift={(0.5,0.5)}] & \alt^{k+2,\ell}\ar[phantom, "\color{blue} e", pos=0, shift={(0.5,0.5)}] \ar[l,"{\kappa}",swap] \ar[dl,"{\mathcal S_{\dagger}}"] \\
         \alt^{k-1,\ell+1} & \alt^{k,\ell+1 }\ar[l,"{\kappa}",swap]  & \alt^{k+1,\ell+1 }\ar[l,"{\kappa}",swap] \ar[phantom, "\color{blue}c", pos=0, shift={(0.5,-0.5)}] & 
    \end{tikzcd}
\end{equation}
Suppose that $a + \kappa b$ lies in the kernel of $\mathcal S_{\dagger}$, where $a \in \mathcal P_{r-1} \alt^{k,\ell}$ and $b \in \mathcal H_{r-1} \alt^{k+1,\ell}$. By the  anticommuting property of $\mathcal S_{\dagger}$ and $\kappa$, it holds that $\mathcal S_{\dagger} a = 0$ and $\kappa \mathcal S_{\dagger} b = 0$. Since $\mathcal S_{\dagger} b \in \mathcal H_{r-1} \alt^{k, \ell + 1}$, it follows from the exactness of the Koszul complex with the spaces of homogeneous polynomials that 
$\mathcal S_{\dagger} b = \kappa c$ for some $c \in \mathcal H_{r-2} \alt^{k+1,\ell+1}$.

When $k = \ell$, $\mathcal S_{\dagger}^{k+1,\ell}$ is an isomorphism. Therefore, it holds that $b = (\mathcal S_{\dagger}^{k+1,\ell})^{-1}\kappa c$. 
When $k < \ell$, $\mathcal S_{\dagger}^{k+1,\ell}$ is onto, with right inverse denoted as  $(\mathcal S_{\dagger}^{k+1,\ell})^{-1}$. Therefore, it holds that $b$ lies in  {
\[\ker(\mathcal S_{\dagger}) + (\mathcal S_{\dagger}^{k+1,\ell})^{-1}(\kappa \mathcal H_{r-2}\alt^{k+1,\ell+1}) = \mathcal H_{r-1}\mathbb W^{k+1,\ell} + (\mathcal S_{\dagger}^{k+1,\ell})^{-1}(\kappa\mathcal H_{r-2}\alt^{k+1,\ell+1}).\]}

To conclude the result, it suffices to show that for each $c \in \mathcal H_{r-2}\alt^{k+1,\ell+1}$, $\kappa \mathcal S_{\dagger}^{-1} \kappa c$ is in $\kappa \mathcal H_{r-1} \mathbb W^{k+1,\ell}$. Since $k +1 \le \ell$, the operator  $\mathcal S_{\dagger}^{k+2,\ell}$ is onto. Therefore, there exists $e$ such that  $\mathcal S_{\dagger}^{k+2,\ell}e = -c$. Let $b' = \mathcal S_{\dagger}^{-1} \kappa c - \kappa e \in \mathcal H_{r-1}\alt^{k+1,\ell}$, then it follows that $\mathcal S_{\dagger}b' = \kappa c - \mathcal S_{\dagger} \kappa e = 0$. Therefore, $b' \in \mathcal H_{r-1} \W^{k+1,\ell}$, and we have 
\[
 \kappa \mathcal S_{\dagger}^{-1} \kappa c = \kappa b' + \kappa \kappa e =  \kappa b'\in \kappa \mathcal H_{r-1}\W^{k+1, \ell}.
\]

 This concludes the proof. 
\end{proof}

\begin{remark}
    Due to the existence of null spaces of $\kappa$, in general $\dim \kappa \mathcal H_{r-1}\Alt^{k+1} \neq \dim \mathcal H_{r-1}\Alt^{k+1}$; likewise, in general $\dim \kappa \mathcal H_{r-1}\mathbb W^{k+1,\ell} \neq \dim \mathcal H_{r-1}\mathbb W^{k+1,\ell}$.
\end{remark}

\subsection{Iterated constructions}
\label{sec:iterated}
We consider the BGG diagram of algebraic forms between row $k$ and row $\ell$:
 \begin{equation}{\label{diagram-nD-algebraic}
\begin{tikzcd}
  \alt^{0,k}    & \alt^{1,k}   & \cdots  &   \alt^{n,k}  \\
  \alt^{0,k+1}  \arrow[ur, "\mathcal S^{0,k+1}"] &\alt^{1,k+1}   \arrow[ur, "\mathcal S^{1,k+1}"] & \cdots   \arrow[ur, "\mathcal S^{n-1,k+1}"] &   \alt^{n,k+1}\\[-15pt]
  \vdots & \vdots & {} & \vdots  {} &{}\\[-15pt]
 \alt^{0,\ell-1}   & \alt^{1,\ell-1}  & \cdots  &   \alt^{n,\ell-1} \\
 \alt^{0,\ell}  \arrow[ur, "\mathcal S^{0,\ell}"] & \alt^{1,\ell}   \arrow[ur, "\mathcal S^{1,\ell}"] & \cdots   \arrow[ur, "\mathcal S^{n-1,\ell}"] &  \alt^{n,\ell} .
\end{tikzcd}
}\end{equation}

Define 
\begin{equation} \label{def:Sp}\mathcal S^{k,\ell}_{[p]}: \alt^{k,\ell} \to \alt^{k+p,\ell - p}\end{equation}
by \[\mathcal S^{k,\ell}_{[p]} = \cS^{k+p-1,\ell-p+1 } \circ \cdots \circ \cS^{k+1,\ell-1} \circ \cS^{k,\ell}.\] 
An illustration of iterated constructions with $p=2$ is shown below. 
\begin{equation}
\begin{tikzcd}
  \alt^{0,k-1} & \alt^{1,k-1} & \alt^{2,k-1} & \cdots & \alt^{n,k-1} \\
  \alt^{0,k} & \alt^{1,k} & \alt^{2,k} & \cdots & \alt^{n,k} \\
  \alt^{0,k+1} \arrow[uurr, "{\color{red} \mathcal{S}^{0,k+1}_{[2]}}", color = red] & \alt^{1,k+1} \arrow[uurr, "{\color{red}\mathcal{S}^{1,k+1}_{[2]}}",color=red] & \alt^{2,k+1} & \cdots & \alt^{n,k+1} \\
  \vdots & \vdots & \vdots & {} & \vdots \\
  \alt^{0,\ell-1} & \alt^{1,\ell-1} & \alt^{2,\ell-1} & \cdots & \alt^{n,\ell-1} \\
  \alt^{0,\ell} \arrow[uurr, "{\color{red}\mathcal{S}^{0,\ell}_{[2]}}",color=red] & \alt^{1,\ell} \arrow[uurr, "{\color{red}\mathcal{S}^{1,\ell}_{[2]}}",color=red] & \alt^{2,\ell} & \cdots & \alt^{n,\ell}
\end{tikzcd}
\end{equation}

Note that for large $p$, the above map can be zero. We also define 
\begin{equation}\label{def:Sdaggerp}\mathcal S^{k,\ell}_{\dagger, [p]} : \alt^{k,\ell} \to \alt^{k-p,\ell + p}\end{equation}
by \[\mathcal S^{k,\ell}_{\dagger, [p]} = \cS^{k-p+1,\ell+p-1 }_{\dagger} \circ \cdots \circ \cS^{k-1,\ell+1}_{\dagger} \circ \cS^{k,\ell}_{\dagger}.\]

Similar to the non-iterated case $p = 1$, we also have the adjointness, and the injectivity/surjectivity. The proof is given in the appendix.
\begin{lemma}
\label{lem:Sp-inj-sur}
We have the following properties.
	\begin{enumerate}
    \item $\mathcal S^{k,\ell}_{\dagger,[p]}$ and $\mathcal S^{k-p, \ell+p}_{[p]}$ are adjoint with respect to the  induced exterior inner products, and therefore $\ker(\mathcal S_{\dagger,[p]}^{k,\ell}) =\ran (\mathcal S^{k-p, \ell+p}_{[p]})^{\perp}$.
			\item  When $k \le \ell+p$, $\mathcal S^{k,\ell}_{\dagger,[p]}$ is surjective, and $\mathcal S^{k-p, \ell+p}_{[p]}$ is injective.
				\item  When $k \geq \ell-p$, $\mathcal S^{k,\ell}_{[p]} $ is surjective while $\mathcal S_{\dagger,[p]}^{k+p, \ell-p}$ is injective.
	\end{enumerate}
\end{lemma}

We then define $\mathbb W_{[p]}^{\bs,\bs}$ {and $\widetilde{\W}_{[p]}^{\bs,\bs}$} as 
\begin{equation}\label{def:Wp}
\W^{k, \ell}_{[p]} := \ker(\cS_{\dagger,[p]}^{k,\ell}) \subseteq \alt^{k,\ell} ,  \text{ when } k \le \ell + p - 1,
\end{equation}
and 
\begin{equation}\label{def:tildeWp}
\widetilde\W^{k, \ell}_{[p]} := \ker(\cS_{[p]}^{k,\ell}) \subseteq \alt^{k,\ell} ,  \text{ when } k \ge \ell - p + 1.
\end{equation}

Therefore, the BGG complexes can be derived for the iterated constructions.  The following complexes are new, enriching the study in \cite{arnold2021complexes}, where the non-iterated case ($p = 1$) was discussed systematically and for $p>1$, only examples in three dimensions were presented.
\begin{theorem}[BGG complexes for iterated constructions]
The following sequence is a complex (referred to as the BGG complexes linking rows $\ell$ and $\ell + p$)  \begin{equation} {\label{BGG-seq-p}
\begin{tikzcd}[column sep = small]
0\ar[r] & C^{\infty} \otimes \W^{0,\ell}_{[p]} \ar[r,"\pi\circ d"] & C^{\infty} \otimes \W^{1,\ell}_{[p]}  \ar[r,"\pi\circ d"] & \cdots   \ar[r,"\pi\circ d"] & C^{\infty} \otimes \W^{\ell + p - 1,\ell}_{[p]}  \ar[r,"d"] & ~ \ar[llld,"\mathcal S_{[p]}^{-1}"] \\
& & ~ \ar[r,"d"] & C^{\infty} \otimes \widetilde\W^{\ell+1,\ell+p}_{[p]}  \ar[r,"d"]  & \cdots  \ar[r,"d"] & C^{\infty} \otimes \widetilde\W^{n,\ell+p}_{[p]} \ar[r] & 0.
\end{tikzcd}
}\end{equation}
The cohomology of \eqref{BGG-seq-p} is isomorphic to $ (\mathcal H_{dR}^{\bs}(\Omega)\otimes \Alt^{\ell}) \oplus (\mathcal H_{dR}^{\bs - p + 1}(\Omega)\otimes  \Alt^{\ell+p})$, where $ \mathcal H_{dR}^{\bs}(\Omega)$ is the de~Rham cohomology. 
\end{theorem}

\begin{proof}
Write $K$ for contraction with $x$ in the second factor. Since $\mathcal S=dK-Kd$ and $d\mathcal S=-\mathcal S d$, induction gives
{
\[
d\mathcal S_{[p]}=(-1)^p\mathcal S_{[p]}d,\qquad
\mathcal S_{[p]}=d(K\mathcal S_{[p-1]})-(-1)^{p-1}(K\mathcal S_{[p-1]})d,
\]}
where $\mathcal S_{[0]}=\mathrm{id}$. Shift the row with coefficient space $\Alt^{\ell+p}$ by $p-1$ degrees and give it differential $(-1)^{p-1}d$. These identities provide the anticommuting linking map and its null-homotopy. By \Cref{lem:Sp-inj-sur}, the linking map is injective before its isomorphism at $\mathcal S_{[p]}^{\ell,\ell+p}$ and surjective afterwards. The algebraic BGG reduction of \cite[Theorems 6 and 7]{arnold2021complexes} therefore gives \eqref{BGG-seq-p}, up to harmless choices of signs on the spaces. Its cohomology is the direct sum of the two input cohomologies with this degree shift.
\end{proof}
\begin{remark}
The cohomology of \eqref{BGG-seq-p} is isomorphic to the sum of the input, i.e., $\mathcal{H}^{k}_{\ell, \ell+p}\cong (\mathcal{H}^{k}_{dR}\otimes \Alt^{\ell})\oplus (\mathcal{H}^{k-(p-1)}_{dR}\otimes \Alt^{\ell+p})$. In particular, on a topologically trivial domain --- where we have $\mathcal{H}^{0}_{dR}\cong \mathbb{R}$ --- for $p>1$, in the middle of the derived complex with $k=p-1$, $\mathcal{H}^{p-1}_{\ell, \ell+p}\cong (\mathcal{H}^{p-1}_{dR}\otimes \Alt^{\ell})\oplus (\mathcal{H}^{0}_{dR}\otimes \Alt^{\ell+p})\cong \Alt^{\ell+p}$. This indicates that even on topologically trivial domain, the iterated complexes have a nontrivial cohomology in the middle. This does not happen with $p=1$. Numerical consequences of this fact can be found in \cite{hu2023spurious}.  
 \end{remark}

Regarding the Koszul spaces, we have the following results. This leads to a candidate for the local shape function space in the iterated diagram. 
\begin{lemma}
\label{lem:Pr-Wp}
The  restrictions $A_{r,[p]}^{k,\ell}:=\mathcal S_{\dagger,[p]}^{k,\ell}|_{\mathcal P_r^-\alt^{k,\ell}}$ are well defined. Their finite-dimensional adjoints are understood below.
\begin{enumerate}
\item When  $k\le\ell+p-1$, $A_{r,[p]}^{k,\ell}$ is surjective and its adjoint is injective.
\item When  $k\ge\ell+p$, $A_{r,[p]}^{k,\ell}$ is injective and its adjoint is surjective.
		\item For $p\le k<\ell+p-1$, the space ${\mathcal P_r^- \W^{k,\ell}_{[p]}}$ is the kernel of $\mathcal S_{\dagger,[p]}^{k,\ell},$ and is characterized as 
\[{\mathcal P_r^- \W^{k,\ell}_{[p]}} = \cP_{r-1} \W^{k,\ell}_{[p]} \oplus \kappa \mathcal H_{r-1} \W^{k+1,\ell}_{[p]} 
.
\]
\item  For $k=\ell+p-1\ge p$, the space ${\mathcal P_r^- \W^{k,\ell}_{[p]}}$ is the kernel of $\mathcal S_{\dagger,[p]}^{k,\ell},$ and is characterized as 
 {
\[{\mathcal P_r^- \W^{k,\ell}_{[p]}} = \cP_{r-1} \W^{k,\ell}_{[p]} \oplus  \kappa (\mathcal S^{k+1,\ell}_{\dagger,[p]})^{-1} \kappa \mathcal H_{r-2} \alt^{k-p+2,\ell+p}. \]}
	\end{enumerate}
For $0\le k<p$, $\mathcal P_r^-\W_{[p]}^{k,\ell}=\mathcal P_r^-\alt^{k,\ell}$.
\end{lemma}
\begin{proof}
    The proofs of (1) and (2) can be found in the appendix. The proofs of (3) and (4) follow a similar argument for \Cref{charact-ker-Sdagger}.
\end{proof}

Finally, we discuss some properties of the symmetric spaces. The spaces $\W_{[p]}^{k,\ell}$ form a nested sequence in the sense that $\W_{[p]}^{k,\ell}\subseteq \W_{[q]}^{k,\ell}$ if $p\leq q$. This can also be observed through the following telescope decomposition. 
Starting from $\alt^{k,\ell}$ with $k \le \ell$, then it follows that 
\[\alt^{k,\ell} = \W^{k,\ell} \oplus \mathcal S^{k-1,\ell + 1} \alt^{k-1,\ell+1}.\]
By the definition of $\mathbb W$ and the adjointness of $\mathcal S$ and $\mathcal S_{\dagger}$, the above decomposition is an orthogonal decomposition.

We can continue decomposing $\alt^{k-1,\ell+1}$ on the right-hand side and obtain
\begin{equation}
\label{altkl}
    \alt^{k,\ell} = \W^{k,\ell} \oplus \mathcal S^{k-1,\ell + 1} \W^{k-1,\ell+1} \oplus \mathcal S^{k-2,\ell+2}_{[2]} \W^{k-2,\ell+2} \oplus \cdots \oplus \mathcal S^{0,\ell + k}_{[k]} \alt^{0,\ell + k}.
\end{equation}

Therefore, we have
 {
\begin{align*}
\mathbb W^{k,\ell} & =  \mathbb W^{k,\ell}, \\ 
\mathbb W^{k,\ell}_{[2]} & = \mathbb W^{k,\ell} \oplus \mathcal S^{k-1,\ell + 1} \W^{k-1,\ell+1}, \\ 
\cdots & = \cdots \\ 
\mathbb W^{k,\ell}_{[p]} & = \mathbb W^{k,\ell} \oplus \mathcal S^{k-1,\ell + 1} \W^{k-1,\ell+1} \oplus \cdots \oplus \mathcal S_{[p-1]}^{k-p+1, \ell + p -1} \mathbb W^{k-p+1, \ell + p -1}.
\end{align*}}
The right-hand side also provides an orthogonal decomposition.

    We take (2,2)-forms as an example, which often appears in the study of geometry. In this case, we have 
    \begin{equation}
    \label{alt22}
    \alt^{2,2} = \W^{2,2} \oplus \mathcal S^{1,3} \W^{1,3} \oplus \mathcal S^{0,4}_{[2]} \Alt^{0,4}.\end{equation}
    The decomposition gives three parts. The first part is the {\it algebraic curvature space}, namely, symmetric (2,2)-forms satisfying the algebraic Bianchi identities. Here, symmetric (2,2)-forms mean 4-tensors which are skew-symmetric with the first two indices and the last two indices, respectively, and symmetric when we exchange the two groups, i.e., $A_{ij, kl}=-A_{ji,kl}=-A_{ij,lk}=A_{kl,ij}$. A standard way to get the algebraic curvature space $\W^{2, 2}$ is to consider the Bianchi map $b : \sym\alt^{2,2} \to \alt^{4} \subseteq  \sym\alt^{2,2}$, which is defined by skew-symmetrization (see \eqref{def:bianchi-map} below). Here, $ \sym\alt^{2,2}$ is the space of {\it symmetric (2,2)-forms}, i.e., 4-tensors $R_{ijkl}$ satisfying $R_{ijkl}=-R_{jikl}=-R_{ijlk}=R_{klij}$. 
      It can be shown that $b$ is a projection and $\W^{2, 2}=\ker(b)$, the kernel of skew-symmetrization \cite{gallot1990riemannian}.  
    We explain this construction using \eqref{alt22}. First, we can show that $\mathcal S^{1,3}\mathbb W^{1,3}$ gives the skew-symmetric (2,2)-forms.
        \begin{lemma}
        $\mathcal S^{1,3}\mathbb W^{1,3} = \skw \alt^{2,2}$, where $ \skw \alt^{2,2}$ is the space of skew-symmetric (2, 2)-forms. 
    \end{lemma}
    \begin{proof}
Let $\omega \in \mathbb W^{1,3}$. Then, for any $X, Y, Z, W $,
     {
\begin{align*}
    \mathcal S^{1,3}& \omega(X,Y)(Z,W) + \mathcal S^{1,3} \omega(Z,W)(X,Y)
\\&     = {\omega(Y)(X,Z,W) - \omega(X)(Y,Z,W) + \omega(W)(Z,X,Y) - \omega(Z)(W,X,Y)} \\&= {-(\mathcal S^{1,3}_{\dagger}\omega)(\cdot)(X,Y,Z,W)=0.}
    \end{align*}}
    This shows    $\mathcal S^{1,3}\mathbb W^{1,3} \subseteq \skw \alt^{2,2}$. 
    Since $\mathcal S^{1,3}$ is injective, it then suffices to check $\dim \W^{1,3} = \dim \skw \alt^{2,2}$, which holds since 
    \[\binom{n}{1} \binom{n}{3} - \binom{n}{4} = \frac{1}{2}(\binom{n}{2})(\binom{n}{2} - 1).\]
            \end{proof}
     Second, the Bianchi  projection, after identifying 4-forms with alternating (2,2)-forms, is $\frac{1}{24}\mathcal S_{[2]}^{0,4}\mathcal S_{\dagger,[2]}^{2,2}$.     Therefore, in terms of the telescoping decomposition, this is the projection from the symmetric $(2,2)$-forms (the sum of the first and third terms, $\W^{2,2} \oplus \mathcal S^{0,4}_{[2]} \Alt^{0,4}$) onto the fully alternating summand $\mathcal S^{0,4}_{[2]} \Alt^{0,4}$. Its kernel is the algebraic curvature space $\W^{2,2}$.

 It can be verified that 
 {
\[\mathcal S^{0,4}_{[2]}\omega (X,Y)(Z,W) = -2 \omega (X,Y,Z,W)\]} 
for any $\omega \in \Alt^4 \cong \Alt^{0,4}.$

The Bianchi map maps $\sym \Alt^{2,2} \to \Alt^{4}$, defined as 
 {
\begin{equation}\label{def:bianchi-map}
b\omega(X,Y,Z,W) = \frac{1}{3} (\omega(X,Y)(Z,W) + \omega(Y,Z)(X,W) + \omega(Z,X)(Y,W)).
\end{equation}}
We can show that 
 {
\[\mathcal S_{\dagger,[2]}^{2,2}\omega = -12b\sym \omega,\]}
and $\mathcal S^{0,4}_{[2]}$ is the inclusion from $\Alt^{4} \to \sym \Alt^{2,2}$ (up to a scalar multiple).

The  projection from $\Alt^{2, 2}$ to the first term in \eqref{alt22} is  the algebraic-curvature projection; the projection to the third term is called the {\it Bianchi symmetrization}  \cite{gallot1990riemannian}.

\begin{remark}\label{rmk:weyl}
    The study of differential geometry involves further decompositions of tensors. For example, the trace operator $\W^{2,2} \to \W^{1,1} = \sym \Alt^{1,1}$  leads to the Ricci decomposition of Riemann tensors (decomposing the Riemann tensor into the Ricci part and the Weyl part).  Up to an overall sign, the trace operator can  also be constructed using the $\mathcal S$ operator: 
     {
\[\Alt^{2,2} \xrightarrow{\star} \Alt^{n-2,2} \xrightarrow{\mathcal S} \Alt^{n-1,1} \xrightarrow{\star} \Alt^{1,1}.\]}
    or, using a Hodge star for the second index, 
     {
\begin{equation}\label{trace-operator}
    \Alt^{2,2} \xrightarrow{\star} \Alt^{2,n-2} \xrightarrow{\mathcal S_{\dagger}} \Alt^{1,n-1} \xrightarrow{\star} \Alt^{1,1}.
    \end{equation}}

    The adjoint of traces (restricted to symmetric (1, 1)-forms) is usually constructed explicitly via {\it Kulkarni–Nomizu operator} $\bigcirc \hspace{-1em} \wedge \hspace{0.2em} g $. For example, dualizing \eqref{trace-operator} by a Hodge star  gives, up to the sign and normalization of the Kulkarni--Nomizu product,
     {
\[(\bigcirc \hspace{-1em} \wedge \hspace{0.2em} g ) \circ \sym : \Alt^{1,1} \xrightarrow{\star\circ\sym} \Alt^{1,n-1} \xrightarrow{\mathcal S} \Alt^{2,n-2} \xrightarrow{\star} \Alt^{2,2}.\]}
    A further investigation of Ricci decomposition through the BGG paradigm will be discussed in the future.
\end{remark}

\section{Generalized traces for differential forms}
\label{sec:trace}

To introduce the continuity condition of the form-valued forms, we generalize the concept of traces of differential forms.  
 Let $\Omega$ be a bounded Lipschitz domain and $F\subseteq {\partial} \Omega$ be a submanifold.  The {\it trace} operator $\iota^*$ is defined as the pullback of the inclusion operator  $\iota: F\rightarrow \overline\Omega$. That is, for  $F\subseteq\overline\Omega$, 
$\iota_{F}^{\ast}: C^{\infty} \otimes \alt^k(\Omega) \to C^{\infty} \otimes \alt^k(F)$ is defined by
\begin{equation}\label{def:trace}
\iota_{F}^{\ast}\omega(v_{1}, \cdots, v_{k})=\omega(\iota_{F,\ast}v_{1}, \cdots, \iota_{F,\ast}v_{k}), \quad \omega\in C^{\infty}\alt^{k}(\Omega),
\end{equation}
where $v_{1}, \cdots v_{k}$ are any vector fields on $F$. Here the pushforward $\iota_{F,\ast}$ is an inclusion from the tangent bundle $TF$ to $T\Omega$.

 When $k > \dim(F)$, the pullback $\iota_{F}^*$ vanishes on $k$-forms. To see that the classical trace operator is not enough for our purpose, consider the elasticity complex (Figure \ref{fig:table}), where 1-form-valued 0-forms (the first nonzero space in the complex) are discretized by a vector Lagrange element. The vertex degrees of freedom of the Lagrange element cannot be interpreted as the trace of  1-form-valued 0-forms, as $\iota_{F}^*$ of a 1-form at vertices vanishes. This demonstrates that {a generalized notion of trace operators} for $k$-forms is necessary when $k > \dim (F)$ and actually also necessary for $k \le \dim(F)$.

To interpret the vertex evaluation degrees of freedom appearing in the Lagrange elements in the context of differential forms, 
 we can define {\it restrictions} of differential forms:
\begin{equation}\label{def:rho}\rho_{F}^{\ast} : C^{\infty}(\Omega) \otimes \alt^k(\mathbb R^n) \to C^{\infty} (F) \otimes \alt^k(\mathbb{R}^{n}),\end{equation} which regards a $k$-form as a $k$-form-valued $0$-form and takes the trace of 0-forms.

To generalize the spaces in Figure \ref{fig:table} to form-valued forms, we need a generalized notion of traces to define the continuity conditions.
Specifically, we will first define the generalized trace $\jmath_{F}^*$, extending the definition of the trace $\iota^{*}_{F}$ to $k$-forms with $k > \dim F$. Note that with the original definition, $\iota^*_{F}$ will vanish in this case. In the next step, we construct a family of linear functionals $\jmath_{F,[p]}^*$ such that they interpolate between the trace $\jmath_{F}^* = \jmath_{F,[1]}^*$ and the restriction  $\rho_F^* = \jmath_{F,[n+1]}^*$. The generalized trace operator $\jmath_{F,[p]}^*$ is used to characterize the continuity of the finite element spaces from the iterated constructions.

\subsection{Generalized trace} 

 We begin with a tangential-normal decomposition formula of alternating forms.

The following result, stating that the exterior algebra of a direct sum is isomorphic to the tensor product of the exterior algebras, is a basic fact of exterior algebra (see,  \cite{lang2002algebra}, for example). For completeness, we include a proof. 
\begin{lemma}[Tangential-normal decomposition on forms]
\label{lem:tan-nor}
For every fixed simplex $F$, the following tangential-normal decomposition holds:
\begin{equation}\label{decomp:altkRn}
\alt^k(\mathbb R^n) \cong \bigoplus_{q = 0}^{k} \alt^q(F) \otimes \Alt^{k-q}(F^{\perp}).
\end{equation}
Here we identify $F$ with its tangent space in $\mathbb{R}^{n}$ and $F^{\perp}$ is the orthogonal complement in the  Euclidean space.
\end{lemma} 

\begin{proof}
The isomorphism \eqref{decomp:altkRn}, denoted as $\bm \Theta$, can be explicitly given as follows. Suppose that $\dim F=m$ and $\bm \alpha_1,\cdots, \bm \alpha_m$ are a basis of $F$, and $\bm\alpha_{m+1},\cdots, \bm \alpha_{n}$ are a basis of $F^{\perp}$. Let $d\alpha^i$ be the dual basis of 1-forms, i.e., $\langle d\alpha^i, \bm \alpha_{j}\rangle =\delta_{ij}$.  Here $\delta_{ij}$ is the Kronecker's delta.  Clearly, $d\alpha^i, i = 1,2,\cdots, m$ span the space of 1-forms on $F$, and $d\alpha^i, i = m+1,\cdots, n$ span the 1-forms on $F^{\perp}$. Set
\[\bm \Theta(d\alpha^{i_1}\wedge \cdots d\alpha^{i_k})  
 = \underbrace{\bigwedge_{s\in [1:k]: i_s \le m} d\alpha^{i_s}}_{\text{tangential part}} \otimes \underbrace{\bigwedge_{s\in [1:k]: i_s > m} d\alpha^{i_s}}_{\text{normal part}},
\]
where $i_1<\cdots<i_k$ and $[1:k]$ is the set $\{ 1, 2, \cdots, k\}$. 
 Then, $\bm \Theta$ maps one basis in $\alt^k(\mathbb R^n)$ to  $\alt^{q}(F) \otimes \alt^{k - q} (F^{\perp})$, where $q := \#\{i_s \le m\}.$ 
Clearly, the decomposition does not depend on the choice of basis of $F$ and $F^{\perp}$.
\end{proof}

\begin{remark}
The result can be also found in \cite[Chap XIX, Proposition 1.4]{lang2002algebra}
\end{remark}

 In particular, when $k = 1$, the above decomposition gives $\alt^1(\mathbb R^n)= \alt^1(F) \oplus \alt^1(F^{\perp})$, indicating that a dual vector can be split into two parts: (dual) tangential part and (dual) normal part. The decomposition of $k$-forms essentially comes from the tensor product of the results of $1$-forms.

 A consequence of \Cref{lem:tan-nor} is that we can define the algebraic projection of a $k$-form to the components with $q$ components tangent to $F$ and $k-q$ components normal to $F$:
\begin{equation}\label{def:vartheta}\vartheta_{F,q}^*: C^{\infty}(\Omega) \otimes \alt^{k}(\mathbb R^n) \to C^{\infty}(F) \otimes \alt^q(F) \otimes \alt^{k-q}(F^{\perp}). \end{equation}
More precisely, for  $1\le i_1<\cdots<i_k\le n$, the map $\vartheta_{F,q}^*$ sends a monomial
\[
 f d\alpha^{i_1}\wedge \cdots \wedge d\alpha^{i_k} \mapsto
 \begin{cases}
  f|_F \bigwedge_{i_s \le m} d\alpha^{i_s} \otimes \bigwedge_{ i_s > m} d\alpha^{i_s},\quad \mbox{~if there are $q$ indices $i_s \le m$,} \\
  0, \quad \mbox{~otherwise~}.
  \end{cases}
 \]
 The extension of $\vartheta_{F,q}^*$ to a combination of monomials  $\sum_{1\le i_1<\cdots<i_k\le n} f d\alpha^{i_1}\wedge \cdots \wedge d\alpha^{i_k}$ is defined by linear combination. 
It is easy to see that $\vartheta_{F,k}^* = \iota_{F}^*$ on $k$-forms. Intuitively, $\vartheta_{F,q}^*$ preserves the $k$ forms that have $q$ tangential components and map others to zero.

  The general idea of the generalized trace operator is to use tangential vectors as much as possible. When  $k \ge \dim(F)$, we  feed all the $\dim F$ tangent vectors of $F$ to the $k$-form, and in addition, we use $k-\dim F$ normal vectors. 
  This leads to the following definition of a \emph{generalized trace} on lower dimensional simplices:
\begin{equation}\label{def:jstar}
\jmath_F^* := \vartheta_{F,\dim F}^*  : C^{\infty}(\Omega) \otimes \alt^{k}(\mathbb R^n) \to C^{\infty}(F) \otimes \alt^{\dim F}(F) \otimes \alt^{k-\dim F}(F^{\perp})\text{ if } \dim F \leq  k.
\end{equation}
Here $\alt^{\dim F}(F) $ is the volume form on $F$, which is unique up to a scalar multiple.  Therefore, if $\dim F \leq k$, the range of $\jmath_F^*$ can be identified with $C^{\infty}(F) \otimes \alt^{k-\dim F}(F^{\perp})$,  understood as  the space of $\alt^{k-\dim F}(F^{\perp})$-valued smooth functions on $F$. For $\dim F > k$, $\jmath_F^*$  are defined to be zero maps (see \Cref{table:restriction}).  In other words, the generalized trace operator $\jmath_{F}^*$ with $\dim F = m$ projects a form to all the $k$-hyperplanes that contain the $m$-simplex $F$. 
{
\begin{remark}
\label{rmk:volume}
    Hereafter, we always make the convention that the volume form of $F$ is identified as 1.
\end{remark}
}

\begin{figure}
\FIG{\includegraphics[width=0.4\linewidth]{\RevisionAssets/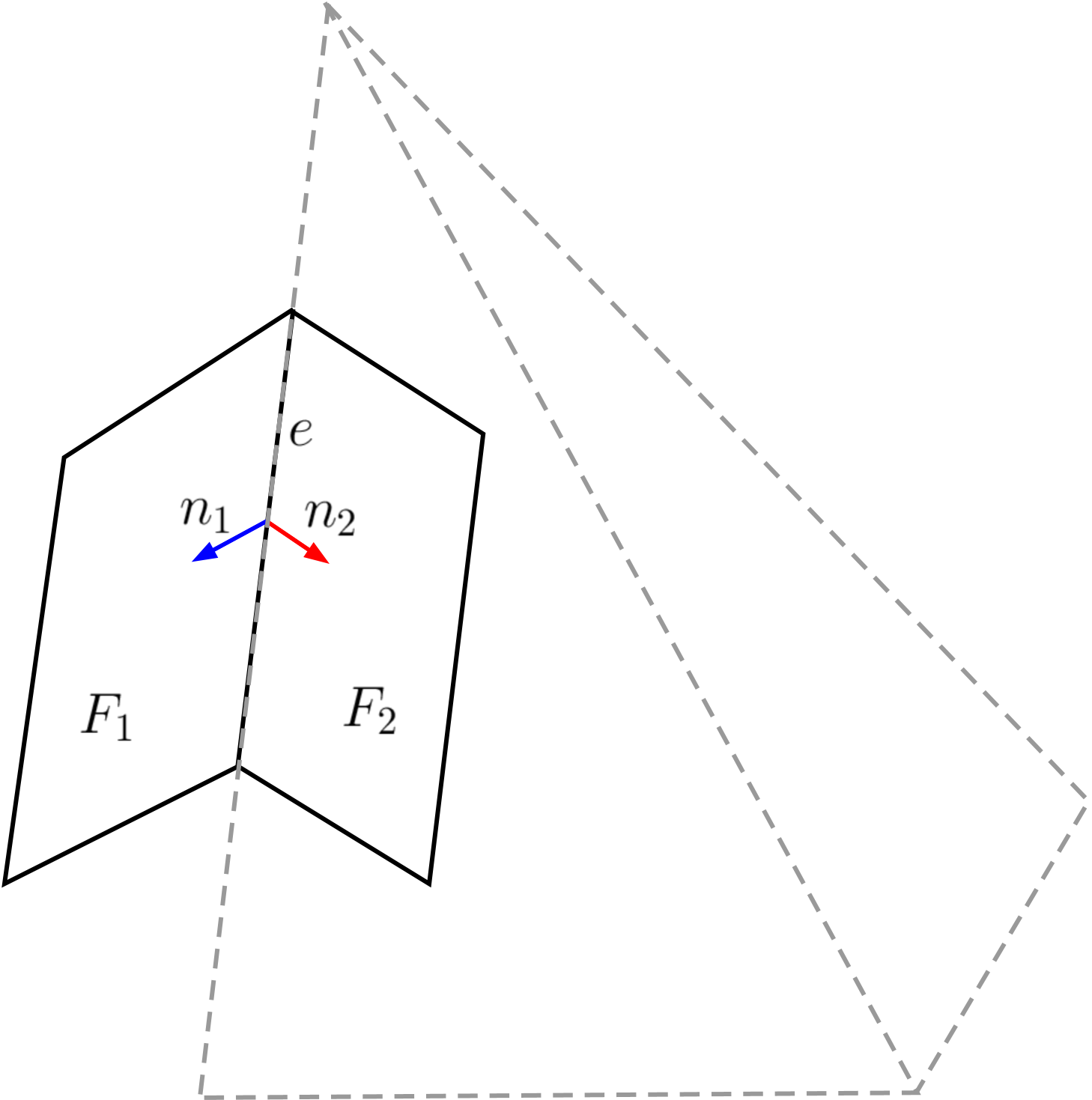}}
{\caption{Generalized trace of $k$-forms on $m$-cells when $k>m$. The figure demonstrates the trace of a 2-form 
$\omega$ on a 1-cell $e$ in $\mathbb{R}^{3}$ (shown in a tetrahedron). In this case, we feed two vectors to $\omega$. There are three possibilities: two normal vectors, and one tangent vector plus one normal (two choices). The generalized trace operator decomposes $\omega$ to these three components {and removes the one corresponding} to the normal-normal component, i.e., the generalized trace projects $\omega$ to the components containing the tangent vector. In general, in $\mathbb{R}^{n}$ there are $m$ tangent vectors to the $m$-cell and $n-m$ normal vectors. The generalized trace operator projects a $k$-form ($k>m$) to the components containing $m$ tangent vectors plus $k-m$ normal vectors. The number of components, i.e., the number of choices of $k-m$ normal vectors, is therefore ${n-m\choose k-m}$. The figure shows the case $n=3$, $m=1$, $k=2$.
}
 \label{fig:generalized-trace}
 }
\end{figure}

\Cref{table:pullback,table:restriction} below summarize the trace $\iota_{F}^*$ and the generalized trace $\jmath_{F}^*$ with the standard vector proxies in $\mathbb{R}^3$.
\begin{table}[htbp]
\TBL{
\begin{center}
\begin{tabular}{|c|ccc|}
\hline
\diagbox[width=\dimexpr \textwidth/8+2\tabcolsep\relax, height=1cm]{ $k$ }{$\dim (F)$}&  0 & 1 & 2 \\\hline
0 & vertex value &edge  value & face value \\ 
1 &0 & edge tangential & face tangential \\ 
2  & 0& 0 & face normal  \\\hline
\end{tabular}
\end{center}}
{\caption{Vector proxies of $\iota^*_{F}$  on $k$-forms in $\mathbb{R}^{3}$. Here 0 means zero maps.}
\label{table:pullback}}
\end{table}

 \begin{table}[htbp]
 \TBL{
 \begin{center}
\begin{tabular}{|c|ccc|}\hline
\diagbox[width=\dimexpr \textwidth/8+2\tabcolsep\relax, height=1cm]{ $k$ }{$\dim (F)$}&  0 & 1 & 2 \\\hline
0 &vertex value &0 & 0\\
1 & vertex value & edge tangential & 0 \\
2  & vertex value & edge normal & face normal  \\\hline
\end{tabular}
\end{center}}
{\caption{Vector proxies of  $\jmath^*_{F}$   on $k$-forms in $\mathbb{R}^{3}$. The diagonal blocks are identical to those in Table \ref{table:pullback}.}
\label{table:restriction}}
\end{table}

Recall that	the composition of trace operators is also the trace. That is, $\iota_{E}^* \circ \iota_{F}^* w = \iota_{E}^* w $ for $E \trianglelefteq F \trianglelefteq K $ and $w \in \alt^k(K)$. Hereafter, we write $F \trianglelefteq K$ to denote that $F$ is a subsimplex of $K$.  For the generalized trace,   the composition of the generalized trace operators and trace operators is a part of the generalized traces. See the following lemma.
\begin{lemma}[Composition of generalized traces]
\label{lem:j-circ-i}
For $E \trianglelefteq F \trianglelefteq K$, and $w \in \alt^k(K)$.  Let $E^{\perp}$ be the normal hyperplane of $E$ in the ambient space $K$, and $E^{\perp} \cap F$ be the normal hyperplane of $E$ in the ambient space $F$. Suppose that $\dim E \le k \le \dim F$, and $q \le k$. Let $\pi_{E, F}$ be the orthogonal projection from the space $\alt^{k - q}(E^{\perp}) \to \alt^{k-q}(E^{\perp} \cap F)$. It holds that $\vartheta_{E,q}^* \circ \iota_{F}^* = \Pi_{E, F} \circ \vartheta_{E,q}^*,$ where \[\Pi_{E, F} : C^{\infty}(E) \otimes \alt^q(E) \otimes \alt^{k-q}(E^{\perp})  \to  C^{\infty}(E) \otimes \alt^q(E) \otimes \alt^{k-q}(E^{\perp} \cap F)\]
is defined as $(id,id,\pi_{E, F})$.
\end{lemma}
\begin{proof} The above result holds since $\iota_{F}^*$ removes the components that are orthogonal to $F$.
\end{proof}

\begin{example}
Let us consider a simple example in three dimensions. Let the edge $E$ be parallel to $x_1$.  A 2-form in three dimensions can therefore be written as:
\[f = f_3 dx^1 \wedge dx^2 + f_1 dx^2 \wedge dx^3 + f_2 dx^3 \wedge dx^1.\] 
When introducing the $\jmath^*$ on the edge $E$, we should extract the  $dx^1 \wedge dx^2$ and $dx^3 \wedge dx^1$  component. That is, $\jmath_E^* f = f_3 dx^1 \otimes dx^2  - f_2 dx^1 \otimes dx^3.$
\end{example}

The above definitions for $\iota^*$ and $\jmath^*$ can be generalized to form-valued forms $C^{\infty}(\Omega)\otimes\alt^{k, \ell}(\Omega)$. For example, we use $\iota^* \iota^*$  to denote the trace operator for both indices. That is, for $\omega \in C^{\infty}(\Omega)\otimes \alt^{k,\ell}(\Omega)$ and $F\subseteq \Omega$, $\iota_{F}^{*} \iota_{F}^{*}: C^{\infty}(\Omega) \otimes \alt^{k,\ell}(\Omega) \to C^{\infty}(F) \otimes \alt^{k,\ell}(F)$ is defined by
\begin{align}\label{double-trace}
\iota_{F}^{*}\iota_{F}^{*}\omega(v_{1}, \cdots, v_{k})(u_{1}, \cdots, u_{\ell}):=\omega(\iota_{F,\ast}v_{1}, \cdots, \iota_{F,\ast}v_{k})(\iota_{F,\ast}u_{1}, \cdots, \iota_{F,\ast}u_{\ell}),\\ \nonumber\quad\forall v_{1}, \cdots, v_{k}, u_{1}, \cdots, u_{\ell}\in C^{\infty}(F;TF). 
\end{align}
We use $\iota^*\jmath^*$ to denote taking $\iota^*$ for the first index in $\alt^{k, \ell}$ and $\jmath^*$ for the second. Similar definitions are used for $\jmath^*\jmath^*$ and $\jmath^*\iota^*$. 
 In vector/matrix proxies in three dimensions, operators on the two indices correspond to row-wise and column-wise operators, e.g.  edge tangential-tangential, edge tangential-normal, face  tangential-normal, face normal-normal, face normal, edge tangential.

 We can specify the range of the generalized traces on form-valued forms. For example, we have 
\[\iota_{F}^* \jmath_{F}^* : C^{\infty}(\Omega) \otimes \alt^{k,\ell}(\Omega) \to C^{\infty}(F) \otimes \alt^{k}(F) \otimes \alt^{\ell - \dim F}(F^{\perp}),\]
where we require that $k \le \dim F \le \ell$ to make sure that both $\iota_{F}^*$ and $\jmath_F^*$ are nontrivial, and use the convention that we  identify the volume form of $F$ as $1$ (\Cref{rmk:volume}). 

\subsection{Generalized trace for $p\geq 1$} 

In the previous subsection, we introduced the generalized trace operator $\jmath_{F}^{\ast}\omega$ for a high-order form $\omega$ on a lower-dimensional cell $F$. (Note that $\jmath^{\ast}$ for  lower-degree forms on higher-dimensional cells is defined to be zero for technical reasons.) The notion involves using all tangent vectors of $F$ to evaluate $\omega$, with the remaining slots filled by vectors normal to $F$. {This is denoted as $\jmath_{F,[1]}^* = \jmath_{F}^*$.} To define finite elements for $p>1$, we need to extend this definition. Specifically, we introduce $\jmath_{F, [p]}^{\ast}$, which uses $p-1$ fewer tangent vectors, substituting them with normal vectors. For instance, when $p=2$, $\jmath_{F, [p]}^{\ast}$ comprises two terms: one is $\jmath_{F}^{\ast}$, which employs all tangent vectors of $F$, and the other is $\vartheta_{F,\dim F-1}^{\ast}$, which omits one tangent vector. In the remainder of this subsection, we provide precise definitions based on this approach.

We review the decomposition   \eqref{decomp:altkRn},
\begin{equation*}
\alt^k(\mathbb R^n) \cong \underbrace{\alt^k(F) \otimes \Alt^{0}(F^{\perp})}_{\text{feed } k \text{ tangential vectors}} \oplus \underbrace{\alt^{k-1}(F) \otimes \Alt^{1}(F^{\perp})}_{\text{feed } k-1 \text{ tangential vectors}} \oplus \cdots \oplus \underbrace{\alt^{0}(F) \otimes \Alt^{k}(F^{\perp})}_{\text{feed } 0 \text{ tangential vectors}}.
\end{equation*}

When both $\dim F$ and $\dim F^{\perp} \ge k$, all terms above are nontrivial. Therefore, we can introduce a family of trace operators to consider those components with at least $k - (p-1)$ tangential vectors, i.e., the beginning $p$ terms of the right-hand side. When $p \ge k + 1$, we actually incorporate all the terms, yielding a restriction operator $\rho^*$. 
Following the discussion on the generalized operator $\jmath^*$, we are interested in the case when $k>\dim F $, where the first several terms in \eqref{decomp:altkRn} are omitted, and the above decomposition becomes 
\begin{equation}\label{decomposition-expansion}
\alt^k(\mathbb R^n) \cong \underbrace{\alt^{\dim F}(F) \otimes \Alt^{k - \dim F }(F^{\perp})}_{\text{feed } \dim F \text{ tangential vectors}} \oplus \underbrace{\alt^{\dim F-1}(F) \otimes \Alt^{k - \dim F + 1}(F^{\perp})}_{\text{feed } \dim F-1 \text{ tangential vectors}} \oplus \cdots .
\end{equation}

 The generalized trace operator $\jmath^*$ is the projection to first term of the decomposition \eqref{decomposition-expansion}. We can continue writing the above sum until a term $\alt^{\dim F-s}(F) \otimes \Alt^{k - \dim F + s}(F^{\perp})$ becomes trivial. This can happen if either $\alt^{\dim F-s}(F)$ becomes  trivial (when $s> \dim F$) or $\Alt^{k - \dim F + s}(F^{\perp})$ becomes trivial (when $k - \dim F + s >\dim (F^{\perp})= n - \dim F$). Now we generalize this definition to a projection to the  first $p$ terms in \eqref{decomposition-expansion}. This leads to a new family of generalized traces $\jmath^*_{[p]}$ ($p=1$ gives $\jmath^*$).

   Given any integer $p\geq 1$ and $\omega \in C^{\infty}(\Omega) \otimes \alt^k(\mathbb R^n)$, set 
\begin{equation} \label{def:jp}\jmath^*_{F,[p]}\omega  :=  \vartheta_{F,\dim F}^* \omega  + \vartheta_{F,\dim F - 1 }^* \omega +  \cdots + \vartheta_{F,\dim F- p+1}^* \omega.\end{equation}

The range of $\jmath_{F,[p]}^*$  is \[\bigoplus_{s= 0}^{p-1} C^{\infty}(F) \otimes \alt^{\dim F - s }(F)  \otimes  \alt^{k - \dim F + s }(F^{\perp}).\]

More precisely, for  $1\le i_1<\cdots<i_k\le n$, the map $\jmath_{F,[p]}^*$ sends a monomial
\[
 f d\alpha^{i_1}\wedge \cdots \wedge d\alpha^{i_k} \mapsto
 \begin{cases}
  f|_F \bigwedge_{i_s \le \dim F} d\alpha^{i_s} \otimes \bigwedge_{ i_s > \dim F} d\alpha^{i_s},\quad \mbox{~if at least $\dim F - p +1$ indices $i_s \le \dim F$} \\
  0, \quad \mbox{~otherwise~}.
  \end{cases}
\]

 The following lemma  specifies the cases where $\jmath_{[p]}^*$ equals to $\iota^*$ or $\rho^*$. 

\begin{lemma}
\label{lem:jp-properties}
Given a $k$-form $w$ and a cell $F$, the trace $\jmath_{F,[p]}^*$ has the following properties:
\begin{enumerate}
	\item $\jmath_{F,[p]}^* w = \iota_{F}^* w$ if $\dim F = k + p - 1$. 
	\item $\jmath_{F,[p]}^* w= 0$ if $\dim F \ge k + p$. 
	\item $\jmath_{F,[p]}^* w= \rho_{F}^*w$ if $p > \dim F$ or $k - (\dim F - p + 1) \ge \dim F^{\perp}$. Suppose that $w$ is defined in $\mathbb R^n$, then the latter condition boils down to $k + p - 1 \ge n$. 
\end{enumerate}
\end{lemma}
\begin{proof}
By definition. 	
\end{proof}

 A direct corollary of \Cref{lem:jp-properties} is that  {for $p\ge2$ no nonzero }operators other than $\iota^*$ and $\rho^*$ can be produced in three dimensions. In fact, the first non-trivial example of the family appears in four dimensions, see the example below. 
\begin{example}
\label{exa:2-forms-in-4D}
We demonstrate an example in four dimensions. Let $F$ be a 2-face lying in the plane spanned by the  $x_1$ and $x_2$ axes. Note that 2-forms in four dimensions can be expressed by the basis $dx^i \wedge dx^j$ (count 6):
\[f = f_{ij} dx^i \wedge dx^j, \{i,j\} \subseteq \{1,2,3,4\}.\] 
The trace $\jmath_F^* = \jmath_{F, [1]}^*$ extracts the $dx^1 \wedge dx^2$ term, the trace $\jmath_{F, [2]}^*$ extracts $dx^1 \wedge dx^2, dx^1 \wedge dx^3, dx^1\wedge dx^4, dx^2\wedge dx^3, dx^2 \wedge dx^4$ terms. The last four terms come from $\vartheta_{F,1}^*$, which contains either $dx^1$ or $dx^2$. Finally, $\jmath_{F, [3]}^* = \rho_{F}^*$ is the restriction. 
\end{example}

The double trace operators $\iota_{F}^{*} \iota_{F}^{*}$ and $\iota_{F}^{*} \jmath_{F}^{*}$ defined in the previous subsection can be similarly defined here, leading to  $\iota_{F}^{*} \jmath_{F, [p]}^{*}$.

\section{Tensorial Whitney forms: construction of spaces}
\label{sec:whitney}
In this section, we present the lowest-order case of our construction, serving as a generalization of the Whitney forms for de~Rham complexes. We refer to this low-order construction as {\it tensorial Whitney forms}. 

The construction follows in three steps. We recall the outline in \Cref{sec:overview}. The first step is to construct finite element spaces for tensors $\alt^{k, \ell}$. The unisolvency of these finite elements with the $\iota^{\ast}\rho^{\ast}$-conformity follows from the results of the Whitney forms for differential forms $C^{\infty}\otimes\alt^{k}$. For later use, we also derive finite elements with the same shape functions but weaker $\iota^{\ast}\iota^{\ast}$ interelement continuity. This step works  uniformly for any $p\geq 1$. The second step is to reduce the finite elements obtained from Step 1 by imposing the extra symmetries using the BGG diagrams, leading to the $\W^{k,\ell}_{[p]}$ spaces. This step is achieved by cancelling some degrees of freedom of the $\iota^{\ast}\iota^{\ast}$-conforming finite elements obtained in Step 1. The resulting finite elements thus also have an $\iota^{\ast}\iota^{\ast}$-conformity. The last step of the construction is to enhance the conformity of the finite elements by moving  degrees of freedom.

 In particular,  the last step leads to finite element discretizations of $C^{\infty} \otimes \mathbb W_{[p]}^{k,\ell}$. We impose the $\iota^*\jmath_{[p]}^*$-conformity on faces in $\mathcal T_{ < \ell + p}$ ($\jmath_{[p]}^*$ is defined on $\ell$-forms only on simplices of dimension less than   $\ell+p$)
      and the $\iota^*\iota^*$-conformity for $\mathcal T_{\ge \ell + p}$. For simplicity, we call such elements $\iota^* \jmath_{[p]}^*$-conforming, which will be discussed in \Cref{subsec:ijp}.

Now we take $r = 1$ in \eqref{def:P-}, yielding that 
\begin{equation}
\label{def:p1-}
\mathcal P^-\alt^{k,\ell} := \mathcal P_1^-\alt^{k,\ell} = \alt^{k,\ell} + \kappa \alt^{k+1,\ell},
\end{equation}
with the following dimension count
\[ \dim \mathcal P^-\Alt^{k,\ell} = \binom{n+1}{k+1} \binom{n}{\ell}.\]

For each $k$-simplex $\sigma$, we recall the Whitney form associated to $\sigma$:
\begin{equation}\label{def:whitney}{\phi}_{\sigma} := \sum_{j=0}^{k}(-1)^{j}\lambda_{\sigma_j}d\lambda_{\sigma_{0}}\wedge\cdots \wedge\widehat{d\lambda_{\sigma_{j}}}\wedge \cdots\wedge d\lambda_{\sigma_{k}}.\end{equation}
Running through all $\sigma$, the Whitney forms give a basis for $\cP^{-}\Alt^k := (\Alt^k + \kappa \Alt^{k+1})$.
We will also use the fact that the pullback $\iota_F^* \phi_{\sigma}$ is the Whitney form $\phi_{\sigma} \in \mathcal P^-\Alt^k(F)$ when $\sigma\trianglelefteq F$, and is zero otherwise. In particular, $\iota^* {\phi}_{\sigma}$ vanishes at $\sigma' \in \mathcal T_k$ whenever $\sigma' \neq \sigma$. 

\subsection{Step 1: $\iota^*\iota^*$-conforming finite elements}

We first investigate spaces with the $\iota^*\iota^*$-conformity and show the unisolvency. 
Correspondingly, define the bubble function spaces: 
\begin{equation}\label{def:B-}
\mathcal B^-\alt^{k,\ell}(K) := \{ \omega \in \mathcal P^- \alt^{k,\ell}(K) : \iota_F^{*} \iota_F^* \omega = 0, \,\, \forall F \trianglelefteq K, F \neq K \}.
\end{equation}

The bubbles can be characterized through the Whitney forms.
\begin{lemma}[Decomposition of the bubble forms]\label{lem:bubble-decomposition}
The following direct sum decomposition holds:
 {
\begin{equation}\label{decomp:B-}
\mathcal B^- \alt^{k,\ell}(K) = \bigoplus_{\sigma \in \mathcal T_k(K)} \operatorname{span}\{\phi_{\sigma}\}\otimes N^{\ell}(\sigma,K).
\end{equation}}
Here, 
\begin{equation}\label{def:Nell}N^{\ell}(\sigma, K) := \{ \omega \in \alt^{\ell}(K): \iota_{F}^*\omega = 0 \mbox{ for all } F \mbox{ such that } \sigma \trianglelefteq F \trianglelefteq_1 K \}\end{equation}
is the  space of constant $\ell$ -forms that vanish on all codimension-one faces $F$ such that $\sigma \trianglelefteq F$. Hereafter, $F \trianglelefteq_1 K$ indicates that $F$ is a subsimplex of $K$, and $\dim K -\dim F=1$.
\end{lemma}
\begin{proof}

Since $\{ \phi_{\sigma}, \sigma \in \mathcal T_{k}(K) \} $ form a basis of $\mathcal P^- \Alt^k$, for $\omega \in \mathcal P^{-}\alt^{k,\ell}$ there exists a unique expression
$\omega = \sum_{\sigma} \phi_{\sigma} \otimes w_{\sigma}$, where $w_{\sigma} \in \alt^\ell$ is a constant $\ell$-form. Thus, the right-hand side of \eqref{decomp:B-} is a direct sum. It then suffices to show that for $\omega \in \mathcal B^-\alt^{k,\ell}(K)$, it holds that $w_{\sigma} \in N^{\ell}(\sigma,K)$ for all  $\sigma \in \mathcal T_k(K)$. 
For each $F$ with $\codim(F)= 1$, we readily see that $\iota_F^{*} \iota_F^* \phi_{\sigma} \otimes w_{\sigma} = 0$ whenever $\iota_{F}^* w_{\sigma} = 0$ or $\iota_{F}^* \phi_{\sigma}{=0}$. The latter holds when $\sigma \not \trianglelefteq F$. Therefore, the right-hand side of \eqref{decomp:B-} is contained in the left-hand side.

Conversely, suppose that $\omega \in \mathcal B^-\alt^{k,\ell}$. Fix $F$ with $ \codim(F)= 1$ and $0 = \iota_F^{*}\iota_F^* \omega  = \sum_{\sigma \in \mathcal T_k(F)} \phi_{\sigma} \otimes  \iota_{F}^* w_{\sigma},$ where we shall not distinguish $\phi_{\sigma}$ and $\iota_{F}^* \phi_{\sigma}$. Since the family $\{\phi_{\sigma}\}_{\sigma \in \mathcal T_k(F)}$ forms a basis of $\mathcal P^-\alt^k(F)$, it holds that $\iota_{F}^*  w_{\sigma} = 0$. 

Therefore, we conclude the desired result. 
\end{proof}

Next, we characterize $N^{\ell}(\sigma, K)$. We first consider some examples in three dimensions by vector proxies. Let $e$ be an edge of $K$, and $F_1^{e}$ and $F_2^{e}$ be two of its adjacent faces. For $\ell = 1$, it holds that $N^{1}(e,K) := \{ \bm w \in \mathbb R^3 : \Pi_{F_i^{e}} \bm w = 0, ~i=1, 2\} = 0$, where $ \Pi_{F_i^{e}}$ is the projection to the tangent space of $F_i^{e}$. Let $F$ be a two-dimensional face. Then $N^1(F,K) := \{  \bm w \in \mathbb R^3 : \Pi_F \bm w = 0  \} = \mathbb R\bm n$. For $\ell = 2$, it holds that $N^{2}(F,K) := \{ \bm w \in \mathbb R^3 : \bm w \cdot \bm n = 0\} = \mathbb R \bm t_1^{F} \oplus \mathbb R\bm t_2^{F}$, where $\bm t_1^{F}$ and $\bm t_2^{F}$ are two linearly independent tangent vectors on $F$.

In general, the following dimension count holds.
\begin{lemma}
\label{lem:N}
For $\dim \sigma = k$ and $\dim K = n$, 
$\dim N^{\ell}(\sigma,K) = \binom{k}{\ell+k-n} = \binom{k}{n - \ell}.$
\end{lemma}
\begin{proof}
The lemma is proved by an explicit count. Suppose that the vertices of $K$ are $X_0, \cdots, X_n$ and the simplex formed by the first $k+1$ vertices $[X_0, X_1,\cdots, X_k] = \sigma$.   Let $dx^i$ be the dual basis of  $(X_i - X_0)$. Clearly, $\alt^\ell$ has a basis $dx^I = \wedge_{i \in I} dx^i$ for $I \subseteq [n] := \{1,2,\cdots,n\}$ and $|I| = \ell$. We can now rewrite $w \in \alt^\ell$ as $w = \sum_{|I| = \ell, I \subseteq [n]} w_I dx^I$. 

For $F$ such that $\codim (F)=1$ and $\sigma \trianglelefteq F$, suppose that $F = [X_0,\cdots, X_{n-1}]$. Then $\iota_F^* w = \sum_{I \subseteq [n-1], |I| = \ell} w_I dx^I = 0$. Therefore, $w_I = 0$ for any $n \not\in I$. 

Similarly, it holds that $w \in N^{\ell}(\sigma, K)$ if and only if $w_I = 0$ for all $I$ not containing at least one of $\{k+1,k+2,\cdots,n\}$. Therefore, the dimension of $N^{\ell}(\sigma, K)$ is equal to $\displaystyle \# \{ I \subseteq [n], \{k+1,\cdots, n\}\subseteq I, |I| = \ell\} = \binom{k}{\ell+k-n}$. 

\end{proof}
\begin{remark}
    Although the definition of $N^{\ell}(\sigma, K)$ depends on $K$ explicitly, it can be verified from the previous proof that $N^{\ell}(\sigma, K)$ does not depend on $K$ if we do not distinguish $\Alt^{\ell}(K)$ for different $K$.
\end{remark}

As a corollary, it holds that
\begin{equation}
\label{eq:bubble-dim-count}
    \dim \mathcal B^-\alt^{k,\ell} (K) = \binom{n+1}{k+1}\binom{k}{\ell+k-n}.
\end{equation}

\begin{corollary}
\label{cor:spanning-whitney}
For a given $n$-simplex $K$, we index its vertex set in $[n+1] = \{1,2,\cdots,n+1\}$.  Let $X(n+1,k)$ be the set of increasing $k$-tuples. 
	For $\sigma$, let $I := [\![ \sigma ]\!] \subseteq [n+1]$ be its corresponding  index set. We will use $\phi_{I}$ to represent $\phi_{\sigma}$.  Then $\phi_{I} \otimes d\lambda_J$ for all  $I \in X(n+1,k+1)$ and $J \in X(n+1,\ell)$ such that $I \cup J = [n+1]$ is a spanning set of $\mathcal B^- \alt^{k,\ell}(K)$. This family has $\binom{n+1}{k+1}\binom{k+1}{\ell+k-n}$ members, whereas \eqref{eq:bubble-dim-count} contains $\binom{k}{\ell+k-n}$ in the second factor. It is also possible to write down a basis, but in general we cannot provide a canonical one due to the linear dependence of $d\lambda_{i}$, see \cite{arnold2009geometric}. 
\end{corollary}

A straightforward corollary of \eqref{eq:bubble-dim-count} is the following.
\begin{corollary}
The bubble space $\mathcal B^-\alt^{k,\ell}(\sigma) = 0$ for an $n$-dimensional simplex $\sigma$, if $\ell + k < n$. 
\end{corollary}

The dimension count  and a face-induction argument imply unisolvency. 
\begin{proposition}[$ C_{\iota^*\iota^*} \mathcal P^- \alt^{k,\ell}$ finite element]
\label{prop:Pminus-ii}
The degrees of freedom
\begin{equation}
\label{eq:dof-ii}
 \langle \iota_{\sigma}^* \iota_{\sigma}^* \omega, b \rangle_{\sigma}, \qquad \forall~~ b \in \mathcal B^-\alt^{k,\ell}(\sigma)
\end{equation}	
for each $\sigma \in \mathcal T(K)$
are unisolvent with respect to the shape function space $\mathcal P^- \alt^{k,\ell}(K)$. The resulting finite element space is $\iota^{*}\iota^*$-conforming, denoted as $ C_{\iota^*\iota^*} \mathcal P^- \alt^{k,\ell}$.
\end{proposition}
By definition,  $\mathcal B^- \alt^{k,\ell}(\sigma) =0$ when $\dim \sigma < \max(k,\ell)$. 
\begin{proof}
 We first prove the dimension count , then use face induction to prove unisolvency and conformity. For $\sigma \in \mathcal T_m(K)$, by \eqref{eq:bubble-dim-count},  it holds that 
	\[\dim \mathcal B^-\alt^{k,\ell} (\sigma) = \binom{m+1}{k+1}\binom{k}{\ell+k-m}.\]
	Therefore,  
\begin{equation} 
\begin{split}	\sum_{\sigma \in \mathcal T(K)} \dim \mathcal B^-\alt^{k,\ell}(\sigma) & = \sum_{m=0}^{n} \binom{n+1}{m+1} \binom{m+1}{k+1}\binom{k}{\ell+k-m} \\ 
& = \sum_{m = 0}^n \frac{(n+1)!}{(n-m)!(m+1)!} \frac{(m+1)!}{(k+1)!(m-k)!} \binom{k}{\ell+k-m} \\ 
& = \sum_{m = 0}^n \binom{n+1}{k+1} \binom {n-k}{n-m} \binom{k}{\ell+k-m} \\ 
& =  \binom{n+1}{k+1}  \binom{n}{\ell} = \dim \mathcal P^- \alt^{k,\ell}(K),
\end{split}	
\end{equation}
where we have used the Vandermonde identity $\sum_{k=0}^{r}\binom{m}{k}\binom{n}{r-k}=\binom{m+n}{r}$.
This completes the  dimension count.

To see the unisolvency, we apply a mathematical induction here. When $\dim \sigma = \max(k,\ell)$, the unisolvency is straightforward. When $\dim \sigma > \max(k,\ell)$, suppose that all the degrees of freedom vanish on  $\omega \in \mathcal P^-\Alt^{k,\ell}(\sigma)$. For each face $F \trianglelefteq_1 \sigma$, we have $\iota_F^*\iota_F^* \omega=0$, therefore  $\omega \in \mathcal B^-\Alt^{k,\ell}(\sigma)$. By the degrees of freedom, we have $\omega = 0$. Therefore, we complete the proof of unisolvency.
\end{proof}

Now we give some examples to show the construction. 
\begin{example}
When $\ell=0$, $C_{\iota^*\iota^*}\mathcal P^-\alt^{k,0}$ is the standard Whitney space $C_{\iota^*}\mathcal P^-\alt^k$. When $k = 0$, $C_{\iota^*\iota^*}\mathcal P^- \alt^{0,\ell}$ gives the  full linear FEEC space $\mathcal P_1 \alt^{\ell}$; when $\ell = n$, $C_{\iota^*\iota^*} \mathcal P^-\alt^{k,n}$ gives the discontinuous space $C^{-1}\mathcal P^- \alt^k$.
\end{example}

\begin{example}[The \emph{full Regge space} $C_{\iota^*\iota^*} \mathcal P^-\alt^{1,1}$]
\label{ex:full-regge}
In this example, we demonstrate the construction of the $C_{\iota^*\iota^*} \mathcal{P}^- \alt^{1,1}$ element function space. In any spatial dimension, the $C_{\iota^*\iota^*} \mathcal{P}^- \alt^{1,1}$ space has one degree of freedom per edge and three degrees of freedom per 2-face. In three dimensions with proxies, the shape function space is $\mathbb{M} + \bm{x} \times \mathbb{M}$, where the degrees of freedom correspond to the edge tangential-tangential component and the moments against three face tangential-tangential bubble functions within each 2-face. The total number of degrees of freedom is $1 \times 6 + 3 \times 4 = 18$. The resulting space is tangential-tangential continuous. See \Cref{fig:regge-reduction} for details.
\end{example}

\begin{remark}
Recall that
 the lowest order Regge finite elements have piecewise constant symmetric matrices as the shape functions \cite{christiansen2011linearization,li2018regge}. One may expect that piecewise constant tensors are a natural candidate for shape functions of the finite element spaces for $\Alt^{k, \ell}$. However,  the discussions above show that this is not the case. For example, let $\mathcal B_0 \alt^{1,1}$ be the space of constant bubbles (matrices with vanishing tangential-tangential components on the boundary). In one dimension, $\dim \mathcal B_0 \alt^{1,1}(e) = 1$; in two dimensions $\dim \mathcal B_0 \alt^{1,1}(f) = 4 -3 = 1$ (a constant matrix has four entries and there is one degree of freedom on each edge). However, continuing this pattern in three dimensions, one has one degree of freedom per edge (corresponding to $\mathcal B_0 \alt^{1,1}(e)$) and one degree of freedom per face (corresponding to $\mathcal B_0 \alt^{1,1}(f)$). This already gives   $4+6 = 10$ degrees of freedom, more than $\dim \alt^{1,1}(\mathbb R^3) = 9$. Therefore introducing additional shape functions, e.g., the above construction with the Koszul operators,  is necessary for constructing $\Alt^{k, \ell}$ finite element spaces.
\end{remark}

\subsection{Step 2: symmetry reduction}

{From the previous step, we have $\iota^{\ast}\iota^{\ast}$-conforming finite element spaces in hand with the shape function space $C_{\iota^*\iota^*}\mathcal P^- \alt^{k,\ell}$ and the degrees of freedom \eqref{eq:dof-ii}. In Step 2 presented below, we follow the BGG diagrams and construction to reduce $C_{\iota^*\iota^*}\mathcal P^- \alt^{k,\ell}$  to $C_{\iota^*\iota^*}\mathcal P^- \W^{k,\ell}$, the spaces with the symmetries encoded in $\ker(\mathcal S_{\dagger})$. To derive the shape functions of the new spaces, we characterize $\mathcal P^- \alt^{k,\ell}\cap \ker(\mathcal S_{\dagger})$.

In this subsection, we assume $k \le \ell$. 
Recall that $\mathcal S_{\dagger}^{k,\ell}: \alt^{k,\ell}\to \alt^{k-1,\ell+1}$ is onto. The kernel space is defined as $\W^{k,\ell}$.  By \Cref{lem:S-poly}, $\mathcal S_{\dagger}^{k,\ell}$ is a surjective map from $\cP^- \alt^{k,\ell}$ to $\cP^- \alt^{k-1,\ell+1}$. By \Cref{charact-ker-Sdagger}, the kernel $\cP^-\W^{k,\ell}:=\ker(\mathcal S_{\dagger}^{k,\ell})$ is characterized as $\W^{k,\ell} + \kappa \W^{k+1,\ell}$ for $0<k<\ell$, and $\W^{k,\ell}$ for $k=\ell>0$. For $k=0$, it is the full space $\cP^-\alt^{0,\ell}$. 

The reduction of the shape function spaces is straightforward. To carry out a similar reduction to the degrees of freedom,   it suffices to show that $\mathcal S_{\dagger}^{k,\ell}$ induces a mapping from $\B^-\alt^{k,\ell}(\sigma)$ to $\B^-\alt^{k-1,\ell+1}(\sigma)$, the spaces involved in \eqref{eq:dof-ii}. This can be verified by the following facts: (1) $\mathcal S_{\dagger}$ commutes with trace, and (2) $\mathcal S$ is injective and $\mathcal S_{\dagger}$ is surjective. We summarize these results in the following lemma.

\begin{lemma}
\label{lem:S-bubble}
For $k \le \ell$, it holds that 
\begin{enumerate}
    \item $\mathcal S_{\dagger}^{k, \ell} : \mathcal P^- \alt^{k,\ell}(\sigma) \to \mathcal P^- \alt^{k-1,\ell+1}(\sigma)$ is onto.  
    \item $\mathcal S_{\dagger}^{k, \ell} : \mathcal B^- \alt^{k,\ell}(\sigma) \to \mathcal B^- \alt^{k-1,\ell+1}(\sigma)$ is onto.  
\end{enumerate}
\end{lemma}

The first statement comes from {the anticommutativity of $\kappa$ and $\mathcal S_{\dagger}$}. The second statement is actually far from trivial, and the proof is presented in the appendix, with the help of \Cref{cor:spanning-whitney}. We then define $\mathcal P^- \mathbb W^{k,\ell}(K) = \ker(\mathcal S_{\dagger} : \mathcal P^- \alt^{k,\ell}(K) \to  \mathcal P^- \alt^{k-1,\ell+1}(K)),$ and 
\begin{equation}\label{def:BW}
\mathcal B^- \mathbb W^{k,\ell}(\sigma) := \ker(\mathcal S_{\dagger} : \mathcal B^- \alt^{k,\ell}(\sigma) \to  \mathcal B^- \alt^{k-1,\ell+1}(\sigma)).
\end{equation}

We first show the symmetry element with respect to the $\iota^*\iota^*$-conformity.

\begin{proposition}[$C_{\iota^*\iota^*} \cP^-\W^{k,\ell}$ finite element]\label{prop:sym-reduction}
For the shape function space $\cP^-\W^{k,\ell}(K)$, the degrees of freedom 
\begin{equation}\label{dof:CiiW}
 \langle \iota^{*}_{\sigma} \iota^*_{\sigma} \omega, b \rangle_{\sigma},\,\,\forall \,\, b \in \mathcal B^- \mathbb W^{k,\ell}(\sigma) 
\end{equation}	
for all $\sigma \in \mathcal T(K)$ are unisolvent. The resulting space
is $\iota^{*}\iota^*$-conforming, denoted as $C_{\iota^*\iota^*} \cP^-\W^{k,\ell}$. 	
\end{proposition}
\begin{proof}
{It suffices to show the dimension count, i.e., the dimension of the reduced space is equal to the number of the new degrees of freedom. By the surjectivity, it holds that 
\[\dim \mathcal P^-\W^{k,\ell} = \dim \mathcal P^- \alt^{k,\ell} - \dim \mathcal P^- \alt^{k - 1,\ell + 1}.\] Similarly, 
\[\dim \mathcal B^-\W^{k,\ell}(\sigma)  = \dim \mathcal B^-\alt^{k,\ell}(\sigma) - \dim \mathcal B^- \alt^{k - 1,\ell + 1}(\sigma).\]
The desired result holds by summing over all $\sigma$.
}
To see the unisolvency, we apply again a mathematical induction. When $\dim \sigma = \ell$, the unisolvency is straightforward. When $\dim \sigma > \ell$, suppose that all the degrees of freedom vanish on  $\omega \in \mathcal P^-\mathbb W^{k,\ell}(\sigma)$. By induction, it holds that  $\iota_{F}^*\iota_F^* \omega \in \mathcal P^-\mathbb W^{k,\ell}(F)$ vanishes on all codimensional 1 faces $F \trianglelefteq_1 \sigma$. Therefore, $\omega \in \mathcal B^- \mathbb W^{k,\ell}(\sigma)$. By the definition of the degrees of freedom, we then have  $\omega = 0$. Therefore, we complete the proof.
\end{proof}

\begin{example}[The Regge element $C_{\iota^*\iota^*} \mathcal P^- \W^{1,1}$]
\label{ex:regge}
We continue \Cref{ex:full-regge} to show how to obtain the symmetric Regge space in three dimensions. Recall that the local shape function space of $C_{\iota^*\iota^*} \mathcal P^- \alt^{1,1}$  is $\mathbb M + \bm x \times \mathbb M$, and we have 1 degree of freedom per edge and 3 degrees of freedom per face. 
For the reduction step, we need to remove the degrees of freedom from $\mathcal P^-\alt^{0,2}$, which is  the full linear $\Alt^2$-valued space (the BDM space in vector proxies). Thus, it has 3 degrees of freedom per 2-face.
The symmetry reduction thus completely removes the face degrees of freedom from $C_{\iota^*\iota^*} \mathcal P^- \alt^{1,1}$ (3-3=0), leading to the symmetric Regge element $C_{\iota^*\iota^*} \mathcal P^- \mathbb W^{1,1}$. 
The resulting local shape function space is $\mathcal P^- \W^{1,1} = \mathbb S$, and we have 1 degree of freedom per edge.
See \Cref{sec:kkform} for more details. 
\end{example}

Next, we consider the symmetry reduction introduced by the iterated operator $\mathcal S^{k,\ell}_{[p]}$. The shape function space is
\[\mathcal P^- \mathbb W^{k,\ell}_{[p]} := \ker(\mathcal S_{\dagger,[p]} : \mathcal P^-\alt^{k,\ell} \to \mathcal P^- \alt^{k-p,\ell+p}).\]

\begin{lemma}
\label{lem:Sp-bubble}
For $k \le \ell + p -1$, it holds that 
\begin{enumerate}
    \item $\mathcal S_{\dagger,[p]} : \mathcal P^- \alt^{k,\ell}(\sigma) \to \mathcal P^- \alt^{k-p,\ell+p}(\sigma)$ is onto.  
    \item $\mathcal S_{\dagger,[p]} : \mathcal B^- \alt^{k,\ell}(\sigma) \to \mathcal B^- \alt^{k-p,\ell+p}(\sigma)$ is onto.  
\end{enumerate}
\end{lemma}
\begin{proof}
    See \Cref{subsec:inj/sur}.
\end{proof}

We first show the symmetric element with the $\iota^*\iota^*$-conformity.

\begin{proposition}[$C_{\iota^*\iota^*} \cP^-\W^{k,\ell}_{[p]}$ finite element]\label{iiWp}
For the shape function space \[\mathcal P^- \mathbb W^{k,\ell}_{[p]}(K) := \ker(\mathcal S^{k, \ell}_{\dagger,[p]} : \mathcal P^- \alt^{k,\ell}(K) \to  \mathcal P^- \alt^{k-p,\ell+p}(K)),\] the degrees of freedom 
\[ \langle \iota^{*}_{\sigma} \iota^*_{\sigma} \omega, b \rangle_{\sigma},\,\,\forall \,\, b \in \mathcal B^- \mathbb W^{k,\ell}_{[p]}(\sigma) := \ker(\mathcal S^{k, \ell}_{\dagger,[p]} : \mathcal B^- \alt^{k,\ell}(\sigma) \to  \mathcal B^- \alt^{k-p,\ell+p}(\sigma))
\]
for all $\sigma \in \mathcal T(K)$ are unisolvent. The resulting space
is $\iota^{*}\iota^*$-conforming.	
\end{proposition}
\begin{proof}
The proof is similar to \Cref{prop:sym-reduction}. 
{It suffices to show the dimension count. By the surjectivity results, we have
\[\dim \mathcal P^-\W_{[p]}^{k,\ell} = \dim \mathcal P^- \alt^{k,\ell} - \dim \mathcal P^- \alt^{k - p,\ell + p},\]  
\[\dim \mathcal B^-\W_{[p]}^{k,\ell}(\sigma)  = \dim \mathcal B^-\alt^{k,\ell}(\sigma) - \dim \mathcal B^- \alt^{k - p,\ell + p}(\sigma).\]
The desired result follows by summing over all  $\sigma$. 
}
\end{proof}

\begin{remark}\label{rmk:sym-reduction-not-happen}
The reduction does not happen in the degrees of freedom on $\mathcal T_{\ell}, \mathcal T_{\ell + 1}, \cdots , \mathcal T_{\ell + p -1} $. This allows us to move the degrees of freedom to lower dimensions in the same way for $\Alt^{k, \ell}$ and $\W^{k, \ell}_{[p]}$; see \Cref{subsec:ijp}. 
\end{remark}

\begin{example}[$C_{\iota^*\iota^*} \mathcal P^- \W^{2,1}_{[2]}$: the MCS element] 
We now demonstrate the example of $C_{\iota^*\iota^*} \mathcal P^- \W^{2,1}_{[2]}$. Again, we use vector proxies to simplify the notation. 
For $\mathcal P^-\alt^{2,1} = \mathbb M + \bm x \otimes\mathbb V$, the construction of ${\iota^*\iota^*}$-conforming elements gives the degrees of freedom of {face normal-tangential moments} (2 per face) plus 4 degrees of freedom inside the tetrahedron. Since $\mathcal P^-\alt^{0,3} = \mathcal P_1$ has 4 degrees of freedom, all the interior degrees of freedom are removed in the reduction. Therefore, the resulting space is the MCS element $C_{\iota^*\iota^*} \mathcal P^- \mathbb W^{2,1}_{[2]}$ \cite{gopalakrishnan2020mass}, where the shape function space is $\mathbb T$, and the {degrees of freedom involve face normal-tangential components}.  
\end{example}

\subsubsection{Finite element spaces with constant shape functions.}

 In the rest of this subsection, we provide some piecewise constant constructions within the framework in this paper, answering Question \eqref{question:constant} raised in the Introduction.

First, based on the spaces we have already obtained, we have the following result. 
\begin{proposition}\label{prop:constant}
  In our construction, the shape function space $C_{\iota^{\ast}\iota^{\ast}}\mathcal P^- \mathbb W^{k,\ell}_{[p]}$ is piecewise constant $\mathbb W^{k,\ell}_{[p]}$ if and only if $k = \ell + p - 1$, provided $p\leq k, \ell\leq n-p$ to ensure that the symmetry reduction is nontrivial\footnote{The symmetry reduction uses $\Alt^{k-p, \ell+p}$ to eliminate part of  $\Alt^{k, \ell}$. Therefore the condition for the reduction being nontrivial is $k-p\geq 0$ and $\ell+p\leq n$.}.
\end{proposition}
\Cref{prop:constant} indicates that only the spaces where the BGG zig-zag occurs are constant. A special case shows that when $p = 1$, only $\mathbb W^{k,k}$ admits a constant local shape function space in our existing construction. Note that \Cref{prop:constant} considers the ``general cases'' under the assumption $p\leq k, \ell\leq n-p$. When we consider some boundary cases of the indices, there are other spaces admitting constant constructions. For example, for $\ell=n$, $\W^{k, n}=\Alt^{k, n}$ since the connecting operator mapping to this space degenerates.  The space $\mathcal P^-\W^{k,n}$ is constant only when $k=n$; for $k<n$ its Koszul summand remains nonzero. 

{
\begin{proof}[Proof of \Cref{prop:constant}]
We have the following diagram:
\begin{equation}
    \begin{tikzcd}
        ~ & \alt^{k,\ell}
         \ar[dl,"{\mathcal S_{\dagger, [p]}^{k, \ell}}",swap] 
         & \alt^{k+1,\ell} \ar[l,"{\kappa}",swap] \ar[dl,"{\mathcal S_{\dagger, [p]}^{k+1, \ell}}"]  \\
         \alt^{k-p,\ell+p} & \alt^{k-p+1,\ell+p }\ar[l,"{\kappa}",swap]  &   
    \end{tikzcd}
\end{equation}
 By \Cref{lem:Pr-Wp} with $r=1$, when $k=\ell+p-1$ the homogeneous summand is zero because $\mathcal H_{-1}=0$. When $k<\ell+p-1$, the homogeneous summand is $\kappa\W_{[p]}^{k+1,\ell}$. Under the stated nontriviality assumptions, $\mathcal S_{\dagger,[p]}^{k+1,\ell}$ is surjective with strictly larger domain than codomain, so this kernel is nonzero. Contraction with $x$ is injective on constant forms of positive first degree. Hence a nonzero linear summand remains, proving both directions.
\end{proof}
}

Recall that in our current construction, we begin with Koszul-type spaces (between constants and linear polynomials) and perform symmetry reduction. Part of the motivation for this construction is to obtain finite element spaces that fit in BGG complexes. \Cref{prop:constant} establishes that this approach yields constant local spaces if and only if $k = \ell + p - 1$ under the assumption $p\leq k, \ell\leq n-p$.   However, we do not necessarily have to start with Koszul-type spaces. Instead, we could begin with constants and explore whether we can achieve $\iota^{\ast}\iota^{\ast}$-conformity and symmetry reduction. Rather than developing this from scratch, we use the results we have so far to derive such spaces.

\begin{example}
We use the (1,2)-form in dimension $n \ge 3$ as an example.  The idea is to first use the identification between (1,2)-forms and (2,1)-forms, noting that (2,1)-forms admit a constant construction with $p=2$, i.e., $\W^{2,1}_{[2]}$ has a constant construction (the MCS element). The three spaces $\Alt^{1,2}$, $\Alt^{2,1}$, and $\Alt^{0,3}$ can be connected via the $\mathcal S$ and $\mathcal S_{\dagger}$ operators. Specifically, we have $\W^{2,1}_{[2]}=\Alt^{2,1}\ominus \mathcal{S}^{0,3}_{[2]}\Alt^{0,3}$. Here, $\ominus$ stands for the symmetry reduction, which takes the orthogonal complement of $\mathcal{S}^{0,3}_{[2]}\Alt^{0,3}$ in $\Alt^{2,1}$ with respect to the Frobenius norm, see \Cref{fig:mirror-space} for an illustration. Applying $\mathcal S_{\dagger}^{2,1}$ to this identity, we obtain $\mathcal S_{\dagger}^{2,1}\W^{2,1}_{[2]}=\Alt^{1,2} \ominus \mathcal S_{\dagger}^{2,1}\mathcal{S}^{0,3}_{[2]}\Alt^{0,3}=\Alt^{1,2} \ominus \mathcal{S}^{0,3}\Alt^{0,3}$. The right-hand side matches the definition of $\W^{1,2}$ from symmetry reduction. This suggests that we can construct $\W^{1,2}$ by pulling back $\W^{2,1}_{[2]}$ (which admits a constant construction) using $\mathcal S_{\dagger}^{2,1}$. The same reasoning applies to the degrees of freedom. This leads to the $\mathrm{MCS}^{\top}$ element, namely, the transpose of the MCS element in the matrix representation. Note again that the piecewise constant construction cannot be directly read from the Koszul-type construction, as in three dimensions the shape function space after the reduction is $\mathbb T + \bm x\times \mathbb S$, rather than constants. 
\end{example}

\begin{figure}[htbp]

\FIG{\includegraphics[width=256.05bp]{\RevisionAssets/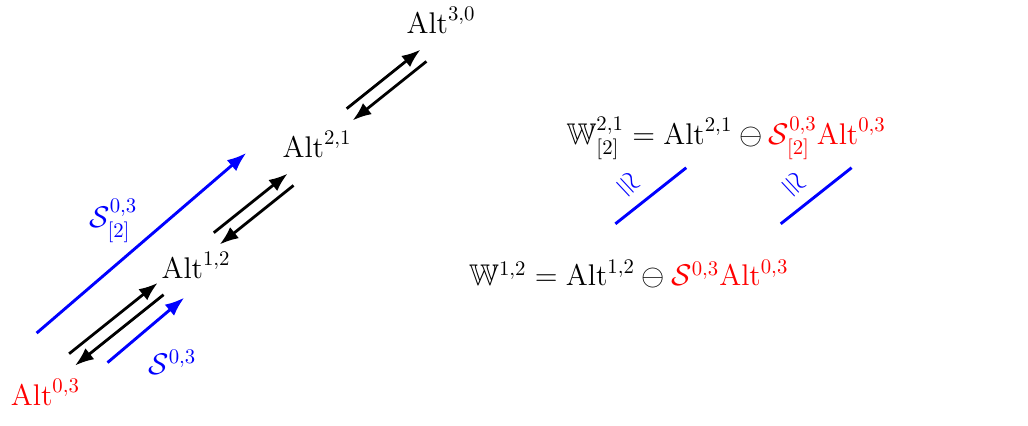}}
{\caption{The construction of a finite element with the constant shape space $\W^{1, 2}$. The figure on the left shows the diagonal in the BGG diagram that contains $\W^{1, 2}$. On this diagonal, we already have a constant construction for $\W^{2, 1}_{[2]}$ following \Cref{prop:constant} ($k=2, \ell=1, p=2$ satisfies the condition therein). Moreover, $\W^{2, 1}_{[2]}$ is the mirror space of $\W^{1, 2}$ in the sense that they are isomorphic to each other (the right part of the figure). Therefore we can pull the degrees of freedom of $\W^{2, 1}_{[2]}$ back, giving those of $\W^{1, 2}$. In general, we can derive constant constructions of $\W^{k, \ell}$ in this way: its mirror space in the diagonal admits a constant construction by \Cref{prop:constant}. For general $p>1$, we can still find a space which already has a constant construction. However, it is not the mirror of $\W^{k, \ell}_{[p]}$. To fix this issue, we decompose $\W^{k, \ell}_{[p]}$ into a sum of $\W$ spaces with $p=1$ \eqref{eq:decomp-Wp}, and then apply the above construction to each of the factors.}
	\label{fig:mirror-space}}
\end{figure}

In summary, the spaces like MCS$^{\top}$ can be obtained directly from our framework, although not following the three steps. 
In general, we have an isomorphism (duality) between $\mathbb W^{k,\ell}$ and $\mathbb W^{\ell,k}_{[1 + \ell - k]}$.
\begin{lemma}\label{lem:W-filp}
For $k \le \ell$, the iterated operators $\mathcal S_{\dagger,[\ell - k]}^{\ell,k}$ and $\mathcal S_{[\ell - k]}^{k,\ell}$ induce an isomorphism between $\mathbb W^{k,\ell}$ and $\mathbb W^{\ell,k}_{[1+\ell - k]}$.
\end{lemma}

\begin{proof}
Put $d=\ell-k$. By \Cref{lem:Sp-inj-sur}, $\mathcal S_{\dagger,[d]}:\Alt^{\ell,k}\to\Alt^{k,\ell}$ is an isomorphism (the identity when $d=0$). Since $\mathcal S_{\dagger}\mathcal S_{\dagger,[d]}=\mathcal S_{\dagger,[d+1]}$, it maps $\ker\mathcal S_{\dagger,[d+1]}$ bijectively onto $\ker\mathcal S_{\dagger}$. The definitions give $[\mathcal S_{\dagger},\mathcal S]=(\ell-k)\mathrm{id}$ on $\Alt^{k,\ell}$. Induction on the power then gives $\mathcal S_{\dagger,[d+1]}\mathcal S_{[d]}\omega=0$ and $\mathcal S_{\dagger,[d]}\mathcal S_{[d]}\omega=(d!)^2\omega$ for $\omega\in\ker\mathcal S_{\dagger}$. Thus both restrictions are isomorphisms.
\end{proof}
The space $\mathbb W^{\ell,k}_{[1+\ell - k]}$ admits a constant construction. We then pull it back to define degrees of freedom for $\mathbb W^{k,\ell}$.
\begin{theorem}\label{thm:constant-p=1}
    Suppose $0 < k \le \ell< n$. For the shape function space $\mathbb W^{k,\ell}(K)$, the following degrees of freedom 
    \begin{equation} \langle \iota^{*}_{\sigma} \iota^*_{\sigma} \omega, b \rangle_{\sigma},\,\,\forall \,\, b \in \mathcal S_{\dagger,[\ell-k]}^{\ell,k} \mathcal B^- \mathbb W^{\ell,k}_{[1+\ell - k]}(\sigma) 
\end{equation}	
for all $\sigma \in \mathcal T(K)$ are unisolvent. The resulting space is $ {\iota^*\iota^*}$-conforming.
\end{theorem}
 \begin{proof}
The operators $\mathcal S$ and $\mathcal S_{\dagger}$ commute with double pullback to a face. By adjointness, a displayed moment of $\omega$ is the corresponding bubble moment of $\mathcal S_{[\ell-k]}\omega$. The latter lies in the mirror space by \Cref{lem:W-filp}, whose moments are unisolvent by \Cref{prop:constant} and the symmetry-reduced construction. Thus vanishing of all the moments implies $\mathcal S_{[\ell-k]}\omega=0$ and hence $\omega=0$. The same facewise argument proves conformity; the number of moments is unchanged under the isomorphism.
\end{proof}
Next, we show that, in general, $\mathbb W^{k,\ell}_{[p]}$ admits an $ {\iota^*\iota^*}$-conforming construction, generalizing \Cref{thm:constant-p=1} with $p=1$. {The key step for deriving the spaces in the case of $p=1$ (\Cref{thm:constant-p=1}) is \Cref{lem:W-filp}: for the space $\W^{k, \ell}=\Alt^{k, \ell}\ominus \mathcal S^{k-1, \ell+1}\Alt^{k-1, \ell+1}$ that we work on, we find another space $\mathbb W^{\ell,k}_{[1+\ell - k]}=\Alt^{\ell, k}\ominus \mathcal S^{k-1, \ell+1}_{[\ell-k+1]}\Alt^{k-1, \ell+1}$ along the ``diagonal'' in the BGG diagram (the spaces connected to $\W^{k, \ell}$ by $\mathcal{S}$ and $\mathcal S_{\dagger}$ operators). Note that $\Alt^{k, \ell}\cong \Alt^{\ell, k}$. Therefore $\W^{k, \ell}$ is isomorphic to $\mathbb W^{\ell,k}_{[1+\ell - k]}$. 

We can repeat this argument for $\mathbb W^{k,\ell}_{[p]}$ with a general $p\geq 1$. However, the specialty of $p=1$ is that the ``mirror space'' $\mathbb W^{\ell,k}_{[1+\ell - k]}$ happens to admit a constant construction (\Cref{prop:constant}). This property is lost for $p>1$. The idea to obtain piecewise constant constructions in the general case is to recall that $\mathbb W^{k,\ell}_{[p]}$ has an orthogonal decomposition into $\W$ spaces:}
\begin{equation}
\label{eq:decomp-Wp}
\mathbb W^{k,\ell}_{[p]} = \mathbb W^{k,\ell} \oplus \mathcal S^{k-1,\ell + 1} \W^{k-1,\ell+1} \oplus \cdots \oplus \mathcal S_{[p-1]}^{k-p+1,\ell + p -1} \mathbb W^{k-p+1,\ell + p -1} .
\end{equation}

Each factor on the right-hand side has $p=1$. Then we can use the previous argument for $p=1$ to find a ``mirror space'' for each term. 
We require $k - p \ge 0$ and $\ell + p \leq n$ to ensure that each term $\mathbb W^{k-s ,\ell+s}$ for $ s= 0,1,\cdots, p-1$ in the right-hand side satisfies the condition in \Cref{thm:constant-p=1}. Note that this condition also guarantees that $\mathbb W_{[p]}^{k,\ell} \neq \Alt^{k,\ell}$, namely, the symmetric reduction is nontrivial.
{Gluing them together, we obtain the degrees of freedom for $\mathbb W^{k,\ell}_{[p]} $ that are unisolvent.}
 For example, for $\mathbb W^{2,2}_{[2]}$ in four and higher dimensions, we can combine the finite elements provided by \Cref{thm:constant-p=1} for $\mathbb W^{2,2}$ and $\mathbb W^{1,3}$ together. 

To show the details, we can use \Cref{lem:W-filp} to rewrite the above decomposition as 
 {
\begin{equation} \label{eq:Wp-decomp-new}
\begin{split}
\mathbb W^{k,\ell}_{[p]} & = \mathcal S_{\dagger,[\ell - k]}^{\ell,k} \mathbb W^{\ell,k}_{[\ell - k+1]} \oplus \mathcal S^{k-1,\ell + 1} \mathcal S_{\dagger,[\ell - k + 2]}^{\ell+1,k-1}\W^{\ell+1,k-1}_{[\ell - k + 3]} \oplus \cdots \\ &\quad \quad \oplus \mathcal S_{[p-1]}^{k-p+1,\ell + p -1} \mathcal S_{\dagger,[\ell - k + 2(p-1)]}^{\ell+p-1,k-p+1} \mathbb W^{\ell + p -1,k-p+1}_{[\ell -k + 2p - 1]} \\ 
& = \bigoplus_{s = 0}^{p-1} \mathcal S_{[s]}^{k-s,\ell+s} \mathcal S_{\dagger,[\ell - k + 2s]}^{\ell + s,k - s} \mathbb W_{[\ell - k + 2s + 1]}^{\ell + s,k - s}.
\end{split}
\end{equation}}
Since each term is isomorphic to the corresponding term in \eqref{eq:decomp-Wp}, \eqref{eq:Wp-decomp-new} is still an orthogonal direct sum. Now we have found the mirror spaces in the sense that each term $\mathbb W_{[\ell - k + 2s + 1]}^{\ell + s,k - s}$ on the right-hand side admits a constant construction.
Therefore, we have the following theorem.
\begin{theorem}\label{thm:constant-construction}
    Suppose $p \le k \le \ell \le n - p$, and let $K$ be an $n$-simplex. For the shape function space $\mathbb W_{[p]}^{k,\ell}(K)$, the following degrees of freedom 
     {
\begin{equation}\label{dofs:constant-p}
     \langle \iota^{*}_{\sigma} \iota^*_{\sigma} \omega, b \rangle_{\sigma},\,\,\forall ~ b \in \bigoplus_{s = 0}^{p-1} \mathcal S_{[s]}^{k-s,\ell+s} \mathcal S_{\dagger,[\ell - k + 2s]}^{\ell + s,k - s} \mathcal B^-\mathbb W_{[\ell - k + 2s + 1]}^{\ell + s,k - s}(\sigma).
\end{equation}}	
for all $\sigma \in \mathcal T(K)$ are unisolvent. The resulting space is ${\iota^*\iota^*}$-conforming.
\end{theorem}

{
\begin{proof}
    Note that the local shape function space of $C_{\iota^{\ast}\iota^{\ast}}\mathcal P^- \W_{[\ell - k + 2s + 1]}^{\ell + s,k - s}$ is $\W_{[\ell - k + 2s + 1]}^{\ell + s,k - s}$, satisfying the conditions in \Cref{prop:constant}. This means that for $\omega \in \mathbb W_{[\ell - k + 2s + 1]}^{\ell + s,k - s}$, the following degrees of freedom are unisolvent:
        \begin{equation} \langle \iota^{*}_{\sigma} \iota^*_{\sigma} \omega, b \rangle_{\sigma},\,\,\forall \,\, b \in \mathcal B^-\mathbb W_{[\ell - k + 2s + 1]}^{\ell + s,k - s}(\sigma).
\end{equation}	
Since  $\mathcal S_{[s]}^{k-s,\ell+s} \mathcal S_{\dagger,[\ell - k + 2s]}^{\ell + s,k - s}$ is injective and commutes with $\iota^*\iota^*$,  and since $\mathcal S_{\dagger,[s]}\mathcal S_{[s]}$ is a positive scalar multiple of the identity on $\W^{k-s,\ell+s}$, adjointness and \Cref{thm:constant-p=1} show that 
{
\[\eta \in \mathcal S_{[s]}^{k-s,\ell+s} \mathcal S_{\dagger,[\ell - k + 2s]}^{\ell + s,k - s} \mathbb W_{[\ell - k + 2s + 1]}^{\ell + s,k - s} = \mathcal S_{[s]}^{k-s,\ell+s} \mathbb W^{k-s,\ell+s}
\]}
 is uniquely determined by the degrees of freedom 
         {
\begin{equation}\label{eq:dof-Ws} \langle \iota^{*}_{\sigma} \iota^*_{\sigma} {\eta}, b \rangle_{\sigma},\,\,\forall \,\, b \in \mathcal S_{[s]}^{k-s,\ell+s} \mathcal S_{\dagger,[\ell - k + 2s]}^{\ell + s,k - s} \mathcal B^-\mathbb W_{[\ell - k + 2s + 1]}^{\ell + s,k - s}(\sigma).
\end{equation}}	

Next, we put back \eqref{eq:dof-Ws} to the decomposition \eqref{eq:Wp-decomp-new} and show that the degrees of freedom \eqref{dofs:constant-p} are unisolvent for $\mathbb W^{k,\ell}_{[p]}$. Since each factor $\mathbb W^{k-s,\ell +s}$ is uniquely determined by \eqref{eq:dof-Ws}, the total number of degrees of freedom is equal to the dimension of $\mathbb W^{k,\ell}_{[p]}$. Therefore, it suffices to show that if $\omega \in \mathbb W^{k,\ell}_{[p]}$ and all the moments in \eqref{dofs:constant-p} vanish on $\omega$, then $\omega = 0$. 

By the decomposition \eqref{eq:decomp-Wp}, we have 
\[\omega = \omega_0 + \omega_1 +\cdots + \omega_{p-1},\]
where $\omega_{t} \in \mathcal S_{[t]}^{k-t,\ell+t} \mathbb W^{k-t,\ell+t}. $ We want to show that vanishing of the moments in \eqref{dofs:constant-p} on $\omega$ implies vanishing of the moments in \eqref{eq:dof-Ws} on $\omega_s$, and hence $\omega_s=0$. In fact, this follows from the fact that the decomposition  \eqref{eq:decomp-Wp} is orthogonal.

This establishes the unisolvency. The conformity follows from the commuting properties of the operators.
\end{proof}
}

\begin{example}
    Let us now see $\mathbb W^{2,2}_{[2]}$ as an example. The degrees of freedom now become 
            \begin{equation} \langle \iota^{*}_{\sigma} \iota^*_{\sigma} \omega, b \rangle_{\sigma},\,\,\forall \,\, b \in  \mathcal B^-\mathbb W^{2,2}(\sigma),
\end{equation}
and
        \begin{equation} \langle \iota^{*}_{\sigma} \iota^*_{\sigma} \omega, b \rangle_{\sigma},\,\,\forall \,\, b \in  \mathcal S^{1,3} \mathcal S_{\dagger, [2]}^{3,1} \mathcal B^-\mathbb W_{[3]}^{3,1}(\sigma) .
\end{equation}

In four dimensions, there is one degree of freedom on each 2-face and there are $2+3=5$ degrees of freedom on each 3-face. The total dimension of $\mathbb W^{2,2}_{[2]}$ is 35. 
\end{example}

}

 \subsection{Step 3: $\iota^*\jmath^*$-conforming finite elements}

To achieve $\iota^* \jmath^*$-conformity for $(k, \ell)$-forms when $k \le \ell$, we adjust the degrees of freedom of both $\mathcal{P}^{-}\Alt^{k, \ell}$ and $\mathcal{P}^{-}\W^{k, \ell}$. Recall that a finite element is $\iota^* \jmath^*$-conforming if, for a simplex $\sigma \in \mathcal{T}_{\le \ell}$, the generalized double trace $\iota^* \jmath^*$ is single-valued, and for $\sigma \in \mathcal{T}_{\ge \ell}$, the standard double trace $\iota^* \iota^*$ is single-valued. Note again that at the index $\ell$, these definitions align.

 Before presenting the details, we show some examples to demonstrate the ideas. 
For $(0,1)$-forms in three dimensions, the shape function space is $\mathcal P^- \alt^{0,1} \cong \mathcal{P}_1 \otimes \mathbb R^{3}$. The $\iota^{*} \iota^*$-conformity translates to the tangential continuity of the vector. Therefore, the global finite element space is exactly the {N\'ed\'elec element of the second kind {\cite{nedelec1986new}}}. On the other hand, we note that the $\iota^*\jmath^*$-conformity means the continuity of every component. Therefore the $\iota^*\jmath^*$-conforming finite element space will be the vector Lagrange element. {Similarly, for $\mathcal P^-\W^{1,2}$ in three dimensions, the $\iota^{*}\iota^{*}$-conformity leads to an MCS$^\top$ element} \cite{gopalakrishnan2020mass} (traceless matrices with tangential-normal continuity on faces), while the $\iota^*\jmath^*$-conformity corresponds to the Hu-Lin-Zhang element with tangential-normal continuity on edges \cite{hu2025distributional}. 

In Step 1, we constructed $\iota^*\iota^*$-conforming finite element spaces. Now, we modify these spaces to achieve the $\iota^* \jmath^*$-conformity. To illustrate, consider $(0,1)$-forms in three dimensions. For each edge, we move the two tangential degrees of freedom from the edge to its two vertices, yielding the vector Lagrange element. Similarly, for $(1,2)$-forms in three dimensions, we redistribute the three degrees of freedom on each face to its three edges, as shown in \Cref{fig:hlz}. The idea of moving degrees of freedom is not new. Recent applications of this idea in the context of complexes can be found in \cite{gillette2020nonstandard,chen2024tangential,chen2024h}. As we see above,  the success of this process relies on matching the number of degrees of freedom to the number of subsimplices. Below we generalize this approach to any $(k, \ell)$-forms. Specifically, we redistribute the degrees of freedom from $\ell$-dimensional simplices to $k$-dimensional ones. In general, moving the degrees of freedom from higher dimensions to lower dimensions would enhance the continuity of the finite elements. In particular, $\iota^*\jmath^*$-continuity ensures $\iota^*\iota^*$-continuity.

  \begin{figure}[htbp]
\FIG{\includegraphics[width=0.5\linewidth]{\RevisionAssets/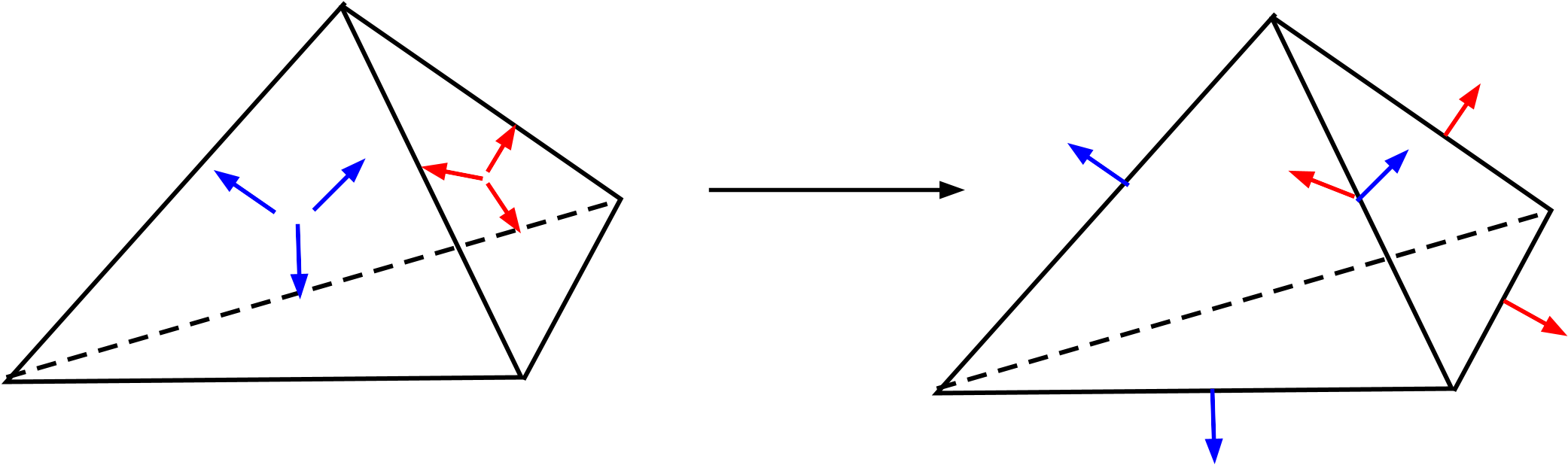}}
{\caption{An illustration for the construction of the Hu-Lin-Zhang traceless element. Here, we move the face degrees of freedom to each edge.}
\label{fig:hlz}}
\end{figure}

Now we provide the precise definition of the degrees of freedom. 
Recall that $\iota_{\sigma}^* \jmath_{\sigma}^*$ maps $C^{\infty} \otimes \alt^{k,\ell}(\mathbb R^n)$ to $C^{\infty}(\sigma) \otimes \alt^{k}(\sigma) \otimes \alt^{\ell - k}(\sigma^{\perp})$, which is the space of smooth sections of a vector bundle over $\sigma$. Correspondingly, we use $\langle \cdot, \cdot \rangle_{\sigma}$ to denote the $L^2$ inner product on these sections. That is, for the inner product $\langle \iota_{\sigma}^* \jmath_{\sigma}^* w, b\rangle_{\sigma}$, for $w\in C^{\infty} \otimes \alt^{k,\ell}(\mathbb R^n)$ and $b \in \alt^k(\sigma) \otimes \alt^{\ell - k}(\sigma^{\perp})$, we first take a pointwise inner product of $ \iota_{\sigma}^* \jmath_{\sigma}^*  w$ and $b$, and integrate on the $k$-dimensional cell $\sigma$ (rather than in the $n$-dimensional space). The degrees of freedom for $\iota^*\jmath^*$-conforming finite element spaces are given by this $L^2$ inner product.

\begin{proposition}[$C_{\iota^*\jmath^*} \cP^- \alt^{k,\ell}$ finite element] \label{prop:moving-dofs-alt}
 For $k\le\ell$, the degrees of freedom
\begin{equation}
\label{eq:dof-ij}
\begin{cases} \langle \iota^{*}_{\sigma}\jmath^{*}_{\sigma}  \omega , b \rangle_{\sigma}, & \forall~~ b \in \alt^k(\sigma) \otimes \alt^{\ell - k}(\sigma^{\perp}), \quad \dim \sigma = k, \\ 
 \langle \iota^{*}_{\sigma} \iota^*_{\sigma} \omega, b \rangle_{\sigma}, & \forall~~ b \in \mathcal B^-\alt^{k,\ell}(\sigma), \quad \dim \sigma > \ell,
\end{cases}	
\end{equation}
for each $\sigma \in \mathcal T(K)$
are unisolvent with respect to the shape function space $\mathcal P^- \alt^{k,\ell}(K)$.
  The resulting finite element space is $\iota^{*}\jmath^*$-conforming, and denoted as $C_{\iota^*\jmath^*} \cP^- \alt^{k,\ell}$. 
\end{proposition}
Note that the degrees of freedom at $\mathcal T_{>\ell}$ are identical to those of $\iota^*\iota^*$-conforming finite elements. 

\begin{proof}

The proof follows from carrying over the unisolvency of \eqref{eq:dof-ii} to that of \eqref{eq:dof-ij} by counting the number of degrees of freedom. First note that we use the same shape function space $\mathcal P^- \alt^{k,\ell}$ as in the case of $C_{\iota^*\iota^*} \cP^-\alt^{k,\ell}$. For the degrees of freedom, the only difference between \eqref{eq:dof-ij} and \eqref{eq:dof-ii} is those on the simplices of dimension $k$ and $\ell$. The dimension count is done once we show that  \eqref{eq:dof-ii} and \eqref{eq:dof-ij} have the same numbers. 

\begin{enumerate}
\item  For the $\iota^*\iota^*$-conforming space $C_{\iota^*\iota^*} \cP^-\alt^{k,\ell}$ , the number of degrees of freedom on each $\ell$-simplex $\sigma_{\ell}$ is \[\dim \mathcal B^- \alt^{k,\ell}(\sigma_{\ell}) = \dim \mathcal P^- \alt^{k,\ell}(\sigma_{\ell}) = \binom{\ell + 1}{\ell + 1} \binom{\ell + 1}{k + 1} = \binom{\ell + 1}{k+1}.\]
Therefore, the total number of degrees of freedom of \eqref{eq:dof-ii} associated with all $\ell$-faces of $K$ 
is $\displaystyle \binom{n+1}{\ell+1} \binom{\ell+1}{k+1}$. 

\item While for \eqref{eq:dof-ij}, the degrees of freedom on each $k$ simplex $\sigma_k$ have the dimension $\displaystyle \dim \alt^{\ell -k}(\sigma_k^{\perp}) = \binom{n - k}{\ell - k} $. Therefore, the total degrees of freedom of \eqref{eq:dof-ij} at $\mathcal T_k(K)$ is 
$\displaystyle  \binom{n-k}{\ell-k} \binom{n+1}{k+1}.$
\end{enumerate} 
By combinatoric identities, the two numbers are identical.

Now it suffices to verify the unisolvency and conformity. Since the second set of degrees of freedom in \eqref{eq:dof-ij} also exists in \eqref{eq:dof-ii}, it suffices to verify the following: for $w \in \cP^- \alt^{k,\ell}$, if $\langle \iota^{*}_{\sigma}\jmath^{*}_{\sigma}  w , b \rangle_{\sigma} = 0$ for all constant $b$ and $\sigma \in \mathcal T_{k}(K)$, then 
for each $F \in \mathcal T_\ell(K)$, $w_F := \iota_F^* \iota_F^* w$ vanishes. 

In fact, by the above vanishing conditions, it holds that $\iota^*_{\sigma} \jmath^*_{\sigma} w = 0$. By \Cref{lem:j-circ-i}, $\iota^*_{\sigma} \jmath^*_{\sigma} w_F = 0$. Note that $w_F$ is in  $\cP^-\alt^{k,\ell}(F) \cong \cP^-\alt^k(F)$. Therefore we have $\iota_{\tau}^* w_F = 0$ for all $\tau \in \mathcal T_{k}(F)$. By the unisolvency of the Whitney form, it then holds that $w_{F} = 0$. The remaining proof is implied in that of \Cref{prop:Pminus-ii}.
\end{proof}

\begin{example}
 For the $C_{\iota^*\iota^*}\mathcal{P}^-\alt^{1,2}$ finite elements, the degrees of freedom (DoFs) consist of face moments against the Raviart-Thomas space (3 per face) and 6 interior degrees of freedom within each tetrahedron. To construct the $C_{\iota^*\jmath^*}\mathcal{P}^-\alt^{1,2}$ element, we redistribute the DoFs from faces (2-simplices) to edges (1-simplices). Each face has 3 degrees of freedom and 3 edges, so we assign 1 degree of freedom to each edge. Consequently, each edge receives 2 degrees of freedom in total from its adjacent faces. This results in the $C_{\iota^*\jmath^*}\mathcal{P}^-\alt^{1,2}$ element, with degrees of freedom defined as moments of edge tangential-normal components (2 per edge) plus 6 interior degrees of freedom per tetrahedron.
\end{example}

  \begin{remark}[Skeletal and bubble degrees of freedom]\label{rmk:skeletal}
 Recall that for finite element differential forms, the lowest order moments on $k$-cells (those degrees of freedom against the lowest-order Whitney forms) are skeletal degrees of freedom, and other degrees of freedom are bubbles, in the sense that the skeletal basis (Whitney forms) gives a complex which carries the cohomology; while the bubbles on each cell lead to exact sequences independent of the topology of the domain or the global triangulation \cite{schoberl2005high,christiansen2010finite}. 
Similarly, we call the first set of degrees of freedom in \eqref{eq:dof-ij} (those on dimension $k$) the skeletal part and the second set (those on dimensions $> \ell$) the bubble part. The notion of \emph{skeletal degrees of freedom} and \emph{bubble degrees of freedom} will be used for all $\iota^* \jmath^*$-conforming and $\iota^*\jmath_{[p]}^*$-conforming elements discussed later.

\begin{equation*}
\begin{aligned} \langle \iota^{*}_{\sigma}\jmath^{*}_{\sigma}  \omega , b \rangle_{\sigma}, & \forall~~ b \in \alt^k(\sigma) \otimes \alt^{\ell - k}(\sigma^{\perp}), \quad \dim \sigma = k,  & \text{ skeletal} \\ 
 \langle \iota^{*}_{\sigma} \iota^*_{\sigma} \omega, b \rangle_{\sigma}, & \forall~~ b \in \mathcal B^-\alt^{k,\ell}(\sigma), \quad \dim \sigma > \ell,  & \text{ bubble}
\end{aligned}	
\end{equation*}
\end{remark}

Next, we consider how to modify the degrees of freedom of $C_{\iota^*\iota^*} \mathcal P^- \W^{k,\ell}$ to obtain $C_{\iota^*\jmath^*} \mathcal P^- \W^{k,\ell}$.
For $\sigma \in \mathcal{T}_{\ell}$, it holds that $\mathcal{B}^- \W^{k,\ell}(\sigma) = \mathcal{B}^- \alt^{k,\ell}(\sigma) = \mathcal{P}^- \alt^{k,\ell}(\sigma)$, where the first identity comes from the fact that $\mathcal S_{\dagger}^{k,\ell}: \Alt^{k,\ell} \to \Alt^{k-1,\ell+1}$ is zero on $\ell$-simplices, and the second equation comes from the fact that the trace operator of an $\ell$-form always vanishes on $(\ell-1)$-dimensional faces. Thus, the degrees of freedom of $\mathcal{P}^- \alt^{k,\ell}$ \eqref{eq:dof-ii} and $\mathcal{P}^- \W^{k,\ell}$ \eqref{dof:CiiW} on $\ell$-dimensional cells are identical. Consequently, the procedure used to transfer degrees of freedom from $\ell$-dimensional to $k$-dimensional cells for $C_{\iota^{\ast}\iota^{\ast}}\mathcal{P}^- \alt^{k,\ell}$ can be applied to $C_{\iota^{\ast}\iota^{\ast}}\mathcal{P}^- \W^{k,\ell}$, yielding the $\iota^\ast\jmath^\ast$-conforming symmetric space $C_{\iota^{\ast}\jmath^{\ast}}\mathcal{P}^- \W^{k,\ell}$. See the following proposition for a precise statement.

\begin{proposition}[$C_{\iota^*\jmath^*} \cP^- \W^{k,\ell}$ finite element] \label{prop:moving-dofs-W}
If $ k \le \ell$, then
the degrees of freedom
\[
\begin{cases} \langle \iota^{*}_{\sigma} \jmath^{*}_{\sigma} \omega, b \rangle_{\sigma}, & \forall~~ b \in \alt^{k}\otimes\alt^{\ell - k}(\sigma^{\perp}), \quad \dim \sigma = k, \\ 
 \langle \iota^{*}_{\sigma} \iota^*_{\sigma} \omega, b \rangle_{\sigma}, & \forall~~ b \in \mathcal B^-\W^{k,\ell}(\sigma) , \quad \dim \sigma > \ell,
\end{cases}	
\]
are unisolvent for $\cP^- \W^{k,\ell}$. The resulting finite element space is $\iota^{*}\jmath^*$-conforming, denoted as $C_{\iota^*\jmath^*} \cP^- \W^{k,\ell}$. 
\end{proposition}
The proof is analogous to that for $\alt^{k,\ell}$ and is therefore omitted. This follows because symmetry reduction occurs only on simplices in $\mathcal{T}_{> \ell}$, whereas redistribution occurs solely on simplices in $\mathcal{T}_{\ell}$.
\begin{example}
As a special case, $C_{\iota^*\jmath^*}\mathcal P^- \alt^{0,\ell}$ gives  $\alt^{\ell}$-valued Lagrange space, which is $\binom{n}{\ell}$ copies of the scalar Lagrange finite element spaces. Moreover,  $C_{\iota^*\iota^*} \mathcal P^-\alt^{k,n}$ gives the discontinuous space $C^{-1}\mathcal P^- \alt^k$.
\end{example}

\begin{example}
	We discuss some nontrivial examples involving symmetries. Again, we consider the symmetric (1,2)-form in three dimensions. The shape function space is then  $\mathbb W^{1,2} + \kappa \mathbb W^{2,2} = \mathbb T + \bm x \times \mathbb S$, which is traceless. For the $C_{\iota^*\iota^*} \mathcal P^- \mathbb W^{1,2}$ element with $\iota^{*}\iota^*$-conformity, the degrees of freedom are moments against face Raviart-Thomas spaces (3 per face) and 2 inside the cell. Note that the reduction only occurs for the interior degrees of freedom. Next, we move the degrees of freedom from faces to edges. For the $C_{\iota^* \jmath^*} \mathcal P^- \W^{1,2}$ element, the degrees of freedom are evaluation of the edge tangential-normal components (2 per edge) and 2 inside the tetrahedron. 
This gives the Hu-Lin-Zhang traceless element in \cite{hu2025distributional}. See \Cref{fig:hlz} {for the step of moving the degrees of freedom}, and \Cref{fig:hlz-reduction} for the whole procedure.
\end{example}

\subsection{$\iota^*\jmath_{[p]}^*$-conforming finite elements}\label{subsec:ijp}

    In this section, we discuss the $\iota^*\jmath_{[p]}^*$-conforming finite elements. All the results developed in this subsection will reduce to the $\iota^*\jmath^*$-conforming finite elements when  $p = 1$. Here, the conformity $\iota ^*\jmath_{[p]}^*$ calls for explanation: 
    \begin{itemize}
        \item On $\mathcal T_{\le \ell + p - 1}$, we require that $\iota^* \jmath^*_{[p]}$ is single-valued,
        \item while on $\mathcal T_{\ge \ell + p}$, we require that the double trace $\iota^*\iota^*$ is single-valued.
    \end{itemize} 
    Note that $\jmath^*_{F,[p]}$ is a direct sum of $p$ terms: $\vartheta_{F, \dim F}^{\ast}$, $\vartheta_{F, \dim F - 1}^{\ast}$, $\cdots$, $\vartheta_{F, \dim F - p +1}^{\ast}$. When $\dim F = k$, the range of $\vartheta_{F,\dim F  - s}^{\ast}$ is in $C^{\infty}(F) \otimes \alt^{k-s}(F) \otimes \alt^{\ell - k + s}(F^{\perp}).$ 
 Note that for large $s$, terms $\vartheta_{F, \dim F-s}^{\ast}$ may be trivial (vanish), leading to some ``degenerate'' case (\Cref{lem:jp-properties}). However, this does not affect the validity of the arguments presented below.

We discuss two cases below: $k < \ell$ and $\ell \le k \le \ell + p - 1$.

\subsubsection{Case 1: $k < \ell$.} We first assume that $k < \ell$. The degrees of freedom of the finite element space $C_{\iota^*\iota^*} \mathcal P^- \alt^{k,\ell}$ are located on the simplices of dimensions greater than or equal to $\ell$, i.e., $\mathcal T_{\ell}, \mathcal T_{\ell +1}, \cdots.$ 
To impose the $\iota^*\jmath_{[p]}^*$-conformity, we move the degrees of freedom on simplices of dimension $\ell, \ell + 1,\cdots, \ell + p - 1$ to $k$-simplices $\sigma$. The idea for this step of construction is: \emph{The new degrees of freedom on $\sigma$ gained from $(\ell + s)$-dimensional simplices will ensure the conformity of the generalized trace $\iota^{\ast}\vartheta_{\sigma,k-s}^*$.} 

To see that it is possible, we first check that the following two sets of degrees of freedom have the same cardinality:
\begin{enumerate}
    \item the number of degrees of freedom on $\mathcal T_{\ell +s}(K)$ before the reassignment, and
    \item the degrees of freedom on $\sigma$ ensuring the $\iota^{\ast}\vartheta_{\sigma,k-s}^*$-conformity on $\mathcal T_{k}(K)$ after the reassignment.
\end{enumerate}

We first consider (1). In fact, for each $(\ell + s)$-simplex $F$, the degrees of freedom before the redistribution are moments against the space $\mathcal B^{-}\alt^{k,\ell}(F)$. By \Cref{lem:bubble-decomposition}, we have  the following decomposition: 
$\mathcal B^{-}\alt^{k,\ell}(F) = \sum \phi_{\sigma} \otimes N^{\ell}(\sigma , F)$. Recall that the dimension of $N^{\ell}(\sigma, F)$ is $\binom{k}{\ell + k - \ell -s } = \binom{k}{k - s}.$ Therefore, it holds that 

\begin{equation*} \begin{split} \sum_{F \in \mathcal T_{\ell + s }(K)} \dim \mathcal B^- \alt^{k,\ell}(F) = & \sum_{F \in \mathcal T_{\ell + s}(K)} \sum_{\substack{\sigma: \sigma\trianglelefteq F\\ \dim \sigma =k }}\dim   \phi_{\sigma} \otimes N^{\ell}(\sigma , F) \\  = & \binom{k}{k-s} \sum_{F \in \mathcal T_{\ell + s}(K)} \sum_{\substack{\sigma: \sigma\trianglelefteq F\\ \dim \sigma =k }} 1.\end{split} \end{equation*}

By swapping the summation of $F$ and $\sigma$, it suffices to consider the number of $F$ that contains $\sigma$ for any $\sigma \in \mathcal T_{k}(K)$. The number of $F$ such that $\sigma \trianglelefteq F$ is then $\binom{n - k}{\ell - k + s} = \dim \alt^{\ell - k + s}(\sigma^{\perp})$. Therefore  
\[
\sum_{\substack{F \in \mathcal T_{\ell +s}(K) : \sigma\trianglelefteq F\\ \dim \sigma =k }}\dim   \phi_{\sigma} \otimes N^{\ell}(\sigma , F) =\binom{n - k}{\ell - k + s} \binom{k}{k - s}=\dim (\alt^{k-s}(\sigma)\otimes  \alt^{\ell - k + s}(\sigma^{\perp})). 
\]
Summing over all $\sigma$, we have 
\[
\sum_{F \in \mathcal T_{\ell +s} (K) }\dim  \mathcal B^{-}\alt^{k,\ell}(F)=\sum_{\sigma \in \mathcal T_{k}(K)}\dim (\alt^{k-s}(\sigma)\otimes  \alt^{\ell - k + s}(\sigma^{\perp}))
\]
for each $s$, where the left-hand side is the number of degrees of freedom moved out in this construction, and the right-hand side is the number of new degrees of freedom added on $k$-simplices. The equality therefore shows that the operation does not change the total number of degrees of freedom in an $n$-simplex. See \Cref{fig:moving-dofs} for an illustration.

\begin{figure}[htbp]
\FIG{\includegraphics[scale = 0.1]{\RevisionAssets/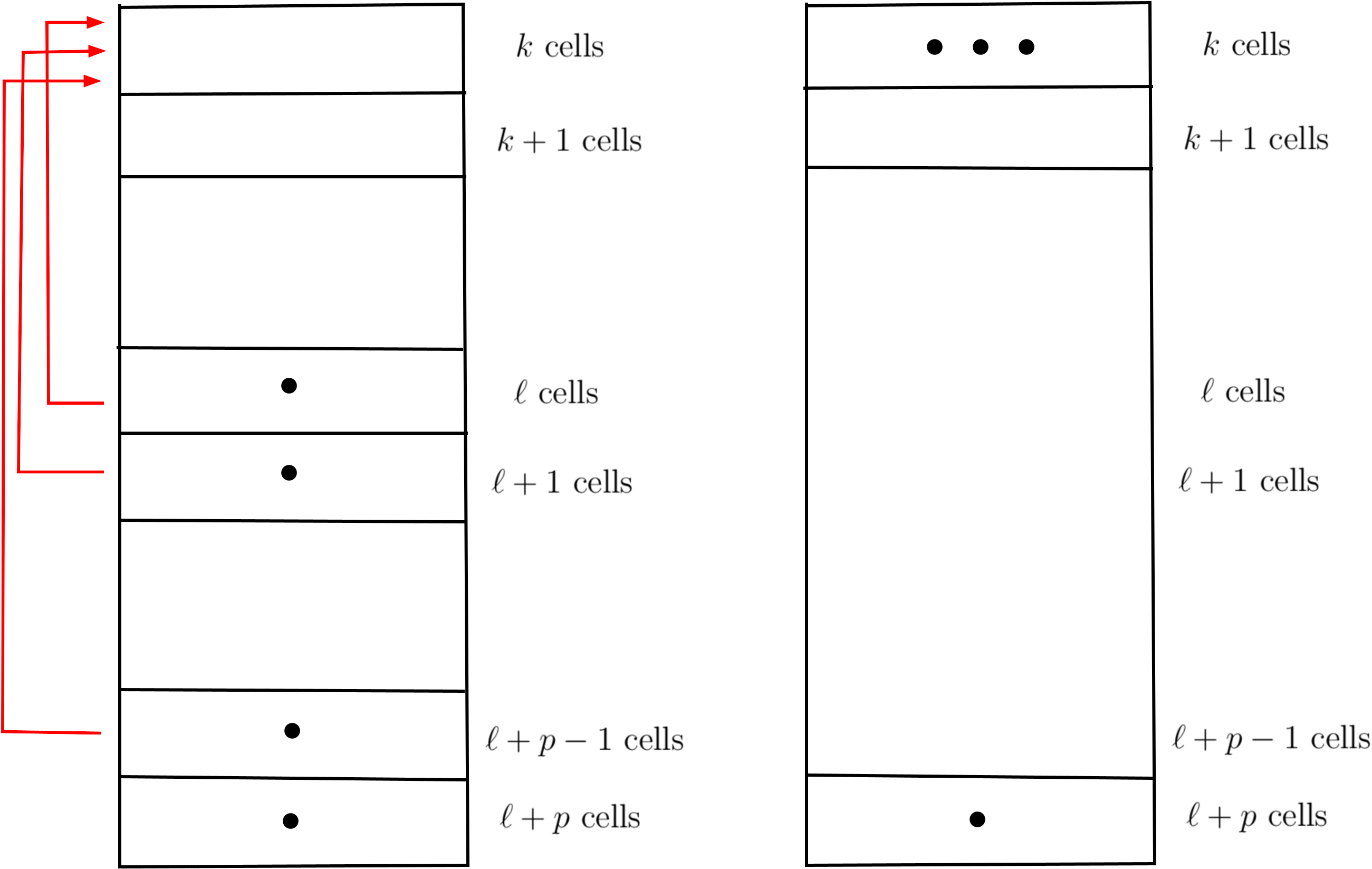}}
{\caption{An illustration for moving the degrees of freedom for the $\jmath_{[p]}^*$ case. Here, we move the degrees of freedom from $\ell$-, $(\ell+1)$-, .... $(\ell + p -1)$- faces to $k$-faces.}
	\label{fig:moving-dofs}}
\end{figure}

 Therefore, the redistribution essentially transfers $p$ sets of degrees of freedom from $\mathcal T_{\ell}(K), \mathcal T_{\ell+1}(K), \cdots, \mathcal T_{\ell+(p-1)}(K)$, and keeps the remaining degrees of freedom (those located in $\mathcal T_{\ge \ell + p}$).  Precisely, we have the following proposition.
\begin{proposition}[$C_{\iota^{*}\jmath^*_{[p]}}\mathcal P^- \alt^{k,\ell}$ finite element] \label{prop:Wp}
 For $k\le\ell+p-1$, the degrees of freedom

\begin{equation}
\label{eq:dof-ijp-alt-kl}
\begin{cases} \langle \iota^{*}_{\sigma}\vartheta^{*}_{\sigma,k} \omega, b \rangle_{\sigma}, & \forall~~ b \in \alt^{k}(\sigma) \otimes \alt^{\ell - k }(\sigma^{\perp}), \quad \dim \sigma = k, \\ 
\langle \iota^{*}_{\sigma}\vartheta^{*}_{\sigma,k - 1} \omega, b \rangle_{\sigma}, & \forall~~ b \in \alt^{k-1}(\sigma) \otimes \alt^{\ell - k + 1}(\sigma^{\perp}), \quad \dim \sigma = k,\\
\langle \iota^{*}_{\sigma}\vartheta^{*}_{\sigma,k - 2} \omega, b \rangle_{\sigma}, & \forall~~ b \in \alt^{k-2}(\sigma) \otimes \alt^{\ell - k + 2}(\sigma^{\perp}), \quad \dim \sigma = k, \\
\cdots \\
\langle \iota^{*}_{\sigma}\vartheta^{*}_{\sigma,k - p + 1 } \omega, b \rangle_{\sigma}, & \forall~~ b \in \alt^{k-p + 1 }(\sigma) \otimes \alt^{\ell - k + p - 1}(\sigma^{\perp}), \quad \dim \sigma = k, \\
 \langle \iota^{*}_{\sigma} \iota^*_{\sigma} \omega, b \rangle_{\sigma}, & \forall~~ b \in \mathcal B^-\alt^{k,\ell}(\sigma), \quad \dim \sigma \ge \ell + p,
\end{cases}	
\end{equation}
or written in a compact form:
\[
\begin{cases} \langle \iota^{*}_{\sigma}\jmath^{*}_{\sigma,[p]} \omega, b \rangle_{\sigma}, & \forall~~ b \in \bigoplus_{s = 0}^{p-1} \alt^{k-s}(\sigma) \otimes \alt^{\ell - k + s}(\sigma^{\perp}), \quad \dim \sigma = k, \\ 
 \langle \iota^{*}_{\sigma} \iota^*_{\sigma} \omega, b \rangle_{\sigma}, & \forall~~ b \in \mathcal B^-\alt^{k,\ell}(\sigma), \quad \dim \sigma \ge \ell + p,
\end{cases}	
\]
are unisolvent with respect to the shape function space $\mathcal P^- \alt^{k,\ell}(K)$. The resulting finite element space is $\iota^{*}\jmath^*_{[p]}$-conforming.	
\end{proposition}
\begin{proof}
We have verified that moving degrees of freedom as described above does not change the total number of degrees of freedom. Therefore, it suffices to show the unisolvency and conformity. 

Suppose that all the degrees of freedom in $\mathcal T_k$ vanish on $w \in \mathcal P^-\alt^{k,\ell}$.  
It suffices to show that $\iota_{F}^*\iota_F^*w = 0$ for any $F \in \mathcal T_{\ell + p -1}$. 
Then, by the last set of degrees of freedom, we can conclude with the unisolvency.
 Fix $F \in \mathcal T_{\ell + p -1}$ and let $w_F = \iota_F^* \iota_F^* w \in \mathcal P^-\alt^{k,\ell}(F)$. By \Cref{lem:j-circ-i}, it holds that $\iota_{E}^{*} \jmath_{E,[p]}^* w_F = 0$ for all $E \in \mathcal T_{k}(F)$. Note that $w_F$ is  defined on an $(\ell+p-1)$-dimensional simplex, therefore, by \Cref{lem:jp-properties}, it holds that $\jmath_{E,[p]}^*= \rho_E^*$. It follows that $w_F = 0$, leading to the unisolvency and conformity.
\end{proof}

\subsubsection{Case 2: $\ell \le k \le \ell + p - 1$} For the case $\ell \le k \le \ell + p - 1$, the finite element $C_{\iota^*\iota^*}\mathcal{P}^-\alt^{k,\ell}(K)$ behaves differently from the previous cases. Here, the degrees of freedom are located only on simplices of dimension $k$ and higher, i.e., $\mathcal{T}_k(K), \mathcal{T}_{k+1}(K), \ldots$, while no degrees of freedom are on  $\mathcal T_{[\ell,k-1]}(K)$ due to the $\iota^*\iota^*$-conformity. To construct the $C_{\iota^*\jmath^*_{[p]}}\mathcal{P}^-\alt^{k,\ell}(K)$ element, we redistribute the degrees of freedom from $\mathcal{T}_{k}(K), \mathcal{T}_{k+1}(K), \ldots, \mathcal{T}_{\ell + p - 1}(K)$ to those defined by the generalized traces $\vartheta_{\sigma, \ell}^{\ast}, \vartheta_{\sigma, \ell - 1}^{\ast}, \ldots, \vartheta_{\sigma, k - p + 1}^{\ast}$ on $\mathcal{T}_k$. This redistribution is reflected in the DoFs specified by \eqref{eq:dof-ijp-alt-kl}, where the traces $\vartheta_{\sigma, k}^{\ast}, \ldots, \vartheta_{\sigma, \ell + 1}^{\ast}$ vanish when $\ell \le k$.

\begin{example}
A trivial case is when $\jmath_{[p]}^* = \rho^*$ (which holds for sufficiently large $p$; see \Cref{lem:jp-properties}). In this case, the construction gives $ C_{\iota^*} \mathcal P^{-} \alt^k\otimes \alt^{\ell}$, i.e., alternating $\ell$-forms-valued finite element $k$-forms $ C_{\iota^*} \mathcal P^{-} \alt^k$ \cite{arnold2018finite,arnold2006finite}. \end{example}

It follows from \Cref{rmk:sym-reduction-not-happen} that the degrees of freedom in $\mathcal T_{< \ell + p}$ are not changed in the symmetry reduction; while the process of moving degrees of freedom only involves $\mathcal T_{< \ell + p}$. Therefore, we can first construct the $\iota^*\iota^*$-conforming space, namely, $C_{\iota^*\iota^*} \mathbb W_{[p]}^{k,\ell}$, and then obtain $C_{\iota^*\jmath_{[p]}^*} \mathbb W_{[p]}^{k,\ell}$ by moving degrees of freedom in the same way as obtaining $C_{\iota^*\jmath_{[p]}^*}   \Alt^{k,\ell}$. Formally speaking, the process of symmetry reduction and moving degrees of freedom commutes.

\begin{proposition} [$C_{\iota^{\ast}\jmath^{\ast}_{[p]}}\mathcal{P}^{-}\W_{[p]}^{k, \ell}$ finite element]\label{ijpW}
If $ k \le \ell + p - 1$, then
the degrees of freedom
\begin{equation}
\label{eq:dof-ijp-W-kl}
\begin{cases} \langle \iota^{*}_{\sigma}\vartheta^{*}_{\sigma,k} \omega, b \rangle_{\sigma}, & \forall~~ b \in \alt^{k}(\sigma) \otimes \alt^{\ell - k }(\sigma^{\perp}), \quad \dim \sigma = k,\\ 
\langle \iota^{*}_{\sigma}\vartheta^{*}_{\sigma,k - 1} \omega, b \rangle_{\sigma}, & \forall~~ b \in \alt^{k-1}(\sigma) \otimes \alt^{\ell - k + 1}(\sigma^{\perp}), \quad \dim \sigma = k, \\
\langle \iota^{*}_{\sigma}\vartheta^{*}_{\sigma,k - 2} \omega, b \rangle_{\sigma}, & \forall~~ b \in \alt^{k-2}(\sigma) \otimes \alt^{\ell - k + 2}(\sigma^{\perp}), \quad \dim \sigma = k, \\
\cdots \\
\langle \iota^{*}_{\sigma}\vartheta^{*}_{\sigma,k - p + 1 } \omega, b \rangle_{\sigma}, & \forall~~ b \in \alt^{k-p + 1 }(\sigma) \otimes \alt^{\ell - k + p - 1}(\sigma^{\perp}), \quad \dim \sigma = k, \\
 \langle \iota^{*}_{\sigma} \iota^*_{\sigma} \omega, b \rangle_{\sigma}, & \forall~~ b \in \mathcal B^-\W_{[p]}^{k,\ell}(\sigma), \quad \dim \sigma \ge \ell + p,
\end{cases}	
\end{equation}
or written compactly,
\[
\begin{cases} \langle \iota^{*}_{\sigma}\jmath^{*}_{\sigma,[p]} \omega, b \rangle_{\sigma}, & \forall~~ b \in \bigoplus_{s = 0}^{p-1} \alt^{k-s}(\sigma) \otimes \alt^{\ell - k + s}(\sigma^{\perp}),\quad \dim \sigma = k, \\ 
 \langle  \iota^{*}_{\sigma}\iota^{*}_{\sigma} \omega, b \rangle_{\sigma}, & \forall~~ b \in \mathcal B^-\W^{k,\ell}_{[p]}(\sigma) , \quad \dim \sigma \ge \ell + p,
\end{cases}	
\]
are unisolvent for {$\mathcal{P}^{-}\W_{[p]}^{k, \ell}$}. The resulting finite element space is $\iota^{*}\jmath_{[p]}^*$-conforming.	

\end{proposition}
\begin{remark}
It is also possible to construct the finite element space $C_{\iota^*\jmath_{[q]}^*} \mathcal P^- \W_{[p]}^{k,\ell}	$ whenever $ q \le p$. The degrees of freedom are
 {
\[
\begin{cases} \langle \iota^{*}_{\sigma}\jmath^{*}_{\sigma,[q]} \omega, b \rangle_{\sigma}, & \forall~~ b \in \bigoplus_{s = 0}^{q-1} \alt^{k-s}(\sigma) \otimes \alt^{\ell - k + s}(\sigma^{\perp}),\quad \dim \sigma = k, \\ 
 \langle  \iota^{*}_{\sigma}\iota^{*}_{\sigma} \omega, b \rangle_{\sigma}, & \forall~~ b \in \mathcal B^-\Alt^{k,\ell}(\sigma) , \quad  \ell + q  \le \dim \sigma \le \ell + p - 1,\\ 
 \langle  \iota^{*}_{\sigma}\iota^{*}_{\sigma} \omega, b \rangle_{\sigma}, & \forall~~ b \in \mathcal B^-\W^{k,\ell}_{[p]}(\sigma) , \quad \dim \sigma \ge \ell + p.
\end{cases}	
\]}
\end{remark}
\begin{example}\label{example:3D-allcases}
The resulting finite elements in three-dimensional space all exist in the literature. For $p \ge 2$, one of the cases discussed in \Cref{lem:jp-properties} holds. See the following examples:
\begin{enumerate}
\item For $\alt^{1,1}$, it holds that $\jmath_{e,[p]}^* = \rho^*$ for $p \ge 2$. In particular, the construction gives three copies of the N\'ed\'elec elements.
\item For $\alt^{2,2}$, it holds that $\jmath_{f,[p]}^* = \rho^*$ for $p \ge 2$. In particular, the construction gives three copies of the Raviart-Thomas elements.
\item For $\alt^{1,2}$, it holds that $\jmath_{e,[p]}^* = \rho^*$ for $p \ge 2$. In particular, the construction gives three copies of the N\'ed\'elec element. 
\item For $\alt^{2,1}$, we can consider the construction when $p \ge 2$. For $p = 2$, $\jmath_{f,[2]}^* = \iota^*$, while for the other cases, $\jmath_{f,[p]}^* = \rho^*$. {The former gives the full normal-tangential continuous element; imposing the $\W_{[2]}^{2,1}$ symmetry reduces it to the MCS element. The latter gives three copies of the Raviart-Thomas element.}
\end{enumerate}

\end{example}

We will see some nontrivial $\iota^*\jmath_{[p]}^*$-conforming finite elements in four dimensions in \Cref{sec:kkform}. 

\section{Tensorial Whitney forms: examples and complexes}

In this section, we present examples of tensorial Whitney forms constructed in the preceding section. We begin by summarizing examples of individual finite elements and complexes in three dimensions. Most constructions in two and three dimensions already exist in the literature. However, extending these patterns to higher dimensions is nontrivial due to increased complexity in continuity conditions and degrees of freedom (DoFs) assignments. In addition to translating the general results from the previous section, we highlight challenges in generalizing to higher dimensions and illustrate key patterns through specific examples. First, we examine the $(k,k)$-forms. Then, moving beyond individual finite elements, we investigate possibilities of fitting these individual finite elements and distributions into discrete BGG complexes. We verify a dimension count condition as an indication of the cohomology being correct.

\subsection{Recap in three dimensions}
In this subsection, we summarize the finite elements in three dimensions{.}
For simplicity, here we mainly focus on the symmetric version  $C_{\iota^*\iota^*} \mathcal P^- \W_{[p]}^{k,\ell}$ and $C_{\iota^*\jmath_{[p]}^*} \mathcal P^- \W_{[p]}^{k,\ell}$, see \Cref{tab:ii-alt,tab:ii-W,tab:ii-W2,tab:ij-W,tab:ij2-W2}.

\begin{table}[h]
\small
\setlength{\tabcolsep}{3pt}
\TBL{
\begin{center}
\begin{tabular}{|c|cccc|}\hline
\diagbox[width=1cm, height = 1cm]{$\ell$}{$k$}&0&1&2&3\\\hline
0&Lagrange  &first type N\'ed\'elec  &RT &DG
\\ 
& $C_{\rho}~\mathcal P_1$ &  $C_{t}~ (\mathbb V + \bm x \times \mathbb V)$ & $C_{n} ~(\mathbb V + \bm x \mathbb R)$ & $C^{-1} ~\mathbb R$ \\
& & & &
\\ 
1&second type N\'ed\'elec &full Regge, \Cref{fig:regge-reduction} (I)& full MCS, \Cref{fig:mcs-reduction} (I)  &vector DG\\
 & $C_{t}~\mathcal P_1\mathbb V$& $C_{tt}$ $(\mathbb{M}+\bm{x}\times\mathbb{M})$ & $C_{nt}$ $(\mathbb{M}+\bm{x}\mathbb{V})$  & $C^{-1}~\mathbb V$ 
 \\ & &  &  & \\ 
2&BDM, \Cref{fig:regge-reduction} (II)& $\mathrm{full~MCS}^{\top}$, \Cref{fig:hlz-reduction} (I)  &full HHJ, \Cref{fig:hhj-reduction} (I)&vector DG\\
& $C_n~\mathcal P_1\mathbb V$ & $C_{tn}~ (\mathbb{M}+\bm{x}\times\mathbb{M})$& $C_{nn} ~(\mathbb{M}+\bm{x}\mathbb{V})$  & $C^{-1}~\mathbb V$  \\ & & &  & \\ 
3&$\text{DG}_1$, \Cref{fig:hlz-reduction,fig:mcs-reduction} (II) &$C^{-1}\mathrm{Ned}$, \Cref{fig:hhj-reduction} (II) &$C^{-1}\mathrm{RT}$&DG\\
& $C^{-1}~\mathcal P_1$ & $C^{-1}~(\mathbb V + \bm x \times \mathbb V) $  & $C^{-1}~(\mathbb V + \bm x  \mathbb R) $ 
 & $C^{-1}~\mathbb R $ \\
\hline
\end{tabular}
\end{center}}
{\caption{$C_{\iota^*\iota^*} \mathcal{P}^-\alt^{k,\ell}$ finite elements.}
\label{tab:ii-alt}}
\end{table}

\begin{table}[h]
\small
\setlength{\tabcolsep}{3pt}
\TBL{
\begin{center}
\begin{tabular}{|c|cccc|}\hline
\diagbox[width=1cm, height = 1cm]{$\ell$}{$k$}&0&1&2&3\\\hline
0&Lagrange& - & - & - \\
&$C_{\rho} ~\mathcal P_1$ & & & \\
& & & &
\\ 
1&second type N\'ed\'elec& Regge, \Cref{fig:regge-reduction} (I-II) & - & - \\
& $C_{t} ~\mathcal P_1\mathbb V$ & $C_{tt}~ \mathbb S$ & & \\ & & &  & \\ 
2&BDM& traceless $\mathrm{MCS}^{\top}$, \Cref{fig:hlz-reduction} (I-II)  & HHJ, \Cref{fig:hhj-reduction} (I-II) & - \\
& $C_n~\mathcal P_1\mathbb V$ &$C_{tn}~ (\mathbb{T}+\bm{x}\times\mathbb{S})$ & $C_{nn}~\mathbb S$ & \\ & & &  & \\ 
3&$\text{DG}_1$ &$C^{-1}\mathrm{Ned}$&$C^{-1}\mathrm{RT}$&DG\\
& $C^{-1}~\mathcal P_1$ & $C^{-1}~(\mathbb V + \bm x \times \mathbb V) $  & $C^{-1}~(\mathbb V + \bm x  \mathbb R) $ 
 & $C^{-1}~\mathbb R $ \\
\hline
\end{tabular}
\end{center}
}
{\caption{ $C_{\iota^*\iota^*} \mathcal{P}^-\W^{k,\ell}$ finite elements.}
\label{tab:ii-W}}
\end{table}

\begin{table}[h]
\small
\setlength{\tabcolsep}{3pt}
\TBL{
\begin{center}
\begin{tabular}{|c|cccc|}\hline
\diagbox[width=1cm, height = 1cm]{$\ell$}{$k$}&0&1&2&3\\\hline
0&Lagrange&first type N\'ed\'elec& - & -\\ 
& $C_{\rho}~\mathcal P_1$& $C_{t}~(\mathbb V + \bm x \times \mathbb V)$ & & \\ 
& &  & & \\ 
1&second type N\'ed\'elec&full Regge & MCS, \Cref{fig:mcs-reduction} (I-II)& - \\
& $C_{t}~\mathcal P_1\mathbb V$ & $C_{tt}~ (\mathbb{M}+\bm{x}\times\mathbb{M})$& $C_{nt}~\mathbb T$ & \\
& & & & \\

2&BDM& ${\mathrm{full~MCS}^{\top}}$  &full HHJ  &vector DG\\
& $C_n~\mathcal P_1 \mathbb V$ &$C_{tn}~(\mathbb{M}+\bm{x}\times\mathbb{M})$ &$C_{nn}~(\mathbb{M}+\bm{x}\mathbb{V})$ & $C^{-1} ~\mathbb V$  \\
& & & & \\
3&$\text{DG}_1$ &$C^{-1}\mathrm{Ned}$&$C^{-1}\mathrm{RT}$&DG\\
& $C^{-1}~\mathcal P_1$ & $C^{-1}~(\mathbb V + \bm x \times \mathbb V) $  & $C^{-1}~(\mathbb V + \bm x  \mathbb R) $ 
 & $C^{-1}~\mathbb R $ \\
\hline
\end{tabular}
\end{center}}
{\caption{ $C_{\iota^*\iota^*} \mathcal{P}^-\W_{[2]}^{k,\ell}$ finite elements.}
\label{tab:ii-W2}}
\end{table}

\begin{table}[h]
\TBL{\begin{center}
\begin{tabular}{|c|cccc|}\hline
\diagbox[width=\dimexpr \textwidth/8 + 2\tabcolsep\relax, height = 1cm]{$\ell$}{$k$}&0&1&2&3\\\hline
0&Lagrange& - & - & - \\
& $C_{\rho} \mathcal P_1$ & & & \\
& & & & \\
1&vector Lagrange & Regge  & - & - \\
& $C_{\rho} \mathcal P_1 \otimes \mathbb V$ & $C_{tt} ~\mathbb S$ & & \\
& & & & \\
2&vector Lagrange & HLZ, \Cref{fig:hlz-reduction}, rightmost& HHJ & - \\
& $C_{\rho} \mathcal P_1 \otimes \mathbb V$ & edge $C_{tn}~ (\mathbb{T}+\bm{x}\times\mathbb{S})$ & $C_{nn}~\mathbb S$ & \\
& & & & \\
3&Lagrange&N\'ed\'elec & RT&DG\\
& $C_{\rho} \mathcal P_1$& $C_{t}~(\mathbb V + \bm x \times \mathbb V)$ & $C_{n}~(\mathbb V+ \bm x \mathbb R)$& $C^{-1}~\mathbb R$ \\
\hline
\end{tabular}\end{center}}
{\caption{ $C_{\iota^*\jmath^*} \mathcal{P}^-\W^{k,\ell}$ finite elements.}
\label{tab:ij-W}}
\end{table}

\begin{table}[h]
 \TBL{
\begin{center}
\begin{tabular}{|c|cccc|}\hline
\diagbox[width=\dimexpr \textwidth/8 + 2\tabcolsep\relax, height = 1cm]{$\ell$}{$k$}&0&1&2&3\\\hline
0&Lagrange&first type N\'ed\'elec& - & -\\
& $C_{\rho}~\mathcal P_1$ & $C_{t}~(\mathbb V + \bm x \times \mathbb V)$& & \\ 
& & & & \\ 
1&vector Lagrange&vector N\'ed\'elec & MCS $\mathbb{T}$& - \\
& $C_{\rho}~\mathcal P_1\otimes \mathbb V$  & $C_{t {\color{blue}\rho}}~ (\mathbb M + \bm x \times \mathbb M)$ & $C_{n{\color{red}t}}~\mathbb T$& \\
& & & & \\
2&vector Lagrange &vector N\'ed\'elec &vector RT &vector DG\\
& $C_{\rho}~\mathcal P_1\otimes \mathbb V$ & $C_{t{\color{blue}\rho}}~ (\mathbb M + \bm x \times \mathbb M)$ & $C_{n{\color{blue}\rho}}~ (\mathbb M + \bm x \otimes \mathbb V)$ & $C^{-1}~\mathbb V$ \\
& & & & \\
3&Lagrange&N\'ed\'elec & RT&DG\\
& $C_{\rho} \mathcal P_1$& $C_{t}~(\mathbb V + \bm x \times \mathbb V)$ & $C_{n}~(\mathbb V+ \bm x \mathbb R)$& $C^{-1}~\mathbb R$ \\
\hline
\end{tabular}
\end{center}
}
{\caption{ $C_{\iota^*\jmath_{[2]}^*} \mathcal{P}^-\W_{[2]}^{k,\ell}$ finite elements. Here, $\jmath_{[2]}^*$ is either $\rho^*$ (in blue) or $\iota^* = t$ (in red).}
\label{tab:ij2-W2}}
\end{table}

\subsection{$(k,k)$-forms in general dimensions}
\label{sec:kkform}

In this section, we consider the $(k,k)$-forms with $k\ge1$ in general dimensions, especially for $k = 1,2,3.$ Specifically, we consider 
\begin{enumerate}
\item $C_{\iota^*\iota^*} \mathcal P^- \alt^{k,k}$	and $C_{\iota^*\iota^*} \mathcal P^- \W_{[p]}^{k,k}$. The most interesting case is $p = 1$, where the local shape function space is $\mathcal P^- \W^{k,k} = \W^{k,k}$. 
\item $C_{\iota^* \jmath_{[p]}^*} \mathcal P^- \alt^{k,k}$ and $C_{\iota^* \jmath_{[p]}^*} \mathcal P^- \W_{[p]}^{k,k}$.
\end{enumerate}

  By \eqref{eq:bubble-dim-count}, we know that no degrees of freedom are put on any  $\sigma \in \mathcal T_{ >2k}$ for  $C_{\iota^*\iota^*}\mathcal P^{-}\Alt^{k,k}$, while for $\sigma \in \mathcal T_{2k}$, the numbers of the degrees of freedom are 
$\binom{2k+1}{k+1}$. 
In the symmetric case, the shape function space of $C_{\iota^*\iota^*} \mathcal P^- \mathbb W^{k,k}$ is constant $\mathcal P^- \mathbb W^{k,k} = \mathbb W^{k,k}$. The symmetry reduction removes the degrees of freedom of $C_{\iota^*\iota^*}\mathcal P^- \alt^{k-1,k+1}$ from those of $C_{\iota^*\iota^*}\mathcal P^- \alt^{k,k}$. For $C_{\iota^*\iota^*}\mathcal P^- \alt^{k-1,k+1}$, no degrees of freedom are located on $\sigma \in \mathcal T_{> 2k}$, while for $\sigma \in \mathcal T_{2k}$, the numbers of degrees of freedom are $\binom{2k+1}{k}$. Removing these numbers from the dimension count of $\mathcal P^- \alt^{k,k}$, we get the following corollary.
\begin{corollary}
 The degrees of freedom of $C_{\iota^*\iota^*}\mathcal P^- \mathbb W^{k,k}$ only appear on $\mathcal T_{[k,2k-1]}.$
\end{corollary}
In particular, for $k = 1$ (the Regge element), this implies that all the degrees of freedom of $C_{\iota^*\iota^*}\mathcal P^- \mathbb W^{1,1}$ are on 1-faces (edges).

Consequently,   the degrees of freedom for ${C_{\iota^{\ast}\iota^{\ast}}}\mathcal P^-\alt^{k,k}$ are 
\begin{equation} \langle \iota^{*}_{\sigma}\iota^{*}_{\sigma} \omega, b \rangle_{\sigma}, \qquad \forall~~ b \in \mathcal B^-\alt^{k,k}(\sigma) ,\,\, \forall \sigma \in  \mathcal T_{[k,2k]},
\end{equation}
while for ${C_{\iota^{\ast}\iota^{\ast}}}\mathcal P^- \mathbb W^{k,k}$, the degrees of freedom are
\begin{equation}
    \langle \iota^{*}_{\sigma}\iota^{*}_{\sigma} \omega, b \rangle_{\sigma}, \qquad \forall~~ b \in \mathcal B^-\mathbb W^{k,k} (\sigma) ,  \,\, \forall \sigma \in  \mathcal T_{[k,2k-1]}.
\end{equation}
Here, $
 \mathcal B^-\mathbb W^{k,k}(\sigma) = \mathcal B^- \alt^{k,k}(\sigma) \cap \ker(\mathcal S_{\dagger}^{k,k})=\mathcal{P}^{-}\W^{k, k}(\sigma)\cap \ker (\iota^{\ast}\iota^{\ast}). $
 See \eqref{def:B-} and \eqref{def:BW}.

\subsubsection*{(1,1)-forms}
For $k = 1$, this result covers the Regge element. In any space dimension, $C_{\iota^{\ast}\iota^{\ast}}\mathcal P^- \alt^{1,1}$ has one degree of freedom per edge and three degrees of freedom per 2-face. 

{For $p=1$, the symmetry reduction completely removes the face degrees of freedom (3-3=0), and thus we obtain the symmetric Regge element.} For $p > 1$, note that $\Alt^{1-p,1+p}$ is trivial. Thus $\W_{[p]}^{1,1} = \Alt^{1,1}$.
The dimension count is summarized in Table \ref{tab:dim-count-11}.
\begin{table}[htbp]
    \centering
    \TBL{
    \begin{tabular}{c|ccc}
        $n$ &  1 & 2 & $\ge 3$ \\ \hline 
     DoFs on $n$-face of $\mathcal P^- \alt^{1,1}$ & 1 & 3 &0\\ \hline
     DoFs on $n$-face of $\mathcal P^- \alt^{0,2}$ & 0 & 3 &  0 \\ 
          DoFs on $n$-face of $\mathcal P^- \mathbb W^{1,1}$ & 1 & {\color{blue} 0} & 0 \\ 
    \end{tabular}}
       {\caption{The dimension count involved in the construction of $C_{\iota^*\iota^*}\mathcal P^- \alt^{1,1}$ and $C_{\iota^*\iota^*}\mathcal P^- \W^{1,1}$. Here, we highlight the dimension (in blue) that the degrees of freedom are modified by the symmetric reduction. The numbers in the last row are obtained from the  first row minus the  second row.}
    \label{tab:dim-count-11}}
\end{table}

\subsubsection*{(2,2)-forms and (3,3)-forms}For $k=2$, we obtain the shape functions and degrees of freedom of $C_{\iota^*\iota^*}\mathcal P^- \alt^{2,2}$ and $C_{\iota^*\iota^*}\mathcal P^-\mathbb W^{2,2}$ by a similar argument. Note that the local shape function space of $C_{\iota^*\iota^*}\mathcal P^-\mathbb W^{2,2} $ is constant $\mathbb W^{2,2}$. 
Moreover, we can consider the kernel of $\mathcal S_{\dagger,[2]}^{2,2}$, yielding $C_{\iota^*\iota^*}\mathcal P^- \mathbb W^{2,2}_{[2]}.$ 
The dimension count is summarized in Table \ref{tab:dim-count-22}, where we list the symmetric space $\mathbb W$ with $p = 1$ and $2$. The case with $k=3$ is summarized in Table \ref{tab:dim-count-33}, where $p$ ranges from 1 to 3. We highlight in blue the number of DoFs after the symmetry reduction. 

\begin{table}[h]
    \centering
    \TBL{
    \begin{tabular}{c|ccccc}
        $n$ &  1 & 2 & 3 & 4 & $\ge 5$ \\ \hline 
     DoFs on $n$-face of $C_{\iota^*\iota^*}\mathcal P^- \alt^{2,2}$ & 0 & 1 & 8 & 10 & 0\\ \hline
     DoFs on $n$-face of $C_{\iota^*\iota^*}\mathcal P^- \alt^{1,3}$ & 0 & 0 & 6 & 10 & 0 \\ 
          DoFs on $n$-face of $C_{\iota^*\iota^*}\mathcal P^- \mathbb W^{2,2}$ & 0 & 1 & {\color{blue} 2} & {\color{blue} 0} & 0 \\ \hline
          DoFs on $n$-face of $C_{\iota^*\iota^*}\mathcal P^- \alt^{0,4}$ & 0 & 0 & 0 & 5& 0 \\ 
            DoFs on $n$-face of $C_{\iota^*\iota^*}\mathcal P^- \W^{2,2}_{[2]}$ & 0 & 1 & 8 & {\color{blue}5} & 0\\ 
    \end{tabular}}
 {\caption{The dimension count in $C_{\iota^*\iota^*}\mathcal P^- \alt^{2,2}$ and $C_{\iota^*\iota^*} \mathcal P^- \W^{2,2}_{[p]}$. The DoFs of $C_{\iota^*\iota^*}\mathcal P^- \mathbb W^{2,2}$ are obtained by $C_{\iota^*\iota^*}\mathcal P^- \mathbb \Alt^{2,2} \ominus C_{\iota^*\iota^*}\mathcal P^- \mathbb \Alt^{1,3}$; the DoFs of $C_{\iota^*\iota^*}\mathcal P^- \mathbb W^{2,2}_{[2]}$ are obtained as $C_{\iota^*\iota^*}\mathcal P^- \mathbb \Alt^{2,2} \ominus C_{\iota^*\iota^*}\mathcal P^- \mathbb \Alt^{0,4}$. }
    \label{tab:dim-count-22}}
\end{table}

\begin{table}[htbp]
    \centering
    \TBL{
    \begin{tabular}{c|ccccccc}
        $n$ &  1 & 2 & 3 & 4 & 5 & 6 & $\ge 7$ \\ \hline 
     DoFs on $n$-face of $\mathcal P^- \alt^{3,3}$ & 0 & 0 & 1 & 15 & 45 & 35 & 0\\ \hline
     DoFs on $n$-face of $\mathcal P^- \alt^{2,4}$ & 0 & 0 & 0 & 10 & 40  & 35 & 0 \\ 
          DoFs on $n$-face of $\mathcal P^- \mathbb W^{3,3}$ & 0 & 0 & 1 & {\color{blue}5} & {\color{blue}5} & {\color{blue}0} & 0 \\ \hline
           DoFs on $n$-face of $\mathcal P^- \alt^{1,5}$ & 0 & 0 & 0 & 0 & 15  & 21 & 0\\ 
            DoFs on $n$-face of $\mathcal P^- \W^{3,3}_{[2]}$ & 0 & 0 & 1 & 15 & {\color{blue}30}  & {\color{blue}14} & 0\\ \hline 
             DoFs on $n$-face of $\mathcal P^- \alt^{0,6}$ & 0 & 0 & 0 & 0 & 0  & 7 & 0\\ 
              DoFs on $n$-face of $\mathcal P^- \W^{3,3}_{[3]}$ & 0 & 0 & 1 & 15 & 45  & {\color{blue}28} & 0\\  
    \end{tabular}}
{\caption{The dimension count  involved in the construction of $C_{\iota^*\iota^*}\mathcal P^- \alt^{3,3}$ and $C_{\iota^*\iota^*}\mathcal P^- \W^{3,3}_{[p]}$.}
    \label{tab:dim-count-33}}
\end{table}

Next, we show how to move the degrees of freedom to obtain $\iota^* \jmath_{[p]}^*$-conforming space for $p \ge 2$. We begin by (1,1)-forms in any dimension. 

 For $\alt^{1,1}$ and  $p=2$, we can move the degree of freedom from 2-faces to 1-faces. Each 2-face has 3 degrees of freedom, and 3 edges. Therefore, each 2-face sends 1 degree of freedom to each of its edges, and each edge receives $(n-1)$ degrees of freedom in total.  The cases with  $p\geq 3$ are trivial.
	\begin{table}[h]
    \centering
    \TBL{
    \begin{tabular}{c|ccc}
        ${m}$ &  1 & 2 & $\ge 3$ \\ \hline 
     DoFs on ${m}$-face for $p = 1$ & 1 & 3 &0\\ \hline
DoFs on ${m}$-face for $p = 2$ & $n$ & 0 &0
    \end{tabular}}
 {\caption{The dimension count involved in the construction of $\mathcal P^- \alt^{1,1}$.}
     \label{tab:dim-count-11-j}}
\end{table}

  In ambient dimension four, for $\Alt^{2,2}$ and $p = 2$, we move the degrees of freedom from 3-faces to 2-faces. Each 3-face has 8 degrees of freedom and four 2-faces. Therefore, each 3-face sends 2 of its degrees of freedom to each of the 2-faces, and each 2-face receives 4 in total. Therefore, $C_{\iota^* \jmath_{[2]}^*} \alt^{2,2}$ has 5 degrees of freedom on each 2-face. 

For $\Alt^{2,2}$ and $p = 3$, we continue moving the degrees of freedom from 4-faces to 2-faces. Each 4-face has 10 degrees of freedom and 10 2-faces. Therefore, each 2-face receives 1 degree of freedom. The construction then has 6 degrees of freedom in each 2-face. The result is $\iota^* \rho^*$-conforming since the 6 degrees of freedom on each face correspond to the  six constant coefficient components of the second $\Alt^2(\mathbb R^4)$ factor. 

\Cref{tab:dim-count-22-j} shows the dimension count. Blue numbers indicate degrees of freedom after symmetry reduction, red numbers denote those redistributed to different cells, and green numbers highlight those that simplices receive from higher-dimensional cells. The blue and red sets of degrees of freedom are disjoint, indicating that the symmetry reduction and redistribution affect sets of degrees of freedom on different dimensions. For instance, {comparing the first and second rows, the redistribution} occurs on 2- and 3-cells, while comparing the second and fourth rows shows that the symmetry reduction affects degrees of freedom on 4-cells. These processes do not overlap. See also \Cref{tab:dim-count-33-j} for (3,3)-forms.

\begin{table}[htbp]
    \centering
    \TBL{
    \begin{tabular}{c|ccccc}
        $m$ &  1 & 2 & 3 & 4 & $\ge 5$ \\ \hline 
     DoF on $m$-face of  $C_{\iota^*\iota^*} \mathcal P^- \alt^{2,2}$& 0 & 1 & 8 & 10 & 0\\ 
     DoF on $m$-face of $C_{\iota^*\jmath^*_{[2]}} \mathcal P^- \alt^{2,2}$  & 0 & {\color{green!50!black}5} & {\color{red} 0} & 10 & 0\\
     DoF on $m$-face of $C_{\iota^*\jmath^*_{[3]}} \mathcal P^- \alt^{2,2}$  & 0 & {\color{green!50!black}6} & {0} & {\color{red}0} & 0\\
      \hline
            DoF on $m$-face of $C_{\iota^*\jmath^*_{[2]}} \mathcal P^- \W^{2,2}_{[2]}$ & 0 & 5 & 0 & {\color{blue} 5} & 0\\ 
    \end{tabular}}
         {\caption{
The dimension count involved in the construction of ${C_{\iota^* \jmath_{[p]}^*}}\mathcal{P}^- \alt^{2,2}$ and ${C_{\iota^* \jmath_{[p]}^*}}\mathcal{P}^- \W^{2,2}_{[p]}$ in ambient dimension four; $m$ denotes the dimension of the supporting face.}

            \label{tab:dim-count-22-j}}
\end{table}

\begin{table}[htbp]
    \centering
    \TBL{
    \begin{tabular}{c|ccccccc}
        ${m}$ &  1 & 2 & 3 & 4 & 5 & 6 & $\ge 7$ \\ \hline 
     DoFs on ${m}$-face for $C_{\iota^*\iota^*} \mathcal P^- \alt^{3,3}$ & 0 & 0 & 1 & 15 & 45 & 35 & 0\\ 
     DoFs on ${m}$-face for $C_{\iota^*\jmath^*_{[2]}} \mathcal P^- \alt^{3,3}$  & 0 & 0 & {\color{green!50!black}10} & {\color{red}0} & 45 & 35 & 0\\ 
     DoFs on ${m}$-face for $C_{\iota^*\jmath^*_{[3]}} \mathcal P^- \alt^{3,3}$  & 0 & 0 & {\color{green!50!black}19} & 0 & {\color{red}0} & 35 & 0\\ 
     DoFs on ${m}$-face for $C_{\iota^*\jmath^*_{[4]}} \mathcal P^- \alt^{3,3}$  & 0 & 0 & {\color{green!50!black}20} & 0 & 0 & {\color{red}0} & 0\\ 
     \hline
            DoFs on ${m}$-face of $C_{\iota^*\iota^*} \mathcal P^- \W^{3,3}_{[2]}$ & 0 & 0 & 1 & 15 & {\color{blue}30}  & {\color{blue}14} & 0\\
                    DoFs on ${m}$-face of $C_{\iota^*\jmath_{[2]}^*} \mathcal P^- \W^{3,3}_{[2]}$ & 0 & 0 & {\color{green!50!black}10} & {\color{red}0} & {\color{blue}30}  & {\color{blue}14} & 0\\
             \hline 
              DoFs on ${m}$-face of $C_{\iota^*\iota^*}  \mathcal P^- \W^{3,3}_{[3]}$ & 0 & 0 & 1 & 15 & 45  & {\color{blue}28} & 0\\ 
                       DoFs on ${m}$-face of $C_{\iota^*\jmath_{[2]}^*}  \mathcal P^- \W^{3,3}_{[3]}$ & 0 & 0 & {\color{green!50!black}10} & {\color{red}0} & 45  & {\color{blue}28} & 0\\  
                                DoFs on ${m}$-face of $C_{\iota^*\jmath_{[3]}^*}  \mathcal P^- \W^{3,3}_{[3]}$ & 0 & 0 & {\color{green!50!black}19} & 0 & {\color{red}0}  & {\color{blue}28} & 0\\   
    \end{tabular}}
 {\caption{{The dimension count involved in the construction of ${C_{\iota^{\ast}\jmath^{\ast}_{[p]}}}\mathcal P^- \alt^{3,3}$ and ${C_{\iota^{\ast}\jmath^{\ast}_{[p]}}}\mathcal P^- \W^{3,3}_{[p]}$ in ambient dimension six; $m$ denotes the dimension of the supporting face.}}
              \label{tab:dim-count-33-j}}
\end{table}

\subsection{Tensor-valued distributions and complexes}
From the pattern in Figure \ref{fig:table}, we observe that in three dimensions, the first part of each complex consists of finite elements in the classical sense (piecewise polynomials), while the second part consists of distributions (Dirac deltas). So far we have constructed finite elements which potentially fit in the first part of complexes (the lower triangular part of Figure \ref{fig:table}). In this subsection, we introduce the distributional spaces and verify the dimension count in any space dimension, which is a necessary condition and strong indication for the discrete complexes to have the correct cohomology.

For $k \ge \ell$, we introduce the following $\iota^*\jmath^*$ distributional spaces:
\begin{equation}\label{def:Dij}
D_{\iota^*\jmath^*} \widetilde{\mathbb W}^{k,\ell} : = \operatorname{span} \{ \omega \mapsto \langle \iota^{*} \jmath^{*} (\star\star) \omega, b \rangle_{\sigma}, \forall~~ b \in \alt^{n-k}(\sigma)\otimes (\alt^{n-k}(\sigma) \otimes \alt^{k - \ell}(\sigma^{\perp})), \sigma \in \mathcal T_{n-k}^{\circ} \},
\end{equation}
where $\star\star : \alt^{k,\ell} \to \alt^{n-k,n-\ell}$ is the two-sided Hodge star operator, and $\mathcal T_{n-k}^{\circ}$ is the set of internal $(n-k)$-simplices. Note that $\sigma$ has dimension $n-k$. Therefore $\alt^{n-k}(\sigma)$ is the volume form on $\sigma$. Therefore the space of $b$, $\alt^{n-k}(\sigma)\otimes (\alt^{n-k}(\sigma) \otimes \alt^{k - \ell}(\sigma^{\perp}))$, is isomorphic to $\alt^{k - \ell}(\sigma^{\perp})$.
The introduction of the Hodge star in the above definition \eqref{def:Dij} is based on the intuition that we obtained in the two and three dimensional cases (see \Cref{fig:table} and \cite{hu2025distributional}): the distributional discretization of $k$-forms is associated with $(n-k)$-cells (dual $k$-cells). For example, the standard de~Rham complex in three dimensions has DoFs at vertices, on edges, faces and cells; the distributional de~Rham complex has DoFs on cells, faces, edges, and vertices, respectively. This is consistent with the spirit of codimensional geometry \cite{Wang:2023:ECIG}.

The distributions above coincide with the skeletal DoFs (see \Cref{rmk:skeletal}) of $C_{\iota^*\jmath^*}\mathcal P^-\W^{n-k,n - \ell}$: 
\begin{equation*}
 \langle \iota^{*}_{\sigma} \jmath^{*}_{\sigma} \omega, b \rangle_{\sigma},  \forall~~ b \in \alt^{n-k}(\sigma)\otimes (\alt^{n-k}(\sigma) \otimes \alt^{k - \ell}(\sigma^{\perp})), \quad \dim \sigma = n - k,
\end{equation*}
while the other (bubble) degrees of freedom of the $\iota^*\jmath^*$-conforming elements do not appear in these distributional spaces, see \Cref{prop:moving-dofs-W}.

We can also use the delta notation to denote the distribution. That is, we use $b\delta_{\iota^*\jmath^*}(\sigma)$ to denote the distribution $\omega \mapsto \langle \iota^{*} \jmath^{*} (\star\star) \omega, b \rangle_{\sigma}$.

The complex now reads 
\begin{equation} \label{BGG-discrete}
\begin{tikzcd}[column sep = small]
0\ar[r] & C_{\iota^*\jmath^*} \mathcal P^- \W^{0,\ell} \ar[r,""] & C_{\iota^*\jmath^*} \mathcal P^- \W^{1,\ell} \ar[r,""] & \cdots   \ar[r,""] & C_{\iota^*\jmath^*}\mathcal P^- \W^{\ell,\ell} \ar[r] &  \ar[lllld,""] \\
~& ~ \ar[r]  & D_{\iota^*\jmath^*}  \widetilde{\W}^{\ell+1,\ell+1} \ar[r,""] & \cdots \ar[r,""] &  D_{\iota^*\jmath^*} \widetilde{\W}^{n-1,\ell+1}\ar[r,""]  & D_{\iota^*\jmath^*}  \widetilde{\W}^{n,\ell+1} \ar[r,""] & 0.
\end{tikzcd}
\end{equation}

More generally, for  $k \ge \ell - p + 1$, we introduce the following distribution spaces:
\begin{equation}\label{def:DijpW}
D_{\iota^*\jmath^*_{[p]}} \widetilde\W^{k,\ell}_{[p]} : = \operatorname{span} \{ \omega \mapsto \langle \iota^{*} \jmath^{*}_{[p]} (\star\star) \omega, b \rangle_{\sigma}, \forall~~ b \in  \bigoplus_{s = 0}^{p-1} \alt^{n - k-s}(\sigma) \otimes \alt^{k - \ell + s}(\sigma^{\perp}), \sigma \in \mathcal T_{n-k}^{\circ} \}. 
\end{equation}

Again, the distribution comes from the skeletal part of $C_{\iota^*\jmath_{[p]}^*}\mathcal P^-\W^{n-k,n-\ell}_{[p]}$, see \Cref{prop:Wp}. The distribution in \eqref{def:DijpW} associated with $b$ is also denoted as $b\delta_{\iota^*\jmath_{[p]}^*}(\sigma)$.

Now we can formally write down the BGG complex linking line $\ell$ and $\ell + p$ in \eqref{diagram-nD-algebraic}. 
 \begin{equation}
{\label{BGG-discrete-p}
\begin{tikzcd}[column sep = small]
0\ar[r] & C_{\iota^*\jmath^*_{[p]}} \mathcal P^- \W^{0,\ell}_{[p]} \ar[r,""] & C_{\iota^*\jmath^*_{[p]}} \mathcal P^- \W^{1,\ell}_{[p]} \ar[r,""] & \cdots   \ar[r,""] & C_{\iota^*\jmath^*_{[p]}}\mathcal P^- \W^{\ell + p - 1,\ell}_{[p]} \ar[r] &  \ar[lllld,""] \\
~& ~ \ar[r] & D_{\iota^*\jmath_{[p]}^*}  \widetilde{\W}^{\ell+1,\ell+p}_{[p]} \ar[r,""] & \cdots \ar[r,""] &  D_{\iota^*\jmath_{[p]}^*} \widetilde{\W}^{n-1,\ell+p}_{[p]} \ar[r,""]  & D_{\iota^*\jmath_{[p]}^*}  \widetilde{\W}^{n,\ell+p}_{[p]} \ar[r,""] & 0 ,
\end{tikzcd}	
}\end{equation}
We leave the details of the definition of the differential operator to a subsequent paper, but the result of the Euler characteristic (dimension count) is given below.

\begin{theorem}
\label{thm:euler}
    Given  a triangulation $\mathcal T$ of a  domain $\Omega$ whose closure is a topological $n$-ball, 
    the Euler characteristic of \eqref{BGG-discrete-p} is equal to that of the smooth BGG complex \eqref{BGG-seq-p}. 
That is,
     {
\begin{equation}
    \label{eq:euler}
    \sum_{\theta = 0}^{\ell +p - 1 } (-1)^{\theta} \dim C_{\iota^* \jmath^*_{[p]}} \mathcal P^- \W^{\theta,\ell}_{[p]}  + \sum_{\theta = 0}^{n-\ell-1} (-1)^{n-\theta+p-1} \dim D_{\iota^* \jmath^*_{[p]}} \widetilde{\W}^{n - \theta,\ell  + p}_{[p]} =  \binom{n}{\ell} + (-1)^{p-1} \binom{n}{\ell+p}. 	
    \end{equation}}

In particular, when  $p=1$, we obtain that the Euler characteristic of \eqref{BGG-discrete} is equal to that of \eqref{BGG-seq}. That is, 
     {
\begin{equation}
    \sum_{\theta = 0}^{\ell} (-1)^{\theta} \dim C_{\iota^* \jmath^*} \mathcal P^- \W^{\theta,\ell}  + \sum_{\theta = 0}^{n-\ell-1} (-1)^{n-\theta} \dim D_{\iota^* \jmath^*} \widetilde{\W}^{n - \theta,\ell + 1}  =  \binom{n+1}{\ell + 1}. 	
    \end{equation}}
  \end{theorem}

The theorem is proved through a two-part counting argument; see the propositions presented below. Recall that the degrees of freedom of $C_{\iota^*\jmath^*}\mathcal P^-\W^{k,\theta}$ are categorized into two parts: the skeletal part and the bubble part. We will show that the skeletal parts contribute to the right-hand side and the bubble sets on each simplex $\sigma$ contribute zero. This is a reminiscence of the property of finite element de~Rham complexes that the (lowest-order) Whitney forms carry the cohomology, while the bubbles fit within exact sequences. 

We first recall the definition of $C_{\iota^*\jmath_{[p]}^*} \W_{[p]}^{k,\ell}$. For each $k$-simplex $\sigma$, the number of skeletal DoFs on $\sigma$ is 
\[\dim \bigoplus_{s = 0}^{p-1} \alt^{k-s}(\sigma) \otimes \alt^{\ell - k + s}(\sigma^{\perp}) = \sum_{s = 0}^{p-1} \binom{k}{s} \binom{n - k}{n - \ell - s}.\]

\begin{proposition}[Dimension count on skeletal degrees of freedom]
\label{prop:skeletal-counting}
 Under the assumptions of \Cref{thm:euler}, let $k = \ell + p$ . The following identity holds:
 {
\begin{equation}
\begin{split} 
\sum_{s = 0}^{p-1} \sum_{\theta = 0}^{k} (-1)^{\theta} \binom{\theta}{s} \binom{n-\theta}{n - \ell -s} f_{\theta} & + (-1)^{n+p-1}  \sum_{s = 0}^{p-1} \sum_{\theta = 0}^{n - \ell}  (-1)^{\theta} \binom{\theta}{s} \binom{n - \theta}{k-s} f_{\theta}^{\circ} = \\ & \binom{n}{\ell} + (-1)^{p-1} \binom{n}{k}.
\end{split}
\end{equation}}
Here, $f_{\theta}$ is the number of $\theta$-simplices of $\mathcal T$, and $f^{\circ}_{\theta}$ is the number of internal $\theta$-simplices of $\mathcal T$. 

\end{proposition}

\begin{proof}
   See \Cref{sec:dof-counting}.
\end{proof}

\begin{proposition}[Dimension count of bubble degrees of freedom]
 For $m=\dim\sigma\ge\ell+p$, it holds that
{
\begin{equation}
\Psi_\sigma:=\sum_{\theta=0}^{\ell+p-1}(-1)^\theta\dim\mathcal B^-\alt^{\theta,\ell}(\sigma)
-\sum_{\theta=p}^{\ell+p-1}(-1)^\theta\dim\mathcal B^-\alt^{\theta-p,\ell+p}(\sigma)=0.
\end{equation}}
\end{proposition}
 \begin{proof}
 For an $m$-simplex $K$ with $m\ge\ell+p$, the corresponding alternating sum of the local shape dimensions is
{
\begin{align*}
&\binom{m}{\ell}\sum_{\theta=0}^{\ell+p-1}(-1)^\theta\binom{m+1}{\theta+1}
-(-1)^p\binom{m}{\ell+p}\sum_{j=0}^{\ell-1}(-1)^j\binom{m+1}{j+1}\\
&=\binom{m}{\ell}\left[1+(-1)^{\ell+p-1}\binom{m}{\ell+p}\right]
-(-1)^p\binom{m}{\ell+p}\left[1+(-1)^{\ell-1}\binom{m}{\ell}\right]\\
&=\binom{m}{\ell}+(-1)^{p-1}\binom{m}{\ell+p}.
\end{align*}}
 An empty sum is zero. By \Cref{prop:skeletal-counting},  the skeletal DoFs have the same alternating sum. Subtracting their contribution leaves
{
\[
\sum_{\substack{\sigma\trianglelefteq K\\\dim\sigma\ge\ell+p}}\Psi_\sigma=0.
\]}
 Induction on $m$, starting at $m=\ell+p$, gives $\Psi_K=0$.
\end{proof}

Now we show the examples in three and four dimensions. 
\subsubsection{Three-dimensional complexes}
In three dimensions, we follow the standard vector proxy notations to reformulate the complexes. In particular, $n$ stands for (face or edge) normal vectors, and $t$ stands for (face  or edge) tangential vectors. The distribution space is spanned by vector and matrix deltas. For example, $\delta_{t}(e)$ stands for a tangential delta, which is a vector-valued distribution; and $\delta_{tt}(e)$ stands for a tangential-tangential delta, which is a matrix-valued distribution.

The distribution space $D_{\iota^*\jmath^*}\widetilde \W^{k,\ell}$ and $D_{\iota^*\jmath_{[2]}^*}\widetilde \W^{k,\ell}_{[2]}$ are shown in \Cref{tab:Dij-W,tab:Dij2-W2}.  Here $\delta_{nn}$ and $\delta_{tt}$ automatically encode symmetries, and $\delta_{nt}$ and $\delta_{tn}$ encode tracelessness. 
\begin{table}[h]
\centering
\TBL{
\begin{tabular}{|c|cccc|}\hline
\diagbox[width=\dimexpr \textwidth/8 + 2\tabcolsep\relax, height = 1cm]{$\ell$}{$k$}&0&1&2&3\\\hline
0& cell $\delta$ & face $\delta_{n}$ & edge $\delta_t$ & vertex $\delta_{\rho} $  \\
1& -  & face $\delta_{nn}$ & edge $\delta_{nt}$  & vertex vector $\delta_{\rho}$ \\
2& -  & - & edge $\delta_{tt}$  & vertex vector $\delta_{\rho}$ \\
3&-&-& -&vertex $\delta_{\rho}$ \\
\hline
\end{tabular}}
{\caption{The delta representations of $D_{\iota^*\jmath^*} \widetilde\W^{k,\ell}$.}
\label{tab:Dij-W}}
\end{table}

\begin{table}[h]
\centering 
\TBL{
\begin{tabular}{|c|cccc|}\hline
\diagbox[width=\dimexpr \textwidth/8 + 2\tabcolsep\relax, height = 1cm]{$\ell$}{$k$}&0&1&2&3\\\hline
0& cell $\delta$ & face $\delta_n$ & edge $\delta_t$ & vertex $\delta_{\rho}$ \\
1& cell vector $\delta$  & face tensor $\delta_{\rho n}$ & edge tensor $\delta_{\rho t}$  & vertex vector $\delta_{\rho}$ \\
2& -  & face $\delta_{tn}$ & edge tensor $\delta_{t\rho}$  & vertex vector $\delta_{\rho}$ \\
3&-&-  &  edge $\delta_t$  &vertex $\delta_{\rho}$\\
\hline
\end{tabular}}
{\caption{The delta representations of $D_{\iota^*\jmath^*_{[2]}} \widetilde\W^{k,\ell}_{[2]}$.}
\label{tab:Dij2-W2}}
\end{table}

We have the following complexes;  unless otherwise specified, the differential operators are defined either in the classical or the distributional sense. 
\begin{enumerate}
\item The discrete Hessian complex (linking rows 0 and 1 of \eqref{diagram-nD}) is 
\begin{equation}
\begin{tikzcd}
0 \ar[r] & \textbf{Lag} \ar[r,"\hess"] & \bigoplus\limits_{f \in \mathcal T_2^{\circ}} \delta_{nn}(f) \ar[r,"\curl"] & \bigoplus\limits_{e \in \mathcal T_1^{\circ}} \delta_{nt}(e) \ar[r,"\div"] & \bigoplus\limits_{v \in \mathcal T_0^{\circ}} \delta(v) \otimes \mathbb V \ar[r] & 0.
\end{tikzcd}
\end{equation}
 Here, $\textbf{Lag} = C_{\iota^*\jmath^*} \mathcal P^- \alt^{0,0}$, see row 0 of \Cref{tab:ij-W} for finite element spaces and row 1 of \Cref{tab:Dij-W} for distributions.
\item The Christiansen--Regge elasticity complex (linking rows 1 and 2 of \eqref{diagram-nD}) is 
\begin{equation}
\begin{tikzcd}
0 \ar[r] & \textbf{Lag} \otimes \mathbb V \ar[r,"\sym\grad"] & \textbf{Reg} \ar[r,"\inc"] & \bigoplus\limits_{e \in \mathcal T_1^{\circ}} \delta_{tt}(e) \ar[r,"\div"] & \bigoplus\limits_{v \in \mathcal T_0^{\circ}} \delta(v) \otimes \mathbb V \ar[r] & 0.
\end{tikzcd}
\end{equation}
Here, $\textbf{Reg} = C_{\iota^*\iota^*}\W^{1, 1}$, see row 1 of \Cref{tab:ij-W} for the finite element spaces and row 2 of \Cref{tab:Dij-W} for the distributions.

\item The discrete divdiv complex (linking rows 2 and 3 of \eqref{diagram-nD}) is 
\begin{equation}
\begin{tikzcd}
0 \ar[r] & \textbf{Lag} \otimes \mathbb V \ar[r,"\dev\grad"] & \textbf{HLZ} \ar[r,"\sym\curl"] & \textbf{HHJ} \ar[r,"\widehat{\div\div}"] & \bigoplus\limits_{v \in \mathcal T_0^{\circ}} \delta(v)  \ar[r] & 0.
\end{tikzcd}
\end{equation}
 Here, $\textbf{HLZ} = C_{\iota^*\jmath^*} \mathcal P^- \W^{1, 2}$ and $\textbf{HHJ} = C_{\iota^*\iota^*} \mathcal P^- \W^{2, 2}$, see row {2} of \Cref{tab:ij-W} for finite element spaces and row {3} of \Cref{tab:Dij-W} for distributions.  
Moreover, the operator $\sym\curl$ is piecewise, and $\widehat{\div\div}$ is a modified distributional $\div\div$ operator (projection of the result of $\div\div$, as Dirac measures on edges, to Dirac measures at vertices). See \cite{hu2025distributional} for details. 
\item The discrete $\grad\curl$ complex (linking rows 0 and 2) is 
\begin{equation}
\begin{tikzcd}0 \ar[r] & \textbf{Lag} \ar[r,"\grad"] & \textbf{Ned} \ar[r,"\grad\curl"] & \bigoplus\limits_{f \in \mathcal T_2^{\circ}} \delta_{tn}(f)  \ar[r,"\curl"] & \bigoplus\limits_{e \in \mathcal T_1^{\circ}} \delta_{t}(e) \otimes \mathbb V \ar[r,"\div"] & \bigoplus\limits_{v \in \mathcal T_0^{\circ}} \delta(v) \otimes \mathbb V \ar[r] & 0.
\end{tikzcd}
\end{equation}
 See row 0 of \Cref{tab:ij2-W2} for the finite element spaces and row 2 of \Cref{tab:Dij2-W2} for the distributions.

\item The discrete $\curl\div$ complex (linking rows 1 and 3) is 
 \begin{equation}
\begin{tikzcd}0 \ar[r] & \textbf{Lag} \otimes \mathbb V \ar[r,"\grad"] & \textbf{Ned} \otimes \mathbb V \ar[r,"\dev\curl"] & \textbf{MCS}  \ar[r,"\curl\div"] & \bigoplus\limits_{e \in \mathcal T_1^{\circ}} \delta_{t}(e)  \ar[r,"\div"] & \bigoplus\limits_{v \in \mathcal T_0^{\circ}} \delta(v)\ar[r] & 0. 
\end{tikzcd}
\end{equation}
See row 1 of \Cref{tab:ij2-W2} for finite element spaces and row 3 of \Cref{tab:Dij2-W2} for distributions. 

\item The discrete $\grad\div$ complex (linking rows 0 and 3) is 
\begin{equation}
\begin{tikzcd}[column sep = small]
0 \ar[r] & \textbf{Lag} \ar[r,"\grad"] & \textbf{Ned} \ar[r,"\curl"] & \textbf{RT} \ar[r,"\grad\div"] &\bigoplus\limits_{f \in \mathcal T_2^{\circ}} \delta_{n}(f)  \ar[r,"\curl"] & \bigoplus\limits_{e \in \mathcal T_1^{\circ}} \delta_{t}(e) \ar[r,"\div"] & \bigoplus\limits_{v \in \mathcal T_0^{\circ}} \delta(v)\ar[r] & 0. 
\end{tikzcd}
\end{equation}
\end{enumerate}
The complexes (1)-(3) have been discussed in \cite{hu2025distributional,christiansen2023extended} (the Christiansen--Regge complex (2) already discussed in \cite{christiansen2011linearization}), while
the complexes (4)-(6) are new. 

\subsubsection{Four-dimensional complexes}
\label{sec:4D-results}

In four dimensions, we also show the construction of the Hessian complex that links rows 0 and 1 in the diagram \eqref{diagram-nD} in four dimensions, and the Regge complex that links rows 1 and 2 in the diagram \eqref{diagram-nD} in four dimensions. The dual of them leads to a discrete ``dual elasticity complex'' that links rows 2 and 3 of \eqref{diagram-nD}, and a divdiv complex that links rows 3 and 4 of \eqref{diagram-nD}. 

Following \cite{NIGAM2024198}, we introduce vector proxies for differential forms in four dimensions. We represent 0-forms and 4-forms using $\mathbb{R}$, 1-forms and 3-forms using $\mathbb{V} \cong \mathbb{R}^4$, and 2-forms using the skew-symmetric matrix space $\mathbb{K}$. 

\begin{remark}
There are two possibilities to identify a 2-form $\omega=\omega_{ij}dx^{i}\wedge dx^{j}$ (with the Einstein summation notation) where $\omega_{ij}=-\omega_{ji}$: one may obtain a skew-symmetric matrix whose $(i, j)$-entry is $\omega_{ij}$; alternatively, one may use the skew-symmetric matrix whose $(k, \ell)$-entry is $\omega_{ij}$.  Here, $i, j, k, \ell$ is any even permutation of 1 to 4. {We use the first convention below.} 

\end{remark}

The vector proxies for the traces $\iota^*$ of 1-, 2-, and 3-forms are $t$ (tangential components), $\gamma$ (cross-tangential components), and $n$ (normal components), respectively.

 The cross tangential component $\gamma$ for a skew-symmetric matrix calls for further explanations. For an oriented 2-face $f$ in $\mathbb{R}^4$ spanned by vectors $t_1$ and $t_2$, we define $\gamma_f = [t_1^i t_2^j]_{ij}$, {i.e., $\gamma_f =t_1\otimes t_2 $}. Given a $\mathbb{K}$-valued function $W = [w_{ij}]$, the trace is computed as $\sum t_1^i w_{ij}|_f t_2^j$. This trace is independent of the choice of $t_1$ and $t_2$, provided they form an  orthonormal basis of $f$ with the correct orientation. 

 Similar to those in three dimensions, these proxies apply to trace operators to faces of varying dimensions; see \Cref{tab:trace-proxy} for a summary.
  For instance, the trace of a 1-form involves extracting its tangential component. On an edge (a one-dimensional cell), this is achieved by computing $u \cdot t$, i.e., the dot product of the vector proxy $u$ of the 1-form with the tangent vector $t$ of the edge. On a 2-dimensional face, the trace corresponds to the tangential component of $u$, which can be obtained by projecting $u$ onto the tangent plane. This projection can be expressed as $\sum_{j=1}^2 (u \cdot t_j) t_j$, where $t_1$ and $t_2$ form an orthonormal basis of the tangent plane. The proxy of $u$ can be expressed by two numbers $u \cdot t_1$ and $u \cdot t_2$. However, in this case, we would view the trace of $u$ as a vector $\sum_{j=1}^2 (u \cdot t_j) t_j$, rather than two numbers (the coefficients). This approach naturally extends to higher dimensions.
 {For example, 
 we can extend the concept of the cross tangential component to 3-faces by contracting the proxy of a 2-form with all possible pairs of tangential vectors. For instance, consider a 3-face \( f \) equipped with an orthonormal basis \( \{t_1, t_2, t_3\} \). The trace operator maps a 2-form \( \sigma \in \K \) (represented as a skew-symmetric matrix) to:
 {
\begin{align*}
2\sum_{1 \leq i < j \leq 3} (t_i \cdot \sigma \cdot t_j) \, t_i \wedge t_j
\end{align*}}
Here, the wedge product is defined as \( t_i \wedge t_j = \frac{1}{2} (t_i \otimes t_j - t_j \otimes t_i) \), and the coefficient \( t_i \cdot \sigma \cdot t_j \) arises from the Frobenius inner product \( \sigma : (t_i \wedge t_j) \), which simplifies to \( t_i \cdot \sigma \cdot t_j \) due to the skew-symmetry of \( \sigma \).
In this context, the range of the trace operator consists of linear combinations of \( t_i \wedge t_j \) for \( 1 \leq i < j \leq 3 \), forming a three-dimensional space of skew-symmetric matrices spanned by these basis elements.
 }

\begin{table}[h]
\centering
\TBL{
    \begin{tabular}{c|ccccc}
        $k$ &  0 & 1 & 2 & 3 & 4\\
        vector proxies & $\mathbb R$ & $\mathbb V$ & $\mathbb K$ & $\mathbb V$ & $\mathbb R$ \\ 
        traces on $k$-cells &  value & $t$ & $\gamma$ & $n$ & integral 
    \end{tabular}}
    {\caption{Proxies of trace operators of $k$-forms in 4D}
    \label{tab:trace-proxy}}
\end{table}
\begin{table}[htbp]
\centering
 \TBL{
\begin{tabular}{cccccc}
\diagbox[width=\dimexpr \textwidth/20+\tabcolsep\relax, height=0.7cm]{ $k$ }{$\ell$}&  0 & 1 & 2 & 3 & 4 \\
0 & $\RR$ & $\VV$ & $\mathbb K$ & $\VV$ & $\RR$ \\
1  & $\VV$ & $\MM$ & $\mathbb V \otimes \mathbb K$ & $\MM$ & $\VV$ \\
2  & $\mathbb K$ & $\mathbb K \otimes \mathbb V$ & $\mathbb K \otimes \mathbb K$ & $\mathbb K \otimes \mathbb V$ & $\mathbb K$ \\
3  & $\VV$ & $\MM$ & $\mathbb V \otimes \mathbb K$ & $\MM$ & $\VV$ \\
4& $\RR$ & $\VV$ & $\mathbb K$ & $\VV$ & $\RR$ \\
\end{tabular}
\hspace{2em}
\begin{tabular}{cccccc}
\diagbox[width=\dimexpr \textwidth/20+\tabcolsep\relax, height=0.7cm]{ $k$ }{$\ell$}&  0 & 1 & 2 & 3 & 4 \\
0 & $\RR$ & $\mathbb V$ & $\mathbb K$ & $\mathbb V$ & $\mathbb R$ \\
1  & $\VV$ & $\mathbb S$ & $\mathbb B$ & $\mathbb T$ & $\mathbb V$ \\
2  & $\mathbb K$ & $\mathbb{B}$ & $\mathbb{AC}$ & $\mathbb N$ & $\mathbb K$ \\
3  & $\VV$ & $\mathbb T$ & $\mathbb N$ & $\mathbb S$ & $\mathbb V$ \\
4& $\RR$ & $\mathbb V$ & $\mathbb K$ & $\VV$ & $\RR$ \\
\end{tabular}\\
}
{\caption{Proxies of $\alt^{k,\ell}$ (left) and $\W^{k,\ell}$ for $k\leq \ell$, $\widetilde{\W}^{k,\ell}$ for $k>\ell$ (right) in 4D.  Here, $\mathbb B := \ker(\mathcal B)$ is the kernel of the algebraic Bianchi symmetrization; $\mathbb N := \ker(\mathcal M)$ satisfies certain trace-free property (see \eqref{def:M} below); $\mathbb{T}$ is the space of 4-by-4 trace-free matrices; $\mathbb{AC}$ is the algebraic curvature (see \eqref{def:AC}).}
\label{tab:proxy4d}}
\end{table}

Next, we consider proxies of the $\W^{k, \ell}$ spaces with symmetries. For $\mathbb K \otimes \mathbb V$, we index its components as $A_{ij,k}$. We then define the \emph{algebraic Bianchi symmetrization operator} 
\[ \mathcal B: \mathbb K \otimes \mathbb V \to \mathbb V, \quad [\mathcal BA]_{q} = [A]_{ij,p} + [A]_{jp,i} + [A]_{pi,j}.\] Here $\text{sgn}(i,j,p,q) = 1$. One can check that $\mathcal B$ is the proxy of  $\mathcal S^{2,1}$. We define
\[
\mathbb B := \ker(\mathcal B).
\]

The proxy of the Hodge star operator that maps 2-forms to 2-forms is denoted by $\varepsilon$, i.e., $[\varepsilon A]_{pq} = [A]_{ij}$ for $\text{sgn} (i,j,p,q) = 1$. In three dimensions, $\star$ corresponds to either $\vskw$ or $\mskw$.
One can check that $\mathcal B = \mathcal M \circ (\varepsilon \otimes id)$, where 
\begin{equation}\label{def:M}
\mathcal M : \mathbb K \otimes \mathbb V \to \mathbb V : [\mathcal MA]_{q} = [A]_{iq,i} + [A]_{jq,j} + [A]_{pq,p}
\end{equation}
is the trace operator.  For example, at $q=4$,
{
\[
[\mathcal M(\varepsilon\otimes\mathrm{id})A]_4=A_{23,1}+A_{31,2}+A_{12,3}=[\mathcal BA]_4,
\]}
and the other components follow by oriented relabelling.
\[
\mathbb N := \ker(\mathcal M).
\]
 The operator $\mathcal M$ is the proxy of $\mathcal S_{\dagger}^{2,3}$ up to a sign.  For simplicity, we do not distinguish the tensor products $\mathbb K \otimes \mathbb V$ and $\mathbb V \otimes \mathbb K$. Therefore, we can still use $\mathcal M$ to represent $\mathcal S^{3,2}$ and use $\mathcal B$ to represent $\mathcal S^{2,1}$ or $\mathcal S^{1,2}_{\dagger}$, and the  kernels are denoted by $\mathbb N$ and $\mathbb B$, respectively. Note that the Hodge star sends $\mathbb B$ to $\mathbb N$, and vice versa.
 The algebraic curvature space $\mathbb{AC}$ is the space of symmetric $\mathbb K \otimes \mathbb K$ with Bianchi's identity. That is, $R = [R]_{ij,k\ell} \in \mathbb{AC}$ if and only if for any $i,j,k,\ell$ it holds
 \begin{equation}\label{def:AC}
 [R]_{ij,k\ell} = [R]_{k\ell,ij}, \quad\mbox{and}\quad [R]_{ij,k\ell} + [R]_{jk,i\ell} + [R]_{ki,j\ell} = 0.
 \end{equation}
In four dimensions, the dimension of $\mathbb{AC}$ is 20.

With the above proxies, the de Rham complex in four dimensions reads
\begin{equation}
\begin{tikzcd}
0 \ar[r] & C^{\infty} \otimes \mathbb R \ar[r,"\grad"] & 	 C^{\infty}\otimes \mathbb V\ar[r,"\text{skwgrad}"] & C^{\infty}\otimes \mathbb K \ar[r,"\curl"] & C^{\infty} \otimes \mathbb V \ar[r,"\div"] & C^{\infty} \otimes \mathbb R \ar[r] & 0.
\end{tikzcd}
\end{equation}

\begin{example}
The smooth Hessian complex (linking rows 0 and 1 of \eqref{diagram-nD}) in four dimensions reads
\begin{equation}
\begin{tikzcd}
0 \ar[r] & C^{\infty} \otimes \mathbb R \ar[r,"\hess"] & 	 C^{\infty}\otimes \mathbb S \ar[r,"\text{skwgrad}"] & C^{\infty}\otimes \mathbb{B} \ar[r,"\curl"] & C^{\infty} \otimes \mathbb T \ar[r,"\div"] & C^{\infty} \otimes \mathbb V \ar[r] & 0.
\end{tikzcd}
\end{equation}

{For the distributions below, let $\nu_f=\varepsilon(t_1\wedge t_2)$ span $\Alt^2(f^\perp)$ for an oriented orthonormal tangent basis $t_1,t_2$ of a 2-face $f$. On an edge $e$, the symbol $\nu$ ranges over $\Alt^2(e^\perp)$. These normal two-form factors arise from the two-sided Hodge star in \eqref{def:Dij}.}

The discrete Hessian complex in four dimensions reads:
\begin{equation}
\label{eq:hessian-4d}
\begin{tikzcd}[column sep = small]
0 \ar[r] & \textbf{Lag} \ar[r,"\hess"] & \bigoplus\limits_{T \in \mathcal T_3^{\circ}} \delta_{nn}(T) \ar[r,"\text{skwgrad}"] & \bigoplus\limits_{f \in \mathcal T_2^{\circ}} {\delta_{n\nu}(f) } \ar[r,"\curl"] & \bigoplus\limits_{e \in \mathcal T_1^{\circ}} \delta_{nt}(e) \ar[r,"\div"] & \bigoplus\limits_{v \in \mathcal T_0^{\circ}} \delta(v) \otimes \mathbb V \ar[r] & 0.
\end{tikzcd}
\end{equation}
Here, we elaborate on the nature of the Dirac deltas involved:
\begin{itemize}
    \item \(\delta_{nn}(T)\) represents the Dirac delta associated with the normal-normal component on the 3-face \( T \). Since each 3-face has one normal vector, it contributes one delta function.
    \item {\(\delta_{n\nu}(f)\) denotes the \(\mathbb{V}\otimes\mathbb{K}\)-valued distributions with coefficients $n\otimes\nu_f$, where $n\in f^\perp$. The normal plane of $f$ has dimension two, so each 2-face contributes two delta functions.}
    \item \(\delta_{nt}(e)\) consists of the Dirac deltas associated with the normal-tangential components on the edge \( e \). Each edge contributes three such deltas.
    \item \(\delta(v) \otimes \mathbb{V}\) comprises Dirac deltas at the vertex \( v \), each associated with a vector in \( \mathbb{V} \).
\end{itemize}
Regarding \eqref{eq:hessian-4d}, the Lagrange finite element space is defined without boundary conditions. In contrast, the spaces of Dirac deltas are restricted to non-boundary entities, effectively imposing zero boundary conditions for functions or forms. This pattern is consistent with examples in two and three dimensions and is supported by functional analytic arguments, as detailed in \cite{hu2025distributional}.

\end{example}

\begin{example}
The smooth elasticity complex (linking rows 1 and 2 of \eqref{diagram-nD}) in four dimensions reads:
\begin{equation}\label{smooth-elasticity-4D}
\begin{tikzcd} 
0 \ar[r] & C^{\infty} \otimes \mathbb V \ar[r,"\sym\grad"] & 	 C^{\infty}\otimes \mathbb S \ar[r,"\text{riem}"] & C^{\infty}\otimes \mathbb{AC} \ar[r,"\curl"] & C^{\infty} \otimes \mathbb N \ar[r,"\div"] & C^{\infty} \otimes \mathbb K \ar[r] & 0.
\end{tikzcd}
\end{equation}
Here, the linearized Riemann operator is defined as $\text{riem} = \skw\grad \top \skw \grad$ (exerting $ \skw\grad$ for both rows and columns of a symmetric matrix).

The discrete elasticity complex in four dimensions is given by:
\begin{equation}
\label{eq:elasticity-4d}
\begin{tikzcd}[column sep = small]
0 \ar[r] & \textbf{Lag} \otimes \mathbb{V} \ar[r, "\sym\grad"] & \textbf{Reg} \ar[r, "\text{riem}"] & \bigoplus_{f \in \mathcal{T}_2^{\circ}} {\delta_{\nu\nu}(f)} \ar[r, "\curl"] & \bigoplus_{e \in \mathcal{T}_1^{\circ}} {\delta_{\nu t}(e)} \ar[r, "\div"] & \bigoplus_{v \in \mathcal{T}_0^{\circ}} \delta(v) \otimes \mathbb{K} \ar[r] & 0.
\end{tikzcd}
\end{equation}

The spaces involving Dirac deltas are defined as follows:
\begin{itemize}
    \item {\(\delta_{\nu\nu}(f)\) is the curvature-valued Dirac delta with coefficient $\nu_f\otimes\nu_f$ on the internal 2-face $f$. Each such 2-face contributes one delta function.}

    \item {\(\delta_{\nu t}(e)\) denotes the Dirac deltas with coefficients $\nu\otimes t$, where $t$ is tangent to the internal edge $e$ and $\nu\in\Alt^2(e^\perp)$. Since $\dim e^\perp=3$, there are $\binom{3}{2}=3$ such delta functions per edge.}
\end{itemize}
\end{example}

\begin{example}

The dual elasticity complex (linking rows 2 and 3 of \eqref{diagram-nD}) in four dimensions
\begin{equation}
\begin{tikzcd}
0 \ar[r] & C^{\infty} \otimes \mathbb K \ar[r,"\pi_{\mathbb B} \circ \grad"] & 	 C^{\infty}\otimes \mathbb{B} \ar[r,"\pi_{\mathbb{AC}} \circ \text{skwgrad} "] & C^{\infty}\otimes \mathbb{AC} \ar[r,"\curl \circ \top \curl"] & C^{\infty} \otimes \mathbb{S} \ar[r,"\div"] & C^{\infty} \otimes \mathbb V \ar[r] & 0,
\end{tikzcd}
\end{equation}
 can be regarded as the dual of elasticity complex \eqref{smooth-elasticity-4D}, 
where $\pi_{\mathbb B}$ and $\pi_{\mathbb{AC}}$ denote the algebraic projections to $\mathbb B$ and $\mathbb{AC}$, respectively.

On the discrete level, we can fit in the spaces constructed in this paper to obtain a discrete dual elasticity complex:
\begin{equation}\label{discrete-dual-elasticity}
\begin{tikzcd}[column sep = small]
0 \ar[r] & \textbf{Lag} \otimes \mathbb K \ar[r] & C_{\iota^*\jmath^*} \mathcal P^- \W^{1,2} \ar[r] & 	C_{\iota^*\iota^*} \mathcal P^- \W^{2,2}  \ar[r] & \bigoplus\limits_{e \in \mathcal T_1^{\circ}} \delta_{tt}(e) \ar[r,""] & \bigoplus\limits_{v \in \mathcal T_0^{\circ}} \delta(v) \otimes \mathbb V\ar[r] & 0.
\end{tikzcd}
\end{equation}
The operators in \eqref{discrete-dual-elasticity} can be regarded as the dual of the operators in \eqref{eq:elasticity-4d}. 
Here, $\bigoplus\limits_{e \in \mathcal T_1^{\circ}} \delta_{tt}(e) $ is the dual of the lowest order Regge space in 4D \cite{christiansen2011linearization,li2018regge}. The shape function space of $C_{\iota^*\jmath^*} \mathcal P^- \mathbb W^{1,2}$ has the form of $\mathbb B + \bm x \times \mathbb{AC}$, whose dimension is 40. The degrees of freedom are
\begin{align} 
\int_e A_{t \gamma},   &  &\text{ count = 3 on  each edge}\\
\int_T A_{t \gamma} : b,\quad \forall\,\, b \in \mathcal B^-\W^{1,2}(T). & &  \text{ count = 2 on each 3-face}
\end{align}	

Here, the shape function space of $C_{\iota^*\iota^*} \mathcal P^-\mathbb W^{2,2}$ is $\mathbb {AC}$, whose dimension is  20. The degrees of freedom are 
\begin{align} 
\int_f A_{\gamma \gamma},   &  &\text{ count = 1 on  each 2-face}\\
\int_T A_{\gamma \gamma} : b, \quad \forall\,\, b \in \mathcal B^-\W^{2,2}(T). & &  \text{ count = 2 on each 3-face}
\end{align}	

\end{example}

\begin{example}
We also consider the divdiv complex, which is the formal adjoint of the Hessian complex: 
\begin{equation}
\begin{tikzcd}
0 \ar[r] & C^{\infty} \otimes \mathbb V \ar[r,"\dev \grad"] & 	 C^{\infty}\otimes \mathbb{T} \ar[r,"\pi_{\mathbb N} \circ \text{skwgrad}"] & C^{\infty}\otimes \mathbb{N} \ar[r,"\sym \curl"] & C^{\infty} \otimes \mathbb{S} \ar[r,"\div\div"] & C^{\infty} \otimes \mathbb R \ar[r] & 0.
\end{tikzcd}
\end{equation}

Similarly, we fit finite element spaces and Dirac deltas in a discrete version with the dual of the operators in \eqref{eq:hessian-4d}: 
\begin{equation}
\begin{tikzcd}[column sep = small]
0 \ar[r] & \textbf{Lag} \otimes \mathbb V \ar[r] & C_{\iota^*\jmath^*} \mathcal P^- \W^{1,3} \ar[r] & 	C_{\iota^*\jmath^*} \mathcal P^- \W^{2,3}  \ar[r] & 	C_{\iota^*\iota^*} \mathcal P^- \W^{3,3}   \ar[r,""] & \bigoplus\limits_{v \in \mathcal T_0^{\circ}} \delta(v)\ar[r] & 0.
\end{tikzcd}
\end{equation}

Here, the shape function space of $C_{\iota^*\jmath^*} \mathcal P^- \mathbb W^{1,3}$ has the form of $\mathbb T + \bm x \times \mathbb N$, whose dimension is 35. The degrees of freedom are
\begin{align} 
\int_e A_{tn},   &  &\text{ count = 3 on  each edge}\\
\text{interior bubbles} & &  \text{ count = 5}
\end{align}	

The shape function space of $C_{\iota^*\jmath^*} \mathcal P^- \mathbb W^{2,3}$ is {$\mathbb N + (\varepsilon\otimes\mathrm{id})\skw(\bm x \otimes \mathbb S)$, whose dimension is 30. Here $\skw$ acts on the first two indices and $\varepsilon$ acts on the resulting skew pair, so the linear summand lies in $\ker\mathcal M$ by $\mathcal M(\varepsilon\otimes\mathrm{id})=\mathcal B$.} The degrees of freedom are
\begin{align} 
\int_f A_{ \gamma n},   &  &\text{ count = 2 on  each 2-face}\\
\text{interior bubbles} & &  \text{ count = 10}
\end{align}	

The shape function space of $C_{\iota^*\jmath^*} \mathcal P^- \mathbb W^{3,3}$ is $\mathbb S$, whose dimension is 10.  The degrees of freedom are
\begin{align} 
\int_T A_{n n},   &  &\text{ count = 1 on  each 3-face}\\
\text{interior bubbles} & &  \text{ count = 5}
\end{align}	
This is the HHJ element in four dimensions.

\end{example}

\begin{remark}
The preceding argument establishes that the differential operators in the discrete Hessian complex \eqref{eq:hessian-4d} and the elasticity complex \eqref{eq:elasticity-4d} can be precisely defined. The remaining complexes can then be derived as their duals. However, in higher dimensions, we have more complexes, and this approach fails to fully specify all differential operators, particularly those associated with zig-zag patterns, such as the Hessian, \(\inc\), and \(\div\div\) operators, except for those in the Hessian and elasticity complexes. To define these operators, we need to investigate integration by parts and calculate the operators in the sense of distributions. This is left as a future work. 
\end{remark}

\section{High-order cases}
\label{sec:high-order}
This section generalizes the idea of the lowest-order case to general degrees $r \ge 1$. Recall that standard finite element exterior calculus covers two types of finite element $k$-forms:
\begin{enumerate}
    \item $C_{\iota^*}\mathcal P_r^-\alt^k$. The dimension is $\displaystyle \binom{r+n}{r+k}\binom{r+k-1}{k}$. The degrees of freedom are represented by a dual basis,
    \[u \mapsto \int_{F} (\iota_F^* u) \wedge g, \quad \forall g \in \mathcal P_{r+k-m-1}\Alt^{m-k}(F),\] 
    for any $F$ with $\dim F := m \ge k$. 
    \item $C_{\iota^*}\mathcal P_r\alt^k$.  The dimension is $\displaystyle \binom{r+n}{r+k}\binom{r+k}{k}$.  The degrees of freedom are represented by a dual basis,
    \[u \mapsto \int_{F} (\iota_F^* u) \wedge g, \quad \forall g \in \mathcal P_{r+k-m}^-\Alt^{m-k}(F),\] 
    for any $F$ with $ \dim F:= m \ge k$.
\end{enumerate}

The bubble spaces of $\mathcal P_r^- \alt^k$ and $\mathcal P_r \alt^k$ are denoted as $\mathcal B_r^- \alt^k$ and $\mathcal B_r \alt^k$, respectively. By the above degrees of freedom, it holds that for $\dim K = m$,  
\begin{equation}\label{eq:dim-Br-}
\dim \mathcal B_r^-\alt^k(K) = \dim \mathcal P_{r+k-m-1}\alt^{m-k}(K) = \binom{r+k-1}{r-1} \binom{r-1}{m-k},
\end{equation}

and 
\begin{equation}\label{eq:dim-Br}
\dim \mathcal B_r\alt^k(K) = \dim \mathcal P_{r+k-m}^-\alt^{m-k}(K) = \binom{r+k}{r} \binom{r - 1}{m - k}.
\end{equation}
As a consequence, $\mathcal B_r^-\Alt^k(K)$ and $\mathcal B_r\Alt^k(K)$ are zero if $m < k$ or $m \ge k + r$. 
In this section, we introduce finite element form-valued forms based on the two spaces above $C_{\iota^*}\mathcal P_r^-\alt^k$ and $C_{\iota^*}\mathcal P_r\alt^k$.

A technical challenge in the high-order case is the linear dependence among the natural barycentric generators.
We now recall the definitions provided in \Cref{cor:spanning-whitney}. In \cite{arnold2009geometric}, high-order Whitney forms were introduced as a basis for the $\mathcal{P}_r\alt^k$ and $\mathcal{P}_r^-\alt^k$ families. For a multi-index $\alpha \in \mathbb{N}_0^{n+1}$, let $\lambda^{\alpha} := \lambda_1^{\alpha_1} \cdots \lambda_{n+1}^{\alpha_{n+1}}$. The support of $\alpha$ is defined as $\supp \alpha := \{ i : \alpha_i \neq 0 \}$.

	From \cite{arnold2009geometric}, we have the following two lemmas, indicating that we can obtain a basis if we break the index symmetry.

    \begin{lemma}
    \label{lem:basis-Pr-}
    The  geometric basis functions of $\mathcal P_r^-\alt^k$ associated with $\sigma$ can be given as \begin{equation}\label{eq:basis-pr-} \lambda^{\alpha} \phi_I : |\alpha| = r-1,  \supp \alpha \cup I = [\![\sigma]\!], \alpha_i = 0 \text{ if } i < \min I,
    \end{equation}
    for all $I$ satisfying $|I|=k+1$.
    \end{lemma}
    \begin{lemma}
     \label{lem:basis-Pr}
	The  geometric basis functions of $\mathcal P_r\alt^k$ associated with $\sigma$ can be given as \begin{equation}\label{eq:basis-pr} \lambda^{\alpha} d\lambda_I : |\alpha| = r, \supp \alpha \cup I = [\![\sigma]\!], \alpha_i = 0 \text{ if } i < \min ~[\![\sigma]\!] \setminus I,
	\end{equation}
    for all $I$ satisfying $|I|=k$.
    \end{lemma}

\subsection{The ${\mathcal P_r^-} \alt^{k,\ell}$ family}

For $r \ge 1$, we define
\begin{equation}
\begin{split}
    \mathcal P_r^{-} \Alt^{k,\ell} := &  \mathcal P_r^- \alt^{k} \otimes \alt^{\ell} \\ = & \mathcal P_{r-1}\Alt^{k,\ell} \oplus \kappa\mathcal H_{r-1} \Alt^{k+1,\ell} \\  = & (\mathcal P_{r-1}\Alt^k \oplus \kappa \mathcal H_{r-1 }\Alt^{k+1}) \otimes \Alt^\ell.
    \end{split}
\end{equation}

It follows that 
\[ \dim \mathcal P_r^-\alt^{k,\ell} = \binom{n+r}{k+r} \binom{r+k-1}{k} \binom{n}{\ell}.\]

We introduce the bubble space of $\mathcal P_r^- \alt^{k,\ell}(K)$ with respect to the double trace $\iota^*\iota^*$:
\begin{equation}
\label{eq:Br-}
\mathcal B_r^-\alt^{k,\ell}(K) := \{ \omega \in \mathcal P_r^- \alt^{k,\ell}(K) : \iota_F^{*} \iota_F^* \omega = 0, \,\, \forall F \trianglelefteq_1 K, F \neq K \}.
\end{equation}

Thanks to the explicit form of the basis in \Cref{lem:basis-Pr-}, we can characterize the high-order bubble spaces in a similar way as \Cref{lem:bubble-decomposition}.
\begin{lemma}
\label{lem:bubble-decomposition-Pr-}
Let $\psi_{\sigma,i}$ be  the geometric basis functions associated with $\sigma$ in \Cref{lem:basis-Pr-}. Then,
\[\mathcal B_r^- \alt^{k,\ell}(K) = \sum_{m = k}^n \sum_{\sigma \in \mathcal T_{m}(K)} \operatorname{span} \{ \psi_{\sigma, i}, i = 1,2,\cdots \} \otimes N^{\ell}(\sigma, K).\]
\end{lemma}

Therefore, the dimension of the bubbles in $n$ dimensions is 
\begin{equation}
\begin{split} 
\dim \mathcal B_{r}^- \alt^{k,\ell}(K) = \sum_{m =k}^n \binom{n+1}{m+1} \Big[ \binom{r+k-1}{r-1} \binom{r-1}{m-k} \Big] \cdot \binom{m}{\ell+m-n},
\end{split}
\end{equation}
since \[
\dim (\mathcal P_{r+k-m-1} \Alt^{m-k})(\mathbb R^m) = \binom{r+k-1}{r-1} \binom{r-1}{m-k},\]
and 
\[ \dim N^{\ell}(\sigma, K) = \binom{m}{\ell + m - n}.\]
For simplicity of notation, we can assume that the summation includes all integers $m$, as the terms with $m < k$ or $m > n$ contribute zero.
\begin{corollary}
$\dim \mathcal B_r^- \alt^{k,\ell}(K) = 0$ if  $\dim K \ge \ell + k + r$.
\end{corollary}

\begin{corollary}
\label{cor:spanning-Pr-}
The set
\[
\lambda^{\alpha} \phi_I \otimes d\lambda_J : |\alpha| = r-1,\supp \alpha \cup I \cup J = [n+1], \quad |I|=k+1, |J|=\ell
\]
is a spanning set of $\mathcal B_r^{-}\alt^{k,\ell}(K).$ 
\end{corollary}

The following theorem focuses on the construction of high-order elements, including $\iota^*\iota^*$-conforming, $\iota^*\jmath^*$-conforming, and $\iota^* \jmath_{[p]}^*$-conforming finite elements. The local shape function space  is $\mathcal P_r^- \alt^{k,\ell}(K)$. The main difference between the high-order construction and the lowest-order case (i.e., $\mathcal P^-\alt^{k,\ell}$) is that the degrees of freedom of $C_{\iota^*}\mathcal P^-\alt^{k}$ are only located on $k$-simplices, while $\mathcal P_r^-\alt^k(K)$ will have degrees of freedom on $k$-, $(k + 1)$-, ...,  $(k + r-1)$-simplices.

\begin{theorem}
\label{thm:prminus}
 Each of the following groups of degrees of freedom is unisolvent with respect to the shape function space $\mathcal P^-_r \alt^{k,\ell}(K)$:
\begin{enumerate}
\item 
\begin{equation} \langle \iota^{*}_{\sigma} \iota^*_{\sigma} \omega, b \rangle_{\sigma}, \qquad \forall~~ b \in \mathcal B_r^-\alt^{k,\ell}(\sigma), \quad \forall \sigma \in \mathcal T(K).
\end{equation}	
The resulting space is $\iota^{*}\iota^*$-conforming, denoted as $C_{\iota^*\iota^*} \mathcal P_r^- \alt^{k,\ell}$.

\item For $k \le \ell$, 
\[
\begin{cases} \langle \iota^{*}_{\sigma} \jmath^{*}_{\sigma} \omega, b \rangle_{\sigma}, & \forall~~ b \in \mathcal B_r^-\alt^k(\sigma) \otimes \alt^{\ell - m}(\sigma^{\perp}), \quad \dim \sigma = m \in [k,\ell],  \\ 
\langle \iota^{*}_{\sigma} \iota^*_{\sigma} \omega, b \rangle_{\sigma}, & \forall~~ b \in \mathcal B_r^-\alt^{k,\ell}(\sigma), \quad \dim \sigma > \ell.
\end{cases}	
\]
 The resulting space is $\iota^{*}\jmath^*$-conforming, denoted as $C_{\iota^*\jmath^*} \mathcal P_r^- \alt^{k,\ell}$.
\item
If $ k \le \ell + p - 1$, then
we have the degrees of freedom
 {
\begin{equation}
\begin{cases} \langle \iota^{*}_{\sigma}\vartheta^{*}_{\sigma,m} \omega, b \rangle_{\sigma}, & \forall~~ b \in \mathcal B_r^-\alt^k(\sigma) \otimes (\alt^{m}(\sigma) \otimes \alt^{\ell - m }(\sigma^{\perp})), \quad \dim \sigma = m \in [k,\ell] ,\\ 
\langle \iota^{*}_{\sigma}\vartheta^{*}_{\sigma,m - 1} \omega, b \rangle_{\sigma}, & \forall~~ b \in \mathcal B_r^-\alt^k(\sigma) \otimes (\alt^{m-1}(\sigma) \otimes \alt^{\ell - m+1}(\sigma^{\perp})), \quad \dim \sigma = m \in  [k,\ell+1] ,\\
\langle \iota^{*}_{\sigma}\vartheta^{*}_{\sigma,m - 2} \omega, b \rangle_{\sigma}, & \forall~~ b \in \mathcal B_r^-\alt^k(\sigma) \otimes (\alt^{m-2}(\sigma) \otimes \alt^{\ell - m + 2}(\sigma^{\perp})), 
\quad \dim \sigma = m \in [k,\ell + 2], \\
\cdots \\
\langle \iota^{*}_{\sigma}\vartheta^{*}_{\sigma,m - p + 1 } \omega, b \rangle_{\sigma}, & \forall~~ b \in \mathcal B_r^-\alt^k(\sigma) \otimes (\alt^{m-p + 1 }(\sigma) \otimes \alt^{\ell - m + p - 1}(\sigma^{\perp})), 
\\ & ~~~~~~~~~~~~~~~~~ ~~~~~~ \quad \dim \sigma = m\in [k,\ell + p - 1], \\
 \langle \iota^{*}_{\sigma} \iota^*_{\sigma} \omega, b \rangle_{\sigma}, & \forall~~ b \in \mathcal B_r^-\alt^{k,\ell}(\sigma), \quad \dim \sigma \ge \ell + p,
\end{cases}	
\label{eq:dof-ijp-alt-kl-r}
\end{equation}}
or written compactly,
\[
\begin{cases} 
\langle \iota^{*}_{\sigma}
\jmath^{*}_{\sigma,[p]} \omega, b \rangle_{\sigma}, & \forall~~ b \in \bigoplus_{s = 0}^{p-1} \mathcal B_r^-\alt^{k}(\sigma) \otimes ( \alt^{m - s}(\sigma) \otimes \alt^{\ell - m + s}(\sigma^{\perp})), \\ &\quad \dim \sigma = m \in [k,\ell + p -1],\\ 
 \langle  \iota^{*}_{\sigma}\iota^{*}_{\sigma} \omega, b \rangle_{\sigma}, & \forall~~ b \in \mathcal B^-_r\alt^{k,\ell}(\sigma) , \quad \dim \sigma \ge \ell + p.
\end{cases}	
\]
 The resulting finite element space is $\iota^{*}\jmath^*_{[p]}$-conforming,  denoted as $C_{\iota^*\jmath_{[p]}^*} \mathcal P_r^- \alt^{k,\ell}$.
  \end{enumerate}
\end{theorem}
Note that the first several items in \eqref{eq:dof-ijp-alt-kl-r} disappear if $k>\ell$. For example, the first item 
\[
 \langle \iota^{*}_{\sigma}\vartheta^{*}_{\sigma,m} \omega, b \rangle_{\sigma},  \forall~~ b \in \mathcal B_r^-\alt^k(\sigma) \otimes (\alt^{m}(\sigma) \otimes \alt^{\ell - m }(\sigma^{\perp}))
\]
does not occur when $k>\ell$, because the index range $m \in [k,\ell]$ is empty. Similar remarks hold for the other items. 

\begin{remark}
    The case  $k > \ell$ is not the only case that will lead to a degenerate degree of freedom. Other cases that lead to degeneration include:
    \begin{enumerate}
        \item If  $m < k$ or $m \ge k + r$, then $\mathcal B_r^-\alt^k(\sigma)$ is zero \eqref{eq:dim-Br-}. 
        \item If $k < p - 1$, then under the requirements $m \ge k$ (such that the first term $\mathcal B_r^-\alt^k(\sigma)$ does not degenerate) and $s \le p-1$,  we can still find $m$ and $s$ such that $m-s<0$, i.e., $\alt^{m-s}(\sigma)$ degenerates.  
           \item Under the requirement $s \leq p-1$, there still exists $s$ satisfying $\ell +s > n$ for large $p$. Then the term $\alt^{\ell - m + s}(\sigma^{\perp})$ degenerates.     
     \end{enumerate}
\end{remark}

\begin{proof}[Proof of \Cref{thm:prminus}, (1)]

 The trace to each face is in the corresponding local space. If all moments vanish, induction on face dimension first gives zero traces on all proper faces; the remaining form lies in $\mathcal B_r^-\alt^{k,\ell}(K)$ and its moment against itself is zero. Thus the moments separate the shape functions. It remains to show the dimension count:
\begin{equation} 
\begin{split}	\sum_{\sigma \in \mathcal T(K)} \dim \mathcal B_r^-\alt^{k,\ell}(\sigma) & = \sum_{n'} \underbrace{\binom{n+1}{n'+1}}_{\text{number of $n'$ simplices}} \sum_{m} \underbrace{ \binom{n'+1}{m+1} \Big[ \binom{r+k-1}{r-1} \binom{r-1}{m-k}\Big] \cdot \binom{m}{\ell+m-n'}}_{\text{dimension of } \mathcal B_r^-\alt^{k,\ell}(\sigma)}\\ \text{\footnotesize (switch $m$ and $n'$)~~} 
& = \sum_{m} \binom{r+k-1}{r-1} \binom{r-1}{m-k}    \sum_{n'}  \binom{n+1}{n'+1}\binom{n'+1}{m+1} \binom{m}{\ell+m-n'}.
\end{split}	
\end{equation}
 We rewrite the inner summand as 
\begin{equation}
\begin{split}
   \sum_{n'}  \binom{n+1}{n'+1}\binom{n'+1}{m+1} \binom{m}{\ell+m-n'} = & \sum_{n'} \binom{n+1}{m+1} \binom{n-m}{n-n'} \binom{m}{n' - \ell} \\
   = & \binom{n+1}{m+1} \binom{n}{n-\ell}.
\end{split}
    \end{equation}
 Here, we used $\binom{a}{b}\binom{b}{c} = \binom{a}{c} \binom{a - c}{b - c}$ twice. 

Therefore,  {
\[\sum_{\sigma \in \mathcal T(K)} \dim \mathcal B_r^-\alt^{k,\ell}(\sigma)  = \sum_m \binom{r+k-1}{r-1} \binom{r-1}{m-k} \binom{n+1}{m+1} \binom{n}{n-\ell}.\]}

It then suffices to show 

\[ \sum_m \binom{r+k-1}{r-1} \binom{r-1}{m-k} \binom{n+1}{m+1}  = \binom{n+r}{k+r} \binom{r+k-1}{k}.\]

Note that this identity is exactly the dimension count of the $\mathcal P_r^-\alt^k$ family, and can be proven by the  Vandermonde identity. Therefore, we show the unisolvency, which also implies the conformity.
\end{proof}

\begin{proof}[Proof of \Cref{thm:prminus}, (2)]
The construction relies on redistributing the degrees of freedom, as illustrated in \Cref{fig:moving-high}. Building on the proof of part (1), we start with an \(\iota^*\iota^*\)-conforming finite element space. The goal of this construction is to reassign the degrees of freedom from \(\ell\)-dimensional simplices to lower-dimensional simplices in \(\mathcal{T}_{[k,\ell]}\).

We now detail the redistribution process. For each \(\ell\)-face \(F\), the degrees of freedom are defined as inner products with elements of \(\mathcal{P}_r^- \alt^k(F)\). Unlike the lowest-order case discussed in \Cref{sec:whitney}, the degrees of freedom for \(\mathcal{P}_r^- \alt^k(F)\) are not confined to \(k\)-simplices, resulting in a more intricate scenario. However, they share a similar idea. 

Consider a fixed \(\ell\)-face \(F\). We have \(\mathcal{B}_r^- \alt^{k,\ell}(F) \cong \mathcal{P}_r^- \alt^k(F)\). Recalling the basis from \Cref{lem:basis-Pr-}, we note that for each \(\sigma \in \mathcal{T}_m\) with \(m \in [k,\ell]\), there are \(\dim \mathcal{B}_r^- \alt^k(\sigma)\) degrees of freedom associated with \(\sigma\). We express this informally as:
\begin{equation}\label{eq:informal-dof}
\#\text{DoFs}(\sigma, F) = \dim \mathcal{B}_r^- \alt^k(\sigma).
\end{equation}

Next, we redistribute these degrees of freedom to the simplex \(\sigma\). Algebraically, the degrees of freedom assigned to \(\sigma\) are the sum of contributions from all \(\ell\)-faces \(F\) containing \(\sigma\), as given by \eqref{eq:informal-dof}. The number of such \(\ell\)-faces is \(\binom{n - m}{\ell - m}\), which corresponds exactly to the dimension of \(\alt^{\ell - m}(\sigma^{\perp})\). This facilitates the dimension count.

Unisolvency and conformity follow from the above reasoning and \Cref{lem:j-circ-i}. Specifically, conformity is verified by checking that \(\iota_F^* \iota_F^* \omega = 0\) on each \(\ell\)-face \(F\) whenever \(\omega\) vanishes at all degrees of freedom.
\end{proof}

\begin{proof}[Proof of \Cref{thm:prminus}, (3)]	
As in the previous proof, we achieve the desired construction by redistributing the degrees of freedom, as illustrated in \Cref{fig:moving-high-p}. We emphasize two key aspects of the high-order \(\iota^*\jmath_{[p]}^*\)-conforming finite elements. First, since we are dealing with the high-order case, degrees of freedom are not solely associated with \(k\)-faces. Second, we may deal with $p>1$, meaning that not only \(\ell\)-faces contribute their degrees of freedom. In essence, the redistribution involves moving degrees of freedom from higher-dimensional simplices to lower-dimensional simplices, both of which are multidimensional.

We begin by assuming \(k \leq \ell\). In this case, the degrees of freedom for the finite element space \(C_{\iota^*\iota^*} \mathcal{P}_r^- \alt^{k,\ell}\) are associated with simplices of dimensions \(\ell\) and higher, specifically those in  \(\mathcal{T}_{\ell}, \mathcal{T}_{\ell +1}, \ldots, \mathcal{T}_{\min(n,\ell+k+r-1)}\). To achieve \(\iota^*\jmath_{[p]}^*\)-conformity, we redistribute the degrees of freedom from simplices of dimensions \(\ell\) through \(\ell + p - 1\) to \(m\)-simplices \(\sigma\) where \(m\) ranges from \(k\) to \(\ell + p - 1\). For the case \(k > \ell\), the degrees of freedom start from \(\mathcal{T}_{k}\), but the redistribution process remains analogous.

\begin{example}[High-order Hu--Lin--Zhang element]
Before presenting the general construction, we show an example of the high-order \((1, 2)\)-forms in three dimensions to show how to achieve $\iota^*\jmath^*$-conformity. The \(\iota^{\ast}\iota^{\ast}\)-conforming finite element has degrees of freedom on faces defined as moments against high-order N\'ed\'elec shape functions. To obtain \(\iota^{\ast}\jmath^{\ast}_{[p]}\)-conforming spaces, we redistribute some of these face degrees of freedom to edges. Specifically:
\begin{itemize}
    \item The degrees of freedom on faces are moments against the planar N\'ed\'elec shape functions, which can be decomposed into components associated with the edges of the face and bubble functions interior to the face.
    \item The edge-associated components are reassigned to the corresponding edges, while the bubble functions remain on the face.
    \item Each face sends some of its degrees of freedom to the edges it contains.
    \item Each edge receives degrees of freedom from all the faces that contain it.
\end{itemize}

More illustrations of the degrees of freedom can be found in \Cref{sec:Prminus-12}. 

\end{example}

This example illustrates the general pattern we will follow below. In the general case, for each \(s \in [0, p-1]\), degrees of freedom are sent from \((\ell + s)\)-simplices and received by \(m\)-simplices \(\sigma\) with \(m \in [k, \ell + p -1]\), specifically those \(\sigma\) contained within the \((\ell + s)\)-simplices. To determine which degrees of freedom are sent from each \((\ell + s)\)-simplex \(F\), we use the geometric decomposition of the bubble space on \(F\).

Fix an \(m\)-simplex \(\sigma\) with \(m \in [k, \ell + p -1]\). For each \((\ell + s)\)-simplex \(F \in \mathcal{T}_{\ell + s}(K)\) that contains \(\sigma\), the original degrees of freedom for \(C_{\iota^*\iota^*} \mathcal{P}_r^- \alt^{k,\ell}\) on \(F\) are inner products with respect to the bubble space \(\mathcal{B}_r^- \alt^{k, \ell}(F)\). This bubble space admits the decomposition (\Cref{lem:bubble-decomposition-Pr-}):
 {
\begin{equation}\label{decomp:Br}
   \mathcal{B}_r^- \alt^{k,\ell}(F) = \sum_{j = k}^{\ell + s} \sum_{\tau \in \mathcal{T}_{j}(F)} [\operatorname{span}_i \psi_{\tau, i}] \otimes N^{\ell}(\tau, F).
\end{equation}}
We identify the degrees of freedom on \(F\) associated with \(\sigma\), specifically those corresponding to the terms in the decomposition where \(\tau = \sigma\), and reassign them to \(\sigma\). Concurrently, we introduce new degrees of freedom on \(\sigma\) corresponding to the generalized trace \(\iota_{\sigma}^{\ast}\vartheta_{\sigma, m - s}^*\), as specified in \eqref{eq:dof-ijp-alt-kl-r}.

We first verify that this redistribution preserves the total number of degrees of freedom. Specifically, for each \(s \in [0, p-1]\), we show that the following two quantities are equal:
\begin{enumerate}
    \item The number of degrees of freedom in \eqref{eq:dof-ijp-alt-kl-r} involving the generalized trace \(\iota_{\sigma}^{\ast}\vartheta_{\sigma, m - s}^*\).
    \item The number of degrees of freedom on \(\sigma\) gained from those on all \((\ell + s)\)-faces \(F\) such that \(\sigma \trianglelefteq F \trianglelefteq K\).
\end{enumerate}
For (1), this number is equal to the dimension of \(\mathcal{B}_r^- \alt^k(\sigma) \otimes (\alt^{m - s}(\sigma) \otimes \alt^{\ell - m + s}(\sigma^{\perp}))\). We now show that (2) matches this dimension.

Consider a fixed \((\ell + s)\)-face \(F\) containing \(\sigma\). The number of degrees of freedom on \(F\) associated with \(\sigma\) is:
\[
\# \text{DoFs}(\sigma, F) = \dim \left( \underbrace{\text{span of bases of } \mathcal{P}_r^- \alt^k(F) \text{ on } \sigma}_{\dim \mathcal{B}_r^- \alt^k(\sigma)} \otimes N^{\ell}(\sigma, F) \right).
\]
By \Cref{lem:N}, for a fixed \(m\)-face \(\sigma\) and \((\ell + s)\)-face \(F\), the dimension of \(N^{\ell}(\sigma, F)\) is:
\[
\dim N^{\ell}(\sigma, F) = \binom{m}{\ell + m - (\ell + s)} = \binom{m}{m - s} = \dim \alt^{m - s}(\sigma).
\]
The number of such \((\ell + s)\)-faces \(F\) containing \(\sigma\) is:
\[
\binom{n - m}{(\ell + s) - m} = \binom{n - m}{n - \ell - s} = \dim \alt^{\ell - m + s}(\sigma^{\perp}),
\]
where the equality \(\binom{n - m}{(\ell + s) - m} = \binom{n - m}{n - \ell - s}\) holds since \(\binom{a}{b} = \binom{a}{a - b}\). Therefore, the total number of degrees of freedom that \(\sigma\) receives from all relevant \(F\) is:
\[
\dim \mathcal{B}_r^- \alt^k(\sigma) \times \dim \alt^{m - s}(\sigma) \times \dim \alt^{\ell - m + s}(\sigma^{\perp}),
\]
which matches the dimension in (1). This confirms that the redistribution preserves the total number of degrees of freedom.

Thus, by redistributing the degrees of freedom as described, we obtain the desired finite element space. The proofs of unisolvency and conformity follow arguments similar to those in \Cref{prop:Wp}.
\end{proof}

\begin{figure}[htbp]
\FIG{
\includegraphics[scale=0.1]{\RevisionAssets/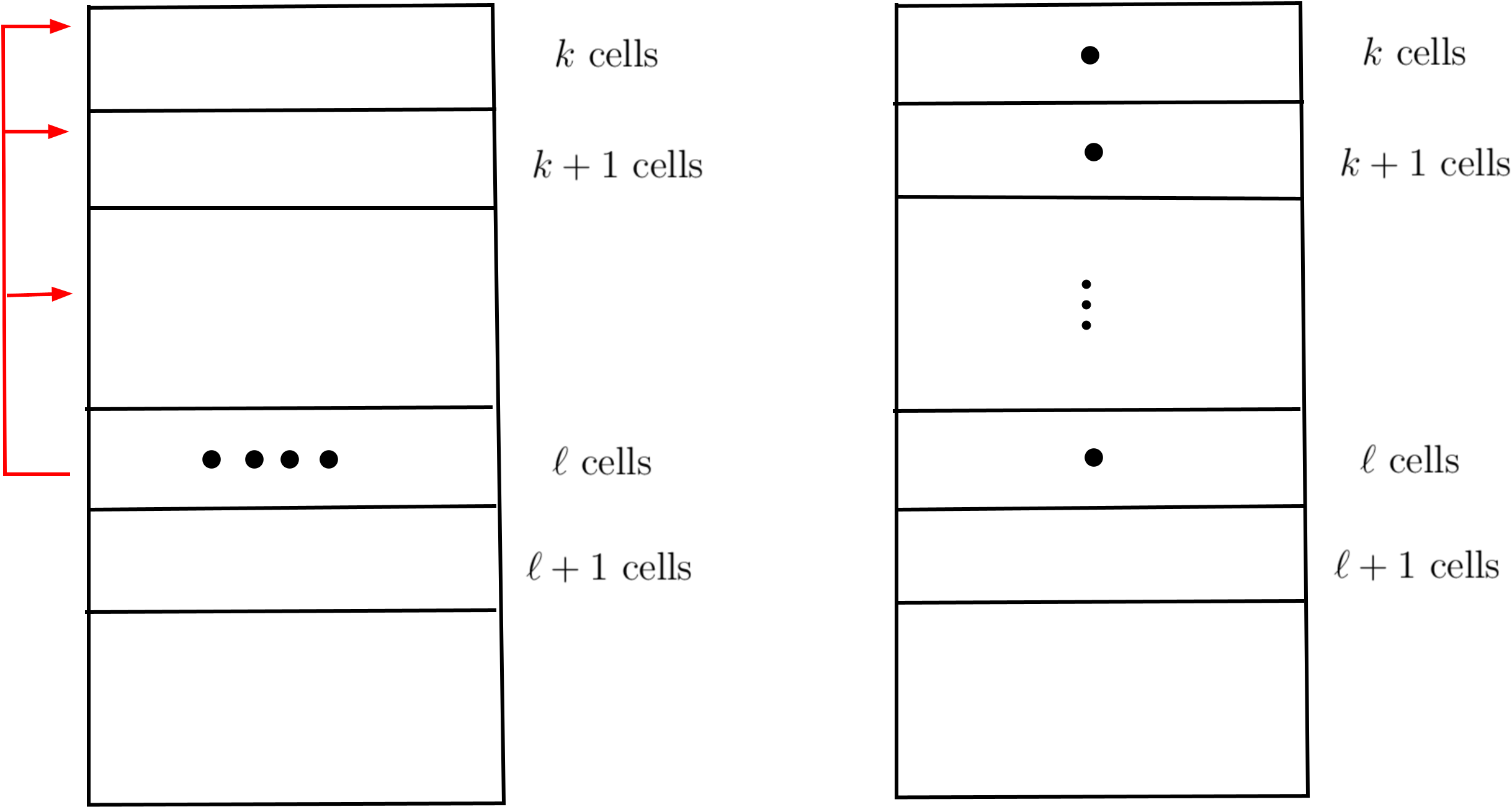}	}
{\caption{An illustration for moving the degrees of freedom in the high-order case for $\iota^*\jmath^*$-conformity. Specifically, for each $\ell$-face $F$, we redistribute its degrees of freedom among $F$ and its subsimplices. In this case, $m$-faces with $m$ ranging from $k$ to $\ell$ can receive the degrees of freedom. If we take a closer look at the basis function of $\mathcal P_r^- \alt^k(F)$ (where $\mathcal P_r^- \alt^k(F)$ represents a certain function space related to our finite element construction), we can find that the possible dimensions $m$ of the faces that receive the degrees of freedom are in the range  $m \in [k,\min(k + r-1,\ell)]$. Note that when  $k+r-1\geq\ell$, some degrees of freedom will in fact stay in the face $F$.}
\label{fig:moving-high}}
\end{figure}

\begin{figure}[htbp]
\FIG{\includegraphics[scale=0.1]{\RevisionAssets/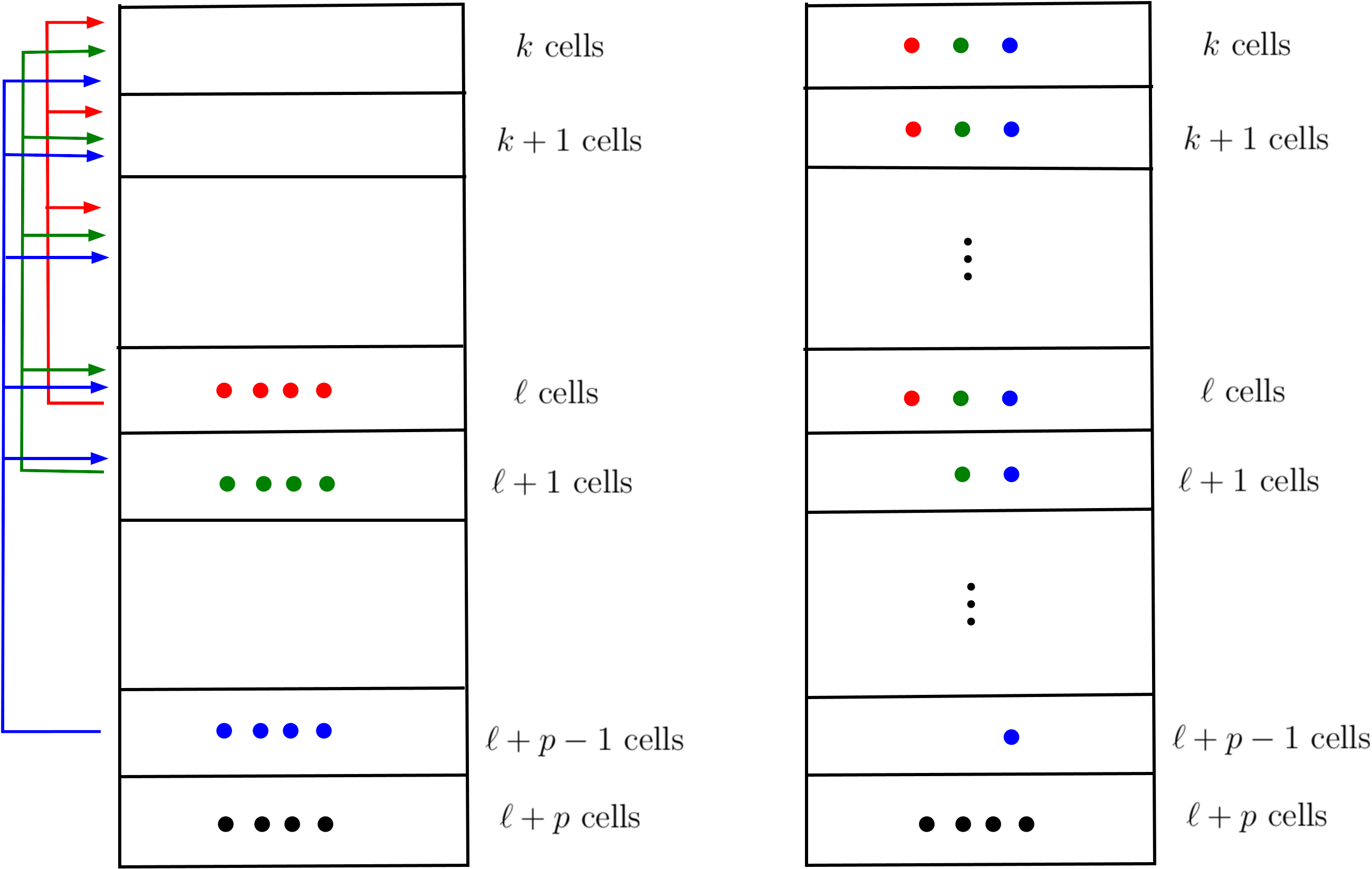}}
{\caption{An illustration for moving the degrees of freedom in the high-order case for $\iota^*\jmath_{[p]}^*$-conformity. For each $\ell + s$ with $s = 0,1,\cdots, p - 1$, we transfer the degrees of freedom to lower-dimensional faces. Note that it is possible that $k > \ell$. When this situation occurs, the $\iota^*\iota^*$-conforming construction will indicate that there are no degrees of freedom for $\ell, \ell + 1,\cdots, k - 1$ faces. Thus, the procedure of moving the degrees of freedom will start from $k$ faces.}
\label{fig:moving-high-p}}
\end{figure}

The previous theorem states all the possible cases for $\alt^{k,\ell}$-valued finite elements. Now, we focus on the BGG symmetry reduction. First, we need to show that the $\mathcal S_{\dagger}$ operator and the iterated version $\mathcal S_{\dagger, [p]}$ are onto. 
\begin{lemma}
\label{lem:S-bubble-Pr-}
If $k \le \ell + p - 1$, then $\mathcal S_{\dagger,[p]} : \mathcal B_r^- \alt^{k,\ell}(\sigma) \to \mathcal B_r^- \alt^{k-p,\ell+p}(\sigma)$ is onto.
\end{lemma}
\begin{proof}
The proof is based on \Cref{cor:spanning-Pr-}, and will be shown in \Cref{subsec:inj/sur}.
\end{proof}

Using this lemma, we can repeat the symmetric reduction procedure in \Cref{sec:whitney}. Let \[\mathcal P_r^- \W_{[p]}^{k,\ell} := \ker(\mathcal S_{\dagger,[p]}^{k,\ell} : \mathcal P_r^- \alt^{k,\ell} \to \mathcal P_r^- \alt^{k-p,\ell + p}).\]
 	Define 
 {
\[\mathcal B_r^- \mathbb W^{k,\ell}_{[p]}(\sigma) := \ker(\mathcal S^{k,\ell}_{\dagger,[p]} : \mathcal B_r^- \alt^{k,\ell}(\sigma) \to \mathcal B_r^- \alt^{k-p,\ell+p}(\sigma)),\]}
whose dimension is 
 {
\[\dim \mathcal B_r^- \mathbb W^{k,\ell}_{[p]}(\sigma) = \dim \mathcal B_r^- \alt^{k,\ell}(\sigma) - \dim \mathcal B_r^- \alt^{k-p,\ell+p}(\sigma).\]}
 The double-trace version $C_{\iota^*\iota^*}\mathcal P_r^-\W_{[p]}^{k,\ell}$, for $k\le\ell+p-1$, has the moments
{
\[
\langle\iota_\sigma^*\iota_\sigma^*\omega,b\rangle_\sigma,
\qquad b\in\mathcal B_r^-\W_{[p]}^{k,\ell}(\sigma),\quad\sigma\in\mathcal T(K).
\]}
These are unisolvent: double pullback commutes with $\mathcal S_{\dagger,[p]}$, so face induction reduces zero moments to a zero bubble norm. The dimension count follows by subtracting the two full-space counts, using \Cref{lem:Pr-Wp,lem:S-bubble-Pr-}.
\begin{theorem}

\begin{enumerate}
\item For $p = 1$, the following degrees of freedom are unisolvent for $\mathcal P^-_r \W^{k,\ell}(K)$:
\[
\begin{cases} \langle \iota^{*}_{\sigma}\jmath^{*}_{\sigma} \omega, b \rangle_{\sigma}, & \forall~~ b \in \mathcal B_r^-\alt^k(\sigma) \otimes \alt^{\ell - m}(\sigma^{\perp}), \quad \dim \sigma = m \in [k,\ell],  \\ 
\langle \iota^{*}_{\sigma} \iota^*_{\sigma} \omega, b \rangle_{\sigma}, & \forall~~ b \in \mathcal B_r^-\W^{k,\ell}(\sigma), \quad \dim \sigma > \ell.
\end{cases}	
\]
 The resulting finite element space is $\iota^{*}\jmath^*$-conforming, and can be denoted as $C_{\iota^*\jmath^*} \mathcal P_r^-\W^{k,\ell}$. 
\item The following degrees of freedom are unisolvent for $\mathcal P_r^- \W_{[p]}^{k,\ell}(K)$: 
\[
\begin{cases} \langle \iota^{*}_{\sigma}\jmath^{*}_{\sigma,[p]} \omega, b \rangle_{\sigma}, & \forall~~ b \in \bigoplus_{s = 0}^{p-1} \mathcal B_r^-\alt^{k}(\sigma) \otimes ( \alt^{m - s}(\sigma) \otimes \alt^{\ell - m + s}(\sigma^{\perp})), \\ &\quad \dim \sigma = m \in [k,\ell + p -1],\\ 
 \langle  \iota^{*}_{\sigma}\iota^{*}_{\sigma} \omega, b \rangle_{\sigma}, & \forall~~ b \in \mathcal B^-_r\W^{k,\ell}_{[p]}(\sigma) , \quad \dim \sigma \ge \ell + p.
\end{cases}	
\]
 The resulting finite element space is $\iota^{*}\jmath^*_{[p]}$-conforming, and can be denoted as $C_{\iota^*\jmath_{[p]}^*} \mathcal P_r^-\W_{[p]}^{k,\ell}$. 
  \end{enumerate}

\end{theorem}
 \begin{proof} On faces of dimension at most $\ell+p-1$, the intrinsic target coefficient space $\Alt^{\ell+p}$ is zero, so symmetry reduction leaves the double-trace spaces unchanged. Thus their moments can be redistributed exactly as in \Cref{thm:prminus} . On higher-dimensional faces the bubble spaces are replaced by their kernels under $\mathcal S_{\dagger,[p]}$; surjectivity in \Cref{lem:S-bubble-Pr-}  gives the displayed dimension subtraction. Face induction, ending with testing the residual kernel bubble against itself, proves unisolvency. The same facewise determination proves the stated trace conformity.\end{proof}

\subsection{The $\mathcal P_r \alt^{k,\ell}$ family}

We have
\[ \dim \mathcal P_r\alt^{k,\ell} = \binom{n+r}{n} \binom{n}{k} \binom{n}{\ell}.\]

We introduce
\begin{equation}
\label{eq:Br}
\mathcal B_r\alt^{k,\ell}(K) := \{ \omega \in \mathcal P_r\alt^{k,\ell}(K) : \iota_F^{*}\iota_F^{*} \omega = 0, \,\, \forall F \trianglelefteq K, F \neq K \}.
\end{equation}

Similar to \Cref{lem:bubble-decomposition-Pr-} and \Cref{lem:bubble-decomposition}, the following lemma holds.

\begin{lemma}
 Let $\psi_{\sigma,i}$ be  the geometric basis functions associated with $\sigma$ in \Cref{lem:basis-Pr}. Then it holds that,
 {
\[\mathcal B_r \alt^{k,\ell}(K) = \sum_{m = k}^n \sum_{\sigma \in \mathcal T_m(K)} [\operatorname{span}_i \psi_{\sigma, i} ] \otimes N^{\ell}(\sigma,K).\]}
\end{lemma}

Therefore, the dimension of the bubble in $n$ dimension is 
\begin{equation}
\begin{split} 
\dim \mathcal B_{r} \alt^{k,\ell}(K) = \sum_{m =k}^n \binom{n+1}{m+1} \Big[ \binom{r+k}{r} \binom{r - 1}{m - k} \Big] \cdot \binom{m}{\ell+m-n},
\end{split}
\end{equation}
since 
\[\dim \mathcal P^-_{r+k-m} \alt^{m-k} (\mathbb R^m) = \binom{r+k}{r} \binom{r - 1}{m - k}.\]
\begin{corollary}
$\dim \mathcal B_r \alt^{k,\ell}(K) = 0$ if $\dim K \ge \ell+ k +r$.
\end{corollary}

\begin{corollary}
\label{cor:spanning-Pr}
The set
\[
\lambda^{\alpha} d\lambda_I \otimes d\lambda_J: |\alpha| = r,\supp \alpha \cup I \cup J = [n+1]\]
is a spanning set of $\mathcal B_r\alt^{k,\ell}(K).$ 
\end{corollary}

\begin{lemma}
\label{lem:S-bubble-Pr}
 For $r\ge1$, if $k \le \ell + p $, then $\mathcal S_{\dagger,[p]} : \mathcal B_r \alt^{k,\ell}(\sigma) \to \mathcal B_r \alt^{k-p,\ell+p}(\sigma)$ is onto.
\end{lemma}
\begin{proof}
The proof is based on \Cref{cor:spanning-Pr},  shown in \Cref{subsec:inj/sur}.
\end{proof}

\begin{theorem}
\label{thm:pr}
The following degrees of freedom are unisolvent with respect to the shape function space $\mathcal P_{r} \alt^{k,\ell}(K)$:
\begin{enumerate}
\item
\begin{equation} \langle \iota^{*}_{\sigma} \iota^*_{\sigma} \omega, b \rangle_{\sigma}, \qquad \forall~~ b \in \mathcal B_r\alt^{k,\ell}(\sigma), \quad \forall \sigma \in \mathcal T(K).
\end{equation}	
 The resulting finite element space is $\iota^{*}\iota^*$-conforming, denoted as $C_{\iota^*\iota^*} \mathcal P_r \alt^{k,\ell}$. 

\item
\[
\begin{cases} \langle \iota^{*}_{\sigma}\jmath^{*}_{\sigma}\omega, b \rangle_{\sigma}, & \forall~~ b \in \mathcal B_r\alt^k(\sigma) \otimes \alt^{\ell - m}(\sigma^{\perp}), \quad \dim \sigma = m \in [k,\ell],  \\ 
\langle\iota^{*}_{\sigma} \iota^*_{\sigma} \omega, b \rangle_{\sigma}, & \forall~~ b \in \mathcal B_r\alt^{k,\ell}(\sigma), \quad \dim \sigma > \ell.
\end{cases}	
\]
 The resulting finite element space is $\iota^{*}\jmath^*$-conforming, denoted as $C_{\iota^*\jmath^*} \mathcal P_r \alt^{k,\ell}$. 
\item

\[
\begin{cases} \langle \iota^{*}_{\sigma}\jmath^{*}_{\sigma,[p]} \omega, b \rangle_{\sigma}, & \forall~~ b \in \bigoplus_{s = 0}^{p-1} \mathcal B_r\alt^{k}(\sigma) \otimes ( \alt^{m - s}(\sigma) \otimes \alt^{\ell - m + s}(\sigma^{\perp})), \\ &\quad \dim \sigma = m \in [k,\ell + p -1],\\ 
 \langle  \iota^{*}_{\sigma}\iota^{*}_{\sigma} \omega, b \rangle_{\sigma}, & \forall~~ b \in \mathcal B_r\alt^{k,\ell}(\sigma) , \quad \dim \sigma \ge \ell + p.
\end{cases}	
\]
 The resulting finite element space is $\iota^{*}\jmath^*_{[p]}$-conforming,  denoted as $C_{\iota^*\jmath_{[p]}^*} \mathcal P_r \alt^{k,\ell}$.
 \end{enumerate}
\end{theorem}
\begin{proof}
 For (1),  the  face-induction argument in the proof of \Cref{thm:prminus} separates the shape functions, so it remains to count the moments. For (2) and (3), use the same facewise redistribution as in \Cref{thm:prminus}, now with the unisolvent scalar $\mathcal P_r\alt^k$ moments and their geometric decomposition. The trace identities are unchanged.
The dimension count reads  
\begin{equation}
\begin{split} 
\sum_{\sigma \in \mathcal T(K)} \dim \mathcal B_{r} \alt^{k,\ell}(\sigma) = &  \sum_{n'} \binom{n+1}{n'+1} \sum_{m =k}^n \binom{n'+1}{m+1} \Big[ \binom{r+k}{r} \binom{r - 1}{m - k} \Big] \cdot \binom{m}{\ell+m-n'} \\
 = & \sum_{m} \Big[ \binom{r+k}{r} \binom{r - 1}{m - k} \Big]  \sum_{n'} \binom{n+1}{n'+1} \binom{n'+1}{m+1}  \cdot \binom{m}{\ell+m-n'} \\ = & \sum_m \binom{r+k}{r} \binom{r - 1}{m - k} \binom{n+1}{m+1} \binom{n}{n-\ell} \\  = & \binom{r+k}{r} \binom{r + n}{r+k} \binom{n}{\ell}. 
\end{split}
\end{equation}
Here, we use the fact $\displaystyle \binom{r - 1}{m - k} = \binom{r-1}{r-1-m+k}$ and the Vandermonde identity. 
\end{proof}

Thus, we can perform the symmetry reduction similarly to the previous subsection, leading to the theorem below.  Let  {
\[\mathcal P_r \W_{[p]}^{k,\ell} := \ker(\mathcal S_{\dagger,[p]}^{k,\ell} : \mathcal P_r \alt^{k,\ell} \to \mathcal P_r \alt^{k-p,\ell + p}).\]}
 	Define 
 {
\[\mathcal B_r \mathbb W^{k,\ell}_{[p]}(\sigma) = \ker(\mathcal S^{k,\ell}_{\dagger,[p]} : \mathcal B_r \alt^{k,\ell}(\sigma) \to \mathcal B_r \alt^{k-p,\ell+p}(\sigma)),\]}
whose dimension is 
 {
\[\dim \mathcal B_r \mathbb W^{k,\ell}_{[p]}(\sigma) = \dim \mathcal B_r \alt^{k,\ell}(\sigma) - \dim \mathcal B_r \alt^{k-p,\ell+p}(\sigma).\]}
For $k\le\ell+p-1$, the space $C_{\iota^*\iota^*}\mathcal P_r\W_{[p]}^{k,\ell}$ is defined by the shape space $\mathcal P_r\W_{[p]}^{k,\ell}$ and the moments
{
\[
\langle\iota_\sigma^*\iota_\sigma^*\omega,b\rangle_\sigma,
\qquad b\in\mathcal B_r\W_{[p]}^{k,\ell}(\sigma),\quad\sigma\in\mathcal T(K).
\]}
The same face-induction argument proves unisolvency, and \Cref{lem:S-bubble-Pr} gives the dimension count by subtraction.
\begin{theorem}

\begin{enumerate}
\item The following degrees of freedom are unisolvent with respect to the shape function space $\mathcal P_r \W^{k,\ell}(K)$:
\[
\begin{cases} \langle \iota^{*}_{\sigma}\jmath^{*}_{\sigma} \omega, b \rangle_{\sigma}, & \forall~~ b \in \mathcal B_r\alt^k(\sigma) \otimes \alt^{\ell - m}(\sigma^{\perp}), \quad \dim \sigma = m \in [k,\ell],  \\ 
\langle \iota^{*}_{\sigma} \iota^*_{\sigma} \omega, b \rangle_{\sigma}, & \forall~~ b \in \mathcal B_r\W^{k,\ell }(\sigma), \quad \dim \sigma > \ell.
\end{cases}	
\]
 The resulting finite element space is $\iota^{*}\jmath^*$-conforming,  denoted as $C_{\iota^*\jmath^*} \mathcal P_r \W^{k,\ell}$.
\item The following degrees of freedom are unisolvent with respect to the shape function space $\mathcal P_r \W^{k,\ell}_{[p]}(K)$:
\[
\begin{cases} \langle \iota^{*}_{\sigma}\jmath^{*}_{\sigma,[p]} \omega, b \rangle_{\sigma}, & \forall~~ b \in \bigoplus_{s = 0}^{p-1} \mathcal B_r\alt^{k}(\sigma) \otimes ( \alt^{m - s}(\sigma) \otimes \alt^{\ell - m + s}(\sigma^{\perp})), \\ &\quad \dim \sigma = m \in [k,\ell + p -1],\\ 
 \langle  \iota^{*}_{\sigma}\iota^{*}_{\sigma} \omega, b \rangle_{\sigma}, & \forall~~ b \in \mathcal B_r\W^{k,\ell }_{[p]}(\sigma) , \quad \dim \sigma \ge \ell + p.
\end{cases}	
\]
 The resulting finite element space is $\iota^{*}\jmath^*_{[p]}$-conforming, denoted as $C_{\iota^*\jmath_{[p]}^*} \mathcal P_r \W_{[p]}^{k,\ell}$.
  \end{enumerate}

\end{theorem}

\begin{proof}
The preceding symmetry-reduction proof applies with $(\mathcal P_r,\mathcal B_r)$ in place of $(\mathcal P_r^-,\mathcal B_r^-)$. Use \Cref{thm:pr} for redistribution and \Cref{lem:S-bubble-Pr} for surjectivity on the remaining bubbles. The low-dimensional face spaces are unchanged since $\Alt^{\ell+p}=0$ there; face induction proves unisolvency and conformity.
\end{proof}

\subsection{Examples of high-order constructions}
To close this section, we show examples in two and three dimensions.

\subsubsection{The $\mathcal P_r\mathbb W^{1,1}$ element}

For $\mathcal P_r\mathbb W^{1,1}$ family, the shape function is $\mathcal P_r \otimes \mathbb S$.
In two dimensions, the dimension of the local shape function space is $\frac{3}{2}(r+2)(r+1)$. The degrees of freedom are 
\begin{align*}
	\int_{e} \bm \sigma_{tt} \cdot q, \quad \forall q \in \mathcal P_{r}(e), & & \text{count} = (r+1)	
	\\
	\int_{f} \bm \sigma_{tt} \cdot q ,\quad \forall q \in \mathcal B_r \mathbb W^{1,1}(f). & & \text{count} = \frac{3}{2}(r+1)r
\end{align*}

In three dimensions, the dimension of the local shape function space is $(r+3)(r+2)(r+1)$. The degrees of freedom are 
\begin{align*}
	\int_{e} \bm \sigma_{tt} \cdot q, \quad \forall q \in \mathcal P_{r}(e), & & \text{count} = (r+1)	
	\\
	\int_{f} \bm \sigma_{tt} \cdot q ,\quad \forall q \in \mathcal B_r \mathbb W^{1,1}(f) , & & \text{count} = \frac{3}{2}(r+1)r\\
	\int_{K} \bm \sigma  \cdot q, \quad \forall q \in  \mathcal B_r \mathbb W^{1,1}(K). & & \text{count} = (r+1)r(r-1)
\end{align*}

This recovers the high-order Regge element, as discussed in \cite{li2018regge}. In fact, we can check that the general definition of high-order $\mathcal P_r$ Regge element is given by $C_{\iota^*\iota^*}\mathcal P_r \W^{1,1}$. See \Cref{sub:regge} for details.  
\begin{figure}[htbp]
\FIG{\includegraphics[scale = 0.12,trim=0 0 0 250pt, clip]{\RevisionAssets/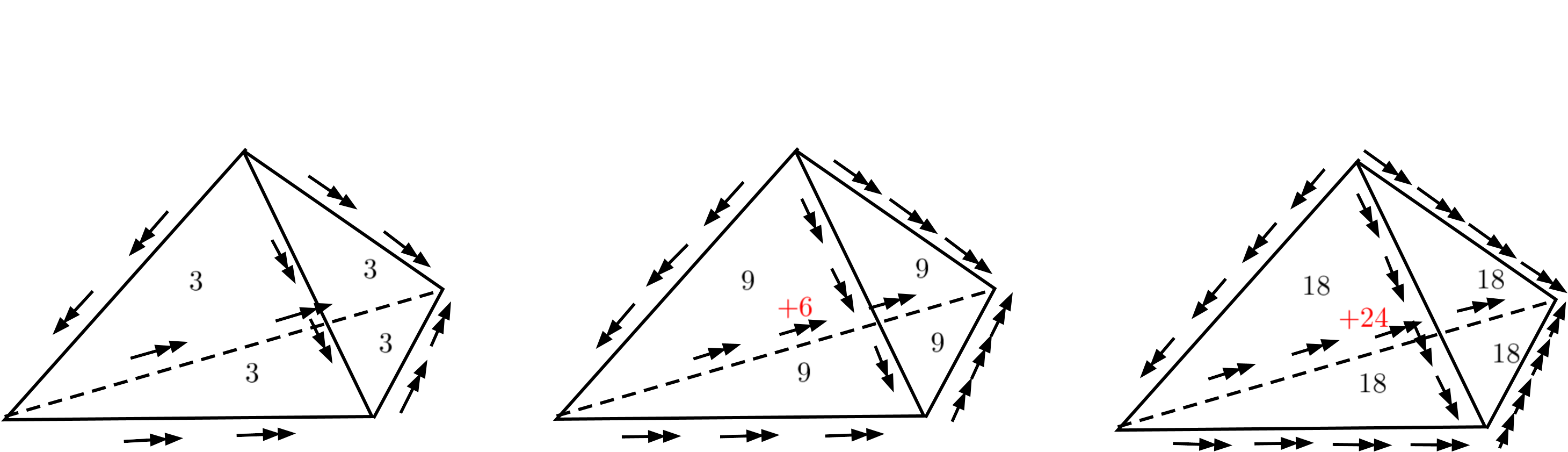}
\vspace{2em}}
{\caption{An illustration for $C_{\iota^*\iota^*}\mathcal P_r \W^{1,1}$ for $r = 1,2,3$. The local shape function is: ($r = 1$) $\mathcal P_1\otimes\mathbb S$, dim = 24; ($r = 2$) $\mathcal P_2 \otimes\mathbb S$, dim = 60; ($r = 3$) $\mathcal P_3 \otimes\mathbb S$, dim = 120. }}
\end{figure}

\subsubsection{The $\mathcal P_r^-\mathbb W^{1,1}$ element} 

In two dimensions, it follows from \Cref{charact-ker-Sdagger} that the shape function space is 
 $(\mathcal  P_{r-1} \otimes \mathbb S^2) \oplus (\bm x^{\perp} \otimes \bm x^{\perp}) \mathcal H_{r-2},$
where $\bm x^{\perp} = [y,-x].$ The dimension of the local shape function space is $\frac{3}{2}(r+1)r + (r-1) = \frac{1}{2}(r+2)(3r-1)$, and 
the degrees of freedom are 
\begin{align*}
	\int_{e} \bm \sigma_{tt} \cdot q ,\quad \forall q \in\mathcal P_{r-1}(e), & & \text{count} = r \\
	\int_{f} \bm \sigma_{tt} \cdot q ,\quad \forall q \in \mathcal B_r^- \mathbb W^{1,1}(f). & & \text{count} = \frac{1}{2}(3 r + 2) (r - 1)
\end{align*}

In three dimensions, the shape function space is $ (\mathcal  P_{r-1} \otimes \mathbb S) \oplus ( \bm x \times (\mathcal H_{r-2} \otimes \mathbb S )\times \bm x) $,
whose dimension is $\frac{1}{2} (r + 2) (r + 3) (2 r - 1).$
The degrees of freedom are 
\begin{align*}
	\int_{e} \bm \sigma_{tt} \cdot q ,\quad \forall q \in\mathcal P_{r-1}(e), & & \text{count} = r \\
	\int_{f} \bm \sigma_{tt} \cdot q ,\quad \forall q \in \mathcal B_r^- \mathbb W^{1,1}(f), & & \text{count} = \frac{1}{2}(3 r + 2) (r - 1) \\
	\int_{K} \bm \sigma  \cdot q, \quad \forall q \in  \mathcal B_r^- \mathbb W^{1,1}(K). & & \text{count} = \frac{1}{2}(r - 1) (2 r^2 - r - 2)
\end{align*}

\begin{figure}[htbp]
\FIG{\includegraphics[scale = 0.12,trim=0 0 0 250pt, clip]{\RevisionAssets/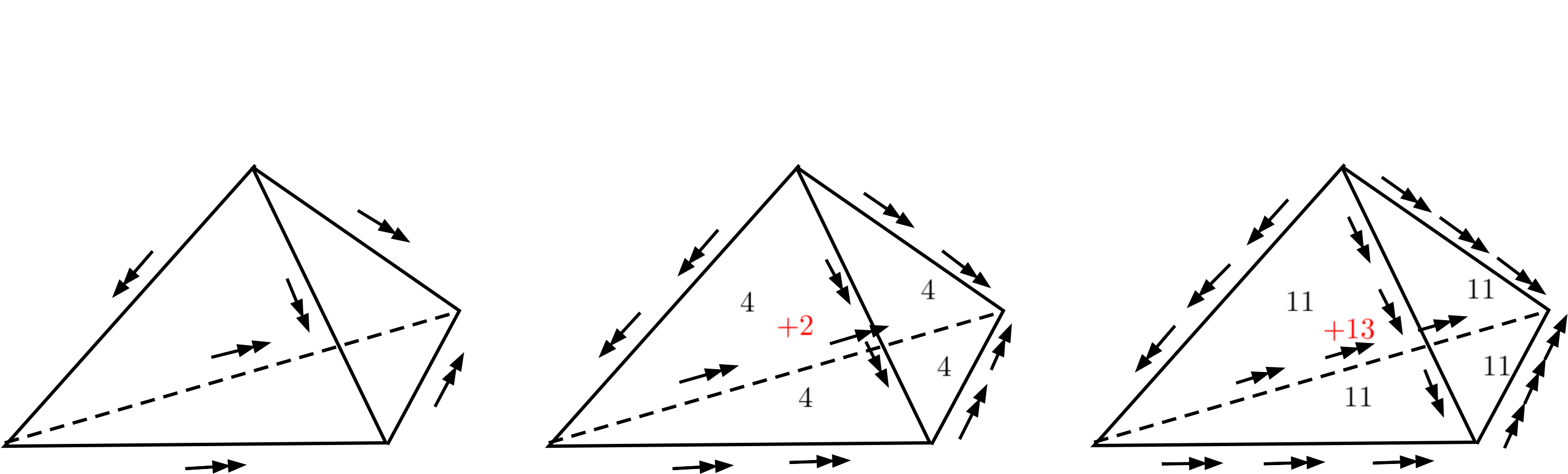}}
{\caption{An illustration for $C_{\iota^*\iota^*}\mathcal P_r^- \W^{1,1}$ for $r = 1,2,3$. The local shape function is: ($r = 1$) $\mathbb S$, dim = 6; ($r = 2$) $\mathcal P_1 \otimes  \mathbb S+ \bm x \times \mathbb S \times \bm x$, dim = 30; ($r = 3$) $\mathcal P_2 \otimes \mathbb S + (\bm x \times (\mathcal H_1 \otimes \mathbb S) \times \bm x)$, dim = 75.}}
\end{figure}

\subsubsection{The $\mathcal P_r\mathbb W^{1,2}$ element} 

For $\mathcal P_r \mathbb W^{1,2}$ family, the shape function space is $\mathcal P_r \otimes \mathbb T$, whose dimension is $\frac{4}{3}(r+3)(r+2)(r+1)$. The degrees of freedom are 
\begin{align*}
	\int_{e} \bm \sigma_{tn} \cdot q,\quad \forall q \in \mathcal P_r(e) \otimes \mathbb R^2, & & \text{count} = 2(r+1) \\ 
	\int_{f} \bm \sigma_{tn} \cdot q ,\quad \forall q \in \mathcal B_r\alt^1(f), & & \text{count} = (r^2-1) \\ 
	\int_{K} \bm \sigma  \cdot q, \quad \forall q \in \mathcal B_r\mathbb W^{1,2}(K). & & \text{count} = \frac{4}{3} (r + 1) (r + 2) r 
\end{align*}

\begin{figure}[htbp]
\FIG{\includegraphics[scale=0.12,trim=0 0 0 250pt, clip]{\RevisionAssets/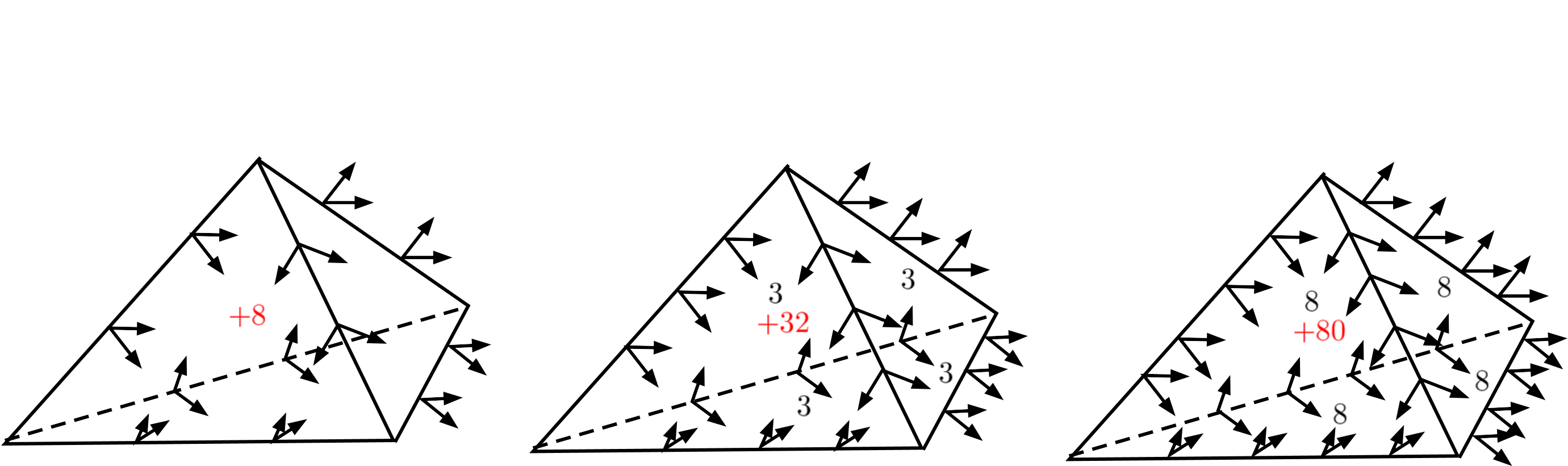}}
{\caption{An illustration for $C_{\iota^*\jmath^*}\mathcal P_r \W^{1,2}$ for $r = 1,2,3$. The local shape function is: ($r = 1$) $\mathcal P_1\otimes \mathbb T$, dim = 32; ($r = 2$) $\mathcal P_2 \otimes \mathbb T$, dim = 80; ($r = 3$) $\mathcal P_3 \otimes \mathbb T$, dim = 160.}}
\end{figure}

\subsubsection{The $\mathcal P_r^-\mathbb W^{1,2}$ element} 
\label{sec:Prminus-12}

For $\mathcal P_r^- \mathbb W^{1,2}$ family, the shape function space is $(\mathcal P_{r-1}  \otimes \mathbb T )\oplus  (\bm x \times (\mathcal H_{r-1} \otimes \mathbb S))$, whose dimension is $\frac{1}{6} (r + 2) (r + 3) (8 r - 1).$ 
The degrees of freedom are 
\begin{align*}
	\int_{e} \bm \sigma_{tn} \cdot q ,\quad \forall q \in \mathcal P_{r-1}(e) \otimes \mathbb R^2, & & \text{count} = 2r \\ 
	\int_{f} \bm \sigma_{tn} \cdot q ,\quad \forall q \in \mathcal B_r^-\alt^1(f), & & \text{count} = r(r-1) \\ 
	\int_{K} \bm \sigma  \cdot q, \quad \forall q \in \mathcal B_r^-\mathbb W^{1,2}(K). & & \text{count} = \frac{1}{6} (r + 2) (8 r^2 - r - 3)
\end{align*}
\begin{figure}[htbp]
\FIG{\includegraphics[scale = 0.12,trim=0 0 0 250pt, clip]{\RevisionAssets/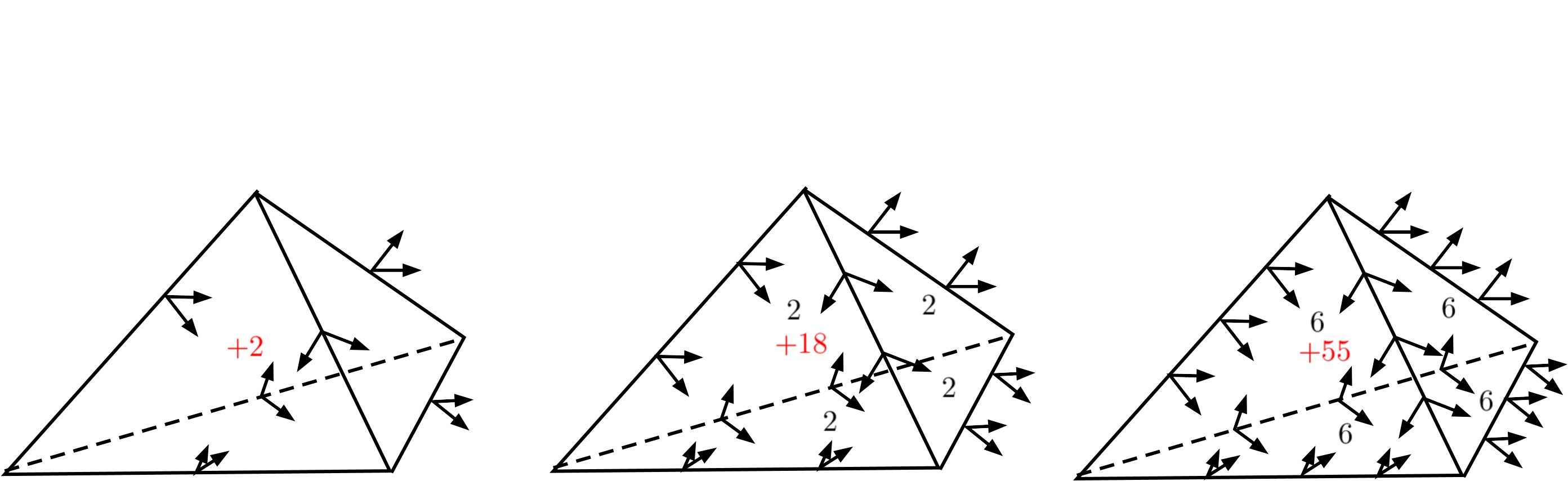}}
{\caption{An illustration for $C_{\iota^*\jmath^*}\mathcal P_r^- \W^{1,2}$ for $r = 1,2,3$. The local shape function is: ($r = 1$) $\mathbb T + \bm x \times \mathbb S$, dim = 14; ($r = 2$) $(\mathcal P_1 \otimes \mathbb T) + (\bm x \times (\mathcal H_1 \otimes \mathbb S))$, dim = 50; ($r = 3$) $(\mathcal P_2 \otimes \mathbb T )+ (\bm x \times (\mathcal H_2 \otimes \mathbb S))$, dim = 115.}}
\end{figure}

\subsubsection{The $\mathcal P_r\mathbb W^{2,2}$ element}  
For $\mathcal P_r \mathbb W^{2,2}$ family, the shape function space is $\mathcal P_r \otimes \mathbb S$, whose dimension is $(r+3)(r+2)(r+1)$. 
The degrees of freedom are
\begin{align*}
	\int_{f} \bm \sigma_{nn} \cdot q ,\quad \forall q \in \mathcal P_r(f), & & \text{count} = \frac{1}{2}(r+2)(r+1) \\ 
	\int_{K} \bm \sigma  \cdot q, \quad \forall q \in \mathcal B_r\mathbb W^{2,2}(K). & & \text{count} = (r + 1)^2 (r + 2) 
\end{align*}

\begin{figure}[htbp]
\FIG{\includegraphics[scale = 0.12,trim=0 0 0 0pt, clip]{\RevisionAssets/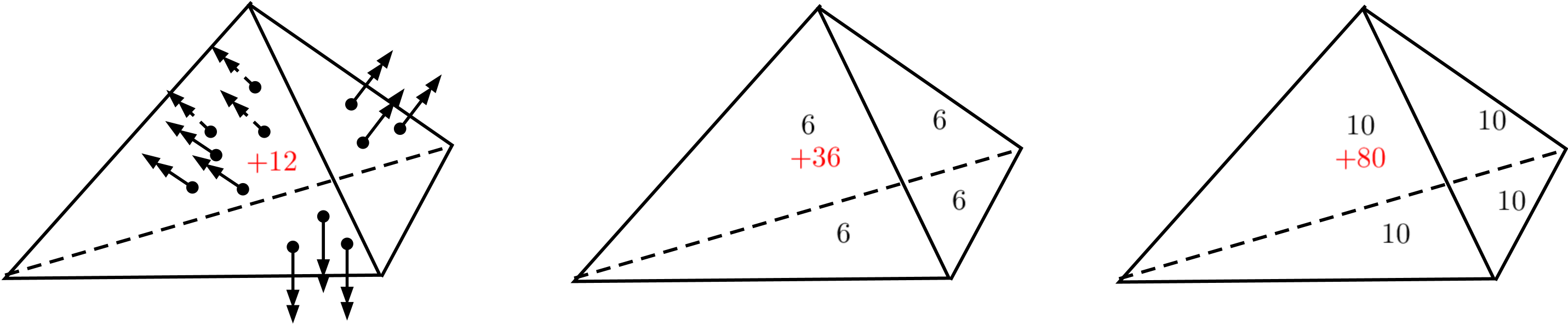}}
{\caption{An illustration for $C_{\iota^*\iota^*}\mathcal P_r\W^{2,2}$ for $r = 1,2,3$. The local shape function is: ($r = 1$) $\mathcal P_1 \otimes \mathbb S$, dim = 24; ($r = 2$) $\mathcal P_2 \otimes \mathbb S$, dim = 60; ($r = 3$) $\mathcal P_3 \otimes \mathbb S$, dim = 120.}}
\end{figure}

\subsubsection{The $\mathcal P_r^-\mathbb W^{2,2}$ element} 

For the $\mathcal P_r^- \mathbb W^{2,2}$ family, the shape function space is $\mathcal P_{r-1}\otimes \mathbb S + \bm x \otimes \bm x\mathcal P_{r-2}$, 
whose dimension is $\frac{1}{2} r (r + 3) (2 r + 1).$ 
The degrees of freedom are 
\begin{align*}
	\int_{f} \bm \sigma_{nn} \cdot q ,\quad \forall q \in \mathcal P_{r-1}(f), & & \text{count} = \frac{1}{2}(r+1)r \\
	\int_{K} \bm \sigma  \cdot q, \quad \forall q \in \mathcal B_r^-\mathbb W^{2,2}(K). & & \text{count} = \frac{1}{2} r (2 r^2 + 3 r - 1)
\end{align*}
\begin{figure}[htbp]
\FIG{\includegraphics[width=345.12bp]{\RevisionAssets/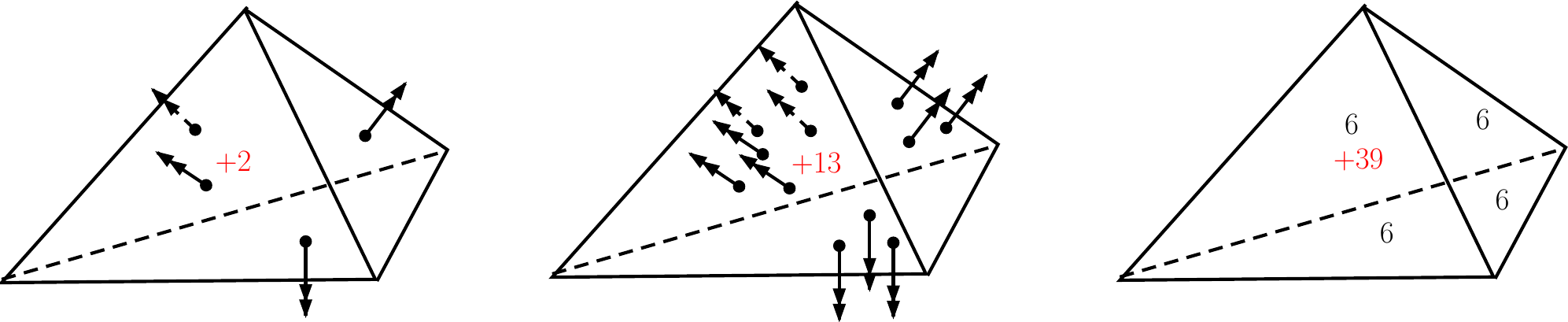}}
{\caption{An illustration for $C_{\iota^*\iota^*}\mathcal P_r^-\W^{2,2}$ for $r = 1,2,3$. The local shape function is: ($r = 1$) $\mathbb S$, dim = 6; ($r = 2$) $\mathcal P_1 \otimes \mathbb S + \bm x \otimes \bm x \mathcal H_0$, dim = 25; ($r = 3$) $\mathcal P_2 \otimes \mathbb S + \bm x \otimes \bm x \mathcal H_1$, dim = 63.}}
\end{figure}

\subsubsection{The $\mathcal P_r\mathbb W_{[2]}^{2,1}$ element}  

For the $\mathcal P_r \mathbb W^{2,1}_{[2]}$ 
family, the shape function space is $\mathcal P_r \otimes \mathbb T $, whose dimension is $\frac{4}{3}(r+3)(r+2)(r+1)$. 
{The} degrees of freedom are 
 {
\begin{align*}
	\int_{f} \bm \sigma_{nt} \cdot q ,\quad \forall q \in \mathcal P_r(f) \otimes \mathbb R^2,  & & \text{count} = (r+2)(r+1) \\
	\int_{K} \bm \sigma  \cdot q, \quad \forall q \in \mathcal B_r\mathbb W_{[2]}^{2,1}(K). & & \text{count} = \frac{4}{3} (r + 1) (r + 2) r
\end{align*}}

\begin{figure}[htbp]
\FIG{\includegraphics[scale = 0.12, trim=0 0 0 0pt, clip]{\RevisionAssets/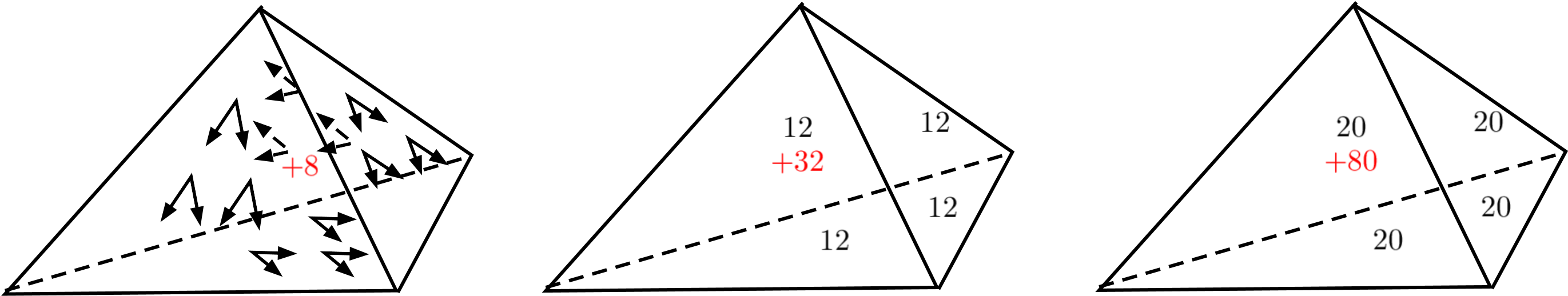}}
{\caption{An illustration for $C_{\iota^*\jmath_{[2]}^*}\mathcal P_r\W^{2,1}_{[2]}$ for $r = 1,2,3$. The local shape function is: ($r = 1$) $\mathcal P_1\otimes\mathbb T$, dim = 32; ($r = 2$) $\mathcal P_2\otimes\mathbb T$, dim = 80; ($r = 3$) $\mathcal P_3\otimes\mathbb T$, dim = 160.}}
\end{figure}

\subsubsection{The $\mathcal P_r^-\mathbb W_{[2]}^{2,1}$ element}  

For  $\mathcal P_r^- \mathbb W_{[2]}^{2,1}$ 
family, the shape function space is  $(\mathcal P_{r-1}\otimes \mathbb T )\oplus (\bm x\otimes (\bm x \times (\mathcal H_{r-2}\otimes \mathbb V))),$ 
whose dimension is $\frac{1}{3} (4 r - 1) (r + 1) (r + 3)$.
The degrees of freedom are 
\begin{align*}
	\int_{f} \bm \sigma_{nt} \cdot q ,\quad \forall q  \in \mathcal P_{r-1} (f) \otimes \mathbb R^2, & & \text{count} = (r+1)r \\ 
	\int_{K} \bm \sigma  \cdot q, \quad \forall q \in \mathcal B_r^{-}\mathbb W_{[2]}^{2,1}(K). & & \text{count} = \frac{1}{3} (r + 1) (4 r + 3) (r - 1)
\end{align*}

\begin{figure}[htbp]
\FIG{\includegraphics[scale = 0.12,trim=0 0 0 0pt, clip]{\RevisionAssets/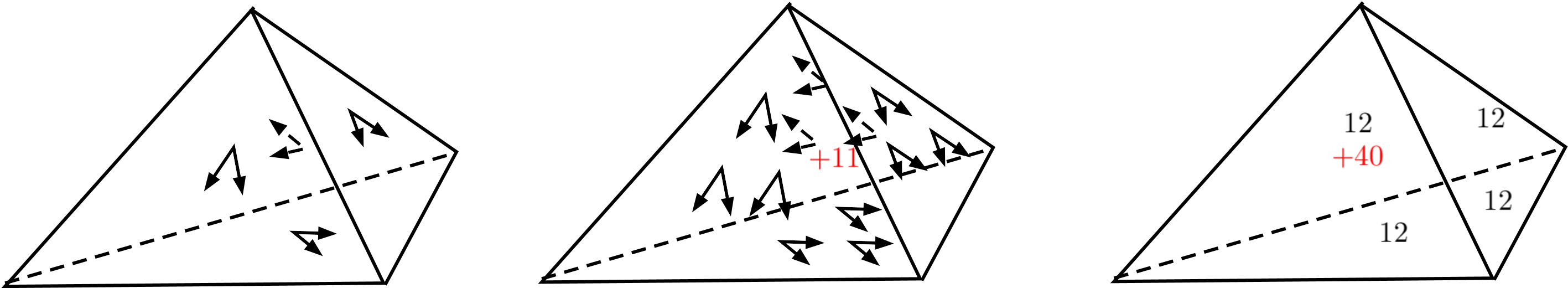}}
{\caption{An illustration for $C_{\iota^*\jmath_{[2]}^*}\mathcal P_r^-\W^{2,1}_{[2]}$ for $r = 1,2,3$. The local shape function is: ($r = 1$) $ \mathbb T$, dim = 8; ($r = 2$)  $(\mathcal P_1 \otimes \mathbb T) \oplus (\bm x\otimes (\bm x \times \mathbb V))$, dim = 35; ($r = 3$)  $(\mathcal P_2 \otimes \mathbb T) \oplus (\bm x\otimes (\bm x \times (\mathcal H_1 \otimes \mathbb V)))$, dim = 88.}}
\end{figure}
 \bibliographystyle{plain}
 \bibliography{\RevisionAssets/ref}

\newpage
\appendix 

\section{Technical Details on the BGG Framework}

In this paper, we show the technical details in 
\Cref{sec:bgg,sec:whitney,sec:high-order}.  
\subsection{Adjointness and  anticommutativity}
In this subsection, we show the adjointness of $\mathcal S$ and $\mathcal S_{\dagger}$, and the  anticommutativity of $\mathcal S_{\dagger}$ and $\kappa$. 
We first show the proof of \Cref{lem:S-inj-suj}, (1).

\begin{lemma}
\label{lem:adjointness}
$\cS$ and $\cS_{\dagger}$ are adjoint with respect to the  induced exterior inner products, i.e., 
\[
(\mathcal S^{k, \ell}\omega, \mu)=(\omega, \mathcal S^{k+1, \ell-1}_{\dagger}\mu), \quad\forall \omega\in \alt^{k, \ell}, \mu\in \alt^{k+1, \ell-1}.
\]
\end{lemma}
\begin{proof}
Use the exterior inner product for which the increasing wedge monomials $dx^I$ are orthonormal, and its tensor-product inner product on double forms. (Thus the raw sum over all tensor indices is divided by $k!\ell!$ on $\Alt^{k,\ell}$.) Let  $\varepsilon_i$ denote exterior multiplication by $dx^i$ and $\iota_i$ contraction with the $i$th unit vector. Then
{
\[
\mathcal S=\sum_i\varepsilon_i\otimes\iota_i,
\qquad \mathcal S_{\dagger}=\sum_i\iota_i\otimes\varepsilon_i.
\]}
 Since $\varepsilon_i^*=\iota_i$, the asserted adjointness follows. With unnormalized tensor Frobenius products, a degree-dependent factor would instead be required.
\end{proof}
\begin{lemma}
The identity \eqref{SdaggerP} holds : $\kappa$ and $\mathcal S_{\dagger}$  anticommute.
\end{lemma}
\begin{proof}
 Writing $\kappa=\sum_j x_j\iota_j\otimes\mathrm{id}$, we have
{
\[
\mathcal S_{\dagger}\kappa+\kappa\mathcal S_{\dagger}
=\sum_{i,j}x_j(\iota_i\iota_j+\iota_j\iota_i)\otimes\varepsilon_i=0.
\]}
 \end{proof}

\subsection{ Combinatorial Preparations}
\label{subsec:comb-preparation}
Next, we show the injectivity and the  surjectivity result.
The proof we adopt here resembles that in the appendix of \cite{arnold2021complexes}. However, we generalize the original proof to the iterated case and the bubble case. 

To this end, we first introduce some notations on combinatorics. Let $X(n,k)$ be the set of increasing $k$-tuples in $[n] := \{1,2,\cdots, n\}$. That is, $X(n,k) := \{\sigma \in [n]^k : \sigma_1 < \sigma_2 < \cdots < \sigma_k\}.$ Let $\mathcal F X(n,k)$ be {the free $\mathbb R$-vector space} generated by $X(n,k)$, and we use $[I]$ to represent the element associated with $I$. 
Therefore, the element in $\mathcal FX(n,k)$ has the form $\sum_{I \in X(n,k)} a_I [I],$ where $a_{I} \in \mathbb R$. 
{It is easy to see that $\dim \mathcal F X(n,k) = \binom{n}{k}.$} We then define $s^k : \mathcal F X(n,k) \to \mathcal F X(n,k+1)$ such that 
\[ s^k([I]) = \sum_{j \in [n] \setminus I}  [I\cup\{j\}], \quad \forall I \in X(n,k).\]
We also define $s_{\dagger}^{k+1} : \mathcal F X(n,k+1) \to \mathcal F X(n,k)$ such that 
\[s^{k+1}_{\dagger}([J]) =  \sum_{j \in J} [J \setminus \{j\}], \quad \forall J \in X(n,k+1).\]

We introduce the inner product on $\mathcal F X(n,k)$ such that $\langle [I], [J] \rangle = \delta_{IJ}.$ Similar to \Cref{lem:adjointness}, we have the following lemma.

\begin{lemma}
\label{lem:adjointness-s}
    $s^k$ and $s_{\dagger}^{k+1}$ are adjoint to each other. 
\end{lemma}
The following lemmas are formulated from the appendix of \cite{arnold2021complexes}.

\begin{lemma}
\label{lem:inner-s}
For all $a, b \in \mathcal F X(n,k)$, it holds that 
\begin{equation}
\label{eq:inner-s} \langle s^ka,s^kb\rangle = \langle s_{\dagger}^ka, s_{\dagger}^kb \rangle + (n-2k) \langle a, b \rangle.
\end{equation}
\end{lemma}
\begin{proof}
For $I, J \in X(n,k)$ it holds that 
\[\langle s^k([I]), s^k([J]) \rangle = \begin{cases} n-k & \mbox{ if } I = J, \\ 1 & \mbox{ if } \# I\cap J = k - 1, \\ 0 & \mbox{ else, } \end{cases}\]
and 
\[\langle s^k_{\dagger}([I]), s^k_{\dagger}([J]) \rangle = \begin{cases} k & \mbox{ if } I = J,  \\ 1 & \mbox{ if } \# I\cap J = k - 1, \\ 0 & \mbox{ else. } \end{cases}\]

It holds that 
\[\langle s^k([I]),s^k([J])\rangle = \langle s^k_{\dagger}([I]), s_{\dagger}^k([J]) \rangle + (n-2k) \langle [I], [J] \rangle.\]
By bi-linearity, we complete the proof.
\end{proof}

Following the identity \eqref{eq:inner-s}, we have the following injectivity/surjectivity on the single index space. 
\begin{lemma}
\label{lem:fx}
The following results hold for $s^k$ and $s^{k}_{\dagger}$:
\begin{enumerate}
\item When $n \ge 2k+1$, $s^k : \mathcal F X(n,k) \to \mathcal F X(n,k+1)$ is injective, $s^{k+1}_{\dagger}: \mathcal F X(n,k+1) \to \mathcal F X(n,k)$ is surjective.
\item When $n \le 2k+1$, $s^k : \mathcal F X(n,k) \to \mathcal F X(n,k+1)$ is surjective, $s^{k+1}_{\dagger}: \mathcal F X(n,k+1) \to \mathcal F X(n,k)$ is injective.
\end{enumerate}
\end{lemma}

\begin{proof}
By \eqref{eq:inner-s}, for $a \in \mathcal F X(n,k)$, we have
\[\|s^ka\|^2  = \|s^k_{\dagger} a\|^2 + (n-2k) \|a\|^2.\]
When $n \ge 2k + 1$, it holds that $\|s^ka\|^2 \ge \|a\|^2$. Therefore, $s^k$ must be injective. 

When $ n \le 2k + 1$, for $a \in \mathcal F X(n,k+1)$, we have \[\|s_{\dagger}^{k+1} a\|^2 = \|s^{k+1}a\|^2 + (2k+2 - n) \|a\|^2 \ge \|a\|^2. \] Therefore, $s^{k+1}_{\dagger}$ is injective. 

The surjectivity comes from the duality argument, see \Cref{lem:adjointness-s}.
\end{proof}

For iterated case, we introduce 
$s_{[p]}^k: \mathcal F X(n,k) \to \mathcal F X(n,k+p)$ such that 
\[ s^k_{[p]}([I]) = \sum_{ P \subseteq [n] \setminus I, |P| = p} [I\cup P], \quad \forall I \in X(n,k)\]
We also define  $s^{k+p}_{\dagger,[p]} : \mathcal F X(n,k+p) \to \mathcal F X(n,k)$ such that 
\[s_{\dagger,[p]}^{k+p}([J]) =  \sum_{ P \subseteq J, |P| = p} [J \setminus P], \quad \forall J \in X(n,k+p).\]
The operator $s_{[p]}^{\bs}$ is $1/p!$ times the $p$-fold composition of $s^{\bs}$; explicitly, \[s^{k+p-1}\circ s^{k+p-2} \circ \cdots \circ s^{k+1} \circ s^k = p!s_{[p]}^k.\] By adjointness, we have 
 {
\[s_{\dagger}^{k+1} \circ s_{\dagger}^{k+2} \circ \cdots \circ s_{\dagger}^{k+p - 1} \circ s_{\dagger}^{k+p} = p! s_{\dagger,[p]}^{k+p}.\]}

\begin{lemma}
\label{lem:fxp}
The following results hold for $s_{[p]}^k$ and $s^{k}_{\dagger,[p]}$:
\begin{enumerate}
\item When $n \ge 2k+p$, $s_{[p]}^k: \mathcal F X(n,k) \to \mathcal F X(n,k+p)$ is injective, and $s_{\dagger,[p]}^{k+p} : \mathcal F X(n,k+p) \to \mathcal F X(n,k)$ is surjective.
\item When $n \le 2k-p$, $s_{[p]}^{k-p}: \mathcal F X(n,k-p) \to \mathcal F X(n,k)$ is surjective, and $s_{\dagger,[p]}^k : \mathcal F X(n,k) \to \mathcal F X(n,k-p)$ is injective.
\end{enumerate}
\end{lemma}
\begin{proof}
The proof is similar to \Cref{lem:fx}, with more combinatoric identities involved. First, we suppose that $ n \ge 2k + p$. For $I, J \in X(n,k)$ and $\# (I \cap J) = k - \theta$ for some non-negative integer $\theta$ such that $0 \le \theta \le k$,  we have

\[\langle s_{[p]}^k ([I]), s_{[p]}^k([J]) \rangle = \binom{n - k - \theta}{p - \theta},\]
and \[\langle s_{\dagger,[p]}^k ([I]), s_{\dagger,[p]}^k([J]) \rangle = \binom{ k - \theta}{p - \theta}.\]

It follows from the Vandermonde's identity that 
\[\binom{n - k - \theta}{p - \theta} = \sum_{q = 0}^{n-2k} \binom{n - 2k}{q} \binom{k - \theta}{p-q-\theta}.\] 
For $ q > p$, $\binom{k - \theta}{p - q - \theta} = 0$ for all non-negative $\theta$. We then can truncate the summation as 
\[\binom{n - k - \theta}{p - \theta} = \sum_{q = 0}^{p} \binom{n - 2k}{q} \binom{k - \theta}{p-q-\theta} =  \sum_{q = 0}^{p} \binom{n - 2k}{p - q} \binom{k - \theta}{q-\theta} \] 
This yields that 
\[\langle s_{[p]}^k( [I]), s_{[p]}^k ([J]) \rangle = \sum_{q = 0}^p \binom{n - 2k}{p - q} \langle s_{\dagger,[q]}^k ([I]), s_{\dagger, [q]}^k( [J]) \rangle,
\]
and therefore, 
\begin{equation}
\label{eq:fxp}
\langle s_{[p]}^k a, s_{[p]}^k b \rangle = \sum_{q = 0}^p \binom{n - 2k}{p - q} \langle s_{\dagger,[q]}^k a, s_{\dagger, [q]}^k b \rangle.
\end{equation}

The remaining proofs are similar. We first show that when $n \ge 2k + p$,  $s_{[p]}^k: \mathcal F X(n,k) \to \mathcal F X(n,k+p)$ is injective. 

Suppose there exists $a$ such that $s_{[p]}^k a = 0$ but $ a \neq 0$, then taking $b = a$ in \eqref{eq:fxp} we obtain
\[
\| s_{[p]}^k a\|^2 = \sum_{q = 0}^p \binom{n - 2k}{p - q} \| s_{\dagger,[q]}^k a\|^2 \ge \binom{n-2k}{p} \|a\|^2.
\]

Therefore, $s_{[p]}^k: \mathcal F X(n,k) \to \mathcal F X(n,k+p)$ is injective. From the adjointness, we know that $s_{\dagger,[p]}^{k+p} : \mathcal F X(n,k+p) \to \mathcal F X(n,k)$ is surjective.

For (2), we switch the role of $s_{[p]}^{\bs}$ and $s_{\dagger,[p]}^{\bs}$. Since $n- k < k$, we have 
\[\binom{k-\theta}{p-\theta} = \sum_{q= 0}^{2k-n} \binom{2k-n}{q} \binom{n - k - \theta}{p - q - \theta} = \sum_{q= 0}^{2k-n} \binom{2k-n}{ p - q} \binom{n - k - \theta}{q- \theta}\]
Similarly, we obtain that \[\langle s_{\dagger,[p]}^k a, s_{\dagger,[p]}^k b \rangle = \sum_{q = 0}^p \binom{2k-n}{p - q} \langle s_{[q]}^k a, s_{[q]}^k b \rangle.
\]
Using this we can finish the proof.

\end{proof}

Next, we generalize the above combinatoric results with double indices, which mimic the $\mathcal S$ and $\mathcal S_{\dagger}$ operator in a combinatoric way. This calls for more notations.

Let $X(n,k,\ell) = X(n,k) \times X(n,\ell)$, and $\mathcal F X(n,k,\ell)$ be {the free $\mathbb R$-vector space} it generates. We define $s_{[p]}^{k,\ell}:\mathcal F X(n,k,\ell) \to \mathcal F X(n,k+ p ,\ell - p )$ such that 

 {
\begin{equation}\label{eq:sp}s_{[p]}^{k,\ell}([I,J]) =   \sum_{P \subseteq J\setminus I, |P| = p} [I \cup P, J \setminus P],\quad \forall I \in X(n,k), J \in X(n,\ell),\end{equation}}
 and $s_{\dagger, [p]}^{k,\ell}:\mathcal F X(n,k,\ell) \to \mathcal F X(n,k- p ,\ell + p )$ such that
 {
\begin{equation}\label{eq:sdaggerp} s_{\dagger,[p]}^{k,\ell} ([I,J]) = \sum_{P \subseteq I \setminus J, |P| = p} [I \setminus P, J \cup P]\quad \forall I \in X(n,k), J \in X(n,\ell).\end{equation}}

\begin{lemma}
\label{lem:fxx}
For double indices, we have the following lemma.
\begin{enumerate}

\item If $k + p \le \ell$, then $s_{[p]}^{k,\ell} : \mathcal F X(n,k,\ell) \to \mathcal F X(n,k+p,\ell - p)$ is injective, $s_{\dagger, [p]}^{k+p,\ell-p}: \mathcal F X(n,k+p,\ell - p) \to \mathcal F X(n,k,\ell)$ is surjective.
\item If $k + p \ge \ell$, then  $s_{[p]}^{k,\ell} : \mathcal F X(n,k,\ell) \to \mathcal F X(n,k+p,\ell-p)$ is surjective,  $s_{\dagger, [p]}^{k+p,\ell-p}: \mathcal F X(n,k+p,\ell - p) \to \mathcal F X(n,k,\ell)$ is injective.
\end{enumerate}
\end{lemma}
The proof of \Cref{lem:fxx} is based on a decomposition on $X(n,k,\ell)$, and the injectivity/surjectivity results of $s^{\bs}_{\bs}$ with a single index. 

To see this, we observe the invariants for $s_{[p]}$ and $s_{\dagger,[p]}$. Note that for the pair $I, J$ and the pair $I' := I\cup P, J' := J\setminus P$, it holds that \[I \cup J = I' \cup J' \text{ and } I' \cap J' = I \cap J.\] Therefore, the double index sets appeared in  either the left-hand sides or the right-hand sides are having the same intersection sets and union sets. This motivates the following definition. For two sets $F$ and $G$, we define 
\[Y_{F,G}(n,k,\ell):= \{ (I,J) : I \in X(n,k), J \in X(n,\ell) : I \cup J = F, I \cap J = G\} \subseteq X(n,k,\ell).\] When $F$ and $G$ are going through all the possibilities, the family of sets $Y_{F,G}(n,k,\ell)$ forms a disjoint union of $X(n,k,\ell)$. 
\begin{lemma}
    It holds that $s_{[p]}^{k,\ell}$ maps $\mathcal FY_{F,G}(n,k,\ell) \to \mathcal FY_{F,G}(n,k+p,\ell - p)$, and $s_{\dagger,[p]}^{k,\ell}$ maps $\mathcal FY_{F,G}(n,k,\ell) \to \mathcal FY_{F,G}(n,k-p,\ell + p)$. 
\end{lemma} 

\begin{proof}
    By the above observation.
\end{proof}

For the mapping for double indices, we have the following injectivity/surjectivity result.
\begin{lemma}
The following results hold.
\label{lem:fy}
\begin{enumerate}
\item If $k + p \le \ell$, then $s_{[p]}^{k,\ell}: \mathcal F Y_{F,G}(n,k,\ell) \to \mathcal F Y_{F,G}(n,k+p,\ell - p)$ is injective, and $s_{\dagger, [p]}^{k+p,\ell-p} : \mathcal F Y_{F,G}(n,k+p,\ell - p) \to \mathcal F Y_{F,G}(n,k,\ell)$ is surjective.
\item  If $k + p \ge \ell$, then $s_{[p]}^{k,\ell}: \mathcal F Y_{F,G}(n,k,\ell) \to \mathcal F Y_{F,G}(n,k+p,\ell - p)$ is surjective, and $s_{\dagger, [p]}^{k+p,\ell - p} : \mathcal F Y_{F,G}(n,k+p,\ell - p) \to \mathcal F Y_{F,G}(n,k,\ell)$ is injective.
\end{enumerate}	
\end{lemma}
\begin{proof}
 Write $|F|=f$ and $|G|=g$. Removing the common indices $G$ identifies a block with $X(f-g,k-g)$; the target is $X(f-g,k-g+p)$. By \Cref{lem:fxp}, the map is injective if  $f-g\ge2(k-g)+p$ and surjective if the reverse inequality holds. Since $f+g=k+\ell$, these conditions are respectively $k+p\le\ell$ and  $k+p\ge\ell$.
\end{proof}

\Cref{lem:fxx} is then a direct consequence of \Cref{lem:fy}.

\subsection{Proof of Injectivity/ Surjectivity}
\label{subsec:inj/sur}

Now we are ready to show the results in the main paper, as shown in \Cref{lem:S-inj-suj} (2)(3), \Cref{lem:Sp-inj-sur} (2)(3). 
For the bubble result, we will prove \Cref{lem:S-bubble} (2) and \Cref{lem:Sp-bubble} (2). 

Note that \Cref{lem:Sp-inj-sur} covers \Cref{lem:S-inj-suj}, and we only need to show \Cref{lem:Sp-inj-sur}.

We first show the injectivity and surjectivity result on the constant form-valued forms.
\begin{proof}[Proof of \Cref{lem:Sp-inj-sur}(2)(3)]
 Fix $F=I\cup J$ and  $G=I\cap J$, and put $U=F\setminus G$. For $j\in U$ define $c_j=(-1)^{\#\{u\in U:u<j\}}$ and set
{
\[
\omega_{I,J}=\left(\prod_{j\in I\setminus G}c_j\right)dx^I\otimes dx^J.
\]}
 Moving an index $j$ between the two factors by $\mathcal S$ or $\mathcal S_{\dagger}$ has sign $c_j$ in the ordinary increasing wedge basis.  The displayed rephasing cancels that sign. On this block the iterated operators are therefore $p!$ times the corresponding unsigned incidence maps in \Cref{lem:fy} . The claimed injectivity and surjectivity follow block by block.
\end{proof}
 \begin{proof}[Proof of \Cref{lem:Sp-bubble}(2)]
Use the spanning family $\phi_I\otimes d\lambda_J$ with $|I|=k+1$, $|J|=\ell$, and  $I\cup J=[n+1]$. Rephase each fixed-union, fixed-intersection block by the same rule as above, now on the vertex index set. In the expansion of $\mathcal S_{\dagger}(\phi_I\otimes d\lambda_J)$, moving $j\in I\setminus J$ has the negative of the corresponding constant-form sign. Indeed, the Whitney form is a contraction of the coordinate wedge form, and contraction anticommutes with $\mathcal S_{\dagger}$. Hence $\mathcal S_{\dagger,[p]}$ intertwines the rephased spanning map with $(-1)^p p!s_{\dagger,[p]}$. The incidence map is onto when $k+1\le\ell+p$. Its surjectivity descends through the spanning maps, without requiring their injectivity.
\end{proof}
The same argument proves \Cref{lem:S-bubble-Pr-}: keep $\alpha$ fixed in $\lambda^\alpha\phi_I\otimes d\lambda_J$, and retain the blocks whose union contains $[n+1]\setminus\supp\alpha$. For \Cref{lem:S-bubble-Pr}, use $\lambda^\alpha d\lambda_I\otimes d\lambda_J$ with $|I|=k$ instead; the incidence-map condition is $k\le\ell+p$.

\section{Counting Skeletal DoFs}
\label{sec:dof-counting}

In this section, we prove \Cref{prop:skeletal-counting}: 
Let $k = \ell + p$, then we have
\begin{equation}
\begin{split} 
\Phi &= \sum_{s = 0}^{p-1} \sum_{\theta = 0}^{k} (-1)^{\theta} \binom{\theta}{s} \binom{n-\theta}{n - \ell -s} f_{\theta} \\ &\quad + (-1)^{n+p-1}  \sum_{s = 0}^{p-1} \sum_{\theta = 0}^{n - \ell}  (-1)^{\theta} \binom{\theta}{s} \binom{n - \theta}{k-s} f_{\theta}^{\circ} \\ &= \binom{n}{\ell} + (-1)^{p-1} \binom{n}{k}.
\end{split}
\end{equation}

We first recall some basic identities about binomial coefficients with negative inputs. 
\begin{lemma}
\label{lem:binom-negative}
For  an integer $a$ and a nonnegative integer $b$, use $\binom{a}{b}=a(a-1)\cdots(a-b+1)/b!$ and $\binom{a}{0}=1$. Then
{
\[
\binom{a}{b}=(-1)^b\binom{b-a-1}{b}.
\]}
Throughout this appendix, a binomial coefficient with negative lower index is zero.
\end{lemma}
\begin{proof}
 Reverse the signs in the $b$ factors of the numerator.
\end{proof}

We now recall the well-known Dehn–Sommerville equation \cite{sommerville1927relations}, see also \cite[Chapter 9.2.2]{grunbaum2003convex}. This equation extends the Euler equation for polytopes to the most general scenario.

\begin{lemma}[Dehn-Sommerville]
\label{lem:dehn-sommerville}
	For any  simplicial $(d-1)$-sphere $P$  (in particular, the boundary of a simplicial $d$-polytope), let $f_i(P)$ be its number of $i$-faces. Then, for each $p$ with $0\le p\le d$, we have
	\[\sum_{i=-1}^{p-1}(-1)^{d+i}\binom{d-i-1}{d-p} f_i(P) = \sum_{i=-1}^{d-p-1}(-1)^i \binom{d-i-1}{p} f_i(P).\]
    Here $f_{-1}(P) = 1$,  $f_d(P)=0$. 
\end{lemma}

Note that by adding some zero terms, we can extend the superscript in the equation above as
 {
\[\sum_{i=-1}^{d-1}(-1)^{d+i}\binom{d-i-1}{d-p} f_i(P) = \sum_{i=-1}^{d-1}(-1)^i \binom{d-i-1}{p} f_i(P).\]}

To start the proof, we first reformulate \Cref{prop:skeletal-counting} to the form that the  Dehn-Sommerville equation can be applied. 

{Let $f_i^{\partial} = f_i - f_i^{\circ}$ be the number of boundary $i$-faces of $\mathcal T$.  Set $f_{-1}^{\partial}=1$ and $f_i^{\partial}=0$ outside $-1\le i\le n-1$. The boundary is an $(n-1)$-sphere, so \Cref{lem:dehn-sommerville}  gives
{
\begin{equation}
\label{eq:DS-partial}
\begin{split}
\sum_{\theta = 0}^{n} (-1)^{\theta} \binom{n - 1 -(\theta-1)}{n-1 + 1 -(p-1) } f_{\theta-1}^{\partial}=   (-1)^{n} \sum_{\theta = 0}^{n} (-1)^{\theta}  \binom{n-\theta}{p - 1} f_{\theta-1}^{\partial},
\end{split}  
\end{equation}} 
for  $1\le p\le n+1$.
Moreover,  cap the boundary by a cone: $C=\mathcal T\cup(v*\partial\mathcal T)$, where $v$ is a new vertex. This is an $n$-sphere, and 
{
\begin{equation} \label{eq:cone-relation}f_i^C := f_i(C) = f_i + f_{i-1}^{\partial},\end{equation}}
and we have 
 {
\begin{equation}
\begin{split}
\sum_{\theta=-1}^{n}(-1)^\theta\binom{n-\theta}{n+1-p}f_\theta^C
+(-1)^n\sum_{\theta=-1}^{n}(-1)^\theta\binom{n-\theta}{p}f_\theta^C=0.
\end{split} 
\end{equation}} 
For technical reasons, we write the term $\theta = -1$ separately as 
\begin{equation}\label{eq:DS-cone}
\begin{split}
\sum_{\theta = 0}^{n} (-1)^{\theta} \binom{n - \theta }{n + 1 - p} f_{\theta}^{C} +   (-1)^{n} \sum_{\theta = 0}^{n} (-1)^{\theta}  \binom{n-\theta}{p } f_{\theta}^C = [1+(-1)^n]\binom{n+1}{p}.
\end{split} 
\end{equation} 
}

The rest of this section is based on \eqref{eq:DS-cone} and \eqref{eq:DS-partial}, as in general these formulas exhaust all the linear constraints of these quantities. We now decouple $\Phi$ as two terms, based on variables $f_{i}^{\partial}$ and $f_i^C$, and we deal with two terms separately. 
\begin{lemma}
Let
\begin{equation}
\label{eq:Phi-C}
\Phi^{C} := \sum_{s = 0}^{p-1} \sum_{\theta = 0}^{k} (-1)^{\theta} \binom{\theta}{s} \binom{n-\theta}{n - \ell -s} f_{\theta}^{C} + (-1)^{n+p-1}  \sum_{s = 0}^{p-1} \sum_{\theta = 0}^{n-\ell}  \binom{\theta}{s} (-1)^{\theta}  \binom{n - \theta}{k-s} f_{\theta}^{C},
\end{equation}
and 
 {
\begin{equation}
\label{eq:Phi-partial}
\Phi^{\partial} := \sum_{s = 0}^{p-1} \sum_{\theta = 0}^{k} (-1)^\theta\binom{\theta}{s} \binom{n-\theta}{n - \ell -s}  f_{\theta-1}^{\partial} + (-1)^{n+p-1}\sum_{s = 0}^{p-1} \sum_{\theta = 0}^{n - \ell} (-1)^\theta\binom{\theta}{s} \binom{n - \theta}{k-s} (f_{\theta}^{\partial} + f_{\theta-1}^{\partial}).
\end{equation}}
Then it holds that 
\[\Phi = \Phi^C - \Phi^{\partial}.\]
\end{lemma}
\begin{proof}
    It follows from substituting \eqref{eq:cone-relation} to $\Phi$.
\end{proof}

We first show the following result, stating how the term $\binom{\theta}{s} \binom{n-\theta}{n-\ell-s}$ can be expressed as a linear combination of $\binom{n-\theta}{n+1-\xi}$.

The following lemma illustrates how to transform the product term $\binom{m}{p} \binom{n}{q} $ for $m < n$, see also~\cite[Chapter 1.4]{riordan1979combinatorial}. 
\begin{lemma}
For $m < n$, it holds that 
\begin{equation}
\label{eq:product-identity}
\binom{m}{p}	\binom{n}{q} = \sum_{\xi} \binom{m-n+q}{p-\xi+q} \binom{\xi}{q} \binom{n}{\xi}.
\end{equation}

\end{lemma}
\begin{proof}
By the Vandermonde's identity, we have 
\[ 
		\binom{m}{p} \binom{n}{q} =  \binom{m-n+q + n -q}{p} \binom{n}{q} 
		= \sum_{\xi} \binom{m-n+q}{p-\xi} \binom{n-q}{\xi} \binom{n}{q}
\]
Since $\binom{q+\xi}{q} \binom{n}{q+\xi} = \binom{n-q}{\xi} \binom{n}{q}$, we have 
\begin{equation}
	\begin{split}
\sum_{\xi} \binom{m-n+q}{p-\xi} \binom{n-q}{\xi} \binom{n}{q} 
		= &\sum_{\xi} \binom{m-n+q}{p-\xi} \binom{q+\xi}{q} \binom{n}{q+\xi} \\
		= &\sum_{\xi} \binom{m-n+q}{p-\xi +q} \binom{\xi}{q} \binom{n}{\xi}.
	\end{split}
\end{equation}	
Here the last equation comes from replacing $\xi$ by $\xi + q$. 
\end{proof}

\begin{lemma}
\label{lem:theta-n-theta}
 For $0\le\theta\le n$ and $0\le\ell<k=\ell+p\le n$, the following identities hold; both sums over $\xi$ run from $\ell+1$ to $k$.
\begin{equation}
\label{eq:ell-prod}
\sum_{s = 0}^{p-1} \binom{\theta}{s} \binom{n-\theta}{n-\ell-s} = \sum_{\xi} (-1)^{k + \xi } \binom{\xi-1}{\ell}\binom{n-\xi}{n-k}\binom{n-\theta}{n+1-\xi},
\end{equation}
and
\begin{equation}
\label{eq:k-prod}
\sum_{s = 0}^{p-1} \binom{\theta}{s} \binom{n-\theta}{k-s} = \sum_{\xi} (-1)^{\ell + \xi + 1} \binom{\xi-1}{\ell}\binom{n-\xi}{n-k}\binom{n-\theta}{\xi}.
\end{equation}
\end{lemma}
\begin{proof}

We first prove \Cref{eq:ell-prod}. We first expand the term $\binom{\theta}{s} \binom{n-\theta}{n - \ell - s}$ as a linear combination of $\binom{n-\theta}{n+1-\xi}$, where the coefficients depend on $\xi$ but are independent of $\theta$. 

By \eqref{eq:product-identity}, it holds that
 {
\begin{equation}
\label{eq:ell-prod1}
\begin{split}
\binom{\theta}{s} \binom{n-\theta}{n - \ell - s} = & (-1)^{s} \binom{-\theta + s - 1}{s} \binom{n-\theta}{n - \ell - s} \\
= & (-1)^s \sum_{\xi} \binom{ (-\theta+s -1)+ (n - \ell -s) - (n - \theta) }{n-\ell-\xi} \binom{\xi}{n-\ell -s}\binom{n-\theta}{\xi}  \\ 
= & (-1)^s \sum_{\xi} \binom{-1-\ell}{n - \ell - \xi}\binom{\xi}{n-\ell-s} \binom{n-\theta}{\xi}. \end{split}
\end{equation}}

Replacing $\xi$ by $n + 1 - \xi$, we have 
\[\binom{\theta}{s} \binom{n-\theta}{n - \ell - s} = (-1)^s \sum_{\xi} \binom{-1-\ell}{-\ell - 1 + \xi} \binom{n + 1 - \xi}{n-\ell-s} \binom{n-\theta}{n + 1 - \xi}.\]

 The term $\xi=n+1$ vanishes because $n-\ell-s\ge1$ for $0\le s\le p-1$, so we may restrict to $\xi\le n$. We may also restrict to $\xi\ge\ell+1$, because the lower index of $\binom{-1-\ell}{-\ell-1+\xi}$ is negative otherwise. Thus
\[ \begin{split} \binom{-1-\ell}{-\ell - 1 + \xi} = & (-1)^{\ell + 1 + \xi } \binom{(1 + \ell)  + (- \ell- 1 + \xi) - 1}{- \ell - 1 + \xi} \\ =& (-1)^{\ell+ 1 + \xi } \binom{\xi - 1}{-\ell - 1 + \xi} \\=& (-1)^{\ell + 1 + \xi } \binom{\xi - 1}{\ell}. \end{split} \] 
The last equation comes from \Cref{lem:binom-negative}.

\begin{equation}
\label{eq:ell-prod2}
\begin{split}
\binom{\theta}{s} \binom{n-\theta}{n - \ell - s} = & (-1)^s \sum_{\xi} \binom{-1-\ell}{-\ell - 1 + \xi} \binom{n + 1 - \xi}{n-\ell-s} \binom{n-\theta}{n + 1 - \xi} \\ 
= & (-1)^s \sum_{\xi} (-1)^{\ell + 1 + \xi } \binom{\xi - 1}{\ell} \binom{n+1-\xi}{n-\ell-s} \binom{n-\theta}{n + 1 - \xi}.
 \end{split} 
\end{equation}

As a result, when expressing $\displaystyle\binom{\theta}{s} \binom{n-\theta}{n - \ell -s}$ as the linear combination of $\displaystyle\binom{n-\theta}{n+1-\xi}$, the coefficient is $\displaystyle (-1)^{\ell + 1 + \xi } (-1)^s \binom{\xi - 1}{\ell} \binom{n+1-\xi}{n-\ell-s}$. 

Therefore, summing over all $s$ about the coefficient of  $\displaystyle\binom{n-\theta}{n+1-\xi}$ we obtain:
 {
\begin{equation}
\label{eq:ell-prod-coeff}
\begin{split}
	  & (-1)^{\ell + 1 + \xi} \sum_{s = 0}^{p-1} (-1)^s \binom{\xi - 1}{\ell} \binom{n+1-\xi}{n-\ell-s}  \\
 = &  (-1)^{\ell + 1 + \xi} \sum_{s = 0}^{p-1} (-1)^s \binom{\xi - 1}{\ell} \left[\binom{n-\xi}{n-\ell-s} + \binom{n-\xi}{n-\ell-s-1}\right]\\ 
 = &  (-1)^{\ell + 1 + \xi} (-1)^{p-1} \binom{\xi - 1}{\ell}\binom{n - \xi}{n - \ell - p} + (-1)^{\ell + 1 + \xi} \binom{\xi - 1}{\ell}\binom{n - \xi}{n - \ell}  \\ 
 = & (-1)^{\ell + 1 + \xi} (-1)^{p-1} \binom{\xi - 1}{\ell}\binom{n - \xi}{n - \ell - p} \\ = & (-1)^{k + \xi} 
 \binom{\xi - 1}{\ell}\binom{n - \xi}{n - k} 
\end{split}
\end{equation}}
Combining both \eqref{eq:ell-prod1} and  \eqref{eq:ell-prod-coeff}, we obtain \eqref{eq:ell-prod}.

Similarly, we have 
\begin{equation}
\begin{split}
\binom{\theta}{s} \binom{n-\theta}{k- s} =  (-1)^s \sum_{\xi} (-1)^{k + \xi} \binom{\xi}{k-s} \binom{n-\xi}{n-k} \binom{n-\theta}{\xi}.
\end{split} 
\end{equation}

Summing over all $s$ about the coefficient of $\displaystyle \binom{n-\theta}{\xi}$ we obtain 
\begin{equation}
\begin{split}
	  & (-1)^{k + \xi} \sum_{s = 0}^{p-1} (-1)^s \binom{\xi}{k-s} \binom{n-\xi}{n-k}  \\
 = & (-1)^{k+ \xi } (-1)^{p-1} \binom{\xi - 1}{\ell}\binom{n - \xi}{n - k}.
\end{split}
\end{equation}
Therefore, we conclude the result. 
\end{proof}

Based on the above identities, now we can compute the concrete value of $\Phi^C$. 

\begin{proposition}
\label{prop:Phi-C}
    It holds that 
    \[\Phi^C := [1 + (-1)^n](\binom{n}{\ell} + (-1)^{\ell - k + 1} \binom{n}{k}). \]
\end{proposition}
\begin{proof}
It follows from the definition of $\Phi^C$ \eqref{eq:Phi-C} that

\begin{equation}
\label{eq:PhiC-0}
    \Phi^{C}= \sum_{\theta = 0}^{k} (-1)^{\theta} f_{\theta}^{C} \underbrace{\sum_{s = 0}^{p-1}\binom{\theta}{s} \binom{n-\theta}{n - \ell -s}}_{\eqref{eq:ell-prod}}  + (-1)^{n+p-1}  \sum_{\theta = 0}^{n-\ell} (-1)^{\theta}   f_{\theta}^{C}  \underbrace{\sum_{s = 0}^{p-1} \binom{\theta}{s}   \binom{n - \theta}{k-s}}_{\eqref{eq:k-prod}}.
\end{equation}

By applying \eqref{eq:ell-prod} and \eqref{eq:k-prod} to both terms, we then have 
\begin{equation}
\label{eq:PhiC-1}
\begin{split} \Phi^{C}
= & 
\sum_{\theta = 0}^{k} (-1)^{\theta} f_{\theta}^{C} 
\sum_{\xi} (-1)^{k + \xi } \binom{\xi-1}{\ell}\binom{n-\xi}{n-k}\binom{n-\theta}{n+1-\xi}  \\
& \quad\quad + (-1)^{n+p-1}  \sum_{\theta = 0}^{n-\ell} (-1)^{\theta}   f_{\theta}^{C}  \sum_{\xi} (-1)^{\ell + \xi + 1} \binom{\xi-1}{\ell}\binom{n-\xi}{n-k}\binom{n-\theta}{\xi}.
\end{split}
\end{equation}
Here, $\xi$ ranges from $\ell + 1$ to $k$.

We now reformulate \eqref{eq:PhiC-1} by interchanging the order of the summation, namely, we first sum on $\theta$, then sum on $\xi$. This leads to 
\begin{equation}
 \Phi^C=  (-1)^k \sum_{\xi} (-1)^{\xi} \binom{\xi - 1}{\ell} \binom{ n - \xi}{n - k}  \Big[ \sum_{\theta = 0}^{k} (-1)^{\theta} \binom{n-\theta}{n+1-\xi} f_{\theta}^{C}  + \sum_{\theta = 0}^{n - \ell} (-1)^{n+\theta} \binom{n - \theta} {\xi}  f_{\theta}^C \Big].
 \end{equation}

The inner summation involving $\theta$ can be calculated with the help of \Cref{eq:DS-cone}. In fact, we have 
\[\sum_{\theta = 0}^{n} (-1)^{\theta} \binom{n - \theta }{n + 1 - \xi} f_{\theta}^{C} +   (-1)^{n} \sum_{\theta = 0}^{n} (-1)^{\theta}  \binom{n-\theta}{\xi } f_{\theta}^C = [1+(-1)^n]\binom{n+1}{\xi}.\]

Therefore, we have 
 {
\begin{equation}
\label{eq:simplified-PhiC}
\begin{split}
    \Phi^C := &  (-1)^k \sum_{\xi} (-1)^{\xi} \binom{\xi - 1}{\ell} \binom{ n - \xi}{n - k}  \Big[ \sum_{\theta = 0}^{\xi - 1} (-1)^{\theta} \binom{n-\theta}{n+1-\xi} f_{\theta}^{C}  + \sum_{\theta = 0}^{n-\xi} (-1)^{n+\theta} \binom{n - \theta} {\xi}  f_{\theta}^C \Big] \\ 
    = & (-1)^k \sum_{\xi} (-1)^{\xi} \binom{\xi - 1}{\ell} \binom{ n - \xi}{n - k} [1+(-1)^n] \binom{n+1}{\xi} \\
    = & (-1)^{k}[1 + (-1)^n] \sum_{\xi} (-1)^{\xi} \binom{\xi - 1}{\ell} \binom{ n - \xi}{n - k}\binom{n+1}{\xi}.
\end{split}
\end{equation}}

Since 
 {
\begin{equation}
\label{eq:Psi_C}
\begin{split}
 & \sum_{\xi} (-1)^{\xi} \binom{\xi-1}{\ell} \binom{n-\xi}{n-k} \binom{n+1}{\xi} \\ 
& = \sum_{\xi} (-1)^{\xi} \binom{\xi-1}{\ell} \binom{n-\xi}{n-k} \binom{n}{\xi} + \sum_{\xi} (-1)^{\xi} \binom{\xi-1}{\ell} \binom{n-\xi}{n-k} \binom{n}{\xi-1} \\
& =  \sum_{\xi} (-1)^{\xi} \binom{\xi - 1}{\ell} \binom{k}{\xi} \binom{n}{k} + \sum_{\xi} (-1)^{\xi} \binom{n}{\ell} \binom{n-\ell}{n-\xi +1} \binom{n - \xi}{n - k} \\
&=\binom{n}{k}(-1)^{\ell+1}+\binom{n}{\ell}(-1)^k \\
& = (-1)^{\ell + 1} \binom{n}{k} + (-1)^{k} \binom{n}{\ell}.
\end{split}	
\end{equation}}
Here  the third line  uses $\binom{a}{b}\binom{b}{c}=\binom{a}{c}\binom{a-c}{b-c}$. For the final evaluation, let
{
\[
A=\sum_{\xi=\ell+1}^{k}(-1)^\xi\binom{\xi-1}{\ell}\binom{k}{\xi}.
\]}
The coefficient of $t^\ell$ in $\sum_{\xi=0}^k(-1)^\xi\binom{k}{\xi}(1+t)^{\xi-1}=(-t)^k/(1+t)$ is zero since $\ell<k$. The $\xi=0$ term is $(-1)^\ell$, so $A=(-1)^{\ell+1}$. The second finite sum is transformed into this same identity by $\eta=n+1-\xi$, with $(k,\ell)$ replaced by $(n-\ell,n-k)$; its value is $(-1)^k$.

Therefore, by \eqref{eq:simplified-PhiC}, we conclude the result. 
\end{proof}

Next, we deal with $\Phi^{\partial}$  \eqref{eq:Phi-partial}:
 {
\begin{equation}
 \Phi^{\partial} := \sum_{s = 0}^{p-1} \sum_{\theta = 0}^{k} (-1)^\theta\binom{\theta}{s} \binom{n-\theta}{n - \ell -s}  f_{\theta-1}^{\partial} + (-1)^{n+p-1}\sum_{s = 0}^{p-1} \sum_{\theta = 0}^{n - \ell} (-1)^\theta\binom{\theta}{s} \binom{n - \theta}{k-s} (f_{\theta}^{\partial} + f_{\theta-1}^{\partial}).
\end{equation}}

\begin{proposition}
\label{prop:Phi-partial}
\[\Phi^{\partial} = (-1)^n \left(\binom{n}{\ell} + (-1)^{\ell - k + 1} \binom{n}{k}\right).\]	
\end{proposition}
\begin{proof}
By  \eqref{eq:ell-prod} and \eqref{eq:k-prod} in \Cref{lem:theta-n-theta}, it holds that 
 {
\begin{equation}
\begin{split}
\Phi^{\partial} = (-1)^k \sum_{\xi} (-1)^{\xi} \binom{\xi - 1}{\ell} \binom{ n - \xi}{n - k}  \Big[ \sum_{\theta = 0}^{n} (-1)^{\theta} \binom{n-\theta}{n+1-\xi} f_{\theta-1}^{\partial}  + \sum_{\theta = 0}^{n} (-1)^{n+\theta} \binom{n - \theta} {\xi}  (f_{\theta}^\partial + f_{\theta-1}^{\partial}) \Big].
\end{split} 
\end{equation}}

We now define 

 {
\[\Psi^{\partial} = \sum_{\theta = 0}^{n} (-1)^{\theta} \binom{n-\theta}{n+1-\xi} f_{\theta-1}^{\partial}  + \sum_{\theta = 0}^{n} (-1)^{n+\theta} \binom{n - \theta} {\xi}  (f_{\theta}^\partial + f_{\theta-1}^{\partial}).\]}

By \Cref{eq:DS-partial}, we have 

\begin{equation}
\begin{split}
\sum_{\theta = 0}^{n} (-1)^{\theta} \binom{n  -\theta}{n+ 1 - \xi } f_{\theta-1}^{\partial}=   (-1)^{n} \sum_{\theta = 0}^{n} (-1)^{\theta}  \binom{n-\theta}{\xi - 1} f_{\theta-1}^{\partial}.
\end{split} 
\end{equation}
Therefore, it holds that 
 {
\begin{equation}
\begin{split}
    \Psi^{\partial} = &  (-1)^n  \sum_{\theta = 0}^{n} (-1)^{\theta} \left[  \binom{n-\theta}{\xi} f_{\theta}^{\partial} +\binom{n-\theta}{\xi}  f_{\theta-1}^{\partial} + \binom{n-\theta}{\xi-1} f_{\theta-1}^{\partial}\right] \\ 
     = & (-1)^n  \sum_{\theta = 0}^{n} (-1)^{\theta} \left[ \binom{n-\theta}{\xi} f_{\theta}^{\partial}  + \binom{n+1 - \theta}{\xi} f_{\theta - 1}^{\partial}\right]\\
     = & (-1)^n \binom{n+1}{\xi}.
    \end{split}
\end{equation}}
By \eqref{eq:Psi_C}, we conclude the result.
\end{proof}
Combining \Cref{prop:Phi-C,prop:Phi-partial},\eqref{eq:Phi-partial}, we have the following result.
\begin{theorem}
    \[\Phi = \Phi^C - \Phi^{\partial} = (\binom{n}{\ell} + (-1)^{\ell - k + 1} \binom{n}{k}).\]
\end{theorem}

\section{High-order Regge elements}
\label{sub:regge}
To close this subsection, we prove that when $k = \ell = 1, p = 1$, the above construction recovers the high-order Regge element in any dimension, cf. \cite{li2018regge}.

For $\mathcal P_r\alt^{1,1}$ form, the numbers of degrees of freedom on simplex $K$ are
 {
\begin{equation}
\footnotesize
\begin{split}
 \dim \mathcal B_{r}  \alt^{1,1}(K) = &   \sum_{m =1}^n \binom{n+1}{m+1} \Big[ \binom{r+1}{r} \binom{r - 1}{m - 1} \Big] \cdot \binom{m}{n-1} \\ 
= &  \binom{n+1}{n} \binom{r+1}{r} \binom{r - 1}{n-2}  \binom{n-1}{n-1} +  \binom{n+1}{n+1} \binom{r+1}{r} \binom{r - 1}{n-1}  \binom{n}{n-1} \\ 
= & (n+1)(r+1) \binom{r-1}{n-2} + n(r+1) \binom{r-1}{n-1} \\ 
= & n(r+1) \binom{r}{n-1} + (r+1) \binom{r-1}{n-2} \\ 
= & n^2 \binom{r+1}{n} +  (r+1) \binom{r-1}{n-2}
\end{split}
\end{equation}}

For $\mathcal P_r\alt^{0,2}$, the numbers of the degrees of freedom on simplex $K$ are
 {
\begin{equation}
\footnotesize
\begin{split}
    \dim \mathcal B_{r}  \alt^{0,2}(K) = &  \sum_{m =0}^n \binom{n+1}{m+1} \Big[ \binom{r}{r} \binom{r-1}{m} \Big] \cdot \binom{m}{n-2},
    \\
    = & \binom{n+1}{n-1} \binom{r-1}{n-2} \binom{n-2}{n-2} + \binom{n+1}{n} \binom{r-1}{n-1} \binom{n-1}{n-2} + \binom{n+1}{n+1} \binom{r-1}{n} \binom{n}{n-2} \\
    = & \binom{n+1}{2} \binom{r-1}{n-2} + (n^2-1) \binom{r-1}{n-1} + \binom{n}{2} \binom{r-1}{n} \\ 
    = & \binom{n+1}{2} \binom{r}{n-1} + \binom{n}{2} \binom{r}{n} - \binom{r-1}{n-1}.
\end{split}
\end{equation}}

For $\mathcal P_r\mathbb W^{1,1}$ form, the numbers of degrees of freedom on simplex $K$ are 
\begin{equation}
\begin{split}
   &   \dim \mathcal B_{r}  \alt^{1,1}(K) - \dim \mathcal B_{r}  \alt^{0,2}(K)  \\
    = & n^2 \binom{r+1}{n} +  (r+1) \binom{r-1}{n-2} - \binom{n+1}{2} \binom{r}{n-1} -  \binom{n}{2} \binom{r}{n} + \binom{r-1}{n-1} \\ 
    = & \binom{n}{2} \binom{r}{n-1} + \binom{n+1}{2} \binom{r}{n} + \binom{r}{n-1} + r\binom{r-1}{n-2} \\ 
    = & \binom{n}{2} \binom{r}{n-1} + \binom{n+1}{2} \binom{r}{n} + \binom{r}{n-1} + (n-1)\binom{r}{n-1} \\
    = & \binom{n}{2} \binom{r}{n-1} + \binom{n+1}{2} \binom{r}{n} + n \binom{r}{n-1} \\ 
    = & \binom{n+1}{2} \binom{r+1}{n}
\end{split}
\end{equation}
which is equal to $\dim \mathcal P_{r-n+1} \mathbb W^{1,1} $, the dimension of degrees of freedom of Regge elements in \cite{li2018regge}.  Moreover, $\W^{1,1}=\mathbb S$, so the local shape space is $\mathcal P_r(K;\mathbb S)$. The double pullback to a face restricts both arguments of this symmetric bilinear form to the tangent space. Its single-valuedness is exactly tangential--tangential continuity, the defining continuity of the Regge space in \cite{li2018regge}. The local shape spaces and assembled continuity conditions therefore coincide.

\end{document}

%% file: preamble.tex
\usepackage{appendix}
\usepackage{geometry}
\usepackage{amssymb}
\usepackage{amsmath}
\usepackage{amsthm}
\usepackage{caption}
\usepackage{subcaption}
\usepackage{diagbox}
\usepackage{mathrsfs}

\usepackage{cite}

\usepackage{tikz}
\usepackage{tikz-cd}
\DeclareFontFamily{U}{rsfs}{\skewchar\font127 }
\DeclareFontShape{U}{rsfs}{m}{n}{%
   <-6> rsfs5
   <6-8> rsfs7
   <8-> rsfs10
}{}
\usepackage{bm}
\usepackage{hyperref}
\hypersetup{
	colorlinks=true,
	linkcolor=cyan!50!blue,
	filecolor=blue,      
	urlcolor=red,
	citecolor=blue,
}
\usepackage{cleveref}
\usepackage{stackengine}
\stackMath

\usepackage{setspace}
\usetikzlibrary{patterns, arrows.meta, calc}

\newcommand*{\be}[1]{\begin{equation}\label{#1}}
\newcommand*{\ee}{\end{equation}}

\newtheorem{theorem}{Theorem}[section]

\newtheorem{lemma}{Lemma}[section]
\newtheorem{proposition}{Proposition}[section]

\newtheorem{corollary}{Corollary}[section]
\theoremstyle{remark}
\newtheorem{remark}{Remark}[section]
\newtheorem{example}{Example}[section]

\definecolor{pink}{RGB}{255,45,115}

\DeclareMathOperator{\grad}{grad}
\DeclareMathOperator{\hess}{hess}
\DeclareMathOperator{\curl}{curl}
\DeclareMathOperator{\inc}{inc}

\DeclareMathOperator{\dev}{dev}
\DeclareMathOperator{\sym}{sym}
\DeclareMathOperator{\diverenge}{div}

\DeclareMathOperator{\tr}{tr}

\newcommand\vskw{\operatorname{vskw}}
\newcommand\mskw{\operatorname{mskw}}

\newcommand\skw{\operatorname{skw}}

\DeclareMathOperator{\Alt}{Alt}

\newcommand\K{\mathbb{K}}
\newcommand\MM{\mathbb{M}}

\newcommand\B{{\mathcal B}}

\newcommand\W{{\mathbb{W}}}

\renewcommand{\div}{\diverenge}

\usepackage{geometry}
\newcommand{\RR}{\mathbb{R}}
\newcommand{\VV}{\mathbb{V}}
\renewcommand{\SS}{\mathbb{S}}
\newcommand{\TT}{\mathbb{T}}

\renewcommand\ker{\mathcal{N}}
 \newcommand{\bs}{{\scriptscriptstyle \bullet}}
 \newcommand\ran{\mathcal{R}}
 \newcommand\alt{\mathrm{Alt}}

 \numberwithin{equation}{section}

\usepackage{lipsum}

\newcounter{quotecount}
\newcommand{\newquote}[1]{\vspace{0.5cm}\refstepcounter{quotecount}%
     \parbox{10cm}{\em #1}\hspace*{2cm}(\arabic{quotecount})\\[0.5cm]}

\newcommand{\xdashleftarrow}[2][]{\ext@arrow 3095\leftarrowfill@@{#1}{#2}}
\newcommand{\xdashleftrightarrow}[2][]{\ext@arrow 3359\leftrightarrowfill@@{#1}{#2}}
\def\rightarrowfill@@{\arrowfill@@\relax\relbar\rightarrow}
\def\leftarrowfill@@{\arrowfill@@\leftarrow\relbar\relax}
\def\leftrightarrowfill@@{\arrowfill@@\leftarrow\relbar\rightarrow}
\def\arrowfill@@#1#2#3#4{%
  $\m@th\thickmuskip0mu\medmuskip\thickmuskip\thinmuskip\thickmuskip
   \relax#4#1
   \xleaders\hbox{$#4#2$}\hfill
   #3$%
}